\documentclass[a4paper,12pt]{scrreprt}
\usepackage[left= 2.5cm,right = 2cm, bottom = 4 cm]{geometry}
\newtheorem{definition}{Definition}[chapter]
\newtheorem{proposition}[definition]{Proposition}
\newtheorem{theorem}[definition]{Theorem}

\newtheorem{corollary}[definition]{Corollary}
\newtheorem{lemma}[definition]{Lemma}
\newtheorem{example}[definition]{Example}

\newtheorem{remark}[definition]{Remark}

\usepackage[
	pdftitle={Port-Hamiltonian},	pdfsubject={},
	pdfauthor={Alexander Kilian},
	pdfkeywords={},	
	hidelinks
]{hyperref}

\usepackage[utf8]{inputenc}
\usepackage[english]{babel}
\usepackage[T1]{fontenc}
\usepackage{graphicx, subfig}

\graphicspath{{img/}}
\usepackage{fancyhdr}
\usepackage{lmodern}
\usepackage[table]{xcolor}
\usepackage{tikz}
\usepackage{amssymb}
\usepackage{multicol}
\usetikzlibrary{shapes,arrows} 

\usepackage{amsfonts}
\usepackage{amsmath}

\usepackage{enumitem} 

\title{Port-Hamiltonsche Systeme}
\author{Alexander Kilian}
\date{}
\makeatletter
\DeclareMathOperator{\rank}{rank}

\DeclareMathOperator{\trace}{tr}
\DeclareMathOperator{\gra}{graph}

\DeclareMathOperator{\ran}{ran}
\DeclareMathOperator{\realpart}{Re}
\DeclareMathOperator{\operatornorm}{op}
\DeclareMathOperator{\extern}{ext}

\DeclareRobustCommand{\qed}{%
  \ifmmode 
  \else \leavevmode\unskip\penalty9999 \hbox{}\nobreak\hfill
  \fi
  \quad\hbox{\qedsymbol}}
\newcommand{\openbox}{\leavevmode
  \hbox to.77778em{%
  \hfil\vrule
  \vbox to.675em{\hrule width.6em\vfil\hrule}%
  \vrule\hfil}}
\newcommand{\qedsymbol}{$\blacksquare$} 
\newenvironment{proof}[1][\proofname]{\par
  \normalfont
  \topsep6\p@\@plus6\p@ \trivlist
  \item[\hskip\labelsep\itshape
    #1.]\ignorespaces
}{%
  \qed\endtrivlist
}

\newcommand{\proofname}{Proof}

\newcommand*{\QEDA}{\hfill\ensuremath{\Diamond}} 
\newcommand*{\QEDB}{\hfill\ensuremath{\blacksquare}}
\makeatother

\begin{document}
	\title{
		\begin{center}
			\includegraphics[width=9cm]{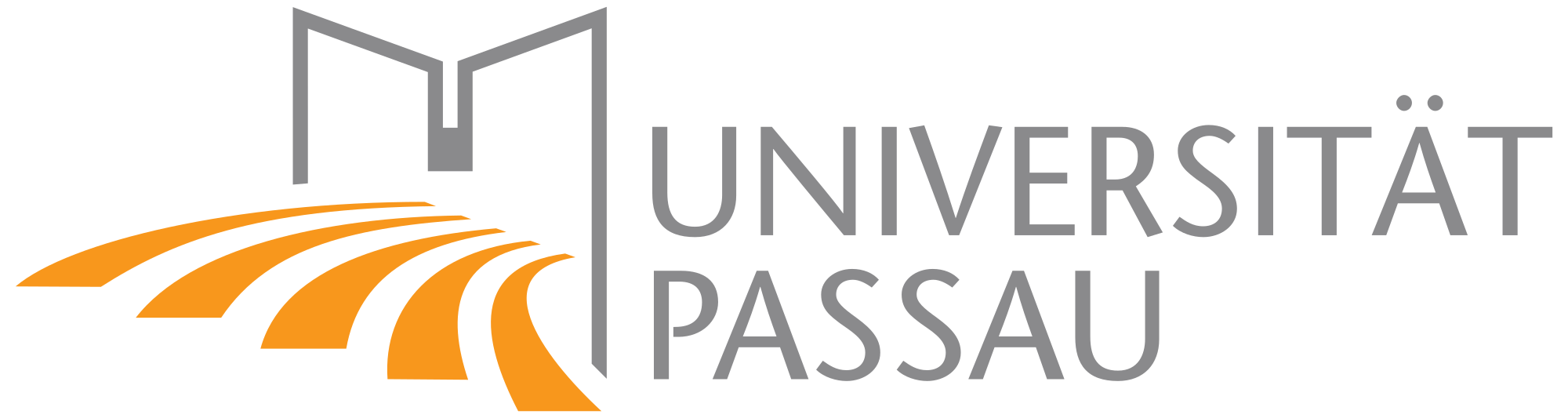} \\
			\vspace{1.5cm} {\large Master's Thesis in Computational Mathematics}
		\end{center}
		\rule{\textwidth}{.5mm}
		{\bfseries Infinite-Dimensional Port-Hamiltonian Systems with a Moving Interface}\\
		\rule{\textwidth}{.5mm}
	}
	
	\author{Alexander Kilian \\
		{\normalsize matriculation number: 77856}\\[1ex]
		{\normalsize supervised by:}\\
		\normalsize Prof. Dr. Fabian Wirth, Dr. Andrii Mironchenko \\
		\normalsize Chair of Dynamical Systems  \\
		\normalsize University of Passau \\[1ex]
	\normalsize Prof. Dr. Bernhard Maschke \\
 \normalsize Laboratoire d’Automatique et de Génie des Procédés \\
\normalsize Claude Bernard University Lyon 1}
	
	\date{December 20, 2021}
	
	\maketitle
	\pagenumbering{gobble}
	\clearpage
	\thispagestyle{empty}
	\setlength{\parindent}{0pt}
	
	\section*{Eidesstattliche Erklärung}
	
	Hiermit erkläre ich, \textbf{Alexander Kilian} (Matrikelnummer~77856), gegenüber der Fakultät für Informatik und Mathematik der Universität Passau, dass ich die Masterarbeit mit dem Thema \textit{Infinite-Dimensional Port-Hamiltonian Systems with a Moving Interface} selbstständig angefertigt und keine anderen als die angegebenen Quellen und Hilfsmittel benutzt habe. Alle wörtlich und sinngemäß übernommenen Ausführungen wurden als solche gekennzeichnet. Weiterhin erkläre ich, dass ich diese Arbeit in gleicher oder ähnlicher Form nicht bereits einer anderen Prüfungsbehörde vorgelegt habe.
	
\vspace{2.5 cm}
	
	Passau, den 20.12.2021 \\
\noindent \begin{flushleft}
	\begin{tabular}{l}
		\includegraphics[height=0.7cm]{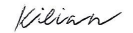} \tabularnewline
		Alexander Kilian 
	\end{tabular}
	\par\end{flushleft}
	\clearpage

\chapter*{Summary}	
This thesis deals with the formulation and analysis of two systems of conservation laws defined on two complementary intervals and coupled by some moving interface as a single infinite-dimensional port-Hamiltonian system. This approach may be regarded as an extension of the well-established mathematical framework of boundary port-Hamiltonian systems, i.e., infinite-dimensional Hamiltonian systems interacting with its physical environment. 
\vspace{0.5 cm}\\
A vast variety of multi-physics systems can be modeled as a Hamiltonian system, a special type of ordinary and partial differential equations. Such systems have a Hamiltonian, which describes the internally stored energy of the system. By allowing such systems to be open, one usually refers to this class of systems as (boundary) port-Hamiltonian systems. For the analysis of infinite-dimensional port-Hamiltonian systems, one combines the energy-based approach by means of Hamiltonians with the poweful framework of strongly continuous semigroups applied in infinite-dimensional systems theory. This approach has led to fundamental well-posedness and stability results in the field of infinite-dimensional port-Hamiltonian systems. Furthermore, the power flow at the spatial boundary has been utilized to control port-Hamiltonian systems.
\vspace{0.5 cm}\\
The great success of this system description motivates to extend the port-Hamiltonian formulation to systems of conservation laws that are coupled by some moving interface. In order to keep track of the power flow at the position of the interface, one has to introduce interface port variables. In this thesis, we provide a well-posed port-Hamiltonian model of two distinct systems of conservation laws defined on two complementary, 1-dimensional spatial domains that are coupled by a fixed interface, and illustrate this result by reference to two transmission lines coupled through a resistor. Furthermore, we discuss under which boundary and interface conditions this system is exponentially stable. Since the formulation in case of a moving interface is quite troubled, we present some results that may be regarded as a first step towards a well-posed model, and debate the main issues that have to be tackled in the future.
\clearpage

\chapter*{Zusammenfassung}
Ziel der Masterarbeit ist es, zwei auf komplementären Intervallen definierte Systeme von Energieerhaltungssätzen, die durch eine mobile Schnittstelle gekoppelt werden, mathematisch als ein unendlich-dimensionales port-Hamiltonsches System zu modellieren und zu analysieren. Dieses Modell kann als eine Verallgemeinerung der port-Hamiltonschen Formulierung linearer unendlich-dimensionaler Systeme, die mit ihrer Umgebung interagieren, verstanden werden. 
\vspace{0.5 cm}\\
Zahlreiche physikalische Prozesse können mathematisch mit Hilfe sogenannter Hamiltonscher Systeme modelliert werden. Diese besitzen eine Hamilton-Funktion, die in vielen Anwendungen die Systemenergie beschreibt. Erlaubt man einen Energieaustausch mit der Umgebung des Systems, so spricht man von port-Hamiltonschen Systemen, die sowohl durch gewöhnliche als auch partielle Differentialgleichungen beschrieben werden können. Bei der Untersuchung unendlich-dimensionaler port-Hamiltonscher Systeme verbindet man den abstrakten, funktionalanalytischen Ansatz mittels der Halbgruppentheorie mit dem auf Hamilton-Funktionen basierten, physikalisch motivierten Ansatz. Mit dieser Methodik ließen sich elementare Wohlgestelltheits- und Stabilitätsresultate auf dem Gebiet unendlich-dimensionaler port-Hamiltonscher Systeme beweisen. Darüber hinaus eignet sich das port-Hamiltonsche Modell aufgrund des Energieaustauschs mit der Umgebung bestens zur energiebasierten Regelung solcher Systeme.
\vspace{0.5 cm}\\
Aufgrund des beachtlichen Erfolgs dieser Systembeschreibung soll in dieser Arbeit der Versuch einer port-Hamiltonschen Formulierung zweier Systeme von Energieerhaltungssätzen, die durch eine mobile Schnittstelle gekoppelt sind, unternommen werden. Dazu werden Schnittstellen-Port-Variablen eingeführt, um den Energiefluss an der sich bewegenden Schnitt- stelle zu modellieren. Wir werden ein wohlgestelltes System des eben beschriebenen Szenarios im Falle einer fixen Schnittstellenposition präsentieren und dieses Resultat anhand zweier Übertragungsleitungen mit unterschiedlichen physikalischen Eigenschaften, die durch einen elektrischen Widerstand gekoppelt werden, illustrieren. Ferner werden wir charakterisieren, unter welchen Rand- und Schnittstellenbedingungen ein solches System exponentiell stabil ist. Die port-Hamiltonsche Formulierung im Falle einer mobilen Schnittstellen ist problembeladen. Wir werden daher Probleme und offene Fragen, die es in Zukunft noch zu klären gilt, ausführlich diskutieren und für den Wohlgestelltheits-Nachweis essentielle Eigenschaften des modellierten Systems präsentieren. 
\clearpage

	\pagenumbering{roman}	
	\tableofcontents

\listoffigures
\listoftables

	\section*{Notation}
	\def\arraystretch{1.5}
	\begin{tabular}{l l}
		\textbf{Symbol} & \textbf{Description} \\
		$\mathbb{N}$, $\mathbb{R}$, $\mathbb{C}$ & the set of natural, real, and complex numbers, respectively  \\
		$\mathbb{R}_{>0}$ & the set of positive real numbers \\
		$\mathbb{R}_+$ & the set of non-negative real numbers  \\
		$\mathbb{F}^{m \times n}$ & the vector space of $m \times n$-matrices  with entries in $\mathbb{F}$ \\
		$\mathbb{F}^n$ & the vector space of $n$-dimensional column vectors with entries in $\mathbb{F}$ \\
		$\realpart(z)$ & real part of $z \in \mathbb{C}$ \\
		$\| \cdot \|_2$ & Euclidean norm on $\mathbb{R}^n$ \\
		$\| \cdot \|_{\operatornorm \colon D(A) \to X}$ & operator norm of an operator $A \colon D(A) \subset X \to X$ \\
		$\| \cdot \|_{\mathcal{L}(X)}$ & operator norm on $\mathcal{L}(X)$ \\
		$X'$ & dual space of a vector space $X$  \\
		$\langle \cdot, \cdot\rangle_{X' \times X}$ & natural pairing between a vector space $X$ and its dual space $X'$ \\
		$\langle \cdot \mid \cdot \rangle$ & power pairing defined on a Hilbert space  \\
		$\ll \cdot , \cdot \gg$ & plus pairing associated with a power pairing  \\
		$M^{\top}$, $v^{\top}$ & transpose of a matrix $M \in \mathbb{F}^{m \times n}$/a vector $v \in \mathbb{F}^n$  \\
		$M > N$ $(M \geq N)$ & positive (semi-)definiteness of $M - N \in \mathbb{F}^{n \times n}$ \\
		$A^{\ast}$ & (formal) Hilbert space adjoint of an operator $A$  \\
		$S^{\bot}$ & orthogonal complement of the space $S$  \\
		$S^{\perp \!\!\! \perp}$ & orthogonal complement of $S$ with respect to the plus pairing   \\
		$L^1([a,b], \mathbb{F}^n)$ & vector space of $\mathbb{F}^n$-valued integrable functions on $[a,b]$ \\
		$L^2([a,b], \mathbb{F}^n)$ & vector space of $\mathbb{F}^n$-valued square integrable functions on $[a,b]$   \\
		$\| \cdot \|_{L^2} = \| \cdot \|_{L^2(a,b)}$ & standard $L^2$-norm on $L^2([a,b], \mathbb{F}^n)$ \\
		$L_{loc}^1([0, \infty), X)$ & vector space of $X$-valued locally integrable functions on $[0, \infty)$  \\
		$H^N([a,b], \mathbb{F}^n)$ & Sobolev space of order $N \in \mathbb{N}$ on $[a,b]$ \\
		$H_{loc}^1([0, \infty), X)$ & vector space of $X$-valued locally  $H^1$-functions on $[0, \infty)$  \\
			$\mathcal{C}([a,b], X)$ & vector space of continuous functions $f \colon [a,b] \to X$ with finite norm \\
		& $\|f\|_{\infty} := \sup\limits_{x \in [a,b]} \|f(x)\|_X$ \\
		\end{tabular}
\newpage
	\def\arraystretch{1.5}
\begin{tabular}{l l}
	$\mathcal{C}^k([a,b], X)$ & vector space of $k$-times continuously differentiable functions $f \colon [a,b] \to X$ \\
	$\mathcal{C}^{\infty}(\mathbb{R}^n, \mathbb{R})$ & vector space of infinitely differentiable functions \\
	$\mathcal{D}(\Omega)$ & vector space of infinitely differentiable functions with compact \\
	& support on $\Omega$ (the set of test functions on $\Omega$)  \\
	$\mathcal{D}'(\Omega)$ & vector space of distributions on $\Omega$ \\
	$\langle u, \varphi \rangle_{\mathcal{D}'(\Omega)}$ & action of $u \in \mathcal{D}'(\Omega)$ on a test function $\varphi \in \mathcal{D}(\Omega)$ \\
	$\mathcal{L}(X,Y)$ & space of linear bounded operators from $X$ to $Y$  \\
	$\mathcal{L}(X)$ & space of linear bounded operators on $X$ \\
	$I_X$ ($I$) & identity of the vector space $X$ \\
	$A_{|\Omega}$ & restriction of $A$ to the set $\Omega \subset D(A)$ \\
	$\overline{D(A)}$ & closure of $D(A) \subset X$ with respect to $\|\cdot\|_X$ \\
	$\frac{dx}{dz}$ & derivative of $x$ with respect to $z$ \\
	$\dot{x} = \frac{dx}{dt}$ & derivative of $x$ with respect to time $t$ \\
	$\ddot{x} = \frac{d^2 x}{dt^2}$ & second derivative of $x$ with respect to time $t$  \\
	$\frac{\partial w}{\partial x_i} = \partial_{x_i} w $ & partial derivative of $w(x_1,\ldots,x_n)$ with respect to $x_i$, $i = 1,\ldots,n$ \\
	$\nabla H(x)$ & gradient column vector of $H \colon \mathbb{R}^n \to \mathbb{R}$ \\
	$\frac{\delta}{\delta x}H = \delta_x H$ & variational derivative of a functional $H \colon L^2([a,b],\mathbb{R}^n) \to \mathbb{R}$ \\
		$T_xX$ & tangent space of $X$ at $x \in X$ \\
		$T_x'X$ & cotangent space of $X$ at $x \in X$ \\
		$A^J$ & set of mappings $J \to A$, with $J$ an index set \\
		$\rho(A)$, $\sigma(A)$ & resolvent set and spectrum of an operator $A \colon D(A) \subset X \to X$ \\
		$f(x_0^+) = \lim\limits_{x \searrow x_0}f(x)$ & right-hand limit of $f$ at $x_0$ \\
		$f(x_0^-) = \lim\limits_{x \nearrow x_0}f(x)$ & left-hand limit of $f$ at $x_0$ \\
		$\mathcal{O}(\epsilon)$ & big O notation with respect to $\epsilon$ \\
		$A \hookrightarrow B$ & embedding of $A \subset B$ into $B$
		
	\end{tabular}
	\clearpage
		\pagenumbering{arabic}
\chapter{Introduction}
\section{Background}

The port-Hamiltonian systems theory \cite{ModCompPhys, JacobZwart, SchaftJeltsema} is a systematic framework for modeling systems comprising various physical domains, such as mechanical, electric, hydraulic, and thermal domains. This approach is realized by classifying every (ideal) system component subject to its main physical characteristics, and by taking energy to be medium of communication between the physical domains in multi-physics systems. The eponymous concepts used for this modeling framework are the system's Hamiltonian (that is, its total energy) stemming from mechanics \cite{Arnold}, and the port-based modeling approach \cite{Paynter}, which is based on the power-conserving interconnection of various system components by means of ports. They consist of physical quantities whose product yields power (e.g., current and voltage in an electrical network, or force and velocity in a mechanical system). That way, the power flow between the system's components can be expressed. The resulting network structure can be geometrically described by a power-conserving interconnection structure, called the Dirac structure, which has proved to be highly valuable in the field of port-based modeling and analysis \cite{Courant, SchaftMaschke95}. 
\vspace{0.5 cm}\\
There are three main features of port-Hamiltonian systems worth highlighting. Firstly, the (power-conserving) interconnection of port-Hamiltonian systems is again a port-Hamiltonian system. This is a crucial property and emphasizes the network structure of this framework. Secondly, this framework allows the inclusion of energy-dissipating elements, which is why port-Hamiltonian models are broadly applicable. Lastly, the port-Hamiltonian framework is greatly suitable for control purposes. It allows to model the interaction of the system with its environment by means of external ports which are, in particular, accessible for controller action.
\vspace{0.5 cm}\\
Based on the theory of infinite-dimensional Hamiltonian systems in \cite{Olver}, the port-Hamiltonian formulation has been extended to infinite-dimensional systems \cite{SchaftMaschke02}. The interconnection structure is generalized to a Stokes-Dirac structure, where the power flow through the boundary of the spatial domain, that is, the interaction with the system's environment, is encompassed. Just as in the finite-dimensional case, this is essential for the interconnection of several Hamiltonian systems, and it is a good starting point for the purpose of control, hence the name boundary port-Hamiltonian systems. For the analysis of infinite-dimensional port-Hamiltonian systems one usually considers abstract Cauchy problems of the form
\begin{align*}
	\dot{x}(t) &= A_{\mathcal{Q}}x(t), \hspace{0.5 cm} t > 0, \\
	x(0) &= x_0 \in X,
\end{align*}
with $A_{\mathcal{Q}} \colon D(A_{\mathcal{Q}}) \subset X \to X$ a linear (usually unbounded) operator on some Hilbert space $X$ that is associated with the port-Hamiltonian system. This port-Hamiltonian operator stores information concerning the interdomain coupling of the modeled multi-physics system as well as the mechanical and physical properties of the system. More precisely, these operators are usually given by $A_{\mathcal{Q}} = \mathcal{J}(\mathcal{Q}\cdot)$, with $\mathcal{J}$ a formally skew-symmetric matrix differential operator modeling the interdomain coupling, and $\mathcal{Q}$ a coercive and bounded multiplication operator representing the physical properties. To prove well-posedness and stability results, one applies the well-documented theory of strongly continuous semigroups \cite{CurtainZwart, EngelNagel, Pazy83}. The most important class of semigroups in the framework of infinite-dimensional port-Hamiltonian systems are contraction semigroups, which are closely related to energy dissipation. Typically, one incorporates dissipation phenomena into the domain of the operator $A_{\mathcal{Q}}$ by means of the port variables, and proves that this operator is the infinitesimal generator of a contraction semigroup. For control purposes, it seems natural to let the controls and observations act at the boundary of the spatial domain, leading to the notion of (port-Hamiltonian) boundary control systems \cite{Fattorini}, \cite[Chapter 10]{Tucsnak-Weiss}. These systems have been extensively studied in the PhD theses \cite{Augner16, Villegas}.
\vspace{0.5 cm}\\ 
Due to the great success of this system description, quite recently it was investigated whether the port-Hamiltonian formulation of infinite-dimensional systems can be extended to first-order systems that are coupled by some moving interface \cite{Diagne}. There are several instances where a system's spatial domain might be subdivided into areas where the intrinsic physical properties or even the aggregate phase differ. For example, one may consider two distinct gases separated by a moving piston, or liquid-solid phase transitions. The mathematical model of the latter is known as the Stefan problem \cite{Gupta}. A port-Hamiltonian formulation of the two-phase Stefan problem has already been tested in \cite{Vincent}. In contrast to boundary port-Hamiltonian systems, this formulation yields an abstract port-Hamiltonian control system with the input being the velocity of the interface, and offers a new field of research in terms of interface control. 
\vspace{0.5 cm}\\
In this thesis, we review the port-Hamiltonian formulation of two systems of conservation laws defined on two complementary intervals that are coupled by a moving interface, as elaborated in \cite{Diagne}, and further examine this class of systems. Our final objective is to collect some useful results concerning the underlying port-Hamiltonian operators by means of the theory of strongly continuous semigroups that may help to prove well-posedness and stability results with respect to the position of the moving interface. The main contributions made in this thesis are the following:
\begin{enumerate}[label = $\bullet$]
	\item We specify the Stokes-Dirac structure of the model proposed in \cite{Diagne} with respect to a formally skew-symmetric matrix differential operator in case of a fixed interface, see Proposition \ref{Interface Paper - Proposition Dirac Structure DI}.
	\item If we allow for dissipation at the boundary and at the interface position, then the port-Hamiltonian operator associated with two systems of conservation laws that are coupled by a fixed interface is the infinitesimal generator of a contraction semigroup on the respective energy space. In particular, the underlying model is well-defined, see Theorem \ref{Interface - Theorem A Generates a Contraction Semigroup}.
	\item Under some more restrictive assumption concerning the boundary condition, the generated semigroup is exponentially stable, see Theorem \ref{Interface - Theorem 2 Exponential Stability}.
	\item We consider a simplified model which is closely related to the one formulated in \cite{Diagne}, and present criteria for the stability of the associated family of port-Hamiltonian operators that encompasses the position of the moving interface, see Subsection \ref{Subsection Stability of the Family of Infinitesimal Generators}. 
\end{enumerate}

\section{Outline of this Thesis}
 The remaining chapters are organized as follows: 
\vspace{0.5 cm}\\
\textbf{Chapter \ref{Chapter Finite-dimensional Port-Hamiltonian Systems}} aims to present the basic concepts and notions of port-based modeling. We deal with the port-Hamiltonian formulation of finite-dimensional multi-physics systems, where we restrict ourselves to linear systems. To that end, we introduce the fundamental concept of a power-conserving interconnection structure, called a Dirac structure, and briefly discuss its constituting port variables. This allows us to define the port-Hamiltonian dynamics by means of the underlying Dirac structure. 
\vspace{0.5 cm}\\
\textbf{Chapter \ref{Chapter Some Background}} covers some well-known notions and results from the literature that are used in the upcoming chapters. We will mainly recall the basics of infinite-dimensional systems theory and, with that said, we will recall the concept of strongly continuous semigroups and some important results that characterize their infinitesimal generators. Furthermore, we will discuss formally adjoint operators of matrix differential operators.
\vspace{0.5 cm}\\
In \textbf{Chapter \ref{Chapter Infinite-dimensional Port-Hamiltonian Systems}} we generalize the port-Hamiltonian approach of Chapter \ref{Chapter Finite-dimensional Port-Hamiltonian Systems} to infinite-dimen- sional systems, where we restrict ourselves to idealized systems with a 1-dimensional spatial domain. We discuss the main differences to finite-dimensional port-Hamiltonian systems, and show how the power flow at the boundary of the spatial domain is incorporated into the definition of the Dirac structure, leading to the notion of a boundary port-Hamiltonian system. This chapter sets a basis for the discussion in the following chapter.
\vspace{0.5 cm}\\
\textbf{Chapter \ref{Chapter Boundary Port-Hamiltonian Systems with a Moving Interface}} is the centerpiece of this thesis and studies the port-Hamiltonian formulation of a system of two conservation laws with a moving interface. The main goal is to provide a well-posed model for a system of two conservation laws on two complementary intervals separated by an interface, where we take the interface's velocity to be the input variable of the system. It turns out that this formulation poses delicate problems. Therefore, we vastly simplify the original problem, and work out some important properties of the simplified system that may help to further investigate the initial one.
\vspace{0.5 cm}\\
In \textbf{Chapter \ref{Chapter Conclusion}}  we summarize the covered topics and reflect the main contributions made in this work. Furthermore, we draw our conclusions from the discussion in Chapter \ref{Chapter Boundary Port-Hamiltonian Systems with a Moving Interface} and give some comments on open problems and potential future work concerning this rather new class of systems.

\chapter{Finite-Dimensional Port-Hamiltonian Systems}
\label{Chapter Finite-dimensional Port-Hamiltonian Systems}

In this chapter, we want to study basic concepts of port-based network modeling and provide the definition of finite-dimensional port-Hamiltonian systems and their underlying Dirac structures. Section \ref{Section A Motivating Example} is devoted to the description of the main characteristics of port-based network modeling, which originates in the work of Paynter \cite{Paynter}. To this end, we consider the model of a simple mass-spring system. This motivates to introduce the port-Hamiltonian formulation of multi-physics systems. The dynamics of port-Hamiltonian systems are geometrically specified by so-called Dirac structures. In Section \ref{Section Dirac Structures} we will introduce this fundamental concept and give some examples of Dirac structures. In the Sections \ref{Section Energy-storing Elements}, \ref{Section Energy-dissipating Elements}, and \ref{Section External Ports} we briefly discuss energy-storing, energy-dissipating, and external ports, respectively. In the finite-dimensional case, these are the constitutive ports of a Dirac structure.
After that, we are finally able to geometrically define the dynamics of port-Hamiltonian systems, which we are going to do in Section \ref{Section Port-Hamiltonian Dynmacis}. As mentioned in the beginning, the port-Hamiltonian framework is particularly convenient for control purposes. Therefore, we will provide a state-space representation of a huge class of port-Hamiltonian systems that is suitable to be taken as the starting point for investigating control theoretical aspects. We will mainly focus on linear port-Hamiltonian systems. 
\vspace{0.5 cm}\\
The whole chapter relies heavily on the Chapters 2, 3, and 4 in \cite{SchaftJeltsema}, and on \cite[Section 2.2]{ModCompPhys}.

\section{A Motivating Example}
\label{Section A Motivating Example}
In order to motivate the framework of port-based modeling, we start with a simple example.
\begin{example}[Mass-Spring System; Example 2.1 in \cite{SchaftJeltsema}]
	\label{Example Mass-Spring System}
 Consider a point mass with mass $m \in \mathbb{R}_{> 0}$, moving in one direction and without friction under the influence of a
	spring force corresponding to a linear spring with spring constant $k \in \mathbb{R}_{>0}$. The usual way of modeling this mass-spring system is by means of the second-order differential equation
	\begin{align}
		m \ddot{z}(t) = - k(z(t) - z_0),
		\label{Example Mass-Spring System - ODE}
	\end{align}
where $z(t) \in \mathbb{R}$ denotes the position of the mass at time instant $t \geq 0$, and $z_0 \in \mathbb{R}_{>0}$ is the length of the spring at rest. In order to present the basic principles of port-based network modeling, we consider this problem from a different perspective. The mass-spring system is now regarded as the interconnection of two subsystems storing energy, namely the spring system storing potential energy and the mass system storing kinetic energy, see Figure~\ref{Figure Mass-Spring-System}. \\
\begin{figure}[h]
	\centering
	\includegraphics[width = 8cm]{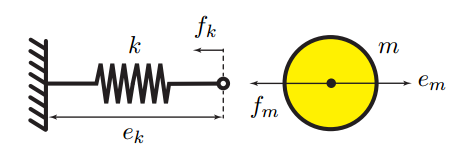}
	\caption{Mass-Spring System Regarded as the Interconnection of two Subsystems \cite[Figure~2.1]{SchaftJeltsema}}
	\label{Figure Mass-Spring-System}
	\end{figure}

For the spring system the potential energy is expressed in terms of the elongation $q = z - z_0 $ of the spring. Assuming a linear
spring, according to Hooke's law the potential energy is given as $E_{\text{pot}}(q) = \frac{1}{2} kq^2$. This leads to the spring system equations
\begin{align*}
	\text{spring system} \colon \begin{cases}
		\dot{q} = -f_k, \\
		e_k = \frac{d}{dq} E_{\text{pot}}(q),
	\end{cases}
\end{align*}
where $-f_k$ denotes the velocity of the endpoint of the spring where
it is attached to the mass, and $e_k = kq$ describes the spring force at this endpoint. \\
For the mass system we obtain similar equations. Denoting  the momentum of the mass by $p \in \mathbb{R}$, the kinetic
energy is given as $E_{\text{kin}}(p) = \frac{1}{2m} p^2$, which leads to the following mass system equations:
\begin{align*}
	\text{mass system} \colon \begin{cases}
		\dot{p} = -f_m, \\
		e_m = \frac{d}{dp} E_{\text{kin}}(p).
	\end{cases}
\end{align*}
Here, $-f_m$ denotes the force exerted on the mass, and $e_m = \frac{p}{m}$ is the velocity of the mass. \\
These separately regarded subsystems are now interconnected through Newton's third law in the following way:
\begin{align*}
		-f_k &= e_m, \\
		f_m &= e_k.
\end{align*}
The system's stored energy is given by the Hamiltonian $H \colon \mathbb{R}^2 \to \mathbb{R}$ defined by
\begin{align}
	H(q,p) = E_{\text{pot}}(q) + E_{\text{kin}}(p) = \frac{1}{2}kq^2 + \frac{1}{2m} p^2.
	\label{Example Mass-Spring System - Hamiltonian}
\end{align}
The interconnected system can be written in terms of the Hamiltonian as follows:
\begin{align}
	\begin{bmatrix}
			\dot{q} \\
			\dot{p}
	\end{bmatrix} = \begin{bmatrix}
	0 & 1 \\
	-1 & 0 
\end{bmatrix} \begin{bmatrix}
\frac{\partial H}{\partial q} (q,p) \\
\frac{\partial H}{\partial p}(q,p) 
\end{bmatrix}.
\label{Example Mass-Spring System - pH System}
\end{align}
These are the so-called Hamiltonian equations for the mass-spring system. Note that \eqref{Example Mass-Spring System - pH System} is equivalent to the second-order model \eqref{Example Mass-Spring System - ODE} by the correspondence $q = z - z_0$: we have
\begin{align*}
	m \ddot{z} = m \ddot{q} = m \dot{\left(\frac{\partial H}{\partial p} (q,p)\right)} = m \dot{\frac{p}{m}} = - \frac{\partial H}{\partial q}(q,p) = - kq = -k(z - z_0). 
\end{align*}
\QEDA
\end{example}

\begin{remark}
	The reason for the minus sign in front of $f_k$ in Example \ref{Example Mass-Spring System} is that we want the product $f_ke_k$ between the spring's velocity and force to be the incoming power with respect to the interconnection. The same applies for the product $f_m e_m$ between the exerted force on the mass and its velocity. This sign convention will be adopted throughout this chapter.
\end{remark}

Despite its simplicity, Example \ref{Example Mass-Spring System} already displays some of the
main characteristics of port-based network modeling: we regard the system as the coupling of the energy-storing spring system with the energy-storing mass system through a power-conserving interconnection where the power is routed from the mass system
to the spring system and vice versa. This approach can be applied to general multi-physics systems, i.e., systems comprising different physical domains.
\vspace{0.5 cm}\\
More generally, in port-based modeling a physical system is considered as the interconnection of three different types of ideal components:
\begin{enumerate}
	\item[(i)] \textit{Energy-storing elements:} typical examples of energy-storing elements are ideal inductors (by means of the magnetic field), capacitors (by means of the electric field), masses, and springs.
	\item[(ii)] \textit{Energy-dissipating} (or \textit{resistive}) elements: examples of elements corresponding to energy-dissipation (due to friction, resistance, etc.) are resistors and dampers.
	\item[(iii)]  \textit{Energy-routing elements:} elements such as gyrators and transformers only redirect the power flow in the overall system \cite[Part XV]{Paynter}.
\end{enumerate}

\begin{figure}[h]
	\centering
	\includegraphics[width = 10cm]{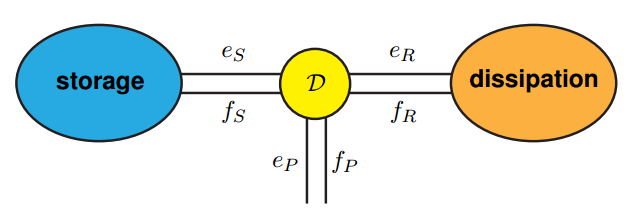}
	\caption{Port-Hamiltonian System \cite[Figure 2.2]{SchaftJeltsema}}
	\label{Figure Port-Hamiltonian System}
	\end{figure}
For the port-Hamiltonian formulation the energy-storing elements will be grouped into a single object denoted by $\mathcal{S}$ ('storage'), and correspondingly the energy-dissipating elements
are grouped into a single object denoted by $\mathcal{R}$ ('resistive'). Moreover, the
interconnection of all energy-routing elements can be considered as one energy-routing structure denoted by $\mathcal{D}$, which we will call the system's underlying \emph{Dirac structure}.
\vspace{0.5 cm}\\
The principle of port-Hamiltonian systems modeling is schematically illustrated in Figure~\ref{Figure Port-Hamiltonian System}: the energy-storing elements $\mathcal{S}$ and the energy-dissipating elements $\mathcal{R}$ are linked to a central interconnection (energy-routing) structure $\mathcal{D}$. This linking is realized via pairs $(f,e)$ of vectors of flow and effort variables of equal dimension. We call a pair $(f,e)$ of vectors of flow and effort variables a port, and the total set of variables $f,e$ is called the set of port variables. The physical meaning of efforts and flows in various physical domains is listed in Table \ref{Table Flows And Efforts}.
\begin{table}
\centering
\small
{\rowcolors{2}{white!100}{black!10!white!90} 
	\begin{tabular}{c c c}
		\hline
		Physical Domain & Flow $f \in \mathcal{F}$  & Effort $e \in \mathcal{E}$ \\ [0.5ex]
		\hline \hline
		electric & current $I$ & voltage $V$ \\
		magnetic & voltage $V$ & current $I$ \\
		potential translation & velocity $v$ & force $F$ \\
		kinetic translation & force $F$ & velocity $v$ \\
		potential rotation & angular velocity $\omega$ & torque $T$ \\
		kinetic rotation & torque $T$ & angular velocity $\omega$ \\
		potential hydraulic & volume flow $\phi$ & pressure $p$ \\
		kinetic hydraulic & pressure $p$ & volume flow $\phi$ \\
		chemical &  chemical potential $\mu$ & molar flow $f_N$ \\
		thermal  & temperature $T$ & entropy flow $f_S$ \\ [0.5ex]
		\hline
	\end{tabular}
}
\caption{Flows and Efforts Corresponding to their Physical Domains \cite[Table 1.2]{ModCompPhys}}
\label{Table Flows And Efforts}
\end{table}
Figure \ref{Figure Port-Hamiltonian System} shows three ports: the port $(f_S, e_S)$ linking to energy-storage, the port $(f_R,e_R)$ corresponding to energy-dissipation, and the
external port $(f_P,e_P)$, by which the system interacts with its environment.
\vspace{0.5 cm}\\
In the following sections, we will provide more details about the fundamental notion of a Dirac structure on finite-dimensional spaces. To this end, we discuss the constituting ports of Dirac structures in the finite-dimensional case. We will treat successively energy-storing elements, energy-dissipating (resistive) elements, and
external ports. After a short comment on state choice of different physical domains, we finally get to the basic geometric definition of a port-Hamiltonian system. Subsequently, we will present some simple yet important examples of such systems.

\section{Dirac Structures}
\label{Section Dirac Structures}
The key property of a Dirac structure is power conservation, meaning that it links the port variables in such a way that the total power entering or leaving a Dirac structure always amounts to zero. For the definition of a Dirac structure, we need to define the spaces of flows and efforts as well as the power flow defined on these spaces. We denote by $\mathcal{F}$ some finite-dimensional linear \emph{space of flows}. Its elements will be denoted by $f \in \mathcal{F}$, and are called \emph{flow} (or \emph{energy}) \emph{variables}. The \emph{space of effort} (or \emph{co-energy}) \emph{variables} is given by the dual linear space $\mathcal{E} = \mathcal{F}'$, and its elements are denoted by $e \in \mathcal{E}$. The total space of flow and effort variables is $\mathcal{B} = \mathcal{F} \times \mathcal{E}$, and is called the \emph{space of port variables} or the \emph{bond space}. The power on the total space of port variables is defined by 
\begin{align}
	P \colon \mathcal{B} \to \mathbb{R}, \hspace{0.5 cm} (f,e) \mapsto \langle e \mid f \rangle,
	\label{Function - Power}
\end{align}
where $\langle \cdot \mid \cdot \rangle \colon \mathcal{E} \times \mathcal{F} \to \mathbb{R}$ is the so-called \emph{power pairing}. In the finite-dimensional case the power pairing $\langle \cdot \mid \cdot \rangle$ equals the natural pairing $\langle \cdot , \cdot \rangle_{\mathcal{F}' \times \mathcal{F}}$ on the respective vector space, see Section \ref{Section Operators}. Usually, one may take $\mathcal{F} = \mathbb{R}^k$ and thus $\mathcal{E} = \mathbb{R}^k$, and the power pairing equals the standard inner product:
\begin{align*}
	\langle e \mid f \rangle =  \langle e, f \rangle_2, \hspace{0.5 cm} (f,e) \in \mathcal{B}.
\end{align*}
With the data provided in Table~\ref{Table Flows And Efforts} one may check that the (inner) product of flow and effort variables indeed equals power. However, we note that there are interesting cases where it is beneficial not to take $\mathcal{F} = \mathbb{R}^k$. For instance, in three-dimensional mechanical systems one usually works with Lie algebras of different matrix groups (see \cite[Section 1.4]{Olver}). For further details we refer to \cite[Chapter 3]{SchaftJeltsema}. We will see in the Chapters \ref{Chapter Infinite-dimensional Port-Hamiltonian Systems} and \ref{Chapter Boundary Port-Hamiltonian Systems with a Moving Interface} that in the infinite-dimensional case, the power pairing does, in general, not coincide with the natural pairing.
\vspace{0.5 cm}\\
Now we are able to define the notion of a Dirac structure in the finite-dimensional case. 
\begin{definition}[Dirac Structure]
	\label{Definition Dirac Structure - Finite-dimensional Case}
	Consider a finite-dimensional linear space of flows $\mathcal{F}$ and its dual space of efforts $\mathcal{E} = \mathcal{F}'$. Let $P$ be the power mapping on the bond space $\mathcal{B} = \mathcal{F} \times \mathcal{E}$ given by \eqref{Function - Power}. A linear subspace $\mathcal{D} \subset \mathcal
	{B}$ is called a (constant) Dirac structure if 
	\begin{enumerate}[label = (\roman*)]
		\item \label{DiracStructureProp1} $P_{|\mathcal{D}} \equiv 0$, i.e., $P(f,e) = 0$ for all $(f,e) \in \mathcal{D}$, and
		\item \label{DiracStructureProp2} $\dim (\mathcal{D}) = \dim (\mathcal{F})$.  
	\end{enumerate}
\end{definition}
Property \ref{DiracStructureProp1} in Definition \ref{Definition Dirac Structure - Finite-dimensional Case} corresponds to power conservation, and states the previously mentioned
fact that the total power flow always equals zero. It can be shown that the maximal dimension of any subspace $\mathcal{D} \subset \mathcal{B}$ satisfying property \ref{DiracStructureProp1} is equal to $\dim (\mathcal{F})$. This claim follows immediately from an equivalent definition of a Dirac structure, which we will state next. Another advantage of this equivalent definition of a Dirac structure is that it can be generalized to the case of an infinite-dimensional vector space $\mathcal{F}$, leading to the definition of an infinite-dimensional Dirac structure, as we will see in Chapter~\ref{Chapter Infinite-dimensional Port-Hamiltonian Systems}.
\vspace{0.5 cm}\\
In order to give this equivalent characterization of a Dirac structure, we define a canonical bilinear form $\ll \cdot, \cdot \gg$ on the space $\mathcal{B}$ which is closely related to the power pairing:
\begin{align}
	\begin{split}
	\ll \cdot, \cdot \gg \colon \mathcal{B} \times \mathcal{B} &\to \mathbb{R}, \\
	((f^1,e^1),(f^2,e^2)) &\mapsto \hspace{0.1 cm}\ll (f^1,e^1) , (f^2,e^2) \gg \hspace{0.1 cm} := \langle e^1 \mid f^2 \rangle + \langle e^2 \mid f^1 \rangle.
	\end{split}
	\label{Plus Pairing}
\end{align}
This bilinear form is \emph{non-degenerate}, that is, if $\ll (f^1,e^1) , (f^2,e^2) \gg = 0$ for all $(f^2,e^2) \in \mathcal{B}$, then necessarily $(f^1,e^1) = 0$. However, it is not positive definite. We call this bilinear form a \emph{plus pairing}.
\vspace{0.5 cm}\\
We will now show that the following equivalent definition of a Dirac structure indeed satisfies the properties stated in Definition \ref{Definition Dirac Structure - Finite-dimensional Case}.
\begin{proposition}[Proposition 2.1 in \cite{SchaftJeltsema}]
	\label{Proposition Characterization Dirac Structure}
		Consider a finite-dimensional linear space of flows $\mathcal{F}$ and its dual space of efforts $\mathcal{E} = \mathcal{F}'$. Let $P$ be the power mapping on the bond space $\mathcal{B} = \mathcal{F} \times \mathcal{E}$ given by \eqref{Function - Power}. A linear subspace $\mathcal{D} \subset \mathcal{B}$ is a Dirac structure if and only if 
		\begin{align}
			\mathcal{D} = \mathcal{D}^{\perp \!\!\! \perp},
			\label{Proposition Characterization Dirac Structure - Equivalent Property}
		\end{align}  where $\mathcal{D}^{\perp \!\!\! \perp}$ denotes the orthogonal complement of $ \mathcal{D}$ with respect to the plus pairing $\ll \cdot, \cdot \gg$.
\end{proposition}
\begin{proof}
	''$\Rightarrow$'': Let $\mathcal{D}$ be a Dirac structure in the sense of Definition \ref{Definition Dirac Structure - Finite-dimensional Case}. Let $(f^1,e^1)$, $(f^2,e^2) \in \mathcal{D}$ be arbitrary port variables. By linearity, we have $(f^1 + f^2, e^1 + e^2) \in \mathcal{D}$. By virtue of property \ref{DiracStructureProp1} in Definition \ref{Definition Dirac Structure - Finite-dimensional Case}, we conclude that
	\begin{align*}
		0 &= \langle e^1 + e^2 \mid f^1 + f^2 \rangle \\
		&= \langle e^1 \mid f^1 \rangle + \langle e^1 \mid f^2 \rangle + \langle e^2 \mid f^1 \rangle + \langle e^2 \mid f^2  \rangle \\
		&= \langle e^1 \mid f^2 \rangle + \langle e^2 \mid f^1 \rangle \\
		&= \hspace{0.2 cm} \ll (e^1,f^1), (e^2,f^2) \gg,
	\end{align*}
whence $\mathcal{D} \subset \mathcal{D}^{\perp \!\!\! \perp}$. Moreover, as the plus pairing $\ll \cdot, \cdot \gg$ is non-degenerate, we have
\begin{align*}
	\dim (\mathcal{D}^{\perp \!\!\! \perp}) = \dim (\mathcal{F} \times \mathcal{E}) - \dim (\mathcal{D}) = 2 \dim (\mathcal{F}) - \dim (\mathcal{D}),
\end{align*}
and due to property \ref{DiracStructureProp2} in Definition \ref{Definition Dirac Structure - Finite-dimensional Case}, it follows that $ \dim(\mathcal{D}) = \dim(\mathcal{D}^{\perp \!\!\! \perp})$ and thus $\mathcal{D} = \mathcal{D}^{\perp \!\!\! \perp}$. 
\vspace{0.5 cm}\\
''$\Leftarrow$'': Let $\mathcal{D}$ satisfy (\ref{Proposition Characterization Dirac Structure - Equivalent Property}). Then for every $(f,e) \in \mathcal{D}$ we have in particular that
\begin{align*}
	0 =	\ll (f,e), (f,e) \gg = 2\langle e \mid f \rangle = 2 P(f,e).
\end{align*}
Thus, property \ref{DiracStructureProp1} is satisfied. Again, as the plus pairing $\ll \cdot, \cdot \gg$ is non-degenerate, we obtain
\begin{align}
	\dim (\mathcal{D}^{\perp \!\!\! \perp}) = 2 \dim (\mathcal{F}) - \dim( \mathcal{D}).
	\label{Proof Proposition Characterization Dirac Structure - Dimension Equation}
\end{align}
Equation (\ref{Proposition Characterization Dirac Structure - Equivalent Property}) now implies that $\dim(\mathcal{D}) = \dim (\mathcal{F})$, and hence property \ref{DiracStructureProp2} of a Dirac structure is satisfied as well. 
\end{proof}
In order to prove the alternative characterization, we first note that we have shown that property \ref{DiracStructureProp1} implies $\mathcal{D} \subset \mathcal{D}^{\perp \!\!\! \perp}$. Together with equation (\ref{Proof Proposition Characterization Dirac Structure - Dimension Equation}) this yields that any subspace $\mathcal{D}$ satisfying property \ref{DiracStructureProp1} has the property that $\dim (\mathcal{D}) \leq \dim (\mathcal{F})$. In conclusion, a Dirac structure indeed is the linear subspace of maximal dimension satisfying property \ref{DiracStructureProp1} of Definition \ref{Definition Dirac Structure - Finite-dimensional Case}.
\vspace{0.5 cm}\\
 There are several examples of constant Dirac structures. We will present only a few of them and refer to \cite[Section 2.2]{SchaftJeltsema} for more.
\begin{example}[Dirac Structures; cf. Section 2.2 in \cite{SchaftJeltsema}]
	\label{Example Dirac Structures}
	Let $\mathcal{F} = \mathbb{R}^n$ and $\mathcal{E} = \mathcal{F}' = \mathbb{R}^n$.  Consider the following examples of Dirac structures:
	\begin{enumerate}
		\item[1.] Let $J \colon \mathcal{E} \to \mathcal{F},$ $e \mapsto f = Je$, with $J \in \mathbb{R}^{n \times n}$ a skew-symmetric matrix, i.e., $J = - J^{\top}$. Then one can readily show that
		\begin{align*}
		\mathcal{D}_J :=	\gra (-J) := \left\{ (f,e) \in \mathcal{F} \times \mathcal{E} \mid -f = Je \right\}
		\end{align*}
	is a Dirac structure. According to Proposition \ref{Proposition Characterization Dirac Structure}, we need to show that $\mathcal{D}_J = \mathcal{D}_J^{\perp \!\!\! \perp}$ with respect to $\ll \cdot , \cdot \gg$ defined in  \eqref{Plus Pairing}. Due to skew-symmetry of $J$, for $(f^1,e^1),(f^2,e^2) \in \mathcal{D}_J$ we have
	\begin{align*}
		\ll (f^1,e^1), (f^2,e^2) \gg &= \langle e^1, f^2 \rangle_2 + \langle e^2 , f^1 \rangle_2 \\
		&= \langle e^2, - Je^1 \rangle_2 + \langle e^1, - Je^2 \rangle_2 \\
		&= \langle e^1, Je^2 \rangle_2 - \langle e^1, J e^2 \rangle_2 \\
		&=0.
	\end{align*}  
We conclude that $\mathcal{D}_J \subset \mathcal{D}_J^{\perp \!\!\! \perp}$. Conversely, let $(f,e) \in \mathcal{D}_J^{\perp \!\!\! \perp}$, that is, for all $(f^1, e^1) \in \mathcal{D}_J$, 
\begin{align*}
	0 &= \ll (f^1, e^1), (f,e) \gg  \\
	&=  \langle e^1, f \rangle_2 + \langle e, f^1 \rangle_2 \\
	&=   \langle e^1, f \rangle_2 + \langle e, -Je^1 \rangle_2 \\
	&= \langle e^1,f \rangle_2 + \langle e^1, Je \rangle_2.
\end{align*} 
Consequently, for this equation to hold for all $(f^1, e^1) \in \mathcal{D}_J$, we need $f = -Je$, whence $(f,e) \in \mathcal{D}_J$ and therefore $\mathcal{D}_J^{\perp \!\!\! \perp} \subset \mathcal{D}_J$. Altogether, $\mathcal{D}_J = \mathcal{D}_J^{\perp \!\!\! \perp}$, and $\mathcal{D}_J$ is a Dirac structure. 
\vspace{0.5 cm}\\
As we will see, graphs of skew-symmetric linear maps in the finite-dimensional case, and (formally) skew-symmetric matrix differential operators in the infinite-dimensional case play a substantial role in the port-Hamiltonian framework.
\item[2.] Let $\mathcal{K} \subset \mathcal{F}$ be any subspace. Define
\begin{align*} 
	\mathcal{K}^{\bot} := \left\{ e \in \mathcal{E} \mid \langle e \mid f \rangle = 0 \hspace{0.1 cm} \text{for all } f \in \mathcal{K} \right\}.
\end{align*}
Then $\mathcal{D} = \mathcal{K} \times \mathcal{K}^{\bot}$ is a so-called separable Dirac structure (cf. \cite[Definition 2.2]{SchaftJeltsema}), i.e., for all $(f^{1},e^{1}),(f^{2},e^{2}) \in \mathcal{D}$ we have
\begin{align*}
	\langle e^{1} \mid f^{2} \rangle = 0.
\end{align*}
One can show that any separable Dirac structure is of the form $\mathcal{K} \times \mathcal{K}^{\bot}$ for an appropriately chosen subspace $\mathcal{K} \subset \mathcal{F}$, see \cite[Proposition 2.2]{SchaftJeltsema}. A typical example for a separable Dirac structure is given by transformers. A transformer is an element linking two scalar bonds with flow and effort variables $(f^{1}, e^{1}), (f^{2}, e^{2}) \in \mathbb{R}^2$ by
\begin{align*}
	f^{2} &= \alpha f^1, \\
	e^1 &= - \alpha e^2,
\end{align*}
with $\alpha \in \mathbb{R} \setminus \{0\}$ the transformer ratio. More precisely, let $\mathcal{F} = \mathcal{E} =  \mathbb{R}^2$ and set $\mathcal{K}_{\alpha} = \left\{ (f, \alpha f) \in \mathcal{F} \mid f \in \mathbb{R}\right\}$. It is clear that $\mathcal{K}_{\alpha}^{\bot} = \left\{ (- \alpha e,e) \in \mathcal{E} \mid e \in \mathbb{R} \right\}$, and that $\dim(\mathcal{K}_{\alpha} \times \mathcal{K}_{\alpha}^{\bot}) = \dim(\mathcal{F}) = 2$, which is why $\mathcal{D}_{\alpha} = \mathcal{K}_{\alpha} \times \mathcal{K}_{\alpha}^{\bot}$ is indeed a separable Dirac structure. This can be easily generalized to the multi-port case.
\end{enumerate}
	\QEDA
\end{example}
As illustrated in Figure \ref{Figure Port-Hamiltonian System}, the Dirac structure describes how energy is exchanged among various components in a physical system. In the following, we want to specify the port variables that constitute a (finite-dimensional) Dirac structure. We will give an overview concerning port variables related to energy-storing and energy-dissipating system elements as well as  external ports.

\section{Energy-Storing Elements}
\label{Section Energy-storing Elements}
The energy-storing multi-port element $\mathcal{S}$ unifies all energy-storing elements of the system. The port variables of the Dirac structure associated with $\mathcal{S} = \mathcal{F}_S \times \mathcal{E}_S$ will be denoted by $(f_S,e_S) \in \mathcal{S}$. The total energy storage of the system is defined by a Hamiltonian function $H \colon X \to \mathbb{R}$ on the state space $X$. We restrict ourselves to the case where the state space is given by  $X = \mathbb{R}^n$. In general, the state space is a smooth (finite-dimensional) manifold, see \cite[Chapter 3]{SchaftJeltsema}, \cite{SchaftMaschke20}. In Chapter \ref{Chapter Infinite-dimensional Port-Hamiltonian Systems} we will discuss the extension to the infinite-dimensional case.
\vspace{0.5 cm}\\
Before providing more details about energy-storing multi-port elements, we want to clarify briefly how the state variables have to be chosen in various physical domains, as they are closely related to the flows and efforts. In the common classification of domains, many domains are characterized by two types of states, namely the \emph{generalized displacement} $q \in X$ and the \emph{generalized momentum} $p \in X'$. Depending on the physical domain, one obtains the generalized displacement $q$ by time-integration of the flow variable $f \in \mathcal{F}$, whereas the generalized momentum $p$ refers to time-integration of the effort variable $e \in \mathcal{E}$.  Table~\ref{Table Flows Efforts States} is an extension of Table~\ref{Table Flows And Efforts} and gives a summary of the generalized states in the respective physical subdomain. For more background information concerning this topic, we recommend the introductory chapter in \cite{ModCompPhys}.
\vspace{0.5 cm}
\begin{table}
	\renewcommand{\arraystretch}{2}
	\tiny
	\centering
	{\rowcolors{3}{black!10!white!90}{white!100}
		\begin{tabular}{ c c c c c} 
			\hline
			Physical Domain & Flow  & Effort  &  Generalized Displacement  & Generalized Momentum  \\ 
			& $f \in \mathcal{F}$ & $e \in \mathcal{E}$ & $q = \int f \, dt$ & $p = \int e  \, dt$ \\ [0.5ex]
			\hline \hline
			electric & current $I$ & voltage $V$ & charge $Q$  & magnetic flux linkage $\lambda$  \\
			magnetic & voltage $V$ & current $I$ & magnetic flux linkage $\lambda$ & charge $Q$ \\ 
			potential translation & velocity $v$ & force $F$  & displacement $x$ & momentum $p$ \\
			kinetic translation & force $F$ & velocity $v$ & momentum $p$ & displacement $x$ \\
			potential rotation & angular velocity  $\omega$ & torque $T$ & angular displacement $\theta$ & angular momentum $b$ \\
			kinetic rotation & torque $T$ & angular velocity $\omega$ & angular momentum $b$ & angular displacement $\theta$ \\
			potential hydraulic & volume flow $\phi$ & pressure $p$ & volume $V$ & momentum of a flow tube $\Gamma$ \\
			kinetic hydraulic & pressure $p$ & volume flow $\phi$ & momentum of a flow tube $\Gamma$ & volume $V$ \\ 
			chemical  & chemical potential $\mu$ & molar flow $f_N$ & number of moles $N$ & $-$\\
			thermal & temperature $T$ & entropy flow $f_S$ & entropy $S$ & $-$  \\ [0.5ex]
			\hline
		\end{tabular}
	}
	\caption{Flows, Efforts, and Generalized State Variables in Different Physical Domains \\ \cite[Table 1.1]{ModCompPhys}}
	\label{Table Flows Efforts States}
\end{table}

The vector of flow variables is given by the rate $\dot{x}$ of the state $x \in X$. Thus, for any (time-dependent) state $x \in X$ the flow vector $\dot{x}$ will be an element of the linear space $\mathcal{F}_S= T_xX$, the \emph{tangent space} of $X$ at $x \in X$. In the given case $X = \mathbb{R}^n$, the tangent space $T_xX$ may be identified with $X$.
\vspace{0.5 cm}\\
Furthermore, the vector of effort variables of the energy-storing multi-port element is given by the gradient vector $\nabla H(x) \in \mathcal{E}_S = T_x'X$, the dual space of the tangent space $T_xX$, which again may be identified with $X' = \mathbb{R}^n$. Note that we define $\nabla H(x)$, $x \in X$, to be a column vector instead of a row vector. This convention will be applied throughout this work.
\vspace{0.5 cm}\\
Now, in case of a time-dependent state $x(t) \in X$, $t \geq 0$, we obtain the following power-balance for the energy-storing multi-port element:
\begin{align}
	\frac{d}{dt}H(x(t)) = \nabla H^{\top}(x(t)) \dot{x}(t).
	\label{Power Balance Energy-Storing Element}
\end{align}
In the present case, the bond space is given by $\mathcal{B} = \mathcal{S} = \mathcal{F}_S \times \mathcal{E}_S$, and the power-pairing $\langle \cdot \mid \cdot \rangle \colon \mathcal{E}_S \times \mathcal{F}_S \to \mathbb{R}$ is given by $\langle \cdot , \cdot \rangle_2$. The interconnection of the energy-storing elements to the storage port $(f_S, e_S) \in X \times X'$ of some Dirac structure $\mathcal{D} \subset X \times X'$ is done by setting
\begin{align}
	f_S = - \dot{x} \hspace{0.3 cm} \text{and} \hspace{0.3 cm} e_S = \nabla H(x).
	\label{Choice of Port Variables fS and eS}
\end{align}
Note that the minus sign in (\ref{Choice of Port Variables fS and eS}) has been placed in order to have a consistent power flow convention: $\nabla H^{\top}(x) \dot{x}$ is the power flowing into the energy-storing elements, whereas $e_S^{\top}f_S$ is the power flowing into the Dirac structure. Hence, with \eqref{Function - Power} and \eqref{Choice of Port Variables fS and eS} the power-balance (\ref{Power Balance Energy-Storing Element}) can also be written as
\begin{align}
	\frac{d}{dt}H(x(t)) = - e_S^{\top}(t)f_S(t) = - \langle e_S(t) \mid f_S(t) \rangle = 0, \hspace{0.5 cm} t \geq 0,
	\label{Power-Balance Energy-Storing Element with fS and eS}
\end{align}
provided that $(f_S(t), e_S(t)) \in \mathcal{D}$ for $t \geq 0$. Equation \eqref{Power-Balance Energy-Storing Element with fS and eS} then expresses power conservation of the system.
\section{Energy-Dissipating Elements}
\label{Section Energy-dissipating Elements}
The second multi-port element $\mathcal{R}$ corresponds to internal energy dissipation, and its port variables are
denoted by $(f_R,e_R)$. These port variables are terminated on a static energy-dissipating (resistive) relation $\mathcal{R}$. Let $\mathcal{F}_R = \mathbb{R}^m$ and $\mathcal{E}_R = \mathcal{F}_R' = \mathbb{R}^m$ be the dual space of $\mathcal{F}_R$. In general, a resistive relation $\mathcal{R} \subset \mathcal{F}_R \times \mathcal{E}_R$ is a set of the form
\begin{align}
	\mathcal{R} = \left\{ (f_R,e_R) \in \mathcal{F}_R \times \mathcal{E}_R \mid R(f_R,e_R) = 0, \hspace{0.1 cm} \langle e_R, f_R \rangle_2 \leq 0 \right\}.
	\label{Resistive Relation}
\end{align}
We will call the subset $\mathcal{R}$ an \emph{energy-dissipating relation} or a \emph{resistive structure}. If we restrict ourselves to linear resistive structures, the mapping $R \colon \mathcal{F}_R \times \mathcal{E}_R \to \mathbb{R}$ will be of the form
\begin{align*}
	R(f_R,e_R) = R_ff_R + R_ee_R,
\end{align*}
where $R_f,R_e \in \mathbb{R}^{m \times m}$. In order to guarantee that the second constitutive relation in (\ref{Resistive Relation}) is satisfied whenever the first one is, we need to impose the following conditions on those matrices:
\begin{enumerate}
	\item[(i)] $R_fR_e^{\top} = R_eR_f^{\top} \geq 0$ and
	\item[(ii)] $\rank \left( \begin{bmatrix}
		R_f & R_e
 	\end{bmatrix} \right)= \dim (\mathcal{F}_R) = m$. 
\end{enumerate}
If these conditions are satisfied, then we can equivalently rewrite the first constitutive condition in (\ref{Resistive Relation}) as the requirement that there exists some vector $\lambda \in \mathbb{R}^m$ such that
\begin{align*}
	f_R = R_e^{\top} \lambda \hspace{0.3 cm} \text{and} \hspace{0.3 cm} e_R = -R_f^{\top}\lambda.
\end{align*}
It is easy to check now that under these conditions we have $\langle e_R , f_R \rangle_2 \leq 0$, and hence the second constitutive condition in (\ref{Resistive Relation}) can be omitted. Altogether, a linear resistive structure is of the form
\begin{align}
	\mathcal{R} = \left\{ (f_R,e_R) \in \mathcal{F}_R \times \mathcal{E}_R \mid R_f f_R + R_ee_R = 0 \right\}
	\label{Resistive Relation - Linear}
\end{align}
with suitably defined matrices $R_f, R_e \in \mathbb{R}^{m \times m}$. For a system with energy-storing and resistive elements, the bond space is given by $\mathcal{B} = \mathcal{F} \times \mathcal{E} =  (X \times \mathcal{F}_R) \times (X' \times \mathcal{E}_R)$, and the power pairing again equals the natural inner product. By definition, the Dirac structure $\mathcal{D} \subset \mathcal{B}$ of such a system satisfies the power-balance
\begin{align}
	\left\langle \begin{pmatrix}
		e_S \\
		e_R 
	\end{pmatrix} \, \Bigg| \,  \begin{pmatrix}
	f_S \\
	f_R
\end{pmatrix} \right\rangle = e_S^{\top}f_S + e_R^{\top}f_R = 0, \hspace{0.5 cm} \left( \begin{pmatrix}
f_S \\
f_R 
\end{pmatrix} , \begin{pmatrix}
e_S \\
e_R
\end{pmatrix} \right) \in \mathcal{D}.
	\label{Power-Balance with Dissipation}
\end{align}
This leads, by substitution of (\ref{Power-Balance Energy-Storing Element with fS and eS}) and (\ref{Power-Balance with Dissipation}), to
\begin{align}
	\frac{d}{dt} H(x(t)) = -e_S^{\top}(t) f_S(t) = e_R^{\top}(t) f_R(t) \leq 0, \hspace{0.5 cm} t\geq0.
	\label{Power-Balance Energy-Storing and Energy-Dissipating Elements}
\end{align}
This inequality expresses the fact that the internally stored energy (the Hamiltonian) monotonically decreases in the presence of resistive elements.
\vspace{0.5 cm}\\
An important special case of energy-dissipating relations occurs when
the resistive relation can be expressed as
\begin{align*}
	f_R = - F(e_R),
\end{align*}
where $F \colon \mathbb{R}^m \to \mathbb{R}^{m}$ satisfies $e_R^{\top} F(e_R) \geq 0$ for all $e_R \in \mathbb{R}^m$. For linear resistive elements, we have $F(e_R) = Re_R$ with $R \in \mathbb{R}^{m \times m}$ being symmetric and positive semi-definite. A typical example of a linear resistive structure is provided by a linear damper in a spring-mass-damper system, where $\mathcal{R}_d$ is given by
\begin{align}
	\label{Resistive Relation Finite-Dimensional Case}
	\mathcal{R}_d = \{ (f_R,e_R) \in \mathbb{R}^2 \mid f_R = -de_R \},
\end{align}
and where $d \in \mathbb{R}_+$ is the damping constant.

\section{External Ports}
\label{Section External Ports}
The \emph{external port} $(f_P,e_P) \in \mathcal{F}_P \times \mathcal{E}_P$, where $\mathcal{F}_P = \mathbb{R}^k$ and $\mathcal{E}_P= \mathcal{F}_P'$, models the interaction of the system with its environment. This comprises different situations. For instance, there are port variables which are accessible for controller action (see \cite[Chapter 5]{ModCompPhys}, \cite[Chapter~7]{SchaftL2}), or there is an external port describing the interaction with the physical environment. Another important type of external port variables corresponds to source connections, as it is often the case in electrical circuits. For instance, in an electrical circuit with voltage source the system's input is the voltage of the source, and the current through the source is set to be the resulting output variable. If one wishes to distinguish between these types of external ports, which in some cases may even be inevitable, one can divide the external port into a control port and an interaction port (cf. \cite[Subsection 2.2.1]{ModCompPhys}). However, we will not distinguish between them.
\vspace{0.5 cm}\\
Now, taking the external port into account, for a Dirac structure
\begin{align*}
	\mathcal{D} \subset X \times \mathcal{F}_R \times \mathcal{F}_P \times X' \times \mathcal{E}_R \times \mathcal{E}_P
\end{align*}
the power-balance (\ref{Power-Balance with Dissipation}) extends to
\begin{align*}
	 \left\langle \begin{pmatrix}
	 	e_S \\
	 	e_R \\
	 	e_P
	 \end{pmatrix} \, \Bigg| \, \begin{pmatrix}
	 f_S \\
	 f_R \\
	 f_P
 \end{pmatrix} \right\rangle = e_S^{\top}f_S + e_R^{\top}f_R + e_P^{\top}f_P  = 0, \hspace{0.5 cm} \left(  \begin{pmatrix}
 f_S \\
 f_R \\
 f_P
\end{pmatrix} , \begin{pmatrix}
e_S \\
e_R \\
e_P
\end{pmatrix} \right) \in \mathcal{D},
\end{align*}
and (\ref{Power-Balance Energy-Storing and Energy-Dissipating Elements}) accordingly extends to
\begin{align*}
	\frac{d}{dt}H(x(t)) = - e_S^{\top}(t) f_S(t) = e_R^{\top}(t)f_R(t) + e_P^{\top}(t)f_P(t) \leq e_P^{\top}(t)f_P(t), \hspace{0.5 cm} t \geq 0.
\end{align*}
 This inequality states that the increase of the internally stored energy is always less
than or equal to the externally supplied power.

\section{Port-Hamiltonian Dynamics}
\label{Section Port-Hamiltonian Dynmacis}
Now that we have discussed the constitutive ports of a Dirac structure and the corresponding power-balance, we want to present the underlying geometrical structure of port-Hamiltonian systems. We will begin with the definition of linear port-Hamiltonian systems by means of Dirac structures, and give a more familiar representation of a big class of such systems. At the end of this section, we will give some final remarks regarding non-linear port-Hamiltonian systems and port-Hamiltonian differential algebraic systems.
\vspace{0.5 cm}\\
We begin with the definition of linear port-Hamiltonian systems with respect to an underlying Dirac structure.

\begin{definition}[Linear Port-Hamiltonian System]
	\label{Definition Linear Port-Hamiltonian System}
Consider the state space $X = \mathbb{R}^n$ and a Hamiltoninan $H \colon X \to \mathbb{R}$ defining energy storage. Furthermore, let $\mathcal{F}_R = \mathcal{E}_R = \mathbb{R}^m$ and $\mathcal{F}_P = \mathcal{E}_P = \mathbb{R}^k$. A linear port-Hamiltonian system on $X$ is defined as a triple $(\mathcal{D}, H, \mathcal{R})$, consisting of a Dirac structure
\begin{align*}
	\mathcal{D} \subset \mathcal{F} \times \mathcal{E} = (X \times \mathcal{F}_R \times \mathcal{F}_P) \times (X'  \times \mathcal{E}_R \times \mathcal{E}_P)
\end{align*}
defined with respect to the power pairing $\langle \cdot \mid \cdot \rangle = \langle \cdot , \cdot \rangle_2 \colon \mathcal{E} \times \mathcal{F} \to \mathbb{R}$, with energy-storing ports $(f_S,e_S) \in X \times X'$, resistive ports $(f_R, e_R) \in \mathcal{F}_R \times\mathcal{E}_R$, external ports $(f_P, e_P) \in \mathcal{F}_p \times \mathcal{E}_P$, the Hamiltonian $H$, and a linear resistive structure
\begin{align*}
	\mathcal{R} \subset \mathcal{F}_R \times \mathcal{E}_R
\end{align*}
of the form \eqref{Resistive Relation - Linear}. Its dynamics are geometrically specified by the requirement that for all $t \geq 0$  we have
\begin{align*}
	\left( \begin{pmatrix}
		-\dot{x}(t) \\
		f_R (t) \\
		f_P (t)
	\end{pmatrix}, \begin{pmatrix}
	\nabla H(x(t)) \\
	e_R(t) \\
	e_P (t) 
\end{pmatrix} \right)  \in \mathcal{D} \hspace{0.3 cm} \text{and} \hspace{0.3 cm} (f_R(t), e_R(t)) \in \mathcal{R}.
\end{align*}
\end{definition}
At the beginning of Chapter \ref{Chapter Finite-dimensional Port-Hamiltonian Systems}, we have already seen how a mass-spring system is modeled as a port-Hamiltonian system. We revisit this example to show that this is indeed a linear port-Hamiltonian system in the sense of Definition \ref{Definition Linear Port-Hamiltonian System}.
\begin{example}[Mass-spring System Revisited]
	\label{Example Mass-spring System Revisited}
Consider the mass-spring system from Example \ref{Example Mass-Spring System}. The dynamics of this system are defined by equation \eqref{Example Mass-Spring System - pH System}. Defining the state space $X = \mathbb{R}^2$ with state vector $x = \begin{bmatrix}
	q \\
	p
\end{bmatrix}$, and taking $\mathcal{F}_S = \mathcal{E}_S = X$, one immediately sees that the underlying Dirac structure of this system is given by the graph of the skew-symmetric matrix $-J \in \mathbb{R}^{2 \times 2}$ with
\begin{align*}
J =  \begin{bmatrix}
	0 & 1 \\
	-1 & 0
\end{bmatrix}, \text{ i.e.,} \hspace{0.3 cm} 	\mathcal{D}_{J} = \{ (f_S,e_S) \in \mathcal{F}_S \times \mathcal{E}_S \mid - f_S = Je_S \}.
\end{align*}
Here we have $- f_S =  \begin{bmatrix}
	\dot{q} \\
	\dot{p} 
\end{bmatrix}$ and $e_S = \begin{bmatrix}
\frac{\partial H}{\partial q} \\
\frac{\partial H}{\partial p}
\end{bmatrix}$, where $H \colon X \to \mathbb{R}$ is the Hamiltonian given by \eqref{Example Mass-Spring System - Hamiltonian}. Note that the matrix $J$ is the one given on the right-hand side in \eqref{Example Mass-Spring System - pH System}. As there are neither external ports nor resistive elements, the dynamics of this linear port-Hamiltonian system is geometrically specified by
\begin{align*}
 (f_S(t), e_S(t)) =	(-\dot{x}(t), \nabla H(x(t))) \in \mathcal{D}, \hspace{0.5 cm}  t \geq 0.
\end{align*}
Moreover, by skew-symmetry of $J$, the power-balance as expected yields
\begin{align*}
	\frac{d}{dt} H(q(t),p(t)) = \left\langle \begin{pmatrix}
		kq(t) \\
		\frac{p(t)}{m}
	\end{pmatrix} \Bigg| \begin{pmatrix}
	\dot{q}(t) \\
	\dot{p}(t)
\end{pmatrix} \right\rangle = - \langle e_S(t) \mid f_S(t) \rangle = 0, \hspace{0.5 cm} t \geq 0.
\end{align*}
\QEDA
\end{example}
A state-space representation, as established for the mass-spring system in Example \ref{Example Mass-Spring System}, is quite common. More general, an important class of linear port-Hamiltonian systems is obtained by considering Dirac structures which are the graph of a skew-symmetric map of the form
\begin{align*}
	\begin{bmatrix}
		e_S \\
		e_R \\
		e_P \\
	\end{bmatrix} \mapsto \begin{bmatrix}
		f_S \\
		f_R \\
		f_P 
	\end{bmatrix} := \begin{bmatrix}
		-J & -G_R & -G \\
		G_R^{\top} & 0 & 0 \\
		G^{\top} & 0 & 0 
	\end{bmatrix} \begin{bmatrix}
		e_S \\
		e_R \\
		e_P \\
	\end{bmatrix} , 
\end{align*}
with $J = - J^{\top} \in \mathbb{R}^{n \times n}$, a linear dissipation relation $f_R = - \bar{R}e_R$ for some positive-semidefinite matrix $\bar{R} \in \mathbb{R}^{m \times m}$, and $G_R \in \mathbb{R}^{n \times m}$, $G \in \mathbb{R}^{n \times k}$. This yields so-called \emph{input-state-output port-Hamiltonian systems} of the form
\begin{align}
	\begin{split}
	\dot{x} &= \left[ J - R \right] \nabla H(x) + G u, \\
	y&= G^{\top} \nabla H(x), 
	\end{split}
	\label{Input-State-Output port-Hamiltonian System}
\end{align}
with $R := G_R \bar{R} G_R^{\top} \in \mathbb{R}^{n \times n}$, and with $u = e_P$ the input vector and $y = f_P$ the output vector. This is a convenient representation for the investigation of control related problems for port-Hamiltonian systems, see \cite[Chapter~5]{ModCompPhys}, \cite[Chapter~7]{SchaftL2}, \cite[Chapter~15]{SchaftJeltsema}. 
\vspace{0.5 cm}\\
The following example aims to emphasize port-based network modeling of multi-physics systems. We also want to point to Example 2.5 in \cite{SchaftJeltsema}, where a port-Hamiltonian formulation for the dynamics of a DC motor is derived.
\begin{example}[Levitated Ball System; Example 2.4 in \cite{SchaftJeltsema}]
	\label{Example - Levitated Ball System}
	\begin{figure}[h]
		\centering
		\includegraphics[width = 6 cm]{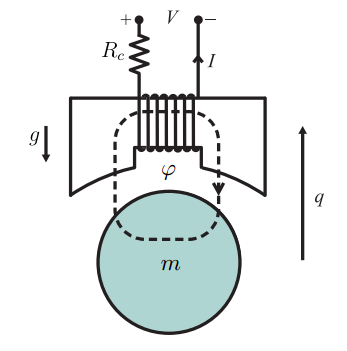}
		\caption{Magnetically Levitated Ball of Example \ref{Example - Levitated Ball System} \cite[Figure 2.4]{SchaftJeltsema}}
		\label{Figure Levitated Ball System}
		\end{figure}
	Consider the dynamics of an iron ball that is levitated by the magnetic field of a controlled inductor, as schematically depicted in Figure \ref{Figure Levitated Ball System}. The Hamiltonian $H \colon \mathbb{R}^3 \to \mathbb{R}$ of this system - with $q \in \mathbb{R}$ the height of the ball, $m >0$ the ball's mass, $p \in \mathbb{R}$ the vertical momentum, $ \varphi \in \mathbb{R}$ the magnetic flux-linkage of the inductor, $R_c >0$ the coil's resistance, $g > 0$ the gravitational constant,  $L(\cdot)$ the inductance depending on the height $q$, $V \in \mathbb{R}$ the voltage across the inductor, and $I \in \mathbb{R}$ the current through the inductor - is defined as
	\begin{align*}
		H(q,p,\varphi) = E_{\text{pot}}(q) + E_{\text{kin}}(p) + E_{\text{mag}}(q,\varphi) = mgq + \frac{p^2}{2m} + \frac{\varphi^2}{2L(q)}.
	\end{align*}
The system dynamics are given by
	\begin{align}
		\begin{split}
		\begin{bmatrix}
			\dot{q} \\
			\dot{p} \\
			\dot{\varphi} 
		\end{bmatrix} &= \begin{bmatrix}
		0 & 1 & 0 \\
		-1 & 0 & 0 \\
		0 & 0 & - R_c 
	\end{bmatrix} \begin{bmatrix}
	\frac{\partial H}{\partial q} \\
	\frac{\partial H}{\partial p} \\
	\frac{\partial H }{\partial \varphi}
\end{bmatrix} + \begin{bmatrix}
0 \\
0 \\
1
\end{bmatrix} V, \\
I &= \frac{\partial H}{\partial \varphi}.
\end{split}
\label{Example Levitated Ball System - pH System}
	\end{align}
Here, $e_P = V$ denotes the input and $f_P = I$ denotes the system's output. Furthermore, the matrices $J$, $R \in \mathbb{R}^{3 \times 3}$, and $G \in \mathbb{R}^{3}$ from \eqref{Input-State-Output port-Hamiltonian System} are given by
\begin{align*}
	J = \begin{bmatrix}
		0 & 1 & 0 \\
		-1 & 0 & 0 \\
		0 & 0 & 0 
	\end{bmatrix}, \hspace{0.3 cm} R = \begin{bmatrix}
	0 & 0 & 0 \\
	0 & 0 & 0 \\
	0 & 0 & R_c
\end{bmatrix}, \hspace{0.3 cm}  \text{and} \hspace{0.3 cm} G = \begin{bmatrix}
0 \\ 
0 \\
1
\end{bmatrix}.
\end{align*}
While the system description \eqref{Example Levitated Ball System - pH System} may at first glance indicate that the mechanical and the magnetic part of the given system are decoupled, they are actually coupled, namely via the Hamiltonian $H$: the magnetic energy $E_{\text{mag}}$ depends both on the flux $\varphi$ and the mechanical variable $q$. Consequently, the evolution of the (mechanical) momentum variable $p$ depends on the magnetic variable $\varphi$, and the evolution of the magnetic variable $\varphi$ depends on the mechanical variable $q$.
\QEDA
\end{example}
Needless to say, there exist nonlinear port-Hamiltonian systems as well. The major difference to linear port-Hamiltonian systems is that  the state space $X$ becomes a smooth manifold, and the Dirac structure is not constant anymore, but is modulated by the state variables. Moreover, the spaces $\mathcal{F}_S$ and $\mathcal{E}_S$ become the tangent space and the cotangent space, respectively, as mentioned in Section \ref{Section Energy-storing Elements}. Furthermore, the input-state-output description \eqref{Input-State-Output port-Hamiltonian System} of nonlinear port-Hamiltonian systems becomes
\begin{align}
	\begin{split}
	\dot{x} &= [J(x) - R(x)]\nabla H(x) + G(x) u, \\
	y &= G^{\top}(x)\nabla H(x). 
\end{split}
\label{Nonlinear PH-System - Finite Dimensional}
\end{align}

For details and some examples we refer to  \cite[Chapter 6]{SchaftL2}, \cite[Chapter 3]{SchaftJeltsema}, and the references specified therein. 
\vspace{0.5 cm}\\
When modeling dynamical systems, the appearance of algebraic equations next to differential equations is ubiquitous. For instance, in network modeling of physical systems, the interconnections between different dynamical components of the overall system almost unavoidably lead to algebraic constraints. Such systems are called \emph{linear port-Hamiltonian descriptor systems} (\emph{port-Hamiltonian differential-algebraic systems}), or short \emph{pHDAEs} \cite{Beattie}. The geometric definition of port-Hamiltonian systems is quite handy, as it can even characterize the dynamics of pHDAEs by choosing the underlying Dirac structure suitably. This is the reason why in Definition \ref{Definition Linear Port-Hamiltonian System} one usually refers to the triple $(\mathcal{D}, H, \mathcal{R})$ as a (linear) port-Hamiltonian differential-algebraic system.
\vspace{0.5 cm}\\
In this chapter, we have presented the basic concepts and main characteristics of the port-Hamiltonian framework on finite-dimensional spaces. We have learned that the centerpiece of this approach is the power-conserving interconnection structure $\mathcal{D}$, called the Dirac structure. We have presented the basic definition of a Dirac structure and discussed its constituting ports in the finite-dimensional case, namely ports associated with energy-storing and energy-dissipating elements as well as external ports. Lastly, we have geometrically specified the dynamics of linear port-Hamiltonian systems by means of their underlying Dirac structures and have given a state-space representation for a huge class of port-Hamiltonian systems. 
\vspace{0.5 cm}\\
Now that we have gained a basic understanding of port-based modeling of physical systems, we want to study infinite-dimensional port-Hamiltonian systems in the upcoming chapters, to finally deal with the main topic of this thesis, namely the port-Hamiltonian formulation of a system of two conservation laws with a moving interface. In order to do so, we need to provide some background information concerning semigroup theory, infinite-dimensional systems theory, the concept of admissibility, and evolution equations. We will treat these subjects in the next chapter. 

\chapter{Some Background on Functional Analysis and Systems Theory}
\label{Chapter Some Background}
In preparation for the main part of this thesis, we recall some well-known concepts and results on selected topics. Most of the time, we will give no proofs, but provide references regarding the respective topics. Furthermore, we will omit various topics that are usually covered in the literature, and only focus on aspects that are useful for the discussion in the following chapters.
\vspace{0.5 cm}\\
This chapter is organized as follows. In Section \ref{Section Operators} we recall the definition of the Hilbert space adjoint of a densely defined operator $A \colon D(A) \subset X \to X$. Moreover, we introduce the notion of the formally adjoint operator of a matrix differential operator, and discuss the difference between the Hilbert space adjoint and the formal adjoint. As we will deal with infinite-dimensional (control) systems, we need a tool for the analysis of such systems. In Section~\ref{Section Strongly Continuous Semigroups} we are going to study strongly continuous semigroups and characterize their infinitesimal generators. In Section \ref{Section The Abstract Cauchy Problem} we cover abstract (i.e. Hilbert space-valued) linear differential equations and give sufficient and necessary conditions for well-posedness of such systems. We will see that the framework of strongly continuous semigroups is a powerful tool for the analysis of so-called abstract Cauchy problems. Afterwards, we are going to shortly discuss infinite-dimensional control systems in Section \ref{Section Linear Control Systems}. For a proper treatment of such systems, we need to introduce the extrapolation space $X_{-1}$ and the concept of admissible control operators. We will give criteria for when the mild solution of an abstract control system is a continuous $X$-valued function. Section \ref{Section Evolution Equations} is the last section of this chapter and deals with linear, time-variant evolution problems. We will study (two-parameter) evolution systems and define the notion of stability for a family of infinitesimal generators of $C_0$-semigroups. We assume that the reader is acquainted with Riemann integrals of vector-valued functions \cite[Section 3.1]{Lorenzi}, the Lebesgue theory of vector-valued measurable functions, and Bochner intergrals \cite[Section 1.1]{Arendt}.
\vspace{0.5 cm}\\
Unless stated otherwise, throughout this chapter we assume that $X$ is a real or complex separable Hilbert space with inner product $\langle  \cdot, \cdot \rangle_{X}$.

\section{Adjoint and Formally Adjoint Operators}
\label{Section Operators}
In this section, we want to define and highlight the difference of the Hilbert space adjoint of a matrix differential operator to its formal adjoint, and provide a simple example to illustrate this difference. For a detailed, more general discussion on formally adjoint operators we refer to \cite[Chapter 13]{Aubin}.
\vspace{0.5 cm}\\
For the first part of this section, we follow \cite[Section 2.9]{Tucsnak-Weiss}.
\begin{definition}
	Let $V$ and $Z$ be Hilbert spaces. An operator $J \in \mathcal{L}(V,Z)$ is called an isomorphism, or a unitary operator, from $V$ to $Z$, if we have
\begin{align*}
	\langle J v, Jw \rangle_Z = \langle v , w \rangle_V, \hspace{0.5 cm} v,w \in V.
\end{align*}
\end{definition}
One may verify that $J \in \mathcal{L}(V,Z)$ is unitary if and only if $\|Jv\|_Z = \|v\|_V$ holds for all $v \in V$, and $\ran(J) = Z$.
\vspace{0.5 cm}\\
For any Hilbert space $V$, we denote  by $ \langle z, v \rangle_{V' \times V}$ the functional $z \in V'$ applied to $v \in V$, so that $ \langle \cdot , \cdot \rangle_{V' \times V}$ is linear in the first component and antilinear in the second component. This pairing is called the \emph{natural (duality) pairing} on $V$. The norm on $V'$ is defined as
\begin{align*}
	\| z \|_{V'} = \sup_{\stackrel{ v \in V}{\|v\|_V = 1}} \vert \langle z, v \rangle_{V' \times V} \vert, \hspace{0.5 cm} z \in V'. 
\end{align*}
There is an operator $J_R \colon V \to V'$ satisfying
\begin{align*}
	\langle J_R v, w \rangle_{V' \times V} = \langle v, w \rangle_V, \hspace{0.5 cm} v, w \in V. 
\end{align*}
According to the Riesz representation theorem, $J_R$ is an isomorphism. We say that we identify the Hilbert space $V$ with its dual $V'$, if we do not distinguish between $v$ and $J_Rv$ for all $v \in V$. 
\vspace{0.5 cm}\\
In the following, we identify $X$ with its dual $X'$. Consider a linear and densely defined operator $A \colon D(A) \subset X \to X$. We wish to associate a \emph{Hilbert space adjoint} $A^{\ast}$ to $A$. Its domain is given by
\begin{align*}
	D(A^{\ast}) := \left\{ y \in X \mid \langle A \cdot , y \rangle_X \colon D(A) \to \mathbb{R}, \, x \mapsto \langle Ax, y \rangle_X \text{ is continuous} \right\}.
\end{align*}
If $y \in D(A^{\ast})$, then the Hahn-Banach theorem extends the functional
\begin{align*}
	x \mapsto \langle A x, y \rangle_X
\end{align*}
to a continuous linear functional on $X$. The Riesz representation theorem yields that there exists a unique element $z_y \in X$ such that this functional has the form
\begin{align*}
	\langle \cdot , z_y \rangle_X. 
\end{align*}
For $y \in D(A^{\ast})$ we define $A^{\ast}y = z_y$.  Hence, it holds that
\begin{align*}
	\langle A x, y \rangle_X = \langle x , A^{\ast}y \rangle_X, \hspace{0.5 cm} x \in D(A), \, y \in D(A^{\ast}). 
\end{align*}
In other words, the domain of the Hilbert space adjoint $A^{\ast}$ of an operator $A$ is given by
\begin{align}
	\label{Adjoint Operator Domain wrt Frechet Riesz}
	D(A^{\ast}) = \left\{ y \in X \mid \exists! \, z_y  \in X \colon \forall x \in D(A) \colon \langle Ax, y \rangle_X = \langle x , z_y \rangle_X \right\}.
\end{align}
The following definitions can be found in Section 3.2 and Section 3.7 in \cite{Tucsnak-Weiss}, and in \cite[Appendix~A.3.2]{CurtainZwart}.
\begin{definition}
	\label{Definition Symmetric and Adjoint Operators}
	Let $A \colon D(A) \subset X \to X$ and $B \colon D(B) \subset X \to X$ be linear and densely defined.
	\begin{enumerate}[label = (\roman*)]
		\item $A$ is called symmetric, if 
		\begin{align*}
			\langle Ax, y \rangle_X = \langle x, Ay \rangle_X, \hspace{0.5 cm} x,y \in D(A).
		\end{align*}
		\item $A$ is called skew-symmetric, if 
		\begin{align*}
			\langle Ax, y \rangle_X = - \langle x, Ay \rangle_X, \hspace{0.5 cm} x,y \in D(A).
		\end{align*} 
	\item If for all $x \in D(A)$, $y \in D(B)$ it holds that
	\begin{align*}
		\langle Ax, y \rangle_X = \langle x, By \rangle_X,
		\end{align*}
	then $B \subset A^{\ast}$, that is, $D(B) \subset D(A^{\ast})$ and $By = A^{\ast}y$ for all $y \in D(B)$.
		\item $A$ is called self-adjoint if $A^{\ast} = A$, that is, $D(A) = D(A^{\ast})$ and $Ax = A^{\ast}x$ for all $x \in D(A)$.
		\item $A$ is called skew-adjoint if $A^{\ast} = -A$.
		\item $A$ is called coercive if it is self-adjoint and there exists some $m > 0$ such that $A \geq mI_X$, i.e.,
		\begin{align*}
			\langle Ax, x \rangle_X \geq  m \|x\|_X^2, \hspace{0.5 cm} x \in D(A).
		\end{align*}
	\end{enumerate}
\end{definition}
For the rest of this section, we restrict ourselves to matrix differential operators defined on a Sobolev space. For this type of operators, we want to introduce the notion of \emph{formally adjoint operators}. Roughly speaking, the only difference to the Hilbert space adjoint is that the constitutive relations of its domain are relaxed in a way such that 
\begin{align*}
	\langle A x, y \rangle_X = \langle x, A^{\ast} y \rangle_X \hspace{0.2 cm} + \text{boundary terms}, \hspace{0.5 cm} x \in D(A), \, y \in D(A^{\ast}).
\end{align*}
The discussion in Section \ref{Section Boundary Port-Hamiltonian Systems Associated with Skew-symmetric Operators} will clarify the necessity of allowing boundary terms concerning the relation between an operator and its formal adjoint.
\vspace{0.5 cm}\\
Now, let $X = L^2([a,b], \mathbb{R}^n)$ be endowed with the inner product
\begin{align*}
	\langle x , y \rangle_X = \langle x , y \rangle_{L^2} = \int_{a}^{b} y^{\top}(z) x(z) \, dz, \hspace{0.5 cm} x,y \in X.
\end{align*}
Let us define the Sobolev spaces
 \begin{align*}
	H^1([a,b], \mathbb{R}^n) &= \left\{ x \in X \, \big| \,  x \text{ is absolutely continuous}, \, \frac{dx}{dz} \in X \right\}, \\
	H^N([a,b], \mathbb{R}^n) &= \left\{ x \in X \, \big| \,  x \text{ is absolutely continuous}, \, \frac{dx}{dz} \in H^{N-1}([a,b], \mathbb{R}^n) \right\}, \hspace{0.5 cm} N \geq 2,
\end{align*}
as well as
\begin{align*}
	H_0^1([a,b], \mathbb{R}^n) &= \left\{ x \in H^1([a,b], \mathbb{R}^n) \mid x(a) = x(b) = 0 \right\}, \\
	H_0^N([a,b], \mathbb{R}^n) &= \left\{ x \in H^N([a,b],\mathbb{R}^n) \cap H_0^{N-1}([a,b], \mathbb{R}^n) \, \big| \, \frac{dx}{dz} \in  H_0^{N-1}([a,b], \mathbb{R}^n)  \right\}, \hspace{0.5 cm} N \geq 2.
\end{align*} 
	Furthermore, define the \emph{boundary trace operator} $\trace \colon  H^N \left( [a,b], \mathbb{R}^n \right) \to \mathbb{R}^{2nN}$ by
\begin{align}
	\trace(x) := \begin{bmatrix}
		x(b) \\
		\vdots \\
		\frac{d^{N-1}x}{d z^{N-1}}(b) \\
		x(a) \\
		\vdots \\
		\frac{d^{N-1}x}{d z^{N-1}}(a)
	\end{bmatrix}, \hspace{0.5 cm} x \in H^N \left( [a,b], \mathbb{R}^n \right).
	\label{Trace Operator - Chapter 3}
\end{align}
The boundary trace operator $\trace$ is linear, bounded, and surjective, see \cite[Section 7.8]{Aubin}. 
It allows us to equivalently write for $N \in \mathbb{N}$,
\begin{align*}
	H_0^N([a,b], \mathbb{R}^n) = \left\{ x \in H^N([a,b], \mathbb{R}^n) \mid \trace(x) = 0 \right\}.
\end{align*}
Now, consider the matrix differential operator $A \colon D(A) \subset X  \to X$ given by
\begin{align}
	\begin{split}
	D(A) &= \left\{ x \in X \mid x \in H^N([a,b], \mathbb{R}^n), \, W \trace(x) = 0 \right\}, \\
	Ax &= \sum_{i = 0}^{N} P_i \frac{d^ix}{dz^i}, \hspace{0.5 cm} x \in D(A),
	\end{split}
\label{Matrix Differential Operator A}
\end{align}
with $P_i \in \mathbb{R}^{n \times n}$, $i = 0, \ldots,N$, and $W \in \mathbb{R}^{n \times 2n}$. In particular, we have $H_0^N([a,b], \mathbb{R}^n) \subset D(A)$.

\begin{definition}
	For the domain of the formal adjoint $A_0^{\ast} \colon D(A_0^{\ast}) \subset X \to X$ of the matrix differential operator $A$ defined in \eqref{Matrix Differential Operator A} it holds that
	\begin{align*}
	D(A_0^{\ast}) = \left\{ y \in X \mid \langle A \cdot, y \rangle_X \colon H_0^N([a,b], \mathbb{R}^n) \to \mathbb{R}, \hspace{0.1 cm}  x \mapsto \langle A x, y \rangle_X \text{ is continuous} \hspace{0.1 cm} \right\}.
\end{align*}
\end{definition} 
In other words, the formal adjoint of the differential operator $A$ defined in \eqref{Matrix Differential Operator A} is simply given by the adjoint of its restriction $A_{|H_0^N([a,b],\mathbb{R}^n)}$. 
\vspace{0.5 cm}\\
Next, we present the formal pendant of some of the notions introduced in Definition \ref{Definition Symmetric and Adjoint Operators}.
\begin{definition}
	\label{Definition Formal Operator Properties}
	Let $A \colon D(A) \subset X \to X$ be the matrix differential operator defined in \eqref{Matrix Differential Operator A}.  
	\begin{enumerate}
		\item[(i)] $A$ is called formally symmetric, if
		\begin{align*}
			\langle A x, y \rangle_X = \langle x, Ay \rangle_X, \hspace{0.5 cm} x,y \in H_0^N([a,b], \mathbb{R}^n).
		\end{align*}
		\item[(ii)] $A$ is called formally skew-symmetric, if
		\begin{align*}
			\langle Ax, y \rangle_X = - \langle x, Ay \rangle_X, \hspace{0.5 cm} x,y \in H_0^N([a,b], \mathbb{R}^n).
		\end{align*}
	\end{enumerate}
\end{definition}
Let us clarify the difference by reference to some simple examples.
\begin{example}
	Consider the following examples.
	\begin{enumerate}
		\item[(i)][cf. Example 13.4 in \cite{Rudin}] Let $X = L^2([a,b], \mathbb{R})$. Consider the operator  $A \colon D(A) \subset X \to X$ given by
		\begin{align*}
			D(A) &= H^1([a,b], \mathbb{R}),  \\
			Ax &= \frac{d}{dz}x, \hspace{0.5 cm} x \in D(A).
		\end{align*}
	We claim that the Hilbert space adjoint and the formal adjoint of the operator $A$ are given by $A_1$ and $A_2$, respectively, with
	\begin{align*}
D(A_1) &= H_0^1([a,b], \mathbb{R}), \\
D(A_2) &= D(A), \\
A_k x &= -  \frac{d}{dz}x, \hspace{0.5 cm} x \in D(A_k), \, k= 1,2. 
\end{align*}
Note that for all $x \in D(A)$, $y \in D(A_1)$ we have
\begin{align*}
	\langle Ax, y \rangle_X = \int_{a}^{b} \frac{d}{dz}x(z) y(z) \, dz = \big[ x(z) y(z) \big]_{a}^{b} - \int_{a}^{b} x(z) \frac{d}{dz}y(z) \, dz = \langle x, A_1 y \rangle_X. 
\end{align*}
Thus, $A_1 \subset A^{\ast}$. In particular, by virtue of the Cauchy-Schwarz inequality, for $y \in D(A_1)$ we have
\begin{align*}
	\sup_{\substack{ x \in D(A) \\
			 \|x\|_X =1 
						 }} \vert \langle Ax, y \rangle_X \vert \leq \sup_{\substack{ x \in D(A) \\
						 \|x\|_X =1  }} \int_{a}^{b} \vert x(z) \vert \left\vert \frac{d}{dz} y(z) \right\vert \, dz \leq 	\sup_{\substack{  x \in D(A) \\
						 \|x\|_X =1  }} \|x \|_X \left\|\frac{d}{dz} y \right\|_X \leq C_y.
	\end{align*}
	 Conversely, suppose that $y \in D(A^{\ast}) \subset X$. Set $\psi = A^{\ast}y$ and define $\Psi = \int_{a}^{\cdot} \psi(r) \, dr \in H^1([a,b], \mathbb{R})$. Then for all $x \in D(A)$ it holds that
\begin{align*}
	\langle A x, y \rangle_X = \langle x, A^{\ast} y \rangle_X = \langle x, \psi \rangle_X	= \langle x, \frac{d}{dz} \Psi \rangle_X = x(b) \Psi(b) - \langle Ax, \Psi \rangle_X. 
\end{align*}
Since this equation holds in particular for all non-zero constant functions, which are obviously contained in $D(A)$, we conclude that $\Psi(b) = 0$. Hence,
\begin{align*}
	\langle A x, y + \Psi \rangle_X = 0
\end{align*}
for all $x \in D(A)$. As a consequence,
\begin{align*}
	y + \Psi \in (\ran (A))^{\bot} = X^{\bot} = \left\{ 0 \right\},
\end{align*}
whence $y = - \Psi$. This implies that $y \in H^1([a,b], \mathbb{R})$, $\trace(y) = 0$, and 
\begin{align*}
	y = - \int_{a}^{\cdot} \psi (r) \, dr = -  \int_{a}^{\cdot} A^{\ast}y(r) \, dr. 
\end{align*} 
By taking the derivative on both sides, we conclude that $A^{\ast}y = - \frac{d}{dz} y$ for all $y \in H_0^1([a,b], \mathbb{R})$. Altogether, the Hilbert space adjoint is given by $A^{\ast} = A_1$. 
\vspace{0.5 cm}\\
To prove that the formal adjoint of $A$ is given by $A_2$, we proceed similarly. Note that for all $x \in H_0^1([a,b], \mathbb{R})$, $y \in D(A_2)$ we have
\begin{align*}
	\langle Ax, y \rangle_X = \langle x, A_2 \rangle_X.
\end{align*}
Conversely, we again suppose that $y \in X$ lies in the domain of the formal adjoint, and define $\psi = A^{\ast}y$ as well as $\Psi = \int_{a}^{\cdot} \psi(r) \, dr$. Since  $x(b) = 0$ for all $x \in H_0^1([a,b], \mathbb{R})$, we find that $y + \Psi \in (\ran (A_{|H_0^1([a,b],\mathbb{R})}))^{\bot}$. Note that $\ran(A_{|H_0^1([a,b],\mathbb{R})})$ consists of all functions $y \in X$ satisfying
\begin{align*}
	\int_{a}^{b} y(z) \, dz = 0. 
\end{align*}
Denoting by $Y$ the one-dimensional subspace of $X$ that contains the constant functions on $[a,b]$, we conclude that
\begin{align*}
	\ran (A_{|H_0^1([a,b],\mathbb{R})}) = Y^{\bot}.
\end{align*}
Hence, $y + \Psi \in Y$. As a result, $y$ is absolutely continuous with $\frac{d}{dz} y \in X$, and therefore lies in $H^1([a,b], \mathbb{R})$. Summing up, the formal adjoint is given by $A_2$. We see that we obtain the domain of the formal adjoint by simply dropping the boundary conditions from the domain of the Hilbert space adjoint.  
\item[(ii)] Let $X = L^2([a,b], \mathbb{R}^2)$. Consider the operator $A \colon D(A) \subset X \to X$ given by
\begin{align*}
	D(A) &=  \left\{ x \in X \mid x \in H^1([a,b], \mathbb{R}^2), \, x(a) = 0 \right\}, \\
	Ax &= \begin{bmatrix}
		0 & -1 \\
		-1 & 0
		\end{bmatrix} \frac{d}{dz} x, \hspace{0.5 cm} x \in D(A). 
\end{align*}
The boundary condition is therefore given by
\begin{align*}
	\begin{bmatrix}
		0 & 0 & 1 & 0 \\
		0 & 0 & 0 & 1
	\end{bmatrix} \trace(x) = 0, \hspace{0.5 cm} x \in D(A). 
\end{align*}
For all $x,y \in D(A)$ it holds that
\begin{align*}
	\langle A x, y \rangle_X &= - \int_{a}^{b} y_1 (z) \frac{d}{dz} x_2(z)   + y_2(z) \frac{d}{dz} x_1(z) \, dz \\
	&= - \big[ x_1(z) y_2(z) + y_1(z) x_2(z) \big]_{a}^{b} + \int_{a}^{b} \frac{d}{dz}y_1(z) x_2(z) + \frac{d}{dz}y_2(z) x_1(z) \, dz \\
	&= - x_1(b) y_2(b) - y_1(b) x_2(b) - \langle x, Ay \rangle_X. 
\end{align*}
Thus, $A$ is not skew-symmetric, but formally skew-symmetric. 
\end{enumerate}
\QEDA
\end{example}
After clarifying the difference between Hilbert space adjoints and formally adjoint operators of matrix differential operators, we are going to introduce the powerful concept of strongly continuous semigroups.

\section{Strongly Continuous Semigroups}
\label{Section Strongly Continuous Semigroups}
The theory of (one-parameter) semigroups of linear operators has a broad variety of applications, and has proven to be a successful tool for the analysis of linear infinite-dimensional (control) systems. For a matrix operator $A \in \mathbb{R}^{n \times n}$, we know that the family $( e^{At})_{t \geq 0}$ of matrix exponentials describes the evolution of the state subject to the linear system
\begin{align*}
	\dot{x}(t) = Ax(t).
\end{align*}
In particular, the mapping $t \mapsto e^{At}$ is differentiable and satisfies
\begin{align*}
	\frac{d}{dt} e^{At} = Ae^{At}.
\end{align*}
The concept of strongly continuous semigroups is commonly used to generalize the above scenario in order to study (linear) systems with values in an arbitrary Hilbert space.
\vspace{0.5 cm}\\
In this section, we give a concise overview concerning strongly continuous semigroups, and collect some facts about the infinitesimal generators of such semigroups as well as some important criteria regarding the generation of strongly continuous semigroups. The material covered in this section can be found in \cite{Arendt}, \cite{CurtainZwart}, \cite{EngelNagel}, \cite{Pazy83}, or \cite{Tucsnak-Weiss}, to name a few.
\subsection{Strongly Continuous Semigroups and their Generators}
We start by introducing the necessary terminology.
\begin{definition}[Semigroup]
	A family $(T(t))_{t \geq 0} \in \mathcal{L}(X)^{[0, \infty)}$ is called a semigroup of bounded linear operators on $X$, or short semigroup, if it satisfies
\begin{enumerate}
	\item[(i)] $T(0) = I_X$, and
	\item[(ii)] the semigroup property:
	\begin{align*}
		T(t+s) = T(t)T(s), \hspace{0.5 cm} t,s \geq 0. 
	\end{align*}
\end{enumerate}
\end{definition}

\begin{definition}
	A semigroup $(T(t))_{t \geq 0}$ on $X$ is called 
	\begin{enumerate}
		\item[(i)] strongly continuous (or a $C_0$-semigroup) if for all $x \in X$, 
		\begin{align*}
			\lim_{t \searrow 0} \|T(t)x - x \|_X = 0, 
		\end{align*}
	\item[(ii)] uniformly continuous if it holds that
	\begin{align*}
		\lim_{t \searrow 0} \|T(t) - I_X \|_{\mathcal{L}(X)} = 0. 
	\end{align*}
	\end{enumerate}
\end{definition}

The strong continuity of $(T(t))_{t \geq 0}$ is equivalent to the mapping $t \mapsto T(t)x$ being continuous for all $x \in X$. Analogously, the mapping $t \mapsto T(t) \in \mathcal{L}(X)$ is continuous provided that $(T(t))_{t \geq 0}$ is uniformly continuous. 
\vspace{0.5 cm}\\
In the following, we restrict ourselves to strongly continuous semigroups, since the presented results hold in particular for uniformly continuous semigroups. An overview of the main differences to strongly continuous semigroups is given in \cite[Section 1.1]{Pazy83}.
\vspace{0.5 cm}\\
An essential property of $C_0$-semigroups is the following. 
\begin{theorem}[Theorem 1.2.2 in \cite{Pazy83}]
	\label{Theorem C0-Semigroup Estimate}
	Let $(T(t))_{t \geq 0}$ be a $C_0$-semigroup on $X$. Then there exist constants $M \geq 1$ and $\omega \in \mathbb{R}$ such that
	\begin{align*}
		\|T(t)\|_{\mathcal{L}(X)} \leq Me^{\omega t}, \hspace{0.5 cm} t \geq 0. 
	\end{align*}
We say that $(T(t))_{t \geq 0}$ is of type $C_0(M, \omega)$.
\end{theorem}

\begin{theorem}[cf. Proposition 2.1.2 in \cite{Tucsnak-Weiss}]
	\label{Theorem T(t)x is continuous}
	For a strongly continuous semigroup $(T(t))_{t \geq 0}$ on $X$, the mapping 
	\begin{align*}
	[0, \infty) \times X \to X, \hspace{0.3 cm} (t,x) \mapsto T(t)x
	\end{align*}
is continuous. 
\end{theorem}

\begin{proof}
	Let $(t_n)_{n \in \mathbb{N}} \in \mathbb{R}_+^{\mathbb{N}}$ and $(x_n)_{n \in \mathbb{N}} \in X^{\mathbb{N}}$ be sequences converging to $t_0 \in \mathbb{R}_+$ and $x_0 \in X$, respectively. Theorem \ref{Theorem C0-Semigroup Estimate} yields 
	\begin{align*}
		\|T(t_n)x_n - T(t_0) x_0 \|_X &\leq \|T(t_n)(x_n - x_0) \|_X + \|T(t_n) x_0 - T(t_0)x_0 \|_X \\
		&\leq Me^{\omega t_n} \|x_n - x_0\|_X + Me^{\omega \min \{ t_n, t_0\}}\|T(\vert t_n - t_0 \vert) x_0  - x_0 \|_X. 
	\end{align*}
Since the sequence $\left( e^{\omega t_n} \right)_{n \in \mathbb{N}}$ is bounded, and since $x_n \to x_0$ by assumption, the first term converges to $0$. The second term converges to $0$ as well, because by the strong continuity of $(T(t))_{ t \geq 0}$ it holds that
\begin{align*}
	\lim_{n \to \infty}T(\vert t_n - t_0 \vert) x_0 = x_0.
\end{align*}
This proves the claim. 
\end{proof}

Theorem \ref{Theorem C0-Semigroup Estimate} gives rise to the notion of the exponential growth bound for $C_0$-semigroups.

\begin{definition}[Growth Bound]
	\label{Definition Growth Bound}
	Let $(T(t))_{t \geq 0}$ be a $C_0$-semigroup on $X$. Its (exponential) growth bound is defined as
	\begin{align*}
		\omega (T(\cdot)) := \inf \left\{ \omega \in \mathbb{R} \mid \exists M_{\omega} \geq 1\colon \forall t \geq 0\colon \, \|T(t)\|_{\mathcal{L}(X)} \leq M_{\omega} e^{\omega t}  \right\}.
	\end{align*}
\end{definition}
In this context, we want to introduce the notion of exponential stability for strongly continuous semigroups, which is one of the most important stability concepts in this framework. For an extensive discussion on stability in general and some well-known characterizations of exponential stability we refer to \cite[Chapter~4]{CurtainZwart}.
\begin{definition}[Exponential Stability]
	\label{Definition Exponential Stability}
	A $C_0$-semigroup $(T(t))_{t \geq 0}$ on $X$ is called exponentially stable if there exist constants $M\geq 1$ and $\alpha > 0$ such that $(T(t))_{t \geq 0}$ is of type $C_0(M, - \alpha)$. 
\end{definition}

Next, we want to discuss what it means for a linear operator $A \colon D(A) \subset X \to X$ to generate a strongly continuous semigroup, and collect some properties of such an operator as well as some relations between a $C_0$-semigroup and its infinitesimal generator.

\begin{definition}[Infinitesimal Generator]
	\label{Definition Infinitesimal Generator}
	Let $(T(t))_{t \geq 0}$ be a $C_0$-semigroup on $X$. The infinitesimal generator $A \colon D(A) \subset X \to X$ of $(T(t))_{t \geq 0}$ is defined by
	\begin{align*}
		D(A) &= \left\{ x \in X \mid \lim_{t \searrow 0} \frac{T(t) - I}{t}x \, \text{exists} \right\}, \\
		Ax &= \lim_{t \searrow 0} \frac{T(t) - I}{t}x, \hspace{0.5 cm} x \in D(A). 
	\end{align*}
\end{definition}

\begin{theorem}[Theorem 1.2.4 in \cite{Pazy83}]
	\label{Theorem Properties of C0-Semigroups}
	Let  $A \colon D(A) \subset X \to X$ be the infinitesimal generator of a $C_0$-semigroup $(T(t))_{t \geq 0}$ on $X$. Then the following assertions are true:
	
	\begin{enumerate}[label = (\roman*)]
		\item \label{C0SemigroupProp1} For all $x \in X$ and $t \geq 0$ it holds that
		\begin{align*}
			\lim_{h \searrow 0} \frac{1}{h} \int_{t}^{t + h} T(\tau)x \, d\tau = T(t)x. 
		\end{align*} 
	\item \label{C0SemigroupProp2} We have $\int_{0}^{t} T(\tau) x \, d\tau \in D(A)$ for all $x \in X$, $ t \geq 0$, with
	\begin{align*}
		A \left( \int_{0}^{t} T(\tau) x \, d\tau \right) = T(t)x - x.
	\end{align*}
\item \label{C0SemigroupProp3} It holds that $T(t) D(A) \subset D(A)$ for all $t \geq 0$. Moreover, for all $x \in D(A)$ we have
\begin{align*}
	\frac{d}{dt}T(t)x  = AT(t)x = T(t)Ax. 
\end{align*}
\item \label{C0SemigroupProp4} For all $x \in D(A)$ and for all $0 \leq \tau \leq t$ we have
\begin{align*}
	T(t)x - T(\tau)x = \int_{\tau}^{t} T(r) Ax \, dr = \int_{\tau}^{t} AT(r) x \, dr. 
\end{align*}
	\end{enumerate}
\end{theorem}

\begin{proof}
	Ad \ref{C0SemigroupProp1}: By continuity of the mapping $t \mapsto T(t)x$ we have for all $ x\in X$, $t \geq 0$,  
	\begin{align*}
		\hspace{0.5 cm} \lim_{h \searrow 0} \left\| \frac{1}{h} \int_{t}^{t+h} T(\tau) x \, d\tau - T(t)x \right\|_X	&= \lim_{h \searrow 0} \left\| \frac{1}{h} \int_{t}^{t+h} T(\tau) x - T(t) x \, d\tau \right\|_X \\
		&\leq \lim_{h \searrow 0} \frac{1}{h} h \sup_{\tau \in [t, t+ h]} \| T(\tau)x - T(t)x \|_X \\
	&= 0. 
	\end{align*}
	Ad \ref{C0SemigroupProp2}: Let $h > 0$. As $T(t) \in \mathcal{L}(X)$, it is closed for all $t \geq 0$. Thus, we have for all $x \in X$,
	\begin{align*}
		\frac{1}{h}(T(h) - I) \int_{0}^{t} T(\tau) x \, d\tau &= \frac{1}{h} \int_{0}^{t} T(\tau + h) x \, d\tau - \frac{1}{h} T(\tau) x \, d\tau \\
		&= \frac{1}{h} \int_{t}^{t +h} T(\tau) x  \, d\tau - \frac{1}{h} \int_{0}^{h} T(\tau)x \, d\tau. 
	\end{align*}
By applying property \ref{C0SemigroupProp1}, this expression converges to $T(t)x - x$ as $h \searrow 0$. This shows \ref{C0SemigroupProp2}. 
	\vspace{0.5 cm}\\
	Ad \ref{C0SemigroupProp3}: Let $x \in D(A)$ and $h > 0$. Then
	\begin{align*}
		\frac{1}{h} (T(h) - I)T(t)x  = T(t) \left( \frac{T(h) - I}{h} \right)x \to T(t) A x
	\end{align*}
as $h \searrow 0$. Thus, by definition, $T(t)x \in D(A)$ with $AT(t)x = T(t)Ax$ for all $t \geq 0$.  Moreover, we have for the right-hand derivative that
\begin{align*}
	\frac{d^+}{dt}T(t)x = AT(t)x = T(t)Ax. 
\end{align*}
Now, let $t,h> 0$ such that $t - h \geq 0$. Then
\begin{align*}
	&\hspace{0.48 cm} \left\| \frac{1}{h} (T(t)x - T(t-h)x) - T(t)Ax \right\|_X \\
	&=  \left\| \frac{1}{h} (T(t)x - T(t-h)x) + T(t-h)Ax - T(t-h)Ax - T(t)Ax \right\|_X \\
	&\leq \|T(t-h)\|_{\mathcal{L}(X)}  \left\| \frac{1}{h} (T(h)x - x) - Ax \right\|_X + \|T(t-h)Ax  - T(t)Ax \|_X
\end{align*}
By strong continuity and the fact that $A$ generates $(T(t))_{t \geq 0}$, taking the limit $h\searrow 0$ yields the same relation for the left-hand derivative.
	\vspace{0.5 cm}\\
	Ad \ref{C0SemigroupProp4}: This follows from property \ref{C0SemigroupProp3} by integration over $[\tau,t]$.  
\end{proof}
From these properties one infers:
\begin{corollary}[Corollary 1.2.5 in \cite{Pazy83}]
	\label{Corollary - Generators Are Closed and Densely Defined}
	If $A \colon D(A) \subset X \to X$ is the infinitesimal generator of a strongly continuous semigroup $(T(t))_{t \geq 0}$ on $X$, then it is densely defined and closed. 
\end{corollary}

Lastly, we note that the generator of a strongly continuous semigroup is uniquely determined.

\begin{theorem}[Theorem 1.2.6 in \cite{Pazy83}]
Let $(T(t))_{t \geq 0}$ and $(S(t))_{t \geq 0}$ be $C_0$-semigroups on $X$ with infinitesimal generators $A$ and $B$, respectively. If $A = B$, then
\begin{align*}
	T(t) = S(t), \hspace{0.5 cm } t \geq 0. 
\end{align*}
\end{theorem}

\subsection{Spectrum and Resolvent of Infinitesimal Generators}
In the following, we want to characterize infinitesimal generators of $C_0$-semigroups. To this end, we introduce the notion of the resolvent operator, and collect some results relevant for the proof of the aforementioned characterization. 
\begin{definition}[Resolvent]
	\label{Definition Resolvent}
	Let $A \colon D(A) \subset X \to X$ be a linear operator. The resolvent set $\rho(A)$ is the set of elements $\lambda \in \mathbb{C}$ which satisfy the following:
	\begin{enumerate}
		\item[(i)] $\overline{(\lambda I - A)(D(A))} = X$. 
		\item[(ii)] $(\lambda I - A)^{-1}$ exists and is bounded in
		\begin{align*}
			D((\lambda I - A)^{-1}) := (\lambda I - A)(D(A))
		\end{align*} 
	with respect to the norm on $X$ induced by $(\lambda I - A)(D(A))$. 
	\end{enumerate} 
If $\rho(A) \neq \emptyset$, for $\lambda \in \rho(A)$ we call
\begin{align*} 
	R(\lambda, A) := (\lambda I - A)^{-1} \in \mathcal{L}(X)
\end{align*}
the resolvent operator of $A$ in $\lambda$. The spectrum of $A$ is
\begin{align*}
	\sigma(A) := \mathbb{C} \setminus \rho(A).
\end{align*}
\end{definition}

\begin{definition}
		Let $A \colon D(A) \subset X \to X$ be a linear operator. The spectrum of $A$ is the disjoint union of the following sets:
		\begin{enumerate}[label = (\roman*)]
			\item The set
			\begin{align*}
				\sigma_p(A) := \left\{ \lambda \in \mathbb{C} \mid (\lambda I - A) \text{ is not injective} \right\}
			\end{align*}
			is called the point spectrum of $A$. Its elements are called eigenvalues, and an element $x_{\lambda} \in D(A)$ with $Ax_{\lambda} = \lambda x_{\lambda}$ is called an eigenfunction with eigenvalue $\lambda \in \sigma_p(A)$.
			\item The set
			\begin{equation*}
				\begin{aligned}
					\sigma_c(A) := & \left\{ \lambda \in \mathbb{C} \mid \overline{(\lambda I - A) (D(A))} =X, \right.                               \\
					& \qquad \hspace{0.7cm} \left. (\lambda I - A) \text{ is injective}, \, (\lambda  I - A)^{-1} \text{ is unbounded} \right\}
				\end{aligned}
			\end{equation*}
			is called the continuous spectrum of $A$.
			\item The set
			\begin{align*}
				\sigma_r(A) := \left\{ \lambda \in \mathbb
				C \mid \overline{(\lambda I - A) (D(A))} \neq X, \, (\lambda I - A) \text{ is injective} \right\}
			\end{align*}
			is called the residual spectrum of $A$.
			\end{enumerate}
\end{definition}

\begin{proposition}[Proposition 2.8.4 in \cite{Tucsnak-Weiss}]
	\label{Proposition 2.8.4 in Tucsnak}
	Let $A \colon D(A) \subset X \to X$ be linear and densely defined, and let $\lambda \in \rho(A)$. Then $\overline{\lambda} \in \rho(A^{\ast})$ and
	\begin{align*}
		\left[ (\lambda I - A)^{-1} \right]^{\ast} = (\overline{\lambda}I - A^{\ast})^{-1}. 
	\end{align*}
\end{proposition}
An important feature of infinitesimal generators is that the corresponding resolvent operator can be written as the \emph{Laplace transform} of the respective $C_0$-semigroup.

\begin{theorem}[Proposition 2.3.1 in \cite{Tucsnak-Weiss}]
	\label{Theorem Laplace Transform}
	Let $(T(t))_{t \geq 0}$ be a $C_0$-semigroup on $X$. Let $A \colon D(A) \subset X \to X$ be its infinitesimal generator. Then the following assertions are true:
	\begin{enumerate}
		\item[(i)] $\mathbb{C}_{\omega(T(\cdot))} = \left\{ z \in \mathbb{C} \mid \realpart(z)  > \omega(T(\cdot)) \right\} \subset \rho(A)$.
		\item[(ii)] For all $x\in X$, $\lambda \in \mathbb{C}_{\omega(T(\cdot))}$ it holds that
		\begin{align*}
			(\lambda I - A)^{-1} x = \int_{0}^{\infty} e^{-\lambda t}T(t)x \, dt. 
		\end{align*}
	\end{enumerate}
\end{theorem}

\begin{proof}
	If $\alpha = \realpart(\lambda) > \omega(T(\cdot))$, then there are some constants $M \geq1$ and $\epsilon > 0$ such that
	\begin{align*}
		\|T(t) \|_{\mathcal{L}(X)} \leq M e^{(\alpha - \epsilon)t}, \qquad t \geq 0.
	\end{align*}
	Thus, the improper Riemann integral exists, since the integrand converges exponentially to $0$ as $t \to \infty$. Now, define
	\begin{align*}
		Z(\lambda)x := \int_{0}^{\infty} e^{-\lambda t} T(t)x \, dt, \qquad x \in X, \, \lambda \in \mathbb{C}_{\omega(T(\cdot))}.
	\end{align*}
	Then we have for $\lambda \in \mathbb{C}_{\omega(T(\cdot))}$ that
	\begin{align*}
		\|Z(\lambda) x \|_X \leq \int_{0}^{\infty} e^{- (\realpart(\lambda))t} M e^{(\realpart(\lambda)- \epsilon)t} \|x\|_X \, dt = \frac{M}{\epsilon} \|x\|_X.
	\end{align*}
	Thus, $Z(\lambda) \in \mathcal{L}(X)$. Furthermore, by substituting $\eta = t+h$ and by Theorem~\ref{Theorem Properties of C0-Semigroups}~\ref{C0SemigroupProp1}, we have
	\begin{align*}
		\frac{1}{h} \left(T(h) - I \right) Z(\lambda) x & = \frac{1}{h} \int_{0}^{\infty} e^{-\lambda t} \left( T(t+h) - T(t) \right) x \, dt                                                                         \\
		& = \frac{e^{\lambda h}}{h} \int_{0}^{\infty} e^{-\lambda \eta} T(\eta)x \, d\eta - \frac{e^{\lambda h}}{h} \int_{0}^{h} e^{-\lambda \eta} T(\eta) x \, d\eta \\
		& \qquad - \frac{1}{h} \int_{0}^{\infty} e^{-\lambda \eta} T(\eta) x \, d\eta                                                                                 \\
		& = \frac{e^{\lambda h}-1}{h} \int_{0}^{\infty} e^{-\lambda t} T(t) x \, dt - \frac{e^{\lambda h}}{h} \int_{0}^{h} e^{-\lambda t} T(t)x \, dt                 \\
		& \to \lambda Z(\lambda) x - x
	\end{align*}
	as $h \to 0$. We conclude that $ Z(\lambda)x \in D(A)$ with
	\begin{align*}
		AZ(\lambda) x = \lambda Z(\lambda) x - x.
	\end{align*}
	Hence, $(\lambda I - A) Z(\lambda) x = x$ for all $x \in X$. As $A$ is closed, for all $x \in D(A)$ the following holds:
	\begin{align*}
		Z(\lambda) A x & = \int_{0}^{\infty} e^{-\lambda t} T(t)Ax \, dt                  \\
		& = \int_{0}^{\infty} e^{-\lambda t} A T(t)x \, dt                 \\
		& = A \left( \int_{0}^{\infty} e^{-\lambda t} T(t)x \, dt  \right) \\
		& = AZ(\lambda) x.
	\end{align*}
	This yields
	\begin{align*}
		Z(\lambda) (\lambda I -A) x = (\lambda I - A)Z(\lambda) x = x.
	\end{align*}
	In conclusion, we have that $Z(\lambda) = (\lambda I -A)^{-1}$, and so $\lambda \in \rho(A)$.
\end{proof}

In the following, we want to present some characteristics of infinitesimal generators of strongly continuous semigroups. To this end, we need to introduce the \emph{Yosida approximation} of a closed and densely defined operator $A \colon D(A) \subset X \to X$. Assume that there exist constants $M > 0$ and $\omega > 0$ such that $(\omega, \infty) \subset \rho(A)$ and such that for all $n \in \mathbb{N}$, 
\begin{align*}
	\|R^n(\lambda,A) \|_{\mathcal{L}(X)} \leq \frac{M}{(\lambda - \omega)^n}, \hspace{0.5 cm} \lambda > \omega.
\end{align*}
If so, it has been shown that for every $\lambda \in \rho(A)$, the Yosida approximation defined as
\begin{align*}
	A_{\lambda} := \lambda^2 R(\lambda,A) - \lambda I = \lambda AR(\lambda,A) \in \mathcal{L}(X)
\end{align*}
has some useful properties which are exploited in order to prove the necessity of the following characterization of infinitesimal generators of $C_0$-semigroups. The sufficiency flollows from Corollary \ref{Corollary - Generators Are Closed and Densely Defined} and Theorem \ref{Theorem Laplace Transform}.
\begin{theorem}[Hille-Yosida, cf. Theorem II.3.8 in \cite{EngelNagel}]
	\label{Theorem Hille-Yosida - C0 Semigroup}
	A linear operator $A \colon D(A) \subset X \to X$ is the infinitesimal generator of a $C_0$-semigroup on $X$ of type $C_0(M, \omega)$ if and only if the following holds:
	\begin{enumerate}
		\item[(i)] $A$ is closed and densely defined.
		\item[(ii)] $(\omega, \infty) \subset \rho(A)$.
		\item[(iii)] For all $\lambda > \omega$ and for all $n \in \mathbb{N}$ one has
		\begin{align*}
			\|R^n(\lambda, A) \|_{\mathcal{L}(X)} \leq \frac{M}{(\lambda - \omega)^n}.
		\end{align*}
	\end{enumerate} 
\end{theorem}

A similar result can be proved for an important class of strongly continuous semigroups, called contraction semigroups. They are especially interesting in the port-Hamiltonian framework of infinite-dimensional systems. 

\begin{definition}[Contraction Semigroup]
	\label{Definition Contraction Semigroup}
	Let $(T(t))_{t \geq0}$ be a $C_0$-semigroup on $X$.
	\begin{enumerate}[label = (\roman*)]
	\item It is called a contraction semigroup, if it is of type $C_0(1,0)$, that is,
	\begin{align*}
		\|T(t)\|_{\mathcal{L}(X)} \leq 1, \hspace{0.5 cm} t \geq 0.
	\end{align*}
\item It is called a unitary semigroup, if for all $x \in X$ it holds that
\begin{align*}
	\|T(t)x \|_X = \|x\|_X, \hspace{0.5 cm} t \geq 0.
	\end{align*}
\end{enumerate}
\end{definition}
\begin{theorem}[Hille-Yosida, Theorem 1.3.1 in \cite{Pazy83}]
	\label{Theorem Hille-Yosida - Contraction Semigroup}
	A linear operator $A \colon D(A) \subset X \to X$ is the infinitesimal generator of a contraction semigroup on $X$ if and only if the following holds:
	\begin{enumerate}
		\item[(i)] $A$ is closed and $\overline{D(A)} = X$. 
		\item[(ii)] $(0, \infty) \subset \rho(A)$, and for every $\lambda > 0$, 
		\begin{align*}
			\|R(\lambda,A) \|_{\mathcal{L}(X)} \leq \frac{1}{\lambda}.
		\end{align*}
\end{enumerate}
\end{theorem}

There is another important characterization for when a linear operator generates a contraction semigroup. To this end, we need to clarify what it means that a linear operator is dissipative. 

\begin{definition}[Dissipative Operator]
	\label{Definition Dissipative Operator}
	A linear operator $A \colon D(A) \subset X \to X$ is called dissipative if for all $x \in D(A)$, 
	\begin{align*}
		\realpart \langle Ax, x \rangle_X \leq 0.
	\end{align*}
\end{definition}
The dissipativity of a linear operator can be characterized in the following way. 

\begin{lemma}[Proposition 6.1.5 in \cite{JacobZwart}]
	\label{Lemma Characterization Dissipativity}
A linear operator $A \colon D(A) \subset X \to X$ is dissipative if and only if for all $x \in D(A)$, $\lambda  >0$ it holds that
\begin{align*}
	\| (\lambda I - A)x \|_X \geq \lambda \| x \|_X. 
\end{align*}
	\end{lemma}

\begin{lemma}[Lemma 6.1.6 in \cite{JacobZwart}]
	\label{Lemma Auxiliary Lemma Lumer Phillips}
	Let $A \colon D(A) \subset X \to X$ be a closed, dissipative operator. Then $\ran (\lambda I - A)$ is closed for all $\lambda > 0$. 
\end{lemma}

Applying Theorem \ref{Theorem Hille-Yosida - Contraction Semigroup} as well as Lemma \ref{Lemma Characterization Dissipativity} and Lemma \ref{Lemma Auxiliary Lemma Lumer Phillips}, one can prove that generators of contraction semigroups can be characterized as follows. 

\begin{theorem}[Lumer-Phillips, Theorem 6.1.7 in \cite{JacobZwart}]
	\label{Theorem Lumer-Phillips}
	A linear operator $A \colon D(A) \subset X \to X$ is the infinitesimal generator of a contraction semigroup on $X$ if and only if the following holds:
	\begin{enumerate}
		\item[(i)] $A$ is dissipative.
		\item[(ii)] $\ran(I-A) = X$.
	\end{enumerate}
\end{theorem}

Now that we have examined some properties of strongly continuous semigroups and their infinitesimal generators, we can finally deal with solving Hilbert space-valued initial value problems. It turns out that $C_0$-semigroups are a vital tool for the analysis of infinite-dimensional systems. 

\section{The Abstract Cauchy Problem}
\label{Section The Abstract Cauchy Problem}
In this section, we introduce the abstract Cauchy problem, i.e., a linear inital value problem with values in a Hilbert space. We will clarify the notion of a solution of such an abstract differential equation, and characterize well-posedness of this class of systems. Again, this is standard material and is covered in \cite[Chapter 5]{CurtainZwart}, \cite[Section II.6]{EngelNagel}, \cite[Chatper 4]{Pazy83}.
\subsection{Linear Homogeneous Initial Value Problems}
\label{Subsection Linear Homogeneous Initial Value Problems}
We begin with the definition of an abstract Cauchy problem.
\begin{definition}[Abstract Cauchy Problem]
	The linear initial value problem 
	\begin{align}
		\begin{split}
		\dot{x}(t) &= Ax(t), \hspace{0.5 cm} t > 0, \\
		x(0) &= x_0 \in X,
		\end{split}
	\label{General Abstract Cauchy Problem}
	\end{align}
with $A \colon D(A) \subset X \to X$ a linear operator, $x_0$ the initial value and $x \colon [0, \infty) \to X$ a time-dependent, Hilbert space-valued function is called an abstract Cauchy problem. The function $x$ is called a (classical) solution, if the following holds:
\begin{enumerate}
	\item[(i)] $x \in \mathcal{C}([0, \infty), X)$.
	\item[(ii)] $x_{| (0,\infty)} \in \mathcal{C}^1((0, \infty) , X)$. 
	\item[(iii)] $x(t) \in D(A)$ for all $ t > 0$.
	\item[(iv)] $x$ satisfies \eqref{General Abstract Cauchy Problem}. 
\end{enumerate}
\end{definition}
We need to agree on a concept of well-posedness for the abstract Cauchy problem \eqref{General Abstract Cauchy Problem}. Just as in the finite-dimensional case, we require for its solutions existence, uniqueness, and continuous dependence on the initial data. This guides us to the following definition.  
\begin{definition}[Well-posedness]
	\label{Definition Well-posedness of an Abstract Cauchy Problem}
	The abstract Cauchy problem \eqref{General Abstract Cauchy Problem} is called well-posed, if the following assertions are true:
	\begin{enumerate}
		\item[(i)] The operator $A$ is densely defined.
		\item[(ii)] For all $x_0 \in D(A)$, there exists a unique classical solution $x(\cdot;x_0) \colon [0, \infty) \to X$.
		\item[(iii)] The classical solution depends continuously on the initial value, that is, for all $\epsilon > 0$, $\tau > 0$, and $x_0 \in D(A)$ there exists some $\delta > 0$ such that for all $x_1 \in D(A)$ the following holds:
		\begin{align*}
			\| x_0 - x_1 \|_X \leq \delta \hspace{0.3 cm} \Rightarrow \hspace{0.3 cm} \sup_{t \in [0,\tau]} \| x(t; x_0) - x(t;x_1) \|_X \leq \epsilon. 
		\end{align*}
	\end{enumerate}
\end{definition}

There are weaker notions of well-posedness as well, but in this thesis, we work with the concept suggested in Definition \ref{Definition Well-posedness of an Abstract Cauchy Problem}. 
\vspace{0.5 cm}\\
The following theorem shows that if the operator $A$ associated with \eqref{General Abstract Cauchy Problem} is the infinitesimal generator of a $C_0$-semigroup, then one can easily construct a classical solution. 
\begin{theorem}[Proposition II.6.2 in \cite{EngelNagel}]
	\label{Theorem A Generator yields Unique Solution}
	Let $A \colon D(A) \subset X \to X$ be the infinitesimal generator of a $C_0$-semigroup $(T(t))_{t \geq 0}$ on $X$. Then, for every initial value $x_0 \in D(A)$, the function 
	\begin{align}
		x = x(\cdot; x_0) \colon [0, \infty) \to X, \hspace{0.3 cm} t \mapsto T(t)x_0,
		\label{Theorem Unique Classical Solutiono of ACP - Unique Solution}
	\end{align}
	is the unique classical solution of \eqref{General Abstract Cauchy Problem}, with $x$ depending continuously on the initial data. 
\end{theorem}
\begin{proof}
	The proof that $x$ defined in \eqref{Theorem Unique Classical Solutiono of ACP - Unique Solution} is a classical solution follows immediately from property (iii) of Theorem \ref{Theorem Properties of C0-Semigroups}. Next, we want to show that $x$ is the only solution of \eqref{General Abstract Cauchy Problem}.
	\vspace{0.5 cm}\\
	Assume that $y \colon [0, \infty) \to X$ solves \eqref{General Abstract Cauchy Problem} as well. For $t > 0$, define the mapping
	\begin{align*}
		\psi \colon \tau \mapsto T(t - \tau)y(\tau), \hspace{0.5 cm} \tau \in [0,t]. 
	\end{align*}
By assumption, $y$ is continuously differentiable and $y(\tau) \in D(A)$ for all $\tau \in (0,t)$. Thus, for all $\tau, h > 0$ such that $\tau + h < t$ it holds that
\begin{align*}
	&\hspace{0.5 cm} \frac{\psi(\tau + h) - \psi(\tau)}{h} \\
	 &= \frac{T(t - \tau - h) y(\tau + h) - T(t - \tau) y(\tau)}{h} \\
	&=  \frac{T(t - \tau - h) y(\tau + h) - T(t - \tau - h)y(\tau) + T(t - \tau - h) y(\tau) - T(t - \tau) y(\tau)}{h} \\
	&= T(t - \tau - h) \frac{y(\tau +h) - y(\tau)}{h} + \frac{T(t - \tau - h )y(\tau) - T(t - \tau)y(\tau)}{h}.
\end{align*} 
Again, since $y(\tau) \in D(A)$ for all $\tau \in (0,t)$, the second term converges to $- AT(t-\tau)y(s)$ as $h \searrow 0$. For the first term, we may find the limit by writing
\begin{align*}
	 T(t - \tau - h) \frac{y(\tau +h) - y(\tau)}{h} &= T(t - \tau) \frac{y(\tau + h) - y(\tau)}{h} \\
	 &\hspace{0.5 cm} + \left[ T(t - \tau - h) - T(t - \tau) \right] \left( \frac{y(\tau + h) - y(\tau)}{h} - \frac{d}{d\tau}y(\tau)   \right) \\
	 &\hspace{0.5 cm} + \left[ T(t - \tau - h) - T(t - \tau) \right]  \frac{d}{d\tau}y(\tau).
\end{align*} 
Now, since $(T(t))_{t \geq 0}$ is strongly continuous, $T(t - \tau - h) - T(t - \tau)$ is uniformly bounded on compact (time) intervals. Consequently, the second term tends to zero as $h \searrow 0$. By virtue of the strong continuity, the same holds for the last expression. Hence, this sum converges to $T(t - \tau) \frac{d}{d \tau}y(\tau)$ as $h \searrow 0$. Altogether, it holds that
\begin{align*}
	\frac{d}{d \tau}  \psi(\tau) &= T(t - \tau) \frac{d}{d\tau} y(\tau) - A T(t - \tau) y(\tau) \\
	&= T(t - \tau) Ay(\tau) - AT(t - \tau) y(\tau) \\
	&=0. 
\end{align*}
The last equation follows from Theorem \ref{Theorem Properties of C0-Semigroups} \ref{C0SemigroupProp3}. Integration yields that
\begin{align*}
	0 = \int_{0}^{t} \frac{d}{d\tau} (T(t- \tau)y(\tau)) \, d\tau = y(t) - T(t)x_0,
\end{align*}
whence $y(t) = T(t)x_0 = x(t)$ for all $t \geq 0$. This proves the uniqueness of the (classical) solution of \eqref{General Abstract Cauchy Problem}. 
\vspace{0.5 cm}\\
Lastly, we need to show the continuous dependence of the solution on the initial data. As $(T(t))_{t \geq 0}$ is strongly continuous, it is of type $C_0(M, \omega)$ for some $M \geq 1$, $\omega \in \mathbb{R}$ (see Theorem~\ref{Theorem C0-Semigroup Estimate}). Thus, we have for $\tau > 0$ and $x_0,x_1 \in D(A)$, 
\begin{align*}
	\sup_{t \in [0,\tau]} \|x(t; x_0) - x(t; x_1) \|_X = \sup_{t \in [0,\tau]} \|T(t) (x_0 - x_1) \|_X \leq Me^{\vert \omega\vert \tau} \| x_0 - x_1 \|_X. 
\end{align*}
This shows that the solution is continuously depending on the initial value and, hence, completes the proof. 
\end{proof}
Theorem \ref{Theorem A Generator yields Unique Solution} shows in particular that it is sufficient for well-posedness of the abstract Cauchy problem \eqref{General Abstract Cauchy Problem} that the associated operator $A$ is the infinitesimal generator of a $C_0$-semigroup. In fact, it is necessary as well. However, we will not prove the necessity and refer to \cite[Theorem II.6.7]{EngelNagel}. 

\begin{theorem}[Corollary II.6.9 in \cite{EngelNagel}]
	\label{Theorem Characterization of Well-Posedness}
	Let $A \colon D(A) \subset X \to X$ be a linear and closed operator. Then the initial value problem \eqref{General Abstract Cauchy Problem} is well-posed in the sense of Definition \ref{Definition Well-posedness of an Abstract Cauchy Problem} if and only if $A$ is the infinitesimal generator of a strongly continuous semigroup $(T(t))_{t \geq 0}$ on $X$. If so, then for any $x_0 \in D(A)$, the unique classical solution is given by \eqref{Theorem Unique Classical Solutiono of ACP - Unique Solution}.
\end{theorem}

Theorem \ref{Theorem Characterization of Well-Posedness} emphasizes the importance of the concept of strongly continuous semigroups, since they are well-suited for the study of abstract Cauchy problems. The starting point for the analysis of abstract Cauchy problems is therefore to check whether the associated operator $A$ generates a strongly continuous semigroup. In this context, let us quickly come back to dissipative operators, and consider the following example.
\begin{example}
	\label{Example Contraction Semigroup}
 Consider the abstract initial value problem \eqref{General Abstract Cauchy Problem} associated with an operator $A \colon D(A) \subset X \to X$ generating a contraction semigroup $(T(t))_{t \geq 0}$, and with initial value $x_0 \in D(A)$. Then, by Theorem \ref{Theorem Characterization of Well-Posedness}, the system has the unique solution $x = T(\cdot)x_0$. Moreover, according to Theorem \ref{Theorem Lumer-Phillips}, $A$ is dissipative. This yields for all $t > 0$, 
\begin{align*}
	\frac{d}{dt} \|x(t)\|_X^2 &= \langle \dot{x}(t), x(t)  \rangle_X + \langle x(t), \dot{x}(t) \rangle_X \\
	&= \langle Ax(t), x(t) \rangle_X + \langle x(t), Ax(t) \rangle_X \\
	&= 2 \realpart \langle Ax(t), x(t) \rangle_X \\
	&\leq 0. 
\end{align*}
Integration on both sides over $[0, \tau]$ yields $\| x(\tau) \|_X^2 \leq \| x(0)\|_X^2$ for all $\tau \geq 0$. In many physical problems, the value $\|x\|_X^2$, $x \in X$, represents the energy level of the state $x$. Thus, for solutions of abstract Cauchy problems associated with a dissipative operator, the energy of the corresponding solution is dissipating. As we will see, this property is essential in the port-Hamiltonian framework of infinite-dimensional systems. 
\QEDA
\end{example}

\subsection{Linear Inhomogeneous Initial Value Problems}

As a preparation for the analysis of abstract control systems, in this section, we shortly discuss inhomogeneous  initial value problems of the following form. 
\begin{definition}[Inhomogeneous Abstract Cauchy Problem]
	The initial value problem 
	\begin{align}
		\begin{split}
			\dot{x}(t) &= Ax(t) + f(t), \hspace{0.5 cm} t > 0, \\
			x(0) &= x_0 \in X,
		\end{split}
		\label{General Abstract Cauchy Problem - Inhomogeneous}
	\end{align}
	with $A \colon D(A) \subset X \to X$ a linear operator, $x_0$ the initial value, and $f \colon [0, \infty) \to X$, $x \colon [0, \infty) \to X$ time-dependent, $X$-valued functions is called an inhomogeneous abstract Cauchy problem. Let $\tau > 0$. The function $x$ is called a (classical) solution on $[0, \tau)$, if the following holds:
	\begin{enumerate}
		\item[(i)] $x_{|[0,\tau]} \in \mathcal{C}([0, \tau], X)$.
		\item[(ii)] $x_{| (0,\tau)} \in \mathcal{C}^1((0, \tau) , X)$. 
		\item[(iii)] $x(t) \in D(A)$ for all $ t \in [0,\tau]$.
		\item[(iv)] $x$ satisfies \eqref{General Abstract Cauchy Problem - Inhomogeneous} on $(0,\tau)$. 
	\end{enumerate}
\end{definition} 

From the discussion in Subsection \ref{Subsection Linear Homogeneous Initial Value Problems} it makes sense to assume that the operator $A$ associated with \eqref{General Abstract Cauchy Problem - Inhomogeneous} generates a $C_0$-semigroup, so that the equation with $f \equiv 0$ has a uniquely determined classical solution. 
\vspace{0.5 cm}\\
We want to show first that if \eqref{General Abstract Cauchy Problem - Inhomogeneous} has a solution, then it has a particular form. More precisely, it is given by the \emph{variation of constants} formula. 
\begin{lemma}[cf. Corollary 4.2.2 in \cite{Pazy83}]
	\label{Lemma Mild Solution}
	Let $A \colon D(A) \subset X \to X$ be the infinitesimal generator of a $C_0$-semigroup $(T(t))_{t \geq 0}$ on $X$. Let $f \in \mathcal{C}([0, \tau], X)$, $x_0 \in D(A)$, and assume that $x \in \mathcal{C}^1([0,\tau], X)$ is a classical solution of \eqref{General Abstract Cauchy Problem - Inhomogeneous} on $[0, \tau)$. Then $x$ is of the form
	\begin{align}
		\label{Mild Solution}
		x(t) = T(t)x_0 + \int_{0}^{t} T(t-s) f(s) \, ds,  \hspace{0.5 cm} t \in [0,\tau].
	\end{align}
\end{lemma}
\begin{proof}
	For $ 0 < t < \tau$, define the mapping
	\begin{align*}
		\psi \colon s \mapsto T(t - s)x(s), \hspace{0.5 cm} s \in [0, t]. 
	\end{align*} 
With similar arguments as in the proof of Theorem \ref{Theorem A Generator yields Unique Solution} one can show that $\psi \in \mathcal{C}^1((0,t), X)$, and its derivative is given by
\begin{align*}
	\frac{d}{ds} \psi(s) = T(T-s)\dot{x}(s) - AT(t-s)x(s), \hspace{0.5 cm} s \in (0,t).
\end{align*}
Using Theorem \ref{Theorem Properties of C0-Semigroups} \ref{C0SemigroupProp3} and the fact that $x$ is a classical solution, we find
\begin{align}
	\label{xyz}
	\frac{d}{ds}\psi(s) = T(t-s)(Ax(s) + f(s)) - T(t-s)Ax(s) = T(t-s)f(s). 
\end{align}
Due to the continuity of $f$, the mapping $s \mapsto T(t-s)f(s)$ is integrable. Thus, by integration of \eqref{xyz} and by definition of $\psi$, we obtain
\begin{align*}
x(t) - T(t)x(0) =\psi(t) - \psi(0) = \int_{0}^{t} T(t-s)f(s) ds.
\end{align*}
In conclusion, $x$ is of the form \eqref{Mild Solution}. This proves the claim.
\end{proof}
The map \eqref{Mild Solution} can be defined for any $x_0 \in X$ and any $f \in L^1([0, \tau], X)$. In this case, it is an element of $\mathcal{C}([0,\tau], X)$ and is called the \emph{mild solution} of \eqref{General Abstract Cauchy Problem - Inhomogeneous}. It is clear that, in general, the mild solution \eqref{Mild Solution} is not a classical solution of \eqref{General Abstract Cauchy Problem - Inhomogeneous}. Furthermore, the continuity of $f$ we required in Lemma \ref{Lemma Mild Solution} is not sufficient to guarantee that there is a classical solution, even if $x_0 \in D(A)$. To see this, let $(T(t))_{t \geq 0}$ be a $C_0$-semigroup with generator $A$ and the property that there exist some $\hat{t} > 0$ and $\hat{x} \in X$ such that $T(\hat{t}) \hat{x} \notin D(A)$. Now, consider the initial value problem
\begin{align*}
	\dot{x}(t) &= Ax(t) + T(t)\hat{x}, \hspace{0.5 cm} t> 0, \\
	x(0) &= 0. 
\end{align*}
One may show that the mild solution is given by $t \mapsto tT(t)\hat{x}$, which is not differentiable on $(0, \infty)$, since $\hat{t}T(\hat{t})\hat{x} \notin D(A)$. Therefore, by Lemma \ref{Lemma Mild Solution}, for this system there does not exist a classical solution at all. 
\vspace{0.5 cm}\\
In the following, we want to present criteria that guarantee the existence of classical solutions of \eqref{General Abstract Cauchy Problem - Inhomogeneous} for any initial value $x_0 \in D(A)$. 
\begin{theorem}[Theorem 4.2.4 in \cite{Pazy83}]
	\label{Theorem Criteria Existence of IACP}
	Let $A \colon D(A) \subset X \to X$ be the infinitesimal generator of a $C_0$-semigroup $(T(t))_{t \geq 0}$ on $X$. Let $f \in L^1([0,\tau], X) \cap \mathcal{C}((0,\tau], X)$ and define
	\begin{align*}
		V(t) = \int_{0}^{t} T(t-s)f(s) \, ds, \hspace{0.5 cm} t \in [0, \tau]. 
	\end{align*}
Then \eqref{General Abstract Cauchy Problem - Inhomogeneous} has a classical solution on $[0,\tau)$ for all $x_0 \in D(A)$ if one of the following properties hold:
\begin{enumerate}[label = (\roman*)]
	\item \label{IHCP1} $V_{|(0,\tau)} \in \mathcal{C}^1((0,\tau), X)$.
	\item \label{IHCP2} For all $t \in (0,\tau)$ it holds that $V(t) \in D(A)$ and $AV(\cdot)$ is continuous on $(0,\tau)$.
\end{enumerate}
Conversely, if \eqref{General Abstract Cauchy Problem - Inhomogeneous} has a solution on $[0,\tau)$ for some $x_0 \in D(A)$, then it satisfies both \ref{IHCP1} and \ref{IHCP2}. 
\end{theorem}
From this theorem we may deduce the following two corollaries, which we will not prove. 
\begin{corollary}[Corollary 4.2.5 in \cite{Pazy83}]
	Let $A \colon D(A) \subset X \to X$ be the infinitesimal generator of a $C_0$-semigroup $(T(t))_{t \geq 0}$ on $X$. If $f \in \mathcal{C}^1([0, \tau], X)$, then for every $x_0 \in D(A)$ the initial value problem \eqref{General Abstract Cauchy Problem - Inhomogeneous} has a unique classical solution on $[0, \tau)$. 
\end{corollary}

\begin{corollary}[Corollary 4.2.6 in \cite{Pazy83}]
		Let $A \colon D(A) \subset X \to X$ be the infinitesimal generator of a $C_0$-semigroup $(T(t))_{t \geq 0}$ on $X$. Let $f \in L^1([0,\tau], X) \cap \mathcal{C}((0,\tau), X)$. If $f(s) \in D(A)$ for all $s \in (0, \tau)$ and $Af(\cdot) \in L^1([0,\tau], X)$, then for every $x_0 \in D(A)$ the initial value problem \eqref{General Abstract Cauchy Problem - Inhomogeneous} has a unique classical solution on $[0, \tau)$. 
\end{corollary}
Next, we want to have a quick discussion on linear control systems. 

\section{Linear Control Systems}
\label{Section Linear Control Systems}

The inhomogeneous system \eqref{General Abstract Cauchy Problem - Inhomogeneous} can be written as an abstract control system, if we replace the function $f(\cdot)$ by $Bu(\cdot) \colon [0, \tau] \to X$, where $u \colon [0, \infty) \to U$ is some input with the input space $U$ being a separable Hilbert space, and $B \in \mathcal{L}(U,X)$ is the input operator. However, the assumption that the input operator is a bounded linear operator from $U$ to $X$ is quite restrictive. In many applications, e.g., boundary and point control problems, the input operators are unbounded. 
\vspace{0.5 cm}\\
For a proper treatment of systems with an unbounded input operator, we make use of the concept of admissible control operators. To this end, we lift the state space $X$ to a larger space $X_{-1}$, called the extrapolation space, that has the property that the $C_0$-semigroup $(T(t))_{t \geq 0}$ generated by $A \colon D(A) \subset X \to X$ extends uniquely to a $C_0$-semigroup $(T_{-1}(t))_{t \geq 0}$ on $X_{-1}$, and its generator $A_{-1}$ is the extension of $A$. We then want to find those $X_{-1}$-valued input operators for which the lifted control system admits uniquely determined, $X$-valued solutions. Hence, admissibility is a well-suited concept to prove well-posedness of abstract control systems.
\vspace{0.5 cm}\\
We proceed as follows. In Subsection \ref{Subsection The Spaces X1 and X-1} we introduce the necessary terminology as well as the concept of duality with respect to a pivot space, where we follow the exposition from Section 2.9 and Section 2.10 in \cite{Tucsnak-Weiss}. Then, in Subsection \ref{Subsection Abstract Control Systems and Admissibility} we give a concise overview about $X_{-1}$-valued control systems and the significance of admissible control operators, which is based on Section 4.1 and Section 4.2 in \cite{Tucsnak-Weiss}.

\subsection{The Spaces $\mathbf{X}_1$ and $\mathbf{X}_{-1}$}
\label{Subsection The Spaces X1 and X-1}
As a preparation for abstract control systems and the definition of a solution of such systems, we are going to introduce the interpolation space and the extrapolation space for a linear operator $A \colon D(A) \subset X \to X$. Besides that, we need to explain what it means that two Hilbert spaces are dual with respect to a pivot space.
\vspace{0.5 cm}\\
Let $V$ and $H$ be Hilbert spaces with $V \subset H$. We say that the \emph{embedding} $V \hookrightarrow H$ is \emph{continuous} provided that there exists some $m \geq 0$ such that for all $v \in V$ it holds that
\begin{align*}
	\|v\|_H \leq m \| v\|_V. 
\end{align*}
In other words, the identity operator from $V$ to $H$ is an element of $\mathcal{L}(V,H)$. 

\begin{theorem}[Proposition 2.9.2 in \cite{Tucsnak-Weiss}]
	\label{Theorem Pivot Space}
	Let $V$ and $H$ be Hilbert spaces with $V \subset H$. Assume that the embedding $V \hookrightarrow H$ is continuous and densely defined. Define the function $\| \cdot \|_{\ast} \colon H \to \mathbb{R}$ by
	\begin{align}
		\label{Norm ast}
		\|z \|_{\ast} = \sup_{\stackrel{v \in V}{\|v\|_V = 1}} \vert \langle z , v \rangle_H \vert, \hspace{0.5 cm} z \in H. 
	\end{align}
Then $\| \cdot \|_{\ast}$ is a norm on $H$. Let $V^{\ast}$ denote the completion of $H$ with respect to $\| \cdot \|_{\ast}$. Define the operator $J \colon V^{\ast} \to V'$ as follows: for any $z \in V^{\ast}$, 
\begin{align}
	\label{Theorem Pivot Space - Limit}
	\langle Jz, v \rangle_{V' \times V} = \lim_{n \to \infty} \langle z_n, v \rangle_H, \hspace{0.5 cm} v \in V, 
\end{align}
with $(z_n)_{n \in \mathbb{N}} \in H^{\mathbb{N}}$ a sequence converging to $z$ in $V^{\ast}$. Then $J$ is an isomorphism from $V^{\ast}$ to $V'$. 
\end{theorem}

\begin{proof}[Sketch of the proof of Theorem \ref{Theorem Pivot Space}]
	We will skip the proofs that $\| \cdot \|_{\ast}$ is a norm, that the limit \eqref{Theorem Pivot Space - Limit} exists and is independent of the choice of the sequence $(z_n)_{n \in \mathbb{N}}$, as long as $z_n \to z$ in $V^{\ast}$. We will show that $\langle Jz, v \rangle_{V' \times V}$ depends continuously on $v \in V$, and thus $Jz \in  V'$. From the definition of $\| \cdot \|_{\ast}$ we deduce
	\begin{align*}
	\vert \langle z, v \rangle_H \vert \leq \|z\|_{\ast} \| v \|_V , \hspace{0.5 cm} z \in H, \, v \in V. 
	\end{align*}
Hence, for any $z \in V^{\ast}$ and any $v \in V$, 
\begin{align*}
	\vert \langle Jz, v \rangle_{V' \times V} \vert = \lim_{n \to \infty} \vert \langle z_n, v \rangle_H \vert \leq \lim_{n \to \infty} \|z_n\|_{\ast} \|v \|_V = \|z\|_{\ast} \| v \|_V. 
\end{align*}
As a consequence, $Jz \in V'$. In particular, $\| Jz \|_{V'} \leq \| z \|_{\ast}$ for all $z \in V^{\ast}$, whence $J \in \mathcal{L}(V^{\ast}, V')$. From the definition of $J$ we infer that
\begin{align*}
\langle	Jz, v \rangle_{V' \times V} = \langle z , v \rangle_H, \hspace{0.5 cm} z \in H, \, v \in V,  
\end{align*}
and so $\| J z\|_{V'} = \|z \|_{\ast}$ for all $z \in H$. Since $H$ is dense in $V^{\ast}$ and $J$ is continuous, we conclude that $\| J z\|_{V'} = \|z \|_{\ast}$ even holds for all $z \in V^{\ast}$. The verification that $J$ is onto can be looked up in the proof of Proposition 2.9.2 in \cite{Tucsnak-Weiss}. Altogether, $J$ is an isomorphism from $V^{\ast}$ to $V'$. 
\end{proof}
If $V$, $H$, and $V^{\ast}$ are as in Theorem \ref{Theorem Pivot Space}, then we say that we identify $V^{\ast}$ with $V'$ if we do not distinguish between $z$ and $Jz$ for all $z \in V^{\ast}$. Thus, we have
\begin{align*}
	V \subset H \subset V',
\end{align*}
with continuous and densely defined embeddings. We call $V'$ the \emph{dual of} $V$ \emph{with respect to the pivot space} $H$. Moreover, the norm $\| \cdot \|_{\ast}$ defined in \eqref{Norm ast} is called the \emph{dual norm of} $\| \cdot \|_V$ \emph{with respect to the pivot space} $H$. 
\vspace{0.5 cm}\\
Note that the space $V$ is uniquely determined by $V'$, as it consists of all $v \in H$ for which the mapping $z \mapsto \langle z, v \rangle_H$ has a continuous extension to $V'$. We therefore call $V$ the \emph{dual of} $V'$ \emph{with respect to the pivot space} $H$.  
\vspace{0.5 cm}\\
Now that we have clarified duality with respect to a pivot space, we can finally turn to the definition of the interpolation space and the extrapolation space of a linear operator. 

\begin{definition}[Interpolation Space $\mathbf{X_1}$]
	\label{Definition Interpolation Space}
	Let $A \colon D(A) \subset X \to X$ be linear and densely defined with $\rho(A) \neq \emptyset$. For a fixed  $\lambda \in \rho(A)$, define the interpolation space $X_1$ = $(D(A), \|\cdot\|_1)$, where 
	\begin{align*}
		\|x\|_1 := \| (\lambda I - A)x \|_X, \hspace{0.5 cm} x \in D(A). 
	\end{align*}
	\end{definition}
	According to \cite[Proposition 2.10.1]{Tucsnak-Weiss} the interpolation space $X_1$ is a Hilbert space which is continuously embedded in $X$. Furthermore, for every $\lambda \in \rho(A)$, $\| \cdot \|_1$ is equivalent to the \emph{graph norm} $\| \cdot \|_{D(A)}$ of $A$, and if $L \in \mathcal{L}(X)$ is such that $L D(A) \subset D(A)$, then $L \in \mathcal{L}(X_1)$. 
	\vspace{0.5 cm}\\
	For the adjoint operator $A^{\ast} \colon D(A^{\ast}) \subset X \to X$ of $A$ one can analogously define the Hilbert space $X_1^d = (D(A^{\ast}), \|\cdot\|_1^d)$ with
	\begin{align*}
		\|x\|_1^d := \| (\overline{\lambda} I - A^{\ast})x \|_X, \hspace{0.5 cm} x \in D(A^{\ast}), 
	\end{align*}
	where $\overline{\lambda} \in \rho(A^{\ast})$ ($\Leftrightarrow \lambda \in \rho(A)$, see Proposition \ref{Proposition 2.8.4 in Tucsnak}).
	\begin{definition}[Extrapolation Space $\mathbf{X_{-1}}$]
		\label{Definition Extrapolation Space}
		Let $A \colon D(A) \subset X \to X$ be linear and densely defined with $\rho(A) \neq \emptyset$. For a fixed $\lambda \in \rho(A)$, define the extrapolation space $X_{-1}$ as the completion of $X$ with respect to the norm
	\begin{align*}
		\|x\|_{-1} := \|(\lambda I - A)^{-1}x\|_X ,\hspace{0.5 cm} x \in X. 
	\end{align*}
\end{definition}
	According to \cite[Proposition 2.10.2]{Tucsnak-Weiss} the extrapolation space $X_{-1}$ is again a Hilbert space, and the norms for different $\lambda \in \rho(A)$ are equivalent. In particular, $X_{-1}$ is independent of the choice of $\lambda$. Moreover, $X_{-1}$ is the dual of $ X_1^d$ with respect to the pivot space $X$. To see the last assertion, define $\|\cdot \|_{-1}$ and $\| \cdot \|_1^d$ with respect to the same $\lambda \in \rho(A)$, and apply Proposition \ref{Proposition 2.8.4 in Tucsnak}. Then, for every $z \in X$, 
	\begin{align*}
\|z \|_{-1} &= \|(\lambda I - A)^{-1}z \|_X \\
			&= \sup_{\stackrel{x \in X}{\|x\|_X = 1}} \vert \langle (\lambda I -A)^{-1}z, x \rangle_X \vert \\
			&= \sup_{\stackrel{x \in X}{\|x\|_X = 1}} \vert \langle z, (\overline{\lambda}I - A^{\ast})^{-1} x \rangle_X \vert \\
			&= \sup_{\stackrel{y \in X_1^d}{\| y \|_1^d = 1}} \vert \langle z , y \rangle_X \vert. 
	\end{align*} 
Next, we present an important result for the upcoming section. Assume that $A$ generates a $C_0$-semigroup. Then this semigroup has a unique extension to $\mathcal{L}(X_{-1})$. 
Subsequently, we show an example where we specify the extrapolation space $X_{-1}$ for an infinitesimal generator. 

\begin{proposition}[Proposition 2.10.3 in \cite{Tucsnak-Weiss}]
	Let $A \colon D(A) \subset X \to X$ be linear and densely defined with $\rho(A) \neq \emptyset$. For $\lambda \in \rho(A)$, define $X_1$ and $X_{-1}$ as in Definition \ref{Definition Interpolation Space} and \ref{Definition Extrapolation Space}, respectively. Then $A \in \mathcal{L}(X_1, X)$ and $A$ has a unique extension $A_{-1} \in \mathcal{L}(X, X_{-1})$. Moreover, 
	\begin{align*}
		(\lambda I - A)^{-1} \in \mathcal{L}(X, X_1) \hspace{0.3 cm} \text{and} \hspace{0.3 cm} (\lambda I - A_{-1})^{-1} \in \mathcal{L}(X_{-1}, X). 
	\end{align*}
In particular, $\lambda \in \rho(A_{-1})$, and these operators are unitary. 
\end{proposition}

\begin{proposition}[Proposition 2.10.4 in \cite{Tucsnak-Weiss}]
		Let $A \colon D(A) \subset X \to X$ be the generator of a strongly continuous semigroup $(T(t))_{t \geq 0}$ on $X$. For $\lambda \in \rho(A)$, define $X_1$ and $X_{-1}$ as in Definition \ref{Definition Interpolation Space} and \ref{Definition Extrapolation Space}, respectively. For $t \geq 0$, the restriction of $T(t)$ to $X_1$ (considered as an operator in $\mathcal{L}(X_1)$) is the image of  $T(t) \in \mathcal{L}(X)$ through the unitary operator $	(\lambda I - A)^{-1} \in \mathcal{L}(X, X_1)$. Therefore, these operators form a strongly continuous semigroup on $X_1$ whose generator is $A_{|D(A^2)}$. The operator $T_{-1}(t) \in \mathcal{L}(X_{-1})$ is the image of $T(t) \in \mathcal{L}(X)$ through the unitary operator $(\lambda I - A_{-1})^{-1} \in \mathcal{L}(X_{-1}, X)$. Therefore, these extended operators form a strongly continuous semigroup $(T_{-1}(t))_{t  \geq 0}$ on $X_{-1}$, whose generator is $A_{-1} \colon D(A_{-1}) \subset X_{-1} \to X_{-1}$ with $D(A_{-1}) = X$. 
\end{proposition}

\begin{example}[Example 2.10.7 in \cite{Tucsnak-Weiss}]
	Consider the left-shift semigroup $(T(t))_{t \geq 0}$ on $X = L^2([0, \infty), \mathbb{C})$ defined by
	\begin{align*}
		T(t)x(\cdot) = x(\cdot + t), \hspace{0.5 cm} t\geq 0, \, x \in X.
	\end{align*}
	Its generator $A \colon D(A) \subset X \to X$ is given by
	\begin{align*}
		D(A) &= H^1((0,\infty), \mathbb{C}), \\
		Ax &= \frac{d}{dz}x, \hspace{0.5 cm} x \in D(A).
	\end{align*}
We want to determine the extrapolation space $X_{-1}$ of this semigroup. To this end, we define the Hilbert space adjoint $A^{\ast} \colon D(A^{\ast}) \subset X \to X$ of $A$, which is given by
	\begin{align*}
		D(A^{\ast}) &= H_0^1((0, \infty), \mathbb{C}), \\
		A^{\ast} x &= - \frac{d}{dz}x , \hspace{0.5 cm} x  \in D(A^{\ast}).
	\end{align*}
From the preceding discussion, we know that the dual of $X_1^d = (D(A^{\ast}), \|\cdot\|_1^d)$ with respect to the pivot space $X$ is given by $X_{-1}$. According to Definition 13.4.7 in \cite{Tucsnak-Weiss}, the dual space of  $H_0^1((0, \infty), \mathbb{C})$ is therefore given by the Sobolev space $H^{-1}((0, \infty), \mathbb{C})$, i.e., 
	\begin{align*}
		X_{-1} = (H_0^1((0, \infty), \mathbb{C}))' = H^{-1}((0, \infty), \mathbb{C}).
	\end{align*}
Next, we want to specify the norm on $X_{-1}$. Choosing $\lambda = 1 \in \rho(A)$, Theorem \ref{Theorem Laplace Transform} yields
	\begin{align*}
		(\lambda I - A)^{-1} x(z) =  \int_{0}^{\infty} e^{-t} x(z + t) \, dt, \hspace{0.5 cm} x \in X, \, z \in (0, \infty). 
	\end{align*}
	Thus, the norm $\|\cdot\|_{-1}$ can be presented as
	\begin{align*}
		\|x\|_{-1} &= \left\| \int_{0}^{\infty} e^{-t} x(\cdot + t) \, dt \right\|_X, \hspace{0.5 cm} x \in X.
	\end{align*}
\QEDA
\end{example}
Having introduced the necessary concepts, we can have a quick discussion on $X_{-1}$-valued control systems in the following segment.
\subsection{Abstract Control Systems and Admissibility}
\label{Subsection Abstract Control Systems and Admissibility}
In this section, we collect some facts about existence and uniqueness of solutions of abstract control systems with values in the extrapolation space $X_{-1}$ of the form
\begin{align*}
	\dot{x}(t) &= A_{-1}x(t) + Bu(t), \hspace{0.5 cm} t >0,  \\
	x(0) &= x_0 \in X_{-1}. 
\end{align*}
In the following, we will denote the extension $A_{-1}$ of $A \colon D(A) \subset X \to X$ and the extension $T_{-1}(t)$ of $T(t)$, $t \geq 0$, by $A$ and $T(t)$ again, respectively. We will not give any proofs and refer to \cite[Chapter 4]{Tucsnak-Weiss}. 
\vspace{0.5 cm}\\
First, we want to clarify what we want to understand under a solution in $X_{-1}$. 

\begin{definition}[Solution in $\mathbf{X_{-1}}$]
	\label{Definition Solution in X-1}
	Consider the differential equation
	\begin{align}
		\label{Abstract System in X-1}
		\dot{x}(t) = Ax(t) + Bu(t), \hspace{0.5 cm} t > 0, 
	\end{align}
with $A \colon D(A) \subset X_{-1} \to X_{-1}$, $u \in L_{\text{loc}}^2([0, \infty), U)$, $B \in \mathcal{L}(U, X_{-1})$, and with $U$ a separable Hilbert space. A solution of \eqref{Abstract System in X-1} is a function
\begin{align*}
	x \in L_{loc}^1([0,\infty), X) \cap \mathcal{C}([0, \infty), X_{-1})
\end{align*}
which satisfies the following equations in $X_{-1}$: 
\begin{align*}
	x(t) - x(0) = \int_{0}^{t} Ax(s)  + Bu(s) \, ds, \hspace{0.5 cm } t \geq 0. 
\end{align*}
\end{definition}
Just as it is the case for the inhomogeneous Cauchy problem, if \eqref{Abstract System in X-1} has a solution, then it has a particular form. 

\begin{proposition}[Proposition 4.1.4 in \cite{Tucsnak-Weiss}]
	With the notation in Definition \ref{Definition Solution in X-1}, suppose that $x$ is a solution of \eqref{Abstract System in X-1} in $X_{-1}$. Then $x$ is given by 
	\begin{align*}
		x(t) = T(t)x_0  + \int_{0}^{t} T(t-s) Bu(s) \, ds, \hspace{0.5 cm} t \geq 0, 
	\end{align*}
and is called the mild solution of \eqref{Abstract System in X-1}. In particular, for every $x_0 \in X$ there exists at most one solution in $X_{-1}$ of \eqref{Abstract System in X-1} satisfying the initial condition $x(0) = x_0$. 
\end{proposition}
Next, we give a criterion for the initial value $x_0$ and the function $f$ that guarantees the existence of a unique solution satisfying the initial condition $x(0) = x_0$. 

\begin{theorem}[Theorem 4.1.6 in \cite{Tucsnak-Weiss}]
	If $x_0 \in X$ and $Bu(\cdot) \in H_{loc}^1([0,\infty), X_{-1})$, then \eqref{Abstract System in X-1} has a unique solution in $X_{-1}$, denoted by $x$, that satisfies $x(0) = x_0$. Moreover, this solution is such that
	\begin{align*}
		x \in \mathcal{C}([0, \infty), X) \cap \mathcal{C}^1([0, \infty), X_{-1}), 
	\end{align*}
and for all $t \geq 0$ it satisfies \eqref{Abstract System in X-1} in a classical sense. 
\end{theorem}
Now, let us consider the following abstract control system:
\begin{align}
\begin{split}
	\dot{x}(t) &= Ax(t) + Bu(t), \hspace{0.5 cm} t >0, \\
	x(0) &= x_0 \in X.
\end{split}
\label{Abstract Control System}
\end{align}
Motivated by the study of systems of the form \eqref{Abstract Control System}, the concept of admissibility was introduced. Generally speaking, one wants to specify all those control operators $B$ for which all mild solutions $x$ of \eqref{Abstract Control System} are continuous $X$-valued functions. Operators $B$ with this property are called admissible.  

\begin{definition}[Admissible Control Operator]
	\label{Definition Admissible Control Operator}
	Let  $(T(t))_{t \geq 0}$ be a $C_0$-semigroup on $X_{-1}$. Let $\tau \geq 0$. Define the mapping $\Phi_{\tau} \in \mathcal{L}(L^2([0, \infty), U), X_{-1})$ by
	\begin{align}
		\Phi_{\tau} u := \int_{0}^{\tau} T(\tau- \eta) Bu(\eta) \, d\eta, \hspace{0.5 cm} u \in L^2([0, \infty), U).
		\label{Phi tau u}
	\end{align}
	An operator $B \in \mathcal{L}(U, X_{-1})$ is called an admissible control operator for $(T(t))_{t \geq 0}$ on $X_{-1}$ if for some $\tau > 0$ it holds that
	\begin{align*}
		\ran (\Phi_{\tau}) \subset X.
	\end{align*}
\end{definition}
In other words: if $B$ is admissible, then in \eqref{Phi tau u} the integrand is $X_{-1}$-valued, but the integral itself is an element of $X$. 

\begin{proposition}[Proposition 4.2.2 in \cite{Tucsnak-Weiss}]
	 Suppose that $B \in \mathcal{L}(U, X_{-1})$ is admissible for the $C_0$-semigroup $(T(t))_{t \geq 0}$ on $X_{-1}$. Then for every $t \geq 0$ it holds that
	 \begin{align*}
	 	\Phi_t \in \mathcal{L}(L^2([0, \infty), U), X). 
	 \end{align*}
\end{proposition}
The operators $\Phi_t$, $t \geq 0$, are called the \emph{input maps corresponding to the system} \eqref{Abstract Control System}, and have the following property.
\begin{proposition}[Proposition 4.2.4 in \cite{Tucsnak-Weiss}]
	 Suppose that $B \in \mathcal{L}(U, X_{-1})$ is admissible for the $C_0$-semigroup $(T(t))_{t \geq 0}$ on $X_{-1}$. Then the mapping
	 \begin{align*}
	 	[0, \infty) \times L^2([0, \infty) ,U), \hspace{0.3 cm} (t,u) \mapsto \Phi_t u 
	 \end{align*}
 is continuous. 
\end{proposition}
The main result of this section states under which criteria the mild solutions of the abstract control system \eqref{Abstract Control System} are continuous $X$-valued functions, and is given below.

\begin{theorem}[Proposition 4.2.5 in \cite{Tucsnak-Weiss}]
	\label{Proposition Well-Posedenss of Abstract Control Systems}
	 Suppose that $B \in \mathcal{L}(U, X_{-1})$ is admissible for the $C_0$-semigroup $(T(t))_{t \geq 0}$ on $X_{-1}$ generated by $A \colon D(A) \subset X_{-1} \to X_{-1}$. Then for every $x_0 \in X$ and every $u \in L_{loc}^2([0, \infty), U)$, the initial value problem \eqref{Abstract Control System} has a unique solution in $X_{-1}$. This solution is given by
	 \begin{align*}
	 	x(t) = T(t)x_0 + \Phi_t u, \hspace{0.5 cm} t \geq 0, 
	 \end{align*}
 and satisfies
 \begin{align*}
 	x \in \mathcal{C}([0, \infty), X) \cap H_{loc}^1((0, \infty), X_{-1}). 
 \end{align*}
\end{theorem}
This theorem shows that admissibility is a valuable tool to prove well-posedness of the abstract control system \eqref{Abstract Control System}. In \cite[Section 4.2]{Tucsnak-Weiss} one can find an example where an admissible operator has been determined for the unilateral right-shift semigroup. 
\vspace{0.5 cm}\\
The following and last section of this chapter deals with the generalization of the abstract Cauchy problem discussed in Section \ref{Section The Abstract Cauchy Problem}, where the operator associated with the system changes over time as well.  

\section{Evolution Equations}
\label{Section Evolution Equations}
In this section, we want to give a concise overview concerning evolution problems. Roughly speaking, evolution problems are abstract, time-variant Cauchy problems, that is, the operator associated with the Cauchy problem is changing over time. We will introduce the necessary terminology and study the concept of stability for this class of systems, as this is a key property that has to be satisfied for proving well-posedness. This section relies substantially on \cite[Chapter 5]{Pazy83}.
\vspace{0.5 cm}\\
As in the autonomous case, let us begin with the definition of an evolution problem.
\begin{definition}[Evolution Problem]
	\label{Definition Classical Solution Evolution Problem}
Let $0 \leq t \leq \tau$, and let $A(t) \colon D(A(t)) \subset X \to X$ be a linear operator. The time-variant initial value problem
\begin{align}
	\begin{split}
	\dot{x}(t) &= A(t) x(t), \hspace{0.5 cm} 0 \leq s < t \leq \tau, \\
	x(s) &= x_0 \in X
\end{split}
\label{Evolution Problem}
\end{align}
is called an evolution problem. A function $x \colon [s,\tau] \to X$ is called a classical solution of the evolution problem \eqref{Evolution Problem} if the following holds:
	\begin{enumerate}
		\item[(i)] $x \in \mathcal{C}([s,\tau], X)$. 
		\item[(ii)] $x_{|(s,\tau)} \in \mathcal{C}^1((s,\tau), X)$.
		\item[(iii)] $x(t) \in D(A(t))$ for all $s < t \leq \tau$.
		\item[(iv)] $x$ satisfies \eqref{Evolution Problem}. 
	\end{enumerate}
\end{definition}
The proof of the following theorem is based on using Picard's iteration method and applying the Banach contraction principle. 
\begin{theorem}[Theorem 5.1.1 in \cite{Pazy83}]
	\label{Theorem 5.1.1 in Pazy}
	\label{Theorem Classical Solution Homogeneous Evolution Problem}
	For $0 \leq t \leq \tau$, let $A(t) \in \mathcal{L}(X)$. If the mapping 
	\begin{align*}
		[0,\tau] \to \mathcal{L}(X), \hspace{0.3 cm} t \mapsto A(t)
	\end{align*}
is continuous, then for every $x_0 \in X$ the initial value problem \eqref{Evolution Problem} has a unique classical solution. 
\end{theorem}

\begin{proof}
	We make use of Picard's iteration theorem. Since the mapping $ t \mapsto A(t)$ is continuous, we may define
	\begin{align*}
		\alpha = \max_{t \in [0, \tau]} \|A(t)\|_{\mathcal{L}(X)}. 
	\end{align*}
Furthermore, for $x_0 \in X$ define the mapping
\begin{align*}
S_{x_0} \colon \mathcal{C}([s,\tau], X) &\to \mathcal{C}([s,\tau], X), \\
	u &\mapsto x_0 + \int_{s}^{\cdot} A(\eta) u(\eta) \, d\eta,  
\end{align*}
where $\mathcal{C}([s,\tau], X)$ is endowed with the norm 
\begin{align*}
	\|u\|_{\infty} = \max_{t \in [s,\tau]} \|u(t) \|_X, \hspace{0.5 cm} u \in \mathcal{C}([s,\tau], X), 
\end{align*}
making $(\mathcal{C}([s,\tau], X),  \|\cdot\|_{\infty})$ a Banach space. 
For all $u,v \in \mathcal{C}([s,\tau], X)$ and for all $s \leq t \leq \tau$ the following estimate holds:
\begin{align*}
	\| S_{x_0}u(t) - S_{x_0} v(t) \|_X &= \left\| x_0 + \int_{s}^{t} A(\eta) u(\eta) \, d\eta - x_0 - \int_{s}^{t} A(\eta) v(\eta) \, d\eta \right\|_X \\
	&= \left\| \int_{s}^{t} A(\eta) \left[ u(\eta) - v(\eta) \right] \, d\eta \right\|_X \\
	&\leq \int_{s}^{t} \|A(\eta)\|_{\mathcal{L}(X)} \| u(\eta) - v(\eta) \|_X \, d\eta \\
	&\leq \alpha (t-s) \|u - v\|_{\infty}. 
\end{align*}
Furthermore, we have
\begin{align*}
	&\hspace{0.5 cm} \| S_{x_0}^2u(t) - S_{x_0}^2v(t) \|_X \\
	&= \left\| S_{x_0} \left( x_0 + \int_{s}^{\cdot} A(\eta_1) u(\eta_1) \, d\eta_1 \right)(t) - S_{x_0} \left( x_0 + \int_{s}^{\cdot} A(\eta_1) v(\eta_1) \, d\eta_1 \right) (t)  \right\|_X \\
	&= \left\| \int_{s}^{t} A(\eta_2) \left(  x_0 + \int_{s}^{\eta_2} A(\eta_1) u(\eta_1) \, d\eta_1 \right) \, d\eta_2 - \int_{s}^{t} A(\eta_2) \left( x_0 + \int_{s}^{\eta_2} A(\eta_1) v(\eta_1) \, d\eta_1 \right) \, d\eta_2 \right\|_X \\
	&= \left\| \int_{s}^{t} A(\eta_2) \left( \int_{s}^{\eta_2} A(\eta_1) \left[ u(\eta_1) - v(\eta_1) \right] d\eta_1  \right) \, d\eta_2 \right\|_X \\
	&\leq \alpha \int_{s}^{t} \int_{s}^{\eta_2} \| A(\eta_1) \left[ u(\eta_1) - v(\eta_1) \right] \|_X \,  d\eta_1 \, d\eta_2 \\
	&\leq \alpha^2 \|u - v\|_{\infty}  \int_{s}^{t} \int_{s}^{\eta_2} 1 \, d\eta_1 \, d\eta_2 \\
	&= \frac{\alpha^2 (t-s)^2}{2} \| u - v\|_{\infty}.
\end{align*}
By induction it follows that 
\begin{align*}
	\|S_{x_0}^n u(t) - S_{x_0}^n v(t) \|_X \leq \frac{\alpha^n (\tau - s)^n}{n!} \|u - v\|_{\infty} 
\end{align*}
for all $n \in \mathbb{N}$, $u,v \in \mathcal{C}([s,\tau], X)$ and $s \leq t \leq \tau$. Consequently, we have
\begin{align*}
	\|S_{x_0}^n u - S_{x_0}^n v \|_{\infty} \leq \frac{\alpha^n (\tau - s)^n}{n!} \|u - v\|_{\infty}.
\end{align*}
Note that for $N \in \mathbb{N}$ large enough we get $\frac{\alpha^N (\tau - s)^N}{N!} < 1$, and so $S_{x_0}^N$ is a contraction mapping. By the Banach contraction principle, $S_{x_0}^{N}$ has a unique fixed point $x = x(\cdot;x_0) \in \mathcal{C}([s,\tau], X)$. One can readily see that
\begin{align*}
	S_{x_0}^N(S_{x_0}(x)) = S_{x_0}(S_{x_0}^N(x))  =S_{x_0}(x).
\end{align*}
Thus, $S_{x_0}(x)$ is a fixed point of $S_{x_0}^N$. As $x$ is the unique fixed point of $S_{x_0}^N$, we infer that $S_{x_0}(x) = x$ and, hence, $x$ is a fixed point of $S_{x_0}$ as well. Finally, the uniqueness of the fixed point of $S_{x_0}^N$ together with the fact that every fixed point of $S_{x_0}$ is a fixed point of $S_{x_0}^N$ yield that $x$ is the unique fixed point of $S_{x_0}$ as well, and therefore satisfies
\begin{align}
	\label{Theorem 5.1.1 in Pazy - Solution}
	x(t) = x_0 + \int_{s}^{t} A(\eta) x(\eta) \, d\eta, \hspace{0.5 cm} t \in [s,\tau]. 
\end{align}  
Now, since $x$ is continuous, the right-hand side of \eqref{Theorem 5.1.1 in Pazy - Solution} is differentiable. Thus, $x$ is differentiable with derivative
\begin{align*}
	\frac{d}{dt}x(t) = \frac{d}{dt}  \left( x_0 + \int_{s}^{t} A(\eta) x(\eta) \, d\eta \right) = A(t)x(t), \hspace{0.5 cm} t \in [s,\tau].  
\end{align*}
Furthermore, we have $x(s) = x_0$, and so $x$ is a solution of the evolution problem \eqref{Evolution Problem}. Note that every solution of \eqref{Evolution Problem} is in particular a solution of \eqref{Theorem 5.1.1 in Pazy - Solution}. Due to the uniqueness of the fixed point $x$ of the mapping $S_{x_0}$, we conclude that the solution of \eqref{Evolution Problem} is unique. 
\end{proof}
Recall that in Theorem \ref{Theorem 5.1.1 in Pazy} we required the operators $A(t)$, $0 \leq t \leq \tau$, to be bounded linear operators on $X$. However, this result gives us some intuition regarding the behavior of the solutions of the evolution problem \eqref{Evolution Problem}. Using the representation \eqref{Theorem 5.1.1 in Pazy - Solution} of such a solution, we may define a solution operator $U(t,s) \colon X \to X$, $0\leq s \leq t \leq \tau$, of the initial value problem \eqref{Evolution Problem} as follows:
\begin{align}
	\label{Pazy - Two-Parameter Family of Operators}
	U(t,s) x_0 = x(t;x_0) = x_0 + \int_{s}^{t} A(\eta) x(\eta) \, d\eta, \hspace{0.5 cm} 0  \leq s \leq t \leq \tau, \, x_0 \in X.
\end{align}
The family $(U(t,s))_{0 \leq s \leq t \leq \tau}$ is a two-parameter family of operators on $X$.
The main properties of the family of solution operators are given in the next theorem. 
\begin{theorem}[Theorem 5.1.2 in \cite{Pazy83}]
	\label{Theorem 5.1.2 in Pazy}
	For $0 \leq t \leq \tau$, let $A(t) \in \mathcal{L}(X)$, and let the mapping $t \mapsto A(t)$ be continuous on $[0,\tau]$. Then for every $0 \leq s \leq t \leq \tau$, the operator $U(t,s)$ defined in \eqref{Pazy - Two-Parameter Family of Operators} is an element of $\mathcal{L}(X)$ and it satisfies the following:
	\begin{enumerate}
		\item[(i)] $\|U(t,s)\|_{\mathcal{L}(X)} \leq \exp \left( \int_{s}^{t} \|A(\eta) \|_{\mathcal{L}(X)} \, d\eta \right)$.
		\item[(ii)] $U(t,t) = I_X$ and for $0 \leq s \leq r \leq t \leq \tau$ we have
		\begin{align*}
			U(t,s) = U(t,r) U(r,s).
		\end{align*}
	\item[(iii)] The mapping $(t,s) \mapsto U(t,s)$ is continuous for $0 \leq s \leq t \leq \tau$. 
	\item[(iv)] For $0 \leq s \leq t \leq \tau$ we have
	\begin{align*}
		\frac{\partial}{\partial t}U(t,s) = A(t) U(t,s). 
	\end{align*}
\item[(v)] For  $0 \leq s \leq t \leq \tau$ we have
\begin{align*}
	\frac{\partial}{\partial s} U(t,s) = - U(t,s) A(s). 
\end{align*}
	\end{enumerate}
\end{theorem}
These properties resemble the ones of (uniformly continuous) one-parameter semigroups studied in Section \ref{Section Strongly Continuous Semigroups}. Indeed, in the time-variant case the two-parameter family of operators $(U(t,s))_{0 \leq s \leq t\leq \tau}$ replaces the one-paramter semigroup $(U(t))_{t \geq 0}$ of the autonomous case, which is the $C_0$-semigroup generated by the (autonomous) operator $A$. This motivates to introduce the notion of evolution systems, which play a substantial role in the framework of evolution problems.
\begin{definition}[Evolution System]
	\label{Definition Evolution System}
	A two-parameter family of operators $U(t,s) \in  \mathcal{L}(X)$, $0 \leq s \leq t \leq \tau$, is called an evolution system if the following holds:
	\begin{enumerate}[label = (\roman*)]
		\item \label{EvolutionSystemProp1} $U(s,s) = I_X$ and $U(t,s) = U(t,r) U(r,s)$ for $0 \leq s \leq r \leq t \leq \tau$. 
		\item \label{EvolutionSystemProp2} The mapping $(t,s) \mapsto U(t,s)$ is strongly continuous for $0 \leq s \leq t \leq \tau$. 
		\end{enumerate}
\end{definition}
Just as for one-paramter semigroups, the strong continuity of $(U(t,s))_{0 \leq s \leq t \leq \tau}$ is equivalent to the mapping $(t,s) \mapsto U(t,s)x$ being continuous for all $x \in X$. 
\vspace{0.5 cm}\\
In the following, we want to present one important criterion on a given family of (usually unbouded) operators $(A(t))_{t \in [0,\tau]}$ on $X$ that guarantees the existence of a unique classical solution of the evolution problem \eqref{Evolution Problem} for a dense subset of initial values in $X$. The existence of a unique solution allows us to specify an evolution system associated with the family $(A(t))_{t \in [0,\tau]}$. Note that the uniqueness of the solution yields property \ref{EvolutionSystemProp1}, whereas the continuity of the solution at the initial data yields property \ref{EvolutionSystemProp2} of the corresponding evolution system. The aforementioned criterion is called stability, and is defined as follows.
\begin{definition}[Stable Family of Generators]
	\label{Definition Stable Family of Generators}
	A family $(A(t))_{t \in [0, \tau]}$ of infinitesimal generators of $C_0$-semigroups on $X$ is called stable if there exist constants $M \geq 1$ and $\omega \in \mathbb{R}$, called stability constants, such that the following holds:
	\begin{enumerate}
		\item[(i)] For $0 \leq t \leq \tau$ we have $(\omega, \infty) \subset \rho(A(t))$.
		\item[(ii)] For every finite sequence $0 \leq t_1 \leq t_2 \leq \ldots \leq t_k \leq \tau$, $k \in \mathbb{N}$, it holds that
		\begin{align}
			 \label{Definition Stable Family of Generators - Condition 2}
			\left\| \prod_{j = 1}^{k} R(\lambda, A(t_j)) \right\|_{\mathcal{L}(X)} \leq \frac{M}{(\lambda - \omega)^k}, \hspace{0.5 cm} \lambda > \omega.
			\end{align}
	\end{enumerate}
\end{definition}

\begin{remark}
	\label{Remark Stable Familiy of Generators}
	We give some short remarks concerning Definition \ref{Definition Stable Family of Generators}.
	\begin{enumerate}[label = (\roman*)]
		\item In general, the operators $R(\lambda, A(t_j))$ do not commute. Hence, the order of the terms given in \eqref{Definition Stable Family of Generators - Condition 2} is important.
		\item The stability of $(A(t))_{t \in [0,\tau]}$ is preserved when the norm $\|\cdot\|_X$ is replaced by an equivalent one. However, the stability constants depend on the particular norm $X$ is endowed with. 
		\item If for $ 0 \leq t \leq \tau$, $A(t)$ generates a $C_0$-semigroup of type $C_0(1, \omega)$ (see Theorem \ref{Theorem C0-Semigroup Estimate}), then one can readily see that the family $(A(t))_{t \in [0,\tau]}$ is stable with stability constants $M=1$ and $\omega \in \mathbb{R}$. In particular, every family $(A(t))_{t \in [0,\tau]}$ of generators of contraction semigroups is stable.
	\end{enumerate}
\end{remark}
In general, it is not easy to determine whether a given family $(A(t))_{t \in [0, \tau]}$ of infinitesimal generators is stable. In the next theorem we give a characterization regarding the stability of such a family. For the proof, one exploits Theorem \ref{Theorem Laplace Transform}. 
\begin{theorem}[Theorem 5.2.2 in \cite{Pazy83}]	
	\label{Theorem 5.2.2 in Pazy}
	For $0 \leq t \leq \tau$, let $A(t)$ be the generator of the $C_0$-semigroup $(S_t(s))_{s \geq 0}$ on $X$. Then the family $(A(t))_{t \in [0, \tau]}$ is stable if and only if there exist constants $M\geq 1$ and $\omega \in \mathbb{R}$ such that $(\omega, \infty) \subset \rho(A(t))$ for all $t \in [0, \tau]$, and either one of the following conditions is satisfied:
	\begin{enumerate}
		\item[(i)] For every finite sequence $0 \leq t_1 \leq t_2 \leq \ldots \leq t_k \leq \tau$, $k \in \mathbb{N}$, and for $s_j \geq 0$, $j = 1,\ldots,k$, 
		\begin{align*}
			\left\| \prod_{j = 1}^{k} S_{t_j}(s_j) \right\|_{\mathcal{L}(X)} \leq M \exp \left( \omega \sum_{j=1}^{k} s_j \right). 
		\end{align*} 
	\item[(ii)] For every finite sequence $0 \leq t_1 \leq t_2 \leq \ldots \leq t_k \leq \tau$, $k \in \mathbb{N}$, and for $\lambda_j > \omega$, $j = 1,\ldots,k$,
	\begin{align*}
		\left\| \prod_{j = 1}^{k} R(\lambda_j, A(t_j)) \right\|_{\mathcal{L}(X)} \leq M \prod_{j=1}^{k}  \frac{1}{\lambda_j - \omega}.
	\end{align*}
	\end{enumerate}
\end{theorem} 
We want to emphasize once again that the stability of a family $(A(t))_{t \in [0,\tau]}$ is not the only condition that has to be satisfied in order to guarantee the existence of a unique solution of the evolution problem \eqref{Evolution Problem} associated with $(A(t))_{t \in [0,\tau]}$. For instance, there has to exist a densely and continuously embedded subspace $Y$ of $X$ such that for all $t \in [0,\tau]$, the space $Y$ is an invariant subspace of the semigroup $(S_t(s))_{s \geq 0}$ generated by $A(t)$, and the family $(\tilde{A}(t))_{t \in [0, \tau]}$ given by $\tilde{A}(t) = A(t)_{|D(A(t)) \cap Y}$, $t \in [0, \tau]$, has to be stable on $Y$. Together with another condition related to $Y$, one may show that there exists a unique $Y$-valued solution of \eqref{Evolution Problem}. This holds equally for inhomogeneous evolution problems of the form
\begin{align*}
		\dot{x}(t) &= A(t) x(t) + f(t), \hspace{0.5 cm} 0 \leq s < t \leq \tau, \\
	x(s) &= x_0 \in X,
\end{align*} 
provided that $f \in \mathcal{C}([s,\tau], X)$. However, we will not provide any further details concerning this topic. Instead, we refer to \cite[Chapter 5]{Pazy83} for more information. 
\vspace{0.5 cm}\\
In this chapter, we have defined formally adjoint operators of differential operators defined on Sobolev spaces, and collected some facts about operator semigroups and their infinitesimal generators. Furthermore, we have examined the abstract Cauchy problem associated with operators generating a $C_0$-semigroup. After a quick discussion on abstract control systems and the concept of admissible control operators, we have studied the extension of abstract Cauchy problems to evolution problems, where the associated operator is replaced by a time-dependent family of infinitesimal generators, and presented one important criterion for well-posedness of such systems. We have now set the foundation for a comprehensive analysis of the systems considered in the following chapters, namely infinite-dimensional port-Hamiltonian systems in Chapter \ref{Chapter Infinite-dimensional Port-Hamiltonian Systems}, and said systems of order 1 which are coupled by a moving interface in Chapter \ref{Chapter Boundary Port-Hamiltonian Systems with a Moving Interface}.

\chapter{Infinite-Dimensional Port-Hamiltonian Systems}
\label{Chapter Infinite-dimensional Port-Hamiltonian Systems}

In this chapter, we are going to investigate infinite-dimensional port-Hamiltonian systems. Throughout the chapter, we will restrict ourselves to systems with a 1-dimensional spatial domain $[a,b] \subset \mathbb{R}$. We start again with some motivational examples in Section \ref{Section Motivating Examples}, specifically the model of a lossless transmission line in Subsection \ref{Subsection The Lossless Transmission Line} and the model of a vibrating string in Subsection \ref{Subsection The Vibrating String}. We will see that the evaluation of the co-energy variables at the boundary of the spatial domain play a fundamental role. Hence, in contrast to the finite-dimensional case, we need to introduce further port variables in order to derive a port-Hamiltonian formulation of infinite-dimensional systems, which will be called boundary port-Hamiltonian systems. The examples can be found in \cite[Chapter~4]{ModCompPhys}, \cite{LeGorrec}, \cite[Chapter~14]{SchaftJeltsema}, or in the introductory chapter of \cite{Villegas}. Subsequently, in Section \ref{Section Boundary Port-Hamiltonian Systems Associated with Skew-symmetric Operators} we will generalize the obtained systems considered in Section \ref{Section Motivating Examples} and study boundary port-Hamiltonian systems which are associated with formally skew-symmetric differential operators. To that end, we will define suitable matrices and operators in order to give a proper representation of the boundary port variables constituting the infinite-dimensional Dirac structure. Following this, we are able to define the bond space and a suitable power pairing to finally define the underlying Dirac structure associated with such systems, and to specify the dynamics with respect to these so-called Stokes-Dirac structures. This section heavily relies on \cite{LeGorrec} and the Sections 2.1 and 2.2 in \cite{Villegas}. At the end of this chapter, we will shortly discuss the analysis of boundary port-Hamiltonian systems.
\vspace{0.5 cm}\\
We want to stress that we will not introduce an abstract mathematical object or a particular solution concept for the systems considered in this chapter. As a consequence, the notion of a system should be understood as a formal concept rather than a general definition.
\vspace{0.5 cm}\\
Unless stated otherwise, we consider the Hilbert space $X = L^2([a,b], \mathbb{R}^n)$ equipped with the inner product
\begin{align}
	\label{L^2 inner product}
	\langle f, g \rangle_{L^2} = \int_{a}^{b} g^{\top}(z) f(z) \, dz, \hspace{0.5 cm} f,g \in X.
\end{align}

\section{Motivating Examples}
\label{Section Motivating Examples}
We want to look for a common structure in the models describing the dynamics of various systems. We seek to rewrite the system models as an abstract Cauchy problem of the form
\begin{align*}
	\dot{x}(t) &= Ax(t), \hspace{0.5 cm} t > 0, \\	
	x(0) &= x_0 \in X, 
\end{align*}
where $x \colon [0, \infty) \to X$ is the state variable, and $A \colon D(A) \subset X \to X$ is a linear operator, as we have discussed in Section \ref{Section The Abstract Cauchy Problem} in more detail. For the analysis of such systems we use the framework of operator semigroups. In this section, we want to present two examples from which we will infer a general structure for a class of partial differential equations.
\vspace{0.5 cm}\\
Beforehand, we want to clarify the notion of the variational derivative for a special class of functionals based on \cite[Section 4.1]{Olver}, and give a quick example.
\begin{definition}[Variational Derivative]
Let $H \colon X \to \mathbb{R}$ be a functional defined by
\begin{align}
	\label{General Functional a Variational Derivative is Applied to}
	H(x) = \int_{a}^{b} \mathcal{H}(x(z)) \, dz, \hspace{0.5 cm} x \in X,
\end{align} 
with $\mathcal{H} \in \mathcal{C}^\infty({\mathbb{R}^n, \mathbb{R}})$. The variational derivative $\frac{\delta H}{\delta x} \colon X \to X$ of $H$ given by
\begin{align*}
	\frac{\delta H}{\delta x}(x) = \left( \frac{\delta H}{\delta x_1}(x), \ldots, \frac{\delta H}{\delta x_n}(x) \right), \hspace{0.5 cm} x \in X,
\end{align*}
is an operator uniquely determined by the requirement that it satisfies
\begin{align*}
	H(x + \epsilon \eta) = H(x) + \epsilon \int_{a}^{b} \left( \frac{\delta H}{\delta x}(x) \right)^{\top}(z) \eta (z) \, dz + \mathcal{O}(\epsilon^2)
\end{align*}
for any $x \in X,$ and for any variation $\epsilon \eta$, where $\epsilon \in \mathbb{R}$, $\eta \in X$, and $\mathcal{O}(\epsilon^2)$ is the big O notation with respect to $\epsilon^2$. 
\end{definition}

\begin{example}[cf. Section 1.6 in \cite{Villegas}]
	\label{Example Variational Derivative}
Consider the functional
\begin{align}
	H(x) = \frac{1}{2} \int_{a}^{b} x^{\top}(z) \left( \mathcal{Q} x \right)(z) \, dz, \hspace{0.5 cm} x \in X,
	\label{Subsection Lossless Transmission Line - Standard Hamiltonian Form}
\end{align}
where $\mathcal{Q} \in \mathcal{L}(X)$ is a coercive operator on $X$. Then we have for all $x,\eta \in X$, $\epsilon \in \mathbb{R}$, 
\begin{align*}
	H(x + \epsilon \eta) &= \frac{1}{2} \int_{a}^{b} (x(z) + \epsilon \eta(z))^{\top}  \left( \mathcal{Q} (x + \epsilon \eta \right)(z) \, dz \\
	&= \frac{1}{2} \int_{a}^{b} x^{\top}(z) \left( \mathcal{Q} x \right) (z) + \epsilon \left( x^{\top}(z) \left( \mathcal{Q} \eta \right)(z) + \eta^{\top}(z) \left( \mathcal{Q}x \right)(z) \right) + \epsilon^2 \eta^{\top}(z) \left( \mathcal{Q} \eta \right)(z) \, dz \\
	&= H(x) + \epsilon \int_{a}^{b} \left(\mathcal{Q} x \right)^{\top}(z) \eta(z) \, dz + \mathcal{O}(\epsilon^2).
\end{align*}
From this we deduce that $\frac{\delta H}{\delta x}(x)= \mathcal{Q}x$.
\QEDA
\end{example}
For functionals of the form \eqref{General Functional a Variational Derivative is Applied to} one can show that the variational derivative is induced by the gradient of the integrand $\mathcal{H}$ in the following way:
\begin{align*}
	\frac{\delta}{\delta x}H(x)(z) = \nabla \mathcal{H}(x(z)), \hspace{0.5 cm} x \in X, \, z \in [a,b]. 
\end{align*}

\subsection{The Lossless Transmission Line}
\label{Subsection The Lossless Transmission Line}
	\begin{figure}[h!]
	\centering
	\includegraphics[width = 8cm]{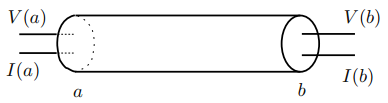}
	\caption{Transmission Line \cite[Figure 7.2]{JacobZwart}}
	\label{Figure Transmission Line}
\end{figure}
Consider an ideal transmission line on the spatial interval $[a,b] \subset \mathbb{R}$, as depicted in Figure~\ref{Figure Transmission Line}. Telegrapher's equations constitute the following system:
\begin{align}
	\begin{split}
	\frac{\partial Q}{\partial t}(z,t) &= - \frac{\partial I}{\partial z} (z,t) = - \frac{\partial }{\partial z} \frac{\varphi(z,t)}{L(z)}, \\
	\frac{\partial \varphi}{\partial t}(z,t) &= -  \frac{\partial V}{\partial z}(z,t) = - \frac{\partial}{\partial z} \frac{Q(z,t)}{C(z)}.
	\end{split}
\label{Subsection Lossless Transmission Line - PDE}
\end{align}
Here, $Q(z,t) \in \mathbb{R}$ and $\varphi(z,t) \in \mathbb{R}$ denote the charge and the magnetic flux-linkage at position $z \in [a,b]$ and at time $t \geq 0$, respectively, and $C(z) >0$ and $L(z) >0$ are the distributed capacity and the distributed inductance of the transmission line at $z \in [a,b]$, repsectively. Furthermore, the current and the voltage are given by $I = \frac{\varphi}{L}$ and $V = \frac{Q}{C}$, respectively. The electro-magnetic energy density $\mathcal{H} \in \mathcal{C}^{\infty}(\mathbb{R}^2, \mathbb{R})$ is given by
\begin{align*}
	\mathcal{H}(Q,\varphi) = \frac{1}{2} \left(\frac{Q^2}{C} + \frac{\varphi^2}{L} \right), \hspace{0.5 cm} (Q, \varphi) \in \mathbb{R}^2.
\end{align*}
 The system's Hamiltonian $H \colon X \to \mathbb{R}$, describing the internally stored energy, is the integral of the electro-magnetic energy density $\mathcal{H}$: letting $x = (Q, \varphi) \in  X$ we have
\begin{align}
	H(Q,\varphi) = \frac{1}{2} \int_{a}^{b} \frac{Q(z)^2}{C(z)} + \frac{\varphi(z)^2}{L(z)} \,  dz  = \frac{1}{2} \int_{a}^{b} \begin{bmatrix}
		Q(z) & \varphi(z)
	\end{bmatrix} \begin{bmatrix}
			\frac{1}{C(z)} & 0 \\
			0 & \frac{1}{L(z)} 
	\end{bmatrix} \begin{bmatrix}
		Q(z) \\
		\varphi (z) 
	\end{bmatrix} \, dz.
\label{Subsection Lossless Transmission Line - Hamiltonian H}
\end{align}
Assuming that both the capacity $C$ and the inductance $L$ are bounded on $[a,b]$, we may define the coercive multiplication operator $\mathcal{Q} \in \mathcal{L}(X)$ as follows:
\begin{align*}
	(\mathcal{Q}x)(z) = \mathcal{Q}(z)x(z) =  \begin{bmatrix}
		\frac{1}{C(z)} & 0 \\
		0 & \frac{1}{L(z)} 
	\end{bmatrix} \begin{bmatrix}
	Q(z) \\
	\varphi(z)
\end{bmatrix}, \hspace{0.5 cm} x \in X, \, z \in [a,b]. 
\end{align*} 
Thus, the Hamiltonian $H$ from \eqref{Subsection Lossless Transmission Line - Hamiltonian H} can be written as $H(x) = \frac{1}{2} \langle x , \mathcal{Q} x \rangle_{L^2}$ for $x \in X$.
\vspace{0.5 cm}\\
As we have seeen in the finite-dimensional case (cf. Section \ref{Section Energy-storing Elements}), the co-energy variables correspond to the gradient of the system's Hamiltonian function. In the infinite-dimensional case, this generalizes to the variational derivative of the Hamiltonian functional.
As shown in Example \ref{Example Variational Derivative}, if a functional is of the form \eqref{Subsection Lossless Transmission Line - Standard Hamiltonian Form}, then the variational derivative is given by
\begin{align*}
	\frac{\delta H}{\delta x} (x) = \mathcal{Q}x = \begin{bmatrix}
		\frac{Q}{C} \\
		\frac{\varphi}{L}
	\end{bmatrix}, \hspace{0.5 cm} x =(Q, \varphi) \in X.
\end{align*}
Altogether, we may equivalently rewrite the system (\ref{Subsection Lossless Transmission Line - PDE}) as
\begin{align}
	\frac{\partial }{\partial t} \begin{bmatrix}
		Q \\
		\varphi 
	\end{bmatrix} = \begin{bmatrix}
	0 & - \frac{\partial }{\partial z} \\
	- \frac{\partial }{\partial z} & 0 
\end{bmatrix} \begin{bmatrix}
\frac{\delta H}{\delta Q}(Q, \varphi) \\
\frac{\delta H}{\delta \varphi}(Q, \varphi)
\end{bmatrix}.
\label{Subsection Lossless Transmission Line - pH System}
\end{align}
This system is called an \emph{infinite-dimensional Hamiltonian system} defined with respect to the matrix differential operator  $\mathcal{J} \colon D(\mathcal{J}) \subset X \to X$ given by
\begin{align*}
	D(\mathcal{J})  &= H^1([a,b], \mathbb{R}^2), \\
	\mathcal{J}x &= \begin{bmatrix}
		0 & - \frac{d}{dz} \\
		- \frac{d}{dz} & 0 
	\end{bmatrix}x,  \hspace{0.5 cm} x \in D(\mathcal{J}),
\end{align*}
and with respect to the Hamiltonian $H$ defined in \eqref{Subsection Lossless Transmission Line - Hamiltonian H}. Necessarily, the matrix differential operator $\mathcal{J}$ has to satisfy two properties, namely formal skew-symmetry and the Jacobi identity, as explained in \cite[Chapter 7]{Olver}. We will only verify the formal skew-symmetry in the sense of Definition \ref{Definition Formal Operator Properties}. To this end, let $e =(e_1,e_2)$, $\tilde{e} = (\tilde{e}_1,\tilde{e}_2) \in H_0^1([a,b], \mathbb{R}^2)$ be some vectors vanishing at the boundary. Then integration by parts yields
\begin{align}
	\begin{split}
	&\hspace{0.5 cm}  \langle \mathcal{J} e, \tilde{e}\rangle_{L^2} + \langle e, \mathcal{J} \tilde{e} \rangle_{L^2} \\
	&= - \int_{a}^{b}  \tilde{e}_1(z)  \frac{d e_2}{dz}(z) + \tilde{e}_2(z) \frac{d e_1}{dz}(z) + e_1(z) \frac{d \tilde{e}_2}{dz}(z) + e_2(z) \frac{d \tilde{e}_1}{dz} (z)  \, dz \\
	&= - \big[ e_1(z)\tilde{e}_2(z) + e_2(z) \tilde{e}_1(z) \big]_{a}^{b} \\
	&= 0.
	\end{split}
	\label{Subsection Lossless Transmission Line - Skew-Symmetry Test}
\end{align}
Since $\mathcal{J}$ is not modulated by the state variable $x$, Corollary 7.5 in \cite{Olver} yields that $\mathcal{J}$ is a Hamiltonian operator (in the sense of \cite[Definition 7.1]{Olver}).
\vspace{0.5 cm}\\
Analogously to the finite-dimensional case, the Hamiltonian structure provides energy conservation: assume that $x \colon [0, \infty) \to X$ is a classical solution of \eqref{Subsection Lossless Transmission Line - pH System} which is vanishing at the boundary. Then by symmetry of the operator $\mathcal{Q}$ and by formal skew-symmetry of $\mathcal{J}$ we obtain 
\begin{align}
	\begin{split}
	\frac{d}{dt} H(x(\cdot,t)) &= \frac{1}{2} \frac{d}{dt} \langle x(\cdot,t), \mathcal{Q} x(\cdot,t) \rangle_{L^2} \\
	  &= \frac{1}{2}  \left\langle \frac{\partial x}{\partial t}(\cdot,t) ,  \mathcal{Q}x (\cdot,t) \right\rangle_{L^2}  + \frac{1}{2} \left\langle \mathcal{Q}  x (\cdot,t), \frac{\partial x}{\partial t}(\cdot,t)  \right\rangle_{L^2} \\
	&=  \frac{1}{2} \langle \mathcal{J} \mathcal{Q} x (\cdot,t), \mathcal{Q} x (\cdot,t) \rangle_{L^2}  + \frac{1}{2} \langle \mathcal{Q} x (\cdot,t), \mathcal{J} \mathcal{Q} x(\cdot,t) \rangle_{L^2}  \\
	&= 0.
\end{split}
\label{Subsection Lossless Transmission Line - Energy Conservation}
\end{align}
We see that the Hamiltonian formulation imposes that there must not occur any energy exchange at the boundaries of the transmission line with the rest of the circuit, and therefore the total energy is conserved. However, this assumption is quite restrictive, and for the purpose of control and for the interconnection of two or more Hamiltonian systems it is essential to take the interaction with
its environment into account. This motivates to introduce a new type of port variables, called \emph{boundary port variables}, and thus to extend the (port-)Hamiltonian framework to boundary port-Hamiltonian systems in the infinite-dimensional case.
\vspace{0.5 cm}\\
Now, if the state variables do not vanish at the boundary of the spatial domain, then from (\ref{Subsection Lossless Transmission Line - Skew-Symmetry Test}) we find that some boundary terms appear. Indeed, the energy is not conserved in general, but satisfies the following energy-balance equation, which again can be derived from (\ref{Subsection Lossless Transmission Line - Skew-Symmetry Test}) and (\ref{Subsection Lossless Transmission Line - Energy Conservation}):
\begin{align}
	\begin{split}
	\frac{d}{dt} H(x(\cdot,t)) &= \frac{1}{2} \langle \mathcal{J} \mathcal{Q} x (\cdot,t), \mathcal{Q} x (\cdot,t) \rangle_{L^2}  + \frac{1}{2} \langle \mathcal{Q} x (\cdot,t), \mathcal{J} \mathcal{Q} x(\cdot,t) \rangle_{L^2}  \\
	 &=- \frac{1}{2}  \left[ 2 \frac{Q(z,t)}{C(z)} \frac{\varphi (z,t)}{L(z)} \right]_{a}^{b} \\
	&=  - \big[ V(z,t) I (z,t) \big]_{a}^{b}.
	\end{split}
\label{Subsection Lossless Transmission Line - Energy Balance Equation}
\end{align}
Note that the product of voltage and current equals power. Hence, equation (\ref{Subsection Lossless Transmission Line - Energy Balance Equation}) states that the change of internally stored energy solely occurs via the spatial boundary. This suggests to introduce the evaluation not of the state variables, but of the co-energy variables $ V = \frac{Q}{C}$ and $I = \frac{\varphi}{L}$ at the spatial boundary as new port variables in order to involve the power flow through the boundary of the system (\ref{Subsection Lossless Transmission Line - pH System}). We will postpone the definition of the boundary port variables to Section \ref{Section Boundary Port-Hamiltonian Systems Associated with Skew-symmetric Operators}.
\begin{remark}[Energy Space] 
	\label{Remark Energy Space}
	As $H(x) = \frac{1}{2} \langle x, \mathcal{Q}x \rangle_{L^2}$, $x \in X$, and as $\mathcal{Q} \in \mathcal{L}(X)$ is a bounded and coercive operator, one usually endows the state space $X$ with the inner product $\langle \cdot, \cdot \rangle_{\mathcal{Q}}$ defined by
	\begin{align}
		\langle x, y \rangle_{\mathcal{Q}} := \frac{1}{2} \langle x, \mathcal{Q} y \rangle_{L^2}, \hspace{0.5 cm} x, y \in X.
		\label{Subsection The Lossless Transmission Line - Inner Product WRT Q}
	\end{align}
	Note that for all $z \in [a,b]$ we have $mI \leq \mathcal{Q}(z) \leq MI$ for some $0 <m \leq M$. Thus, the standard $L^2$-norm is equivalent to the norm defined in \eqref{Subsection The Lossless Transmission Line - Inner Product WRT Q}, as we have for all $f \in L^2([a,b], \mathbb{R})$,
	\begin{align*}
		\frac{m}{2} \|f\|_{L^2}^2 \leq \|f\|_{\mathcal{Q}}^2 \leq \frac{M}{2} \|f\|_{L^2}^2.
	\end{align*}
	The main reason for doing so is that in this way, the squared norm of the state $x \in X$ equals the Hamiltonian, i.e., the internally stored energy, of the system:
	\begin{align}
		\|x(\cdot,t) \|_{\mathcal{Q}}^2 = \frac{1}{2}\langle x (\cdot,t), \mathcal{Q}x(\cdot,t) \rangle_{L^2} = H(x(\cdot,t)), \hspace{0.5 cm} t \geq 0.
		\label{Subsection The Lossless Transmission Line - Representation Hamiltonian WRT Q norm}
	\end{align}
	For this choice of the inner product, $X$ is known as the energy space. We have already mentioned the usefulness of the relation \eqref{Subsection The Lossless Transmission Line - Representation Hamiltonian WRT Q norm} in Example \ref{Example Contraction Semigroup}. 
\end{remark}


\subsection{The Vibrating String}
\label{Subsection The Vibrating String}
	\begin{figure}[h]
		\centering
		\includegraphics[width = 8 cm]{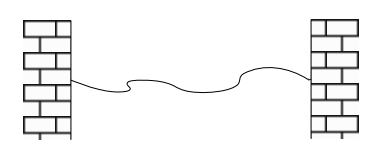}
		\caption{Vibrating String}
		\label{Figure Vibrating String}
	\end{figure}
Consider a vibrating string of the length $L = b -a >0$, $a,b \in \mathbb{R}$, held stationary at both ends and free to vibrate transversely. The scenario is illustrated in Figure \ref{Figure Vibrating String}. The vibration of the string can be modeled by
	\begin{align}
		\frac{\partial w^2}{\partial t^2}(z,t) = \frac{1}{(\rho A)(z)} \frac{\partial}{\partial z} \left( T(z)  \frac{\partial w}{\partial z}(z,t) \right), 
		\label{Example Vibrating String - Vibrating String Equation}
	\end{align}
where $w(z,t) \in \mathbb{R}$ denotes the vertical displacement of the string at position $z \in [a,b]$, and at time $t \geq 0$, and $T(z) >0$ and $(\rho A) (z) > 0$, $z \in [a,b]$ denote the Young's modulus of the string and the linear mass density, respectively, which are assumed to be bounded. We take the momentum $p  = (\rho A) \frac{\partial w}{\partial t}$  and the strain $\epsilon = \frac{\partial w}{\partial z} $ to constitute the state variable $x=(p,\epsilon) \in X$. The Hamiltonian $H \colon X \to \mathbb{R}$ of this system is then given by the sum of the kinetic and the elastic energies in the following way: 
\begin{align}
	H(p,\epsilon) = \frac{1}{2} \int_{a}^{b}  \frac{1}{(\rho A)(z)} p(z)^2 + T(z) \epsilon(z)^2 \, dz = \frac{1}{2} \int_{a}^{b} \begin{bmatrix}
		p(z) & \epsilon(z)
	\end{bmatrix} \begin{bmatrix}
			\frac{1}{(\rho A)(z)} & 0 \\
			0 & T(z)
\end{bmatrix} \begin{bmatrix}
p(z) \\
\epsilon(z)
\end{bmatrix} \, dz.
	\label{Subsection The Vibrating String - Energy}
\end{align}
Again, we may define the coercive multiplication operator $\mathcal{Q} \in \mathcal{L}(X)$ given by
\begin{align*}
	(\mathcal{Q}x)(z) = \mathcal{Q}(z) x(z) = \begin{bmatrix}
		\frac{1}{(\rho A)(z)} & 0 \\
		0 & T(z)
	\end{bmatrix} \begin{bmatrix}
	p(z) \\
	\epsilon(z)
\end{bmatrix}, \hspace{0.5 cm} x \in X, \, z \in [a,b].
\end{align*}
Additionally, as the Hamiltonian is once again of the form (\ref{Subsection Lossless Transmission Line - Standard Hamiltonian Form}), the variational derivative of the Hamiltonian $H$ equals the coercive matrix operator $\mathcal{Q}$:
\begin{align*}
	\frac{\delta H}{\delta x}(x) = \begin{bmatrix}
		\frac{\delta H}{\delta p}(p,\epsilon) \\
		\frac{\delta H}{\delta \epsilon}(p,\epsilon)
	\end{bmatrix} = \begin{bmatrix} 
	\frac{p}{\rho A} \\
T\epsilon
\end{bmatrix} = \begin{bmatrix}
\frac{\partial w}{\partial t} \\
T \frac{\partial w}{\partial z}
\end{bmatrix}, \hspace{0.5 cm} x \in X. 
\end{align*}
Altogether, we may equivalently rewrite equation (\ref{Example Vibrating String - Vibrating String Equation}) as an infinite-dimensional, first-order Hamiltonian system of the form
\begin{align}
	\frac{\partial}{\partial t} \begin{bmatrix}
			p \\
			\epsilon
	\end{bmatrix} = \begin{bmatrix}
			0 & \frac{\partial}{\partial z} \\
			\frac{\partial }{\partial z} & 0 
		\end{bmatrix} \begin{bmatrix}
			\frac{\delta H}{\delta p}(p,\epsilon)  \\
			\frac{\delta H}{\delta \epsilon}(p,\epsilon)
	\end{bmatrix} 
	\label{Subsection The Vibrating String - Vibrating String Equation in pH form}
\end{align}
with respect to the (Hamiltonian) matrix differential operator $\mathcal{J} \colon D(\mathcal{J}) \subset X \to X$ given by
\begin{align*}
	D(\mathcal{J}) &= H^1([a,b], \mathbb{R}^2), \\
	\mathcal{J}x &=  \begin{bmatrix}
		0 & \frac{d}{dz} \\
		\frac{d}{dz} & 0 
	\end{bmatrix}x,  \hspace{0.5 cm} x\in D(\mathcal{J}),
\end{align*}
and with respect to the Hamiltonian $H$ given in \eqref{Subsection The Vibrating String - Energy}. The proof that $\mathcal{J}$ is indeed a Hamiltonian operator (see again \cite[Definition 7.1]{Olver}) works exactly as in the previous example. For a classical solution $x \colon [0, \infty) \to X$ of \eqref{Subsection The Vibrating String - Vibrating String Equation in pH form}, the energy-balance equation of the system \eqref{Subsection The Vibrating String - Vibrating String Equation in pH form} is given by
\begin{align}
	\frac{d}{dt} H(x(\cdot,t)) = \left[ \frac{p(z,t)}{(\rho A)(z)} T(z)  \epsilon(z,t) \right]_{a}^{b},
	\label{Subsection The Vibrating String - Balance Equation}
\end{align}
where the product of the co-energy variables, namely the velocity $\frac{\delta H}{\delta p}(p,\epsilon) = \frac{p}{\rho A}$ and the stress $\frac{\delta H}{\delta \epsilon}(p,\epsilon) = T \epsilon$, again equals power. In our particular scenario, the ends of the string are both fixed, that is, the velocity at the boundary is zero. In particular, $p(a,t) = p(b,t) = 0$, $t \geq 0$, which is why in this case, the energy is conserved. However, if at least one end of the string may move freely or some force may be applied to it, the power flow through the boundary has to be encompassed, as depicted in \eqref{Subsection The Vibrating String - Balance Equation}.
\begin{remark}
	\label{Remark Skew-Symmetry J and Operator Q}
We note that the model \eqref{Subsection The Vibrating String - Vibrating String Equation in pH form} of this system has the same structure as the model \eqref{Subsection Lossless Transmission Line - pH System} of the lossless transmission line: the multiplication operator $ \mathcal{Q} \in \mathcal{L}(X)$ is coercive and represents the mechanical and physical properties of the system. Furthermore, the Hamiltonian operators can be written as $\mathcal{J} = P_1 \frac{d}{dz}$, where the matrix $P_1 \in \mathbb{R}^{2 \times 2}$ is symmetric (and invertible). Given the vibrating string example, the formal skew-symmetry of the differential operator $\mathcal{J}$ implies that the kinetic energy is transformed into elastic energy and vice versa. Hence, the internally stored energy is conserved in the absence of any power flow through the boundary. The formally skew-symmetric differential operators can be regarded as the infinite-dimensional counterpart of the skew-symmetric matrices in the finite-dimensional case.
\end{remark}
Just as in the finite-dimensional case there may also occur (internal) energy dissipation in infinite-dimensional systems. We refer to \cite[Subsection 4.1.3]{ModCompPhys} and \cite[Section 1.4]{Villegas} and the references mentioned there. In the sequel, we will generalize the systems considered in Section \ref{Section Motivating Examples} and only study idealized systems where any dissipation phenomena are neglected.

\section{Boundary Port-Hamiltonian Systems Associated with Skew-symmetric Operators}
\label{Section Boundary Port-Hamiltonian Systems Associated with Skew-symmetric Operators}
We infer from the preceding examples that it is possible to describe a variety of PDEs with the following structure:
\begin{align}
	\begin{split}
		\frac{\partial x}{\partial t} (z,t) &= \mathcal{J} (\mathcal{Q}(z)x(z,t)), \hspace{0.5 cm} t > 0, \\
		x(z,0) &= x_0(z)
	\end{split}
	\label{General Lossless Port-Hamiltonian System - Infinite-Dimensional Case}
\end{align}
on the state space $X = L^2 ([a,b], \mathbb{R}^n)$, where $ \frac{\delta}{\delta x} H =\mathcal{Q} \in \mathcal{L}(X)$ is the variational derivative of the Hamiltonian $H \colon X \to \mathbb{R}$ given by
\begin{align}
	H (x) = \int_{a}^{b} \mathcal{H}(x(z)) \, dz =  \frac{1}{2} \int_{a}^{b} x^{\top}(z) (\mathcal{Q} x)(z) \, dz,
	\label{General Lossless Port-Hamiltonian System - Hamiltonian H}
\end{align} 
where $\mathcal{H} \in \mathcal{C}^{\infty}(\mathbb{R}^n, \mathbb{R})$ is the Hamiltonian density. Furthermore, the (multiplication) operator $\mathcal{Q}$ is assumed to be coercive.
Let the differential operator $\mathcal{J} \colon D(\mathcal{J}) \subset X \to X$ be given by
\begin{align}
	\begin{split}
	D(\mathcal{J}) &= H^N \left( [a,b], \mathbb{R}^n \right), \\
	\mathcal{J} x &= \sum_{i=0}^{N} P_i \frac{d^i x}{d z^i}, \hspace{0.5 cm} x \in D(\mathcal{J}),
	\end{split}
	\label{Skew-Symmetric Differential Operator J - N-Dimensional}
\end{align}
 with $P_i \in \mathbb{R}^{n \times n}$, $i = 0,\ldots,N$, where $P_N$ is supposed to be invertible. Furthermore, we assume that
\begin{align}
	P_i = (-1)^{i+1} P_i^{\top}, \hspace{0.5 cm} i =0,\ldots,N,
	\label{Skew-Symmetric Differential Operator J - N-Dimensional - Corresponding Matrices}
\end{align} 
that is, the constituting matrices of $\mathcal{J}$ are skew-symmetric and symmetric in an alternating fashion. We will see that condition (\ref{Skew-Symmetric Differential Operator J - N-Dimensional - Corresponding Matrices}) implies that the differential operator $\mathcal{J}$ is formally skew-symmetric. Recall from Remark \ref{Remark Skew-Symmetry J and Operator Q} that the formal skew-symmetry corresponds to an energy-conserving exchange of energy between different physical domains. Finally, following Remark \ref{Remark Energy Space}, we endow the state space $X$ with the inner product $\langle \cdot, \cdot \rangle_{\mathcal{Q}}$ defined in (\ref{Subsection The Lossless Transmission Line - Inner Product WRT Q}), so that the Hamiltonian $H$ can be rewritten as in (\ref{Subsection The Lossless Transmission Line - Representation Hamiltonian WRT Q norm}).
\vspace{0.5 cm}\\
In this section, we want to define the Dirac structure which is associated with systems of the form \eqref{General Lossless Port-Hamiltonian System - Infinite-Dimensional Case}. As the examples in Section \ref{Section Motivating Examples} suggested, we need to introduce boundary port variables to encompass the power flow through the boundary as well. In Subsection~\ref{Subsection Boundary Port Variables} we will define this new type of port variables by means of suitably chosen matrices and operators, and give a representation of the evaluation of the effort variables at the spatial boundary with respect to the boundary port variables. Subsection \ref{Subsection Dirac Structure and Boundary Port-Hamiltonian Systems} deals with the right choice of the bond space $\mathcal{B} = \mathcal{F} \times \mathcal{E}$ and a suitable power pairing $\langle \cdot \mid \cdot \rangle \colon \mathcal{E} \times \mathcal{F} \to \mathbb{R}$ constituting the bilinear form (the plus pairing) $\ll \cdot, \cdot \gg \colon \mathcal{B} \times \mathcal{B} \to \mathbb{R}$ in order to define the underlying Stokes-Dirac structure associated with infinite-dimensional linear systems of the form \eqref{General Lossless Port-Hamiltonian System - Infinite-Dimensional Case}, leading to the notion of boundary port-Hamiltonian systems.

\subsection{Boundary Port Variables}
\label{Subsection Boundary Port Variables}
Let us, to begin with, show that the differential operator $\mathcal{J}$ defined by \eqref{Skew-Symmetric Differential Operator J - N-Dimensional}-\eqref{Skew-Symmetric Differential Operator J - N-Dimensional - Corresponding Matrices} is indeed formally skew-symmetric. 
\begin{theorem}[Theorem 2.1 in \cite{Villegas}]
	\label{Theorem Stokes-Like Theorem}
	Let $\mathcal{J}$ be the differential operator defined in \eqref{Skew-Symmetric Differential Operator J - N-Dimensional}-\eqref{Skew-Symmetric Differential Operator J - N-Dimensional - Corresponding Matrices}. Then for any two functions $e_1,e_2 \in D(\mathcal{J})$, we have
	\begin{align}
	\langle \mathcal{J} e_1 , e_2 \rangle_{L^2} + \langle e_1 , \mathcal{J} e_2 \rangle_{L^2} = \left[ \begin{bmatrix}
			e_1^{\top}(z), \ldots, \frac{d^{N-1} e_1^{\top}}{d z^{N-1}} (z)
		\end{bmatrix} \mathcal{P} \begin{bmatrix}
		e_2(z) \\
		\vdots \\
		\frac{d^{N-1} e_2}{d z^{N-1}}(z)
	\end{bmatrix} \right]_{a}^{b},
\label{Theorem Stokes-Like Theorem - Equation}
	\end{align}
with the symmetric matrix $\mathcal{P} \in \mathbb{R}^{nN \times nN}$ given by
\begin{align}
	\mathcal{P} = \begin{bmatrix}
		P_1 & P_2 & P_3 & \cdots & P_{N-1} & P_N  \\
		-P_2 & -P_3 & - P_4 & \cdots & - P_{N} & 0 \\
		P_3 & P_4 & \ddots & \ddots & 0 & 0 \\
		\vdots &  \ddots & \ddots & \ddots &  & \vdots \\
		(-1)^{N-1}P_N & 0 & \cdots & \cdots & &  0 
	\end{bmatrix}.
\label{Matrix P - Distributed-Parameter}
\end{align} 
In particular, $\mathcal{J}$ is formally skew-symmetric.
\end{theorem}
\begin{proof}
	The proof is rather tedious. The claim basically follows from iteratively applying integration by parts. We refer to \cite[Theorem 3.1]{Memorandum}.
\end{proof}
In order to define a power-conserving interconnection structure in the infinite-dimensional case, we need to incorporate the evaluation of the effort variables at the spatial boundary into the definition of the respective power pairing, which we will specify later on. To this end, we will write the right-hand side of \eqref{Theorem Stokes-Like Theorem - Equation} subject to the boundary port variables. Note that these port variables will depend on the constituting matrices of $\mathcal{J}$, which are captured in the matrix $\mathcal{P}$ defined in \eqref{Matrix P - Distributed-Parameter}. Since we suppose that the matrix $P_N$ is invertible, it is easy to check that $\mathcal{P}$ is invertible as well. 
\vspace{0.5 cm}\\
In order to define the boundary port variables, we need to introduce further matrices.   
First, let the matrix $\mathcal{P}_{\extern} \in \mathbb{R}^{2nN \times 2nN}$ associated with the differential operator $\mathcal{J}$ be defined as
	\begin{align*}
		\mathcal{P}_{\extern} = \begin{bmatrix}
			\mathcal{P} & 0 \\
			0 & -\mathcal{P}
			\end{bmatrix},
	\end{align*}
where $\mathcal{P}$ is the symmetric and invertible matrix from (\ref{Matrix P - Distributed-Parameter}). Next, we factorize $\mathcal{P}_{\extern}$ in such a way that it allows us to define the port variables, and simultaneously guarantees that the bilinear form - which has yet to be defined - becomes independent of the (matrix) coefficients of the operator $\mathcal{J}$, as we will see shortly.
\begin{lemma}[Lemma 2.4 in \cite{Villegas}]
	\label{Lemma Auxillary Lemma Port-Boundary Variables}
	The matrix $R_{\extern} \in \mathbb{R}^{2nN \times 2nN}$ defined as
	\begin{align}
		R_{\extern} = \frac{1}{\sqrt{2}} \begin{bmatrix}
			\mathcal{P} & -\mathcal{P} \\
			I & I 
		\end{bmatrix}
	\label{Matrix Rext}
	\end{align}
is invertible and satisfies
\begin{align*}
	\mathcal{P}_{\extern} = R_{\extern}^{\top} \Sigma R_{\extern},
\end{align*}
where
\begin{align}
	\label{Matrix Sigma}
	\Sigma = \begin{bmatrix}
		0 & I \\
		I & 0
	\end{bmatrix} \in \mathbb{R}^{2nN \times 2nN}.
	\end{align}
Furthermore, all possible matrices $R \in \mathbb{R}^{2nN \times 2nN}$ satisfying $\mathcal{P}_{\extern} = R^{\top} \Sigma R$ are of the form
\begin{align*}
	R = U R_{\extern},
\end{align*}
with $U \in \mathbb{R}^{2nN \times 2nN}$ satisfying $U^{\top} \Sigma U = \Sigma$. 
\end{lemma}
The proof of Lemma \ref{Lemma Auxillary Lemma Port-Boundary Variables} is fairly easy. For details we refer to the proof of \cite[Lemma 3.4]{LeGorrec}.
\vspace{0.5 cm}\\
The essential part of defining the Dirac structure associated with the operator $\mathcal{J}$ is to introduce the boundary port variables, which we are finally able to do.

\begin{definition}
\label{Definition Boundary Port Variables in The Linear Case}
The boundary port variables associated with the differential operator $\mathcal{J}$ given by \eqref{Skew-Symmetric Differential Operator J - N-Dimensional}-\eqref{Skew-Symmetric Differential Operator J - N-Dimensional - Corresponding Matrices} are the vectors $f_{\partial}, e_{\partial}  \in \mathbb{R}^{nN}$ defined by 
\begin{align}
	\label{Boundary Flow and Effort Linear Case} 
	\begin{bmatrix}
		f_{\partial,e} \\ 
		e_{\partial,e} 
	\end{bmatrix} := R_{\extern} \trace (e), \hspace{0.5 cm} e \in D(\mathcal{J}),
\end{align}
where $\trace \colon H^N([a,b], \mathbb{R}^n) \to \mathbb{R}^{2nN}$ is the boundary trace operator defined in \eqref{Trace Operator - Chapter 3}.
\end{definition}
We usually write $\begin{bmatrix}
	f_{\partial} \\
	e_{\partial}
\end{bmatrix}$ instead of $\begin{bmatrix}
	f_{\partial, e} \\
	e_{\partial, e}
\end{bmatrix}$ whenever the dependence of the boundary port variables on the variable $e \in D(\mathcal{J})$ is obvious. We call $f_{\partial} \in \mathbb{R}^{nN}$ and $e_{\partial} \in \mathbb{R}^{nN}$ the \emph{boundary flow} and the \emph{boundary effort}, respectively. 
\vspace{0.5 cm}\\
Using Definition \ref{Definition Boundary Port Variables in The Linear Case} and Lemma \ref{Lemma Auxillary Lemma Port-Boundary Variables}, we can rewrite the right-hand side of \eqref{Theorem Stokes-Like Theorem - Equation} in terms of the boundary port variables in the following way: for $e_1,e_2 \in H^N \left( [a,b], \mathbb{R}^n \right)$ we have
\begin{align*}
\begin{bmatrix}
		f_{\partial,e_1}^{\top} & e_{\partial, e_1}^{\top}
	\end{bmatrix} \Sigma \begin{bmatrix}
		f_{\partial, e_2} \\
		e_{\partial, e_2}
	\end{bmatrix} &= \trace(e_1)^{\top} R_{\extern}^{\top} \Sigma R_{\extern} \trace(e_2) \\
&= \trace(e_1)^{\top} \mathcal{P}_{\extern} \trace(e_2) \\
&=  \trace(e_1)^{\top} \begin{bmatrix}
\mathcal{P} & 0 \\
0 & -\mathcal{P}
\end{bmatrix} \trace(e_2) \\
&= \left[ \begin{bmatrix}
	e_1^{\top}(z) & \cdots & \frac{d^{N-1} e_1^{\top}}{d z^{N-1}} (z)
\end{bmatrix} \mathcal{P} \begin{bmatrix}
	e_2(z) \\
	\vdots \\
	\frac{d^{N-1} e_2}{d z^{N-1}}(z)
\end{bmatrix} \right]_{a}^{b}.
\end{align*}  
Thus, for all $e_1, e_2 \in D(\mathcal{J})$, equation (\ref{Theorem Stokes-Like Theorem - Equation}) becomes
\begin{align}
	\begin{split}
		\langle \mathcal{J}e_1 , e_2 \rangle_{L^2} + \langle e_1, \mathcal{J} e_2 \rangle_{L^2} &=  \begin{bmatrix}
			f_{\partial,e_1}^{\top} & e_{\partial, e_1}^{\top}
		\end{bmatrix} \Sigma \begin{bmatrix}
		f_{\partial, e_2} \\
		e_{\partial, e_2}
	\end{bmatrix} \\
&= f_{\partial, e_1}^{\top} e_{\partial, e_2} + e_{\partial, e_1}^{\top} f_{\partial, e_2}. 
\end{split}
\label{Theorem Stokes-Like Theorem - Equation Representation 2}
\end{align}

Now, we are able to define the Dirac structure on a suitable bond space corresponding to the formally skew-symmetric operator $\mathcal{J}$.

\subsection{Dirac Structure and Boundary Port-Hamiltonian Systems}
\label{Subsection Dirac Structure and Boundary Port-Hamiltonian Systems}
We have already discussed that in the finite-dimensional case the space of flows $\mathcal{F}$ and the space of efforts $\mathcal{E}$ correspond to a finite-dimensional vector space and its dual space, with the power pairing $\langle \cdot \mid \cdot \rangle \colon \mathcal{E} \times \mathcal{F} \to \mathbb{R}$ of an element of the flow space and of an element of the effort space being equal to physical power (see Section \ref{Section Dirac Structures}). In the present case the flow space $\mathcal{F}$ is now an infinite-dimensional Hilbert space, and the effort space $\mathcal{E}$ is again to be defined as the dual space of the space of flows,  $\mathcal{E} = \mathcal{F}'$. Throughout, we will identify $\mathcal{F}$ with its dual $\mathcal{E}$. 
\vspace{0.5 cm}\\
In order to define a Dirac structure, we need to endow the bond space $\mathcal{B} = \mathcal{F} \times \mathcal{E}$ not with the canonical inner product $\langle \cdot , \cdot \rangle_{\mathcal{B}}$, but with the plus pairing $\ll \cdot, \cdot \gg$ defined in \eqref{Plus Pairing}, i.e., 
\begin{align}
	\ll (f^1,e^1), (f^2, e^2) \gg \hspace{0.1 cm}= \langle e^1 \mid f^2 \rangle + \langle e^2 \mid f^1 \rangle, \hspace{0.5 cm} ((f^1, e^1), (f^2, e^2)) \in \mathcal{B} \times \mathcal{B}.
	\label{Plus Pairing - Distributed-Parameter}
\end{align}
This symmetric bilinear form depends on the choice of the power pairing $\langle \cdot \mid \cdot \rangle \colon \mathcal{E} \times \mathcal{F} \to \mathbb{R}$, which we will, together with the flow and effort spaces, specify shortly. 
\vspace{0.5 cm}\\
We may now define the notion of a Dirac structure on an infinite-dimensional Hilbert space.
\begin{definition}
	\label{Definition Dirac Structure - Infinite-dimensional Case}
	Let the flow space $\mathcal{F}$ be an infinite-dimensional Hilbert space, and let the effort space $\mathcal{E} = \mathcal{F}'$ be its dual space. A linear subspace $\mathcal{D} \subset \mathcal{B}$ of the bond space $\mathcal{B} =  \mathcal{F} \times \mathcal{E}$ is called a Dirac structure if 
	\begin{align*}
		\mathcal{D} = \mathcal{D}^{\perp \!\!\! \perp}, 
	\end{align*}
	where $\mathcal{D}^{\perp \!\!\! \perp}$ denotes the orthogonal complement of $\mathcal{D}$ with respect to the plus pairing $\ll \cdot, \cdot \gg$ defind in \eqref{Plus Pairing - Distributed-Parameter}.
\end{definition}
Note that this is precisely the characterization of a Dirac structure in the finite-dimensional case (cf. Proposition \ref{Proposition Characterization Dirac Structure}), and that this implies that the Dirac structure $\mathcal{D}$ captures the power-conserving interconnection structure of a system. For a general definition of an infinite-dimensional port-Hamiltonian system with total energy \eqref{General Lossless Port-Hamiltonian System - Hamiltonian H}, see \cite[Definition~1.13]{Villegas}. 
\vspace{0.5 cm}\\
In Subsection \ref{Subsection Boundary Port Variables} we have defined boundary port variables which we need to incorporate into the definition of the power pairing and thus the plus pairing. By Definition \ref{Definition Dirac Structure - Infinite-dimensional Case}, this is necessary for defining the Dirac structure associated with $\mathcal{J}$ defined by \eqref{Skew-Symmetric Differential Operator J - N-Dimensional}-\eqref{Skew-Symmetric Differential Operator J - N-Dimensional - Corresponding Matrices}. First, we need to select the flow space and the effort space appropriately. Because the bilinear form contains both state space elements and their evaluation at the boundary, we set
\begin{align}
	\mathcal{F} = \mathcal{E} = L^2 \left( [a,b], \mathbb{R}^n \right) \times \mathbb{R}^{nN},
	\label{Flow And Effort Space - Linear Infinite-dimensional Systems}
\end{align}
endowed with the canonical inner product:
\begin{align*}
	\langle (f^1,f_{\partial}^1), (f^2,f_{\partial}^2) \rangle_{\mathcal{F}} = \langle f^1, f^2 \rangle_{L^2} + \langle f_{\partial}^1, f_{\partial}^2 \rangle_2, \hspace{0.5 cm} (f^1, f_{\partial}^1), (f^2,f_{\partial}^2) \in \mathcal{F}.
\end{align*}
In contrast to the finite-dimensional case, the power pairing does not necessarily coincide with the canonical duality pairing/inner product, but is slightly modified. Here, we choose the power pairing
\begin{align}
	\left\langle \begin{pmatrix}
		e \\
		e_{\partial}
	\end{pmatrix} \Bigg| \begin{pmatrix}
	f \\
	f_{\partial}
\end{pmatrix} \right\rangle = \langle e , f \rangle_{L^2} - \langle e_{\partial}, f_{\partial} \rangle_2, \hspace{0.5 cm} \left( \begin{pmatrix}
e \\
e_{\partial}
\end{pmatrix}, \begin{pmatrix}
f \\
f_{\partial}
\end{pmatrix} \right) \in \mathcal{E} \times \mathcal{F}. 
\label{Power Pairing - Infinite-Dimensional Case}
\end{align}
The reason behind that modification will be clear soon. Altogether, we endow the bond space $\mathcal{B} = \mathcal{F} \times \mathcal{E}$ with the following plus pairing:
\begin{align}
	\ll (f^1, f_{\partial}^1, e^1, e_{\partial}^1), (f^2, f_{\partial}^2, e^2, e_{\partial}^2) \gg \hspace{0.1 cm} := \langle e^2, f^1 \rangle_{L^2} - \langle e_{\partial}^2, f_{\partial}^1 \rangle_2  + \langle e^1, f^2 \rangle_{L^2}  - \langle e_{\partial}^1, f_{\partial}^2 \rangle_2, 
	\label{Plus Pairing - Infinite-Dimensional Case}
\end{align}
where $(f^i, f_{\partial}^i, e^i, e_{\partial}^i) \in \mathcal{B}$, $i=1,2$.
\vspace{0.5 cm}\\
Now we have everything required to define the Dirac structure associated with the formally skew-symmetric operator $\mathcal{J}$ given by (\ref{Skew-Symmetric Differential Operator J - N-Dimensional})-(\ref{Skew-Symmetric Differential Operator J - N-Dimensional - Corresponding Matrices}).
\begin{theorem}[Theorem 2.7 in \cite{Villegas}]
	\label{Theorem Dirac Structure Associated With Skew-Symmetric Operator J}
	Consider the formally skew-symmetric operator
	$\mathcal{J}$ given by  \eqref{Skew-Symmetric Differential Operator J - N-Dimensional}-\eqref{Skew-Symmetric Differential Operator J - N-Dimensional - Corresponding Matrices}. Let $\mathcal{B} = \mathcal{F} \times \mathcal{E}$ be the bond space with respect to the flow space $\mathcal{F}$ and the effort space $\mathcal{E}$ defined in \eqref{Flow And Effort Space - Linear Infinite-dimensional Systems}. Then the subspace $\mathcal{D}_{\mathcal{J}} \subset \mathcal{B}$ given by
	\begin{align}
		\mathcal{D}_{\mathcal{J}} := \left\{ \left( \begin{pmatrix}
			f \\
			f_{\partial} \end{pmatrix},
		\begin{pmatrix}
			e \\
			e_{\partial}
		\end{pmatrix} \right) \in \mathcal{B} \hspace{0.1 cm} \Bigg| \hspace{0.1 cm} e \in H^N \left( [a,b], \mathbb{R}^n \right), \hspace{0.1 cm} f = \mathcal{J}e , \hspace{0.1 cm} \begin{bmatrix}
		f_{\partial, e} \\
		e_{\partial, e}
	\end{bmatrix} = R_{\extern} \trace (e) \right\}
\label{Stokes-Dirac structure - Linear PH Systems}
	\end{align}
is a Dirac structure with respect to the plus pairing
$\ll\cdot,\cdot\gg \colon \mathcal{B} \times \mathcal{B} \to \mathbb{R}$ defined in \eqref{Plus Pairing - Infinite-Dimensional Case}, and is called the Stokes-Dirac structure.
\end{theorem}
For the proof one has to verify that  $\mathcal{D}_{\mathcal{J}} = \mathcal{D}_{\mathcal{J}}^{\perp \!\!\! \perp}$ holds with respect to the plus pairing $\ll \cdot, \cdot \gg$, according to Definition \ref{Definition Dirac Structure - Infinite-dimensional Case}. We refer to the proof of Theorem 3.6 in \cite{LeGorrec}.
\vspace{0.5 cm}\\
The Stokes-Dirac structure $\mathcal{D}_{\mathcal{J}}$ is the underlying Dirac structure of the class of systems of the form \eqref{General Lossless Port-Hamiltonian System - Infinite-Dimensional Case}. Consequently, and analogously to the finite-dimensional case, the dynamics of these systems can be geometrically specified with respect to $\mathcal{D}_{\mathcal{J}}$, which we are going to do next.
\vspace{0.5 cm}\\
Beforehand, we explain the reason behind the specific choice of the power pairing \eqref{Power Pairing - Infinite-Dimensional Case}: assume that
$(f^i, f_{\partial}^i , e^i, e_{\partial}^i) \in \mathcal{D}_{\mathcal{J}}$ for $i = 1, 2$. As $\mathcal{D} = \mathcal{D}^{\perp \!\!\! \perp}$, we get
\begin{align*}
	0 &= \hspace{0.1 cm} \ll (f^1, f_{\partial}^1, e^1, e_{\partial}^1), (f^2, f_{\partial}^2, e^2, e_{\partial}^2) \gg \\
	&= \hspace{0.1 cm}	\ll (\mathcal{J}e^1, f_{\partial}^1, e^1, e_{\partial}^1), (\mathcal{J}e^2, f_{\partial}^2, e^2, e_{\partial}^2) \gg \\
	&=  \langle \mathcal{J}e^1, e^2 \rangle_{L^2} + \langle e^1, \mathcal{J}e^2 \rangle_{L^2} - \langle  e_{\partial}^2, f_{\partial}^1 \rangle_2  - \langle e_{\partial}^1, f_{\partial}^2 \rangle_2 \\
	&=  \langle \mathcal{J}e^1, e^2 \rangle_{L^2} + \langle e^1, \mathcal{J}e^2 \rangle_{L^2} - \begin{bmatrix}
		f_{\partial,e^1}^{\top} & e_{\partial, e^1}^{\top}
	\end{bmatrix} \Sigma \begin{bmatrix}
		f_{\partial, e^2} \\
		e_{\partial, e^2}
	\end{bmatrix},
\end{align*}
with $\Sigma$ defined in \eqref{Matrix Sigma}.
This is exactly equation \eqref{Theorem Stokes-Like Theorem - Equation Representation 2}, which, in particular, has been exploited in the proof of Theorem \ref{Theorem Dirac Structure Associated With Skew-Symmetric Operator J}. Now, comparing the system description \eqref{General Lossless Port-Hamiltonian System - Infinite-Dimensional Case} and the constitutive relations \eqref{Stokes-Dirac structure - Linear PH Systems} of $\mathcal{D}_{\mathcal{J}}$, we infer:
\begin{enumerate}
	\item[(i)] Define the temporal change of the state variables as the flow variables, i.e., $f = \frac{\partial x}{\partial t}$. 
	\item[(ii)] Define the variational derivative of $H$ as the effort variables, $e = \frac{\delta H}{\delta x}(x) = \mathcal{Q}x$. 
	\item[(iii)] Encompass the two external boundary port variables $f_{\partial,e}, e_{\partial,e}$ defined with respect to the effort variables $e$, see Definition \ref{Definition Boundary Port Variables in The Linear Case}.
\end{enumerate}
Thus, we are able to specify the dynamics of the system \eqref{General Lossless Port-Hamiltonian System - Infinite-Dimensional Case} with respect to the Stokes-Dirac structure $\mathcal{D}_{\mathcal{J}}$ in \eqref{Stokes-Dirac structure - Linear PH Systems}. By doing so, the power flow through the boundary of the spatial domain is encompassed, which is why we call such systems \emph{boundary port-Hamiltonian systems}. 
\begin{corollary}
	Consider the Hamiltonian system \eqref{General Lossless Port-Hamiltonian System - Infinite-Dimensional Case} with Hamiltonian \eqref{General Lossless Port-Hamiltonian System - Hamiltonian H}, which is augmented with the boundary port variables given in Definition \ref{Definition Boundary Port Variables in The Linear Case}. The dynamics of the resulting boundary port-Hamiltonian system are geometrically specified by the requirement that for all $t >0$ we have
	\begin{align*}
		\left( \begin{pmatrix}
			\frac{\partial}{\partial t}x(\cdot,t) \\
			f_{\partial}(t) 
		\end{pmatrix}, \begin{pmatrix}
		\mathcal{Q}x(\cdot,t) \\
		e_{\partial}(t)
	\end{pmatrix} \right) \in \mathcal{D}_{\mathcal{J}},
	\end{align*}
with $f_{\partial}(t) = f_{\partial, \mathcal{Q}x(\cdot,t)}$, $e_{\partial}(t) = e_{\partial, \mathcal{Q}x(\cdot,t)}$, and $\mathcal{D}_{\mathcal{J}}$ the Stokes-Dirac structure \eqref{Stokes-Dirac structure - Linear PH Systems} defined with respect to the plus pairing \eqref{Plus Pairing - Infinite-Dimensional Case}. 
\end{corollary}
Recall that the energy space $X$ is endowed with the inner product $\langle \cdot , \cdot \rangle_{\mathcal{Q}}$ defined in \eqref{Subsection The Lossless Transmission Line - Inner Product WRT Q}. Thus, for a classical solution $x \colon [0, \infty) \to X$ of the system \eqref{General Lossless Port-Hamiltonian System - Infinite-Dimensional Case} with the Hamiltonian $H \colon X \to \mathbb{R}$ given in \eqref{General Lossless Port-Hamiltonian System - Hamiltonian H} we obtain the following balance equation with respect to the boundary port variables: applying \eqref{Theorem Stokes-Like Theorem - Equation Representation 2}, we have for all $t \geq 0$,
\begin{align}
	\begin{split}
	\frac{d}{dt} \|x(\cdot,t) \|_{\mathcal{Q}}^2 = \frac{d}{dt} H(x(\cdot,t)) &= \frac{1}{2} \frac{d}{dt} \langle x(\cdot,t), \mathcal{Q}x(\cdot,t) \rangle_{L^2} \\
	&= \frac{1}{2} \langle \mathcal{J}\mathcal{Q}x(\cdot,t), \mathcal{Q}x(\cdot,t) \rangle_{L^2} + \frac{1}{2} \langle \mathcal{Q}x(\cdot,t), \mathcal{J}\mathcal{Q}x(\cdot,t) \rangle_{L^2} \\
	&= \frac{1}{2} \begin{bmatrix} 
		f_{\partial}^{\top}(t) & e_{\partial}^{\top}(t)
	\end{bmatrix}
\Sigma \begin{bmatrix}
	f_{\partial}(t) \\
	e_{\partial}(t)
\end{bmatrix} \\
&= \langle e_{\partial}(t), f_{\partial}(t) \rangle_2. 
\end{split}
\label{Power Balance for General Linear PH-Systems - Infinite Dimensions}
\end{align}
This equation expresses the fact that the change of internal power solely happens at the boundary of the spatial domain, and therefore generalizes the balance equations \eqref{Subsection Lossless Transmission Line - Energy Balance Equation} and \eqref{Subsection The Vibrating String - Balance Equation} we have derived in the introductory section of this chapter. 
\begin{remark}
	In Section \ref{Section Energy-storing Elements} we pointed out that, in order to distinguish between the power flowing into the energy-storing elements and the power flowing into the Dirac structure, we need to put a minus sign in front of the rate of change of the state $x$, see \eqref{Choice of Port Variables fS and eS}. In the infinite-dimensional case, the power flowing into the energy-storing elements would correspond to the inner product $\langle \frac{\delta H}{\delta x}(x),  \frac{\partial x}{\partial t} \rangle_{L^2}$, and the power flow into the Dirac structure would correspond to the inner product $\langle e_S, f_S \rangle_{L^2}$, leading to the relations $f = f_S = - \frac{\partial x}{\partial t}$ and $e = e_S =\frac{\delta H}{\delta x}(x)$. Necessarily, some minor changes have to be made concerning the Dirac structure and the plus pairing. To be more accurate, the plus pairing in \eqref{Plus Pairing - Infinite-Dimensional Case} would need to be changed to 
	\begin{align*}
		\ll (f^1, f_{\partial}^1, e^1, e_{\partial}^1), (f^2, f_{\partial}^2, e^2, e_{\partial}^2) \gg \hspace{0.1 cm} :=  \langle f^1, e^2 \rangle_{L^2} + \langle f_{\partial}^1, e_{\partial}^2 \rangle_2  + \langle e^1, f^2 \rangle_{L^2}  + \langle e_{\partial}^1, f_{\partial}^2 \rangle_2, 
	\end{align*}
that is, the plus pairing is induced by the canonical power pairing $\langle \cdot \mid \cdot \rangle$ on the Hilbert space $\mathcal{F} = \mathcal{E}$, and not by the one defined in \eqref{Power Pairing - Infinite-Dimensional Case}. Moreover, the Stokes-Dirac structure given in \eqref{Stokes-Dirac structure - Linear PH Systems} would be of the form
	\begin{align*}
	\mathcal{D}_{-\mathcal{J}} = \left\{ \left( \begin{pmatrix}
		f \\
		f_{\partial} \end{pmatrix},
	\begin{pmatrix}
		e \\
		e_{\partial}
	\end{pmatrix} \right) \in \mathcal{B} \hspace{0.1 cm} \Bigg| \hspace{0.1 cm} e \in H^N \left( [a,b], \mathbb{R}^n \right), \hspace{0.1 cm} - f = \mathcal{J}e , \hspace{0.1 cm} \begin{bmatrix}
		f_{\partial, e} \\
		e_{\partial, e}
	\end{bmatrix} = R_{\extern} \trace (e) \right\}.
\end{align*}
With these simple adjustments one could continue the aforementioned power flow convention from Chapter \ref{Chapter Finite-dimensional Port-Hamiltonian Systems}. The proof that $\mathcal{D}_{- \mathcal{J}}$ is indeed a Dirac structure with respect to the modified plus pairing works analogously. However, as both sign conventions are common in the literature (cf.  \cite{LeGorrec, MaschkeLagrange} and \cite{SchaftMaschke02, SchaftMaschke20}), we will stick with the sign convention applied throughout this chapter for the rest of the thesis.  
\end{remark}
This is a good starting point for the analysis of the system \eqref{General Lossless Port-Hamiltonian System - Infinite-Dimensional Case} with Hamiltonian \eqref{General Lossless Port-Hamiltonian System - Hamiltonian H}. To this end, one defines the (maximal) \emph{port-Hamiltonian operator} $A_{\mathcal{Q}} \colon D(A_{\mathcal{Q}}) \subset X \to X$ given by
\begin{align*}
	D(A_{\mathcal{Q}}) &= \left\{ x \in X \mid \mathcal{Q}x \in D(\mathcal{J}) \right\},  \\
	A_{\mathcal{Q}}x &= \mathcal{J}(\mathcal{Q}x), \hspace{0.5 cm} x \in D(A_{\mathcal{Q}}), 
\end{align*}
where $\mathcal{J}$ is the formally skew-symmetric operator defined in \eqref{Skew-Symmetric Differential Operator J - N-Dimensional}-\eqref{Skew-Symmetric Differential Operator J - N-Dimensional - Corresponding Matrices}, and studies the abstract Cauchy problem
\begin{align}
	\begin{split}
		\dot{x}(t) &= A_{\mathcal{Q}}x(t), \hspace{0.5 cm} t >0, \\
		x(0) &= x_0 \in X,
	\end{split}
	\label{Evolution Equation wrt pH-Operator AQ}
\end{align}
as discussed in the beginning of Chapter \ref{Chapter Infinite-dimensional Port-Hamiltonian Systems}. As for every partial differential equation, it is clear that one has to impose boundary conditions in order to guarantee well-posedness of the system \eqref{Evolution Equation wrt pH-Operator AQ}. Indeed, in \cite[Section 4]{LeGorrec} it has been shown that formulating the boundary conditions with respect to the boundary port variables, one can characterize those boundary conditions for which the associated port-Hamiltonian operator generates a contraction semigroup. 
\vspace{0.5 cm}\\
What we may deduce from equation \eqref{Power Balance for General Linear PH-Systems - Infinite Dimensions} is that the change of energy, and therefore the change of the state, is equal to the power flow at the boundary. Hence, it seems natural to let some input (or at least a part of it) act on the boundary of the spatial domain, and to define boundary control systems associated with boundary port-Hamiltonian systems, see \cite[Chapter 11]{JacobZwart}. Some control related problems of such systems have been successfully studied in the PhD theses \cite{Augner16} and \cite{Villegas}. 
\vspace{0.5 cm}\\
In this chapter, we have worked out the main differences of port-based modeling of  finite-dimensional systems to infinite-dimensional systems on a 1-dimensional spatial domain. We give a short overview of the major changes:
\begin{enumerate}
	\item[(i)] The gradient of the Hamiltonian function is replaced by the variational derivative $\frac{\delta H}{\delta x}$ of the Hamiltonian functional (the gradient $\nabla \mathcal{H}$ of the Hamiltonian density). 
	\item[(ii)] The skew-symmetric matrix $J$ is replaced by the formally skew-symmetric differential operator $\mathcal{J}$. 
	\item[(iii)] One has to incorporate the power flow at the boundary of the spatial domain (by means of the boundary port variables $f_{\partial}$ and $e_{\partial}$) into the definition of the (Stokes-) Dirac structure. 
	\item[(iv)] The power pairing $\langle \cdot \mid \cdot \rangle \colon \mathcal{E} \times \mathcal{F} \to \mathbb{R}$ does, in general, not coincide with the canonical inner product on $\mathcal{F} = \mathcal{E}$, but is slightly modified. 
\end{enumerate}
Moreover, in Theorem \ref{Theorem Dirac Structure Associated With Skew-Symmetric Operator J} we have defined the underlying power-conserving interconnection structure, namely the Stokes-Dirac structure $\mathcal{D}_{\mathcal{J}}$, of the system \eqref{General Lossless Port-Hamiltonian System - Infinite-Dimensional Case}. 
\vspace{0.5 cm}\\
Now that we have gained a basic understanding concerning port-based modeling of infinite-dimensional systems, we can finally move on to the main part of this thesis. We want to define two first order Hamiltonian systems that are coupled through a moving interface as a single boundary port-Hamiltonian system.

At this stage it is significant to mention two points:

\chapter{Boundary Port-Hamiltonian Systems with a Moving Interface}
\label{Chapter Boundary Port-Hamiltonian Systems with a Moving Interface}
The main chapter of this thesis is devoted to the formulation of two systems of two conservation laws defined on two complementary intervals that are coupled by some moving interface as a single boundary port-Hamiltonian system on the composed domain. Given the great success of the study of boundary port-Hamiltonian systems from Chapter \ref{Chapter Infinite-dimensional Port-Hamiltonian Systems}, in \cite{Diagne} one wanted to investigate whether the port-Hamiltonian formulation can be extended to two systems that are coupled by a moving interface. As we have already mentioned in the introductory chapter of this thesis, this approach has been motivated by the fact that there are several cases where the system is heterogeneous in the considered spatial domain, and thus may be subdivided in several phases. Due to the distinct physical properties on the respective subdomains, there might arise some discontinuities of some variables at the interface, e.g., some pressure discontinuities of two gases separated by a moving piston, or different heat fluxes at the interface of a liquid-solid transition. These discontinuities are modeled by a set of interface relations, and result in a change of the position of the interface. According to \cite{Diagne}, the port-Hamiltonian formulation of this scenario yields an abstract control system, where the input is defined as the velocity of the moving interface position. 
\vspace{0.5 cm}\\
\emph{Our utmost objective is to rewrite the systems defined in \cite{Diagne} as evolution problems and to analyze them by means of the framework of strongly continuous semigroups}, see Chapter \ref{Chapter Some Background}. It should be mentioned beforehand that we will not be able to cover all problems regarding this fairly new system description, as this is beyond the scope of this thesis. Furthermore, the concept of a system used in this chapter remains purely formal. However, we seek to provide some useful results that may help to prove well-posedness of this class of systems, which one may call boundary port-Hamiltonian interface control systems. This chapter is organized as follows. In Section \ref{Section Port-Hamiltonian System of two Conservation Laws} we recall the port-Hamiltonian formulation of a system of two conservation laws with an arbitrary Hamiltonian functional. Section \ref{Section Two Port-Hamiltonian Systems Coupled by an Interface} deals with the port-Hamiltonian formulation of two systems that are coupled by some interface at a fixed position. We will introduce interface port variables associated with the predefined interface relations, and so-called color functions in order to express the state variables on the total spatial domain. Furthermore, we will define the interconnected system as a port-Hamiltonian system associated with a formally skew-symmetric operator, and present the underyling Dirac structure of this class of systems. Section \ref{Section Port-Hamiltonian Systems Coupled through a Moving Interface} aims to repeat the preceding procedure in case of a moving interface. We will see that this approach yields a system that has the form of an abstract control system, with the input being the velocity of the displacement of the interface position. However, there are several difficulties concerning the proper formulation and analysis of this class of systems. This is why we tackle some of these problems in the upfollowing sections. In Section \ref{Section A Simplified System} we will greatly simplify the systems obtained in Section \ref{Section Two Port-Hamiltonian Systems Coupled by an Interface}, and show that, under certain boundary and interface conditions, the associated port-Hamiltonian operator generates a contraction semigroup, see Theorem~\ref{Interface - Theorem A Generates a Contraction Semigroup}. Afterwards, we present criteria for when this operator even generates an exponentially stable (contraction) semigroup. Furthermore, we specify its Hilbert space adjoint and the associated resolvent operator. Lastly, in Subsection \ref{Subsection Stability of the Family of Infinitesimal Generators} we will come back to the case where the interface is moving over time. Using the simplified system discussed throughout Section~\ref{Section A Simplified System}, we define a family of port-Hamiltonian operators that encompasses the position of the moving interface, and prove the stability of this family of operators.
\vspace{0.5 cm}\\
We assume that some basic concepts about weak derivatives and distributions are known to the reader. For a brief overview, we refer to the first two sections in \cite[Chapter 13]{Tucsnak-Weiss}. For a comprehensive exposition we suggest \cite{Vladimirov}. We will denote the space of test functions by  $\mathcal{D}(a,b) := \mathcal{C}_{c}^{\infty}((a,b),\mathbb{R})$, and the vector space of distributions by $\mathcal{D}'(a,b)$. Throughout, we will identify a function $f \in L^2([a,b], \mathbb{R})$ with its regular distribution $\Lambda_{f} \colon \mathcal{D}(a,b) \to \mathbb{R}$. In particular, we will sometimes regard the space $L^2([a,b], \mathbb{R})$ as a subspace of $\mathcal{D}'(a,b)$. The action of a regular distribution $\Lambda_{f}$ on a test function $\varphi \in \mathcal{D}(a,b)$ is denoted by 
\begin{align*}
	\langle \Lambda_f , \varphi \rangle_{\mathcal{D}'(a,b)} := \langle f, \varphi \rangle_{L^2} = \int_{a}^{b} f(z) \varphi(z) \, dz. 
\end{align*}

\section{Port-Hamiltonian System of two Conservation Laws}
\label{Section Port-Hamiltonian System of two Conservation Laws}

In this section, we briefly recall the port-Hamiltonian formulation of a system of two conservation laws describing an energy-conserving exchange of energy between different physical domains. We have already come across this type of systems in Section \ref{Section Motivating Examples}, namely the first-order systems of the lossless transmission line and the vibrating string. However, in this section, we will slightly generalize these systems by considering arbitrary Hamiltonian functionals.
\vspace{0.5 cm}\\
Let $X = L^2([a,b], \mathbb{R}^2)$ be endowed with the standard $L^2$-inner product $\langle \cdot, \cdot \rangle_{L^2}$ defined in \eqref{L^2 inner product}. As mentioned before, let $H \colon X \to \mathbb{R}$ be some arbitrary Hamiltonian functional of the form
\begin{align}
	H(x) = \int_{a}^{b} \mathcal{H}(x(z)) \, dz, \hspace{0.5 cm} x \in X,
	\label{Interface Paper - Hamiltonian H}
\end{align}
with Hamiltonian density $\mathcal{H} \in \mathcal{C}^{\infty}(\mathbb{R}^2, \mathbb{R})$ and variational derivative $\frac{\delta }{\delta x}H \colon X \to X$, as described in the beginning of Section \ref{Section Motivating Examples}. Now, consider a system of two conservation laws
\begin{align}
	\frac{\partial}{\partial t} x(z,t)  + \frac{\partial}{\partial z} \mathcal{N}(x(\cdot,t))(z) = 0, \hspace{0.5 cm} z \in [a,b], \, t >0,
	\label{Interface Paper - Standard PDE - Conservation Law}
\end{align}
on the state space $X$, with the effort variables defined as
\begin{align*}
	\mathcal{N}(x) = \begin{bmatrix}
		\mathcal{N}_1(x) \\
		\mathcal{N}_2(x)
	\end{bmatrix} = \begin{bmatrix}
		0 & 1 \\
		1 & 0 
	\end{bmatrix} \begin{bmatrix}
	\delta_{x_1} H(x) \\
	\delta_{x_2} H(x)
\end{bmatrix}, \hspace{0.5 cm} x \in X.
\end{align*}
Then the system of conservation laws \eqref{Interface Paper - Standard PDE - Conservation Law} may be equivalently rewritten as the Hamiltonian system
\begin{align}
\frac{\partial}{\partial t} x(z,t) = \mathcal{J} \frac{\delta}{\delta x} H(x(\cdot,t))(z)
	\label{Interface Paper - Standard PDE}
\end{align}
generated by the Hamiltonian \eqref{Interface Paper - Hamiltonian H} and defined with respect to the differential operator 
$\mathcal{J} \colon D(\mathcal{J}) \subset X \to X$ given by
\begin{align*}
	\begin{split}
	D(\mathcal{J}) &= H^1([a,b], \mathbb{R}^2), \\
	\mathcal{J}x &= \begin{bmatrix}
		0 & - \frac{d}{dz} \\
		-\frac{d}{dz} & 0 
	\end{bmatrix}x, \hspace{0.5 cm} x \in D(\mathcal{J}).
\end{split}
\end{align*}
Since the matrix
\begin{align}
P_1 = \begin{bmatrix}
	0 & - 1 \\
	-1 & 0
\end{bmatrix} \in \mathbb{R}^{2 \times 2}
\label{Interface Paper - Matrix P1}
\end{align} 
is symmetric, recall that, by Theorem \ref{Theorem Stokes-Like Theorem}, the matrix differential operator $\mathcal{J} = P_1 \frac{d}{dz}$ is formally skew-symmetric. 
\vspace{0.5 cm}\\
Next, we need to augment the Hamiltonian system \eqref{Interface Paper - Standard PDE} with boundary port variables. In contrast to \cite{Diagne}, we will stick with the definition from Section \ref{Section Boundary Port-Hamiltonian Systems Associated with Skew-symmetric Operators}. This will have an impact on the definition of the Dirac structure and the corresponding plus pairing $\langle \cdot \mid \cdot \rangle$, and hence slightly deviates from the execution in \cite{Diagne}. Now, since $N=1$, the matrix $\mathcal{P}$ defined in \eqref{Matrix P - Distributed-Parameter} reduces to $\mathcal{P} = P_1$, and the invertible matrix $R_{\extern}$ defined in \eqref{Matrix Rext} is therefore given by
\begin{align}
	\label{Interface Paper - Matrix Rext}
	R_{\extern} = \frac{1}{\sqrt{2}} \begin{bmatrix}
		P_1 & - P_1 \\
		I & I 
	\end{bmatrix} \in \mathbb{R}^{4 \times 4}. 
\end{align}
According to \eqref{Boundary Flow and Effort Linear Case}, the boundary flow variable $f_{\partial} = f_{\partial, e} \in \mathbb{R}^2$ and boundary effort variable $e_{\partial} = e_{\partial, e} \in \mathbb{R}^2$ are for all effort variables $ e  \in D(\mathcal{J})$ given by
\begin{align*}
	\begin{bmatrix}
		f_{\partial} \\
		e_{\partial}
	\end{bmatrix} = R_{\extern} \trace(e) = \begin{bmatrix}  \frac{1}{\sqrt{2}} (P_1e(b) - P_1 e(a)) \\
		\frac{1}{\sqrt{2}} ( e(b) + e(a))
	\end{bmatrix}. 
\end{align*}
The only difference is that the geometric specification of the dynamics of the system \eqref{Interface Paper - Standard PDE} with respect to the Stokes-Dirac structure $\mathcal{D}_{\mathcal{J}}$ defined in \eqref{Stokes-Dirac structure - Linear PH Systems} becomes
	\begin{align*}
		\left( \begin{pmatrix}
			\partial_t x_1 (\cdot,t) \\
			\partial_t x_2 (\cdot,t) \\
			f_{\partial}(t) 
		\end{pmatrix}, \begin{pmatrix}
		\delta_{x_1} H(x(\cdot,t)) \\
		\delta_{x_2} H(x(\cdot,t)) \\
		e_{\partial}(t)
	\end{pmatrix} \right) \in \mathcal{D}_{\mathcal{J}}, \hspace{0.5 cm} t >0,
	\end{align*}
where $f_{\partial}(t) = f_{\partial, \delta_xH(x(\cdot,t))}$ and $e_{\partial}(t) = e_{\partial, \delta_x H(x(\cdot,t))}$. 
\vspace{0.5 cm}\\
Due to the constitutive relation of the Dirac structure $\mathcal{D}_{\mathcal{J}}$, the rate of energy change along a solution $x \colon [0, \infty) \to X$ of \eqref{Interface Paper - Standard PDE} is still given for all $t \geq 0$ by the following balance equation:
\begin{align*}
	&\hspace{0.5 cm} \frac{d}{dt} H(x(\cdot,t)) \\
	&= \frac{d}{dt} \int_{a}^{b} \mathcal{H}(x(z,t)) \, dz  \\
	&= \int_{a}^{b} \frac{\partial}{\partial t} \mathcal{H}(x(z,t)) \, dz \\
	&= \int_{a}^{b} \delta_{x_1} H(x(z,t))  \frac{\partial}{\partial t} x_1(z,t) + \delta_{x_2} H(x(z,t)) \frac{\partial}{\partial t} x_2 (z,t) \, dz \\
	&= \int_{a}^{b}  - \delta_{x_1} H(x(z,t)) \frac{\partial}{\partial z} \big( \delta{x_2} H(x(z,t)) \big) - \delta_{x_2} H(x(z,t)) \frac{\partial}{\partial z} \big( \delta_{x_1}H(x(z,t)) \big) \, dz \\
	&= \int_{a}^{b} - \frac{\partial}{\partial z} \left[ \delta_{x_1 }H(x(z,t)) \delta_{x_2}H(x(z,t)) \right] \, dz \\
	&= - \big[ \delta_{x_1}H(x(z,t)) \delta_{x_2}H(x(z,t)) \big]_{a}^{b} \\
	&= \frac{1}{2} \big[ \delta_x H(x(z,t))^{\top} P_1 \delta_x H(x(z,t)) \big]_{a}^{b} \\
	&= \langle e_{\partial}(t) , f_{\partial}(t) \rangle_2,
\end{align*}
where $\delta_x H(x(z,t)) = \delta_x H(x(\cdot,t))(z)$. So, just as in Subsection \ref{Subsection Dirac Structure and Boundary Port-Hamiltonian Systems}, this equality states that the system's change of energy equals the power supplied to the system through the boundary.  
\vspace{0.5 cm}\\
Now that we have clarified the port-Hamiltonian formulation of systems of two conservation laws defined by an arbitrary Hamiltonian functional $H$ of the form \eqref{Interface Paper - Hamiltonian H}, we may deal with our main concern next.

\section{Two Port-Hamiltonian Systems Coupled by an Interface}
\label{Section Two Port-Hamiltonian Systems Coupled by an Interface}

In the next two sections we examine two boundary port-Hamiltonian systems defined on two complementary intervals which are coupled by a moving interface. We seek to derive a formulation of the interconnected system as a boundary port-Hamiltonian system. It is well-known that the (power-conserving) interconnection of port-Hamiltonian systems again defines a port-Hamiltonian system. This interconnection is realized by the composition of the underyling Dirac structures through a number of mutually shared port variables. The internally stored energy of the interconnected system, described by the Hamiltonian on the product space, is simply the sum of the Hamiltonians of the constituting subsystems. We refer to \cite[Chapter 6]{SchaftJeltsema} for the finite-dimensional case as well as to \cite[Chapter 7]{Villegas} and the references specified therein for the infinite-dimensional case. In \cite[Section 2.2]{Diagne} it was depicted how port-Hamiltonian systems are coupled through their (fixed) boundaries.
\vspace{0.5 cm}\\
However, as we consider moving interfaces, we will have to deal with time-varying spatial domains, and thus we have to take a different approach. In order to keep track of the interface position, we are going to take the characteristic functions of the time-varying spatial domains of each subsystem to be part of the state variables. In preparation for the time-varying case studied in Section \ref{Section Port-Hamiltonian Systems Coupled through a Moving Interface}, this section is devoted to the aforementioned formulation as a boundary port-Hamiltonian system in case of a fixed interface position, and is organized as follows: In Subsection \ref{Subsection Prolongation with Color Functions} we introduce the characteristic functions of the complementary subdomains of the interconnected subsystems, called color-functions, as part of the state variables. Furthermore, we define the Hamiltonian and its variational derivative on the extended state space, and we specify interface port variables with respect to some predefined interface relations. In Subsection \ref{Subsection Balance Equations of the State Variables} we state the balance equations of the constituting state variables. The main difficulty relies on the integration of the interface relations into the definition of these balance equations. Finally, in Subsection \ref{Subsection Port-Hamiltonian Formulation and Dirac Structure} we are going to write the system of balance equations as a (generalized) Hamiltonian system with respect to some formally skew-symmetric operator on the extended state space. Moreover, by taking the interface and boundary port variables into account, we may define the underlying Dirac structure and state a balance equation of the total system with respect to the boundary and the interface port variables.

\subsection{Color Functions and Interface Port Variables}
\label{Subsection Prolongation with Color Functions}
Our starting situation is as follows: let $H^- \colon L^2((a,0), \mathbb{R}^2) \to \mathbb{R}$ and
$H^+ \colon L^2((0,b), \mathbb{R}^2) \to \mathbb{R}$ be two Hamiltonian functionals given by
\begin{align}
	H^-(x^-) &= \int_{a}^{0} \mathcal{H}^-(x^-(z)) \, dz, \hspace{0.5 cm} x^- \in L^2((a,0), \mathbb{R}^2), \hspace{0.3 cm} \text{and} \label{Interface Paper - Hamiltonian H-} \\
	H^+(x^+) &= \int_{0}^{b} \mathcal{H}^+(x^+(z)) \, dz, \hspace{0.5  cm} x^+ \in L^2((0,b), \mathbb{R}^2), \label{Interface Paper - Hamiltonian H+}
\end{align}
with the distinct energy densities $\mathcal{H}^{\pm} \in \mathcal{C}^{\infty}(\mathbb{R}^2, \mathbb{R})$, respectively, indicating that the physical properties of the system differ on the respective domain. We assume that we may, in a way, extend these Hamiltonians to Hamiltonian functionals $H^{\pm} \colon X \to \mathbb{R}$. Now, consider the two systems of conservation laws
\begin{align}
	\begin{split}
		\frac{\partial}{\partial t} x^- + \frac{\partial}{\partial z} \mathcal{N}^-(x^-) &= 0, \hspace{0.5 cm}  z \in [a,0), \hspace{0.1 cm} t >0,  \\
		\frac{\partial}{\partial t} x^+ + \frac{\partial}{\partial z} \mathcal{N}^+(x^+) &= 0, \hspace{0.5 cm} z \in (0,b], \hspace{0.1 cm} t > 0,
	\end{split} 
	\label{Interface Paper - Two Systems of Conservation Laws on Respective Domain} 
\end{align}
defined on the respective interval $[a,0)$ or $(0,b]$, where
\begin{align}
	\mathcal{N}^-(x^-) = \begin{bmatrix}
		0 & 1 \\
		1 & 0
	\end{bmatrix} \frac{\delta}{\delta x^-}H^-(x^-) \hspace{0.3 cm} \text{and} \hspace{0.3 cm} \mathcal{N}^+(x^+) = \begin{bmatrix}
	0 & 1 \\
	1 & 0 
\end{bmatrix} \frac{\delta}{\delta x^+}H^+(x^+)
\label{Interface Paper - Relation Flux Variable and Variational Derivative}
\end{align} are the \emph{flux variables} corresponding to the Hamiltonians $H^-$ and $H^+$ of the respective system defined in \eqref{Interface Paper - Hamiltonian H-} and \eqref{Interface Paper - Hamiltonian H+}, respectively. Our objective is to interconnect these two systems to a single port-Hamiltonian system on the composed spatial domain $[a,b]$, where the position of the interface is assumed to be at the fixed position $z=0$. In order to do so, we need to introduce suitable state variables defined on the composed domain, and we need to define some \emph{interface relations} at $z =0$. These interface relations simultaneously constitute the \emph{interface port variables} $(e_I, f_I) \in \mathbb{R}^2$ which we will need do define the underyling Dirac structure of the coupled system. These are given by
\begin{align}
	f_I &= \mathcal{N}_1^+(x^+)(0^+) = \mathcal{N}_1^-(x^-)(0^-), \label{Interface Paper - Interface Variable fI} \\
	-e_I &= \mathcal{N}_2^+(x^+)(0^+) - \mathcal{N}_2^-(x^-)(0^-). \label{Interface Paper - Interface Variable eI}
\end{align}
Equation \eqref{Interface Paper - Interface Variable fI} is a continuity equation of the flux variable $\mathcal{N}_1$, called the privileged variable, and \eqref{Interface Paper - Interface Variable eI} is a balance equation of the flux variable $\mathcal{N}_2$ with external (source) term $e_I$.
\vspace{0.5 cm}\\
In order to keep track of the interface position, we use the characteristic functions of the domains of the two systems given in \eqref{Interface Paper - Two Systems of Conservation Laws on Respective Domain}, called \emph{color functions}:
\begin{align}
	c_0(z,t) = \begin{cases} 
1, & z \in [a,0), \\
0, & z \in [0,b],
\end{cases} \hspace{0.3 cm} \text{and} \hspace{0.3 cm} \overline{c}_0(z,t) =  \begin{cases} 
0, & z \in [a,0], \\
1, & z \in (0,b].
\end{cases}
\label{Interface Paper - Characteristic Functions c0 and c0overline}
\end{align}
This allows us to write the state variables of the coupled system as the sum of prolongations of the variables of each subsystem to the composed domain $[a,b]$:
\begin{align}
	\begin{split}
	x(z,t) &= x^-(z,t) + x^+(z,t) \\
	&= c_0(z,t)x(z,t) + \overline{c}_0(z,t)x(z,t).
\end{split}
\label{Interface Paper - State Variable x Split}
\end{align}
Next, we want to define the Hamiltonian and, hence,  the flux variables of the total system. To this end, we introduce the extended state space $\tilde{X} =L^2([a,b], \mathbb{R}^4)$. The extended state variable is denoted by
\begin{align}
	\tilde{x} = (x, c, \overline{c}) \in \tilde{X},
	\label{Interface Paper - Extended State Variable xtilde}
\end{align}
where we want to stress that $\overline{c}$ is fully independent of $c$. While in the fixed interface scenario, it might be unnecessary to consider the extended state space $\tilde{X}$ (the color functions $c_0$ and $\overline{c}_0$ are time-independent), this is essential in case of a moving interface discussed in Section~\ref{Section Port-Hamiltonian Systems Coupled through a Moving Interface}. So, in preparation for the moving interface scenario, we will (for now) formulate the coupled system on the extended state space $\tilde{X}$.
\vspace{0.5 cm}\\ 
We start with the energy density $\mathcal{H} \in \mathcal{C}^{\infty}(\mathbb{R}^4, \mathbb{R})$, which is given by
\begin{align}
	\mathcal{H}(x,c, \overline{c}) = c \mathcal{H}^-(x) + \overline{c} \mathcal{H}^+(x), \hspace{0.5 cm} \tilde{x} \in \mathbb{R}^4.
	\label{Interface Paper - Hamiltonian Density Split}
\end{align}
Next, define $H \colon \tilde{X} \to \mathbb{R}$ as
\begin{align}
	H(x,c, \overline{c}) = \int_{a}^{b} \mathcal{H}(x(z), c(z), \overline{c}(z)) \, dz, \hspace{0.5 cm} \tilde{x} \in \tilde{X}.
	\label{Interface Paper - Hamiltonian H of Coupled System}
\end{align}
Since $H^{\pm} \colon X \to \mathbb{R}$ are functionals, one can readily see that $H$ is a functional as well. The variational derivative $\delta_{\tilde{x}} H \colon \tilde{X} \to \tilde{X}$ is again induced by the gradient of the energy density $\mathcal{H}$ defined in \eqref{Interface Paper - Hamiltonian Density Split}, i.e., 
\begin{align}
	\label{Interface Paper - Extended Variational Derivatives}
	\delta_{\tilde{x}} H(\tilde{x}) = \begin{bmatrix}
		\delta_{x} H(x, c, \overline{c}) \\
		\delta_{c} H(x, c, \overline{c}) \\
		\delta_{\overline{c}} H(x, c, \overline{c})
	\end{bmatrix} = \begin{bmatrix}
		c \delta_{x} H^-(x) +  \overline{c} \delta_{x} H^+(x) \\
		\mathcal{H}^-(x) \\
		\mathcal{H}^+(x)
	\end{bmatrix}, \hspace{0.5 cm } \tilde{x} \in \tilde{X}.
\end{align}
If we  now choose $c$ and $\overline{c}$ as the color functions defined in \eqref{Interface Paper - Characteristic Functions c0 and c0overline}, then we obtain for all $(x, c_0, \overline{c}_0) \in \tilde{X}$,
\begin{align*}
	H(x, c_0, \overline{c}_0) &= \int_{a}^{b} c_0(z) \mathcal{H}^-(x(z)) + \overline{c}_0(z) \mathcal{H}^+(x(z)) \, dz \\
	&= \int_{a}^{0} \mathcal{H}^-(x(z)) \, dz + \int_{0}^{b} \mathcal{H}^+(x(z)) \, dz \\
	&= H^-(x^-) + H^+(x^+). 
\end{align*}
So, $H(\cdot ,c_0, \overline{c}_0)$ describes the internally stored energy, which is given by the sum of the Hamiltonians of the subsystems defined in \eqref{Interface Paper - Two Systems of Conservation Laws on Respective Domain}, and therefore is the Hamiltonian of the coupled system.
\vspace{0.5 cm}\\
The vector of flux variables of the coupled system becomes
\begin{align}
	\mathcal{N}(x,c_0, \overline{c}_0) = c_0 \mathcal{N}^-(x) + \overline{c}_0 \mathcal{N}^+(x), \hspace{0.5 cm} x \in X, 
	\label{Interface Paper - Flux Variable N Split}
\end{align}
which obviously satisfies 
\begin{align*}
	c_0 \mathcal{N}(x,c_0, \overline{c}_0) =  c_0 \mathcal{N}^-(x) \hspace{0.3 cm} \text{and} \hspace{0.3 cm} 	\overline{c}_0 \mathcal{N}(x,c_0, \overline{c}_0)  =  \overline{c}_0 \mathcal{N}^+(x).
\end{align*} 
For the sake of brevity, for $i = 1,2,$ we will sometimes write $\mathcal{N}_i^{\pm}(z,t)$ instead of $\mathcal{N}_i^{\pm}(x)(z,t)$, and $\mathcal{N}_i(z,t)$ instead of
\begin{align*}
	 \mathcal{N}_i(x,c_0, \overline{c}_0)(z,t) = \mathcal{N}_i(x(\cdot,t), c_0(\cdot,t), \overline{c}_0(\cdot,t))(z). 
\end{align*} 
Note that the first row of \eqref{Interface Paper - Extended Variational Derivatives} corresponds to the definition of the flux variable (\ref{Interface Paper - Flux Variable N Split}) in the following way:
\begin{align}
	\begin{split}
	\begin{bmatrix}
		0 & 1 \\
		1 & 0 
	\end{bmatrix} \delta_x H(x, c_0, \overline{c}_0) &= \begin{bmatrix}
	0 & 1 \\
	1 & 0 
\end{bmatrix}  c_0 \delta_x H^-(x) + \overline{c}_0  \delta_x H^+ (x)  \\
 &= c_0 \mathcal{N}^-(x) + \overline{c}_0 \mathcal{N}^+(x)  \\
 &= \mathcal{N}(x, c_0, \overline{c}_0). 
\end{split}
\label{Interface Paper - Relation Flux Variable and Variational Derivative - Composed Domain}
\end{align}
In order to define the coupled system as a Hamiltonian-like system on the extended state space $\tilde{X}$, we need to define the balance equation of each variable constituting the extended state variable $\tilde{x} = (x, c_0, \overline{c}_0)$, and to include the interface relations \eqref{Interface Paper - Interface Variable fI}-\eqref{Interface Paper - Interface Variable eI}. This is what we are going to do next.

\subsection{Balance Equations of the State Variables}
\label{Subsection Balance Equations of the State Variables}

The coupling of the two systems of conservation laws \eqref{Interface Paper - Two Systems of Conservation Laws on Respective Domain} at the (interface) position $z=0$ with the interface relations \eqref{Interface Paper - Interface Variable fI}-\eqref{Interface Paper - Interface Variable eI} led us to the introduction of the prolongated flux variables \eqref{Interface Paper - Flux Variable N Split} defined on the composed domain $[a,b]$. Consequently, the flux variable $\mathcal{N}_1(x, c_0, \overline{c}_0)$ satisfies a continuity equation at the interface, whereas the flux variable $\mathcal{N}_2(x, c_0, \overline{c}_0)$ satisfies a balance equation at the interface. In the following, we want to state a conservation law for each of the variables $x_1$, $x_2$, $c_0$, and $\overline{c}_0$ on the whole domain $[a,b]$ having regard to the interface relations.
\paragraph*{Conservation laws of $\mathbf{c_0}$ and $\mathbf{\overline{c}_0}$.} First of all, let us state the balance equations of the color functions $c_0$ and $\overline{c}_0$ defined in \eqref{Interface Paper - Characteristic Functions c0 and c0overline} which are part of the extended state variable $\tilde{x} \in \tilde{X} = L^2([a,b], \mathbb{R}^4)$ defined in \eqref{Interface Paper - Extended State Variable xtilde} as well. It is clear that, for all $t \geq 0$, they satisfy the trivial conservation laws
\begin{align}
	\partial_t c_0 = 0 \hspace{0.3 cm} \text{and} \hspace{0.3 cm} \partial_t \overline{c}_0 = 0,
	\label{Interface Paper - Conservation Laws of Color Functions with Fixed Interface Position}
\end{align}
respectively.
\paragraph*{Conservation law of $\mathbf{x_1}$.} Next, consider the conservation law of the state variable $x_1$ on $[a,b]$. Assuming that the individual conservation laws \eqref{Interface Paper - Two Systems of Conservation Laws on Respective Domain} are satisfied, with the aid of the flux variable \eqref{Interface Paper - Flux Variable N Split}, it may be formally written as
\begin{align}
	\begin{split}
	\partial_t x_1 &= - \partial_z \mathcal{N}_1(x, c_0, \overline{c}_0) \\
	&= - \partial_z ( c_0 \mathcal{N}_1^-(x) + \overline{c}_0 \mathcal{N}_1^+(x)  ) \\
	&= - \partial_z \left( c_0 \mathcal{N}_1(x, c_0, \overline{c}_0) + \overline{c}_0 \mathcal{N}_1(x, c_0, \overline{c}_0)  \right) \\
	&= \underbrace{- \left[ \partial_z ( c_0 \cdot) + \partial_z (\overline{c}_0 \cdot) \right]}_{=: \mathbf{d}_0} \mathcal{N}_1(x, c_0, \overline{c}_0).
\end{split}
\label{Interface Paper - Conservation Law of x1 - Nonmoving Interface}
\end{align}
To begin with, for all $t >0$, equation (\ref{Interface Paper - Conservation Law of x1 - Nonmoving Interface}) holds in $\mathcal{D}'(a,b)$: let $\varphi \in \mathcal{D}(a,b)$. Then, due to \eqref{Interface Paper - Two Systems of Conservation Laws on Respective Domain} and the continuity condition \eqref{Interface Paper - Interface Variable fI}, the following holds for all $t > 0$:
	\begin{align*}
		\allowdisplaybreaks
	&\hspace{0.5 cm} \int_{a}^{b} \left[ \mathbf{d}_0  \mathcal{N}_1(x,c_0, \overline{c}_0) (z,t) \right] \varphi(z) \, dz \\
	&= \int_{a}^{b} - \left[ \partial_z (c_0(z,t) \mathcal{N}_1(z,t)) + \partial_z (\overline{c}_0 (z,t) \mathcal{N}_1(z,t)) \right] \varphi(z) \,  dz  \\
	&= \int_{a}^{b} - \left[ \partial_z (c_0(z,t) \mathcal{N}_1 (z,t)) \right] \varphi(z) \, dz + \int_{a}^{b} - \left[  \partial_z (\overline{c}_0 (z,t) \mathcal{N}_1(z,t)) \right] \varphi(z) \, dz  \\
\text{(part. int.)} \hspace{0.3 cm} &= \int_{a}^{b} c_0(z,t) \mathcal{N}_1(z,t) \frac{d}{dz} \varphi(z) \, dz + \int_{a}^{b} \overline{c}_0(z,t) \mathcal{N}_1(z,t) \frac{d}{dz} \varphi(z) \, dz \\
	&= \int_{a}^{0} \mathcal{N}_1^-(z,t) \frac{d}{dz} \varphi(z) \, dz + \int_{0}^{b}  \mathcal{N}_1^+(z,t) \frac{d}{dz} \varphi(z) \, dz \\
\text{(part. int.)} \hspace{0.3 cm}	&=  \mathcal{N}_1^-(0^-,t) \varphi(0) + \int_{a}^{0}  - \partial_z \mathcal{N}_1^-(z,t)  \varphi(z) \,  dz \\
	&\hspace{0.4 cm} - \mathcal{N}_1^+(0^+,t) \varphi(0)   + \int_{0}^{b} - \partial_z \mathcal{N}_1^+(z,t)  \varphi(z) \, dz \\
\text{\eqref{Interface Paper - Two Systems of Conservation Laws on Respective Domain}} \hspace{0.3 cm}	&= \mathcal{N}_1^-(0^-,t) \varphi(0) + \int_{a}^{0}  \partial_t x_1^-(z,t)  \varphi(z) \,  dz \color{white}\text{abcdefghijklmnopqrstuvw.} \\
	&\hspace{0.4 cm} - \mathcal{N}_1^+(0^+,t) \varphi(0)   + \int_{0}^{b} \partial_t x_1^+(z,t)  \varphi(z) \, dz \\
\text{\eqref{Interface Paper - Interface Variable fI}} \hspace{0.3 cm}	&= \int_{a}^{b} \partial_t x_1(z,t) \varphi(z) \, dz. 
\end{align*}

Note that for all $t >0$, $\varphi \in \mathcal{D}(a,b)$, one computes
\begin{align*}
- \int_{a}^{b} \frac{\partial}{\partial z} c_0(z,t) \mathcal{N}_1(z,t) \varphi(z) \, dz =  \int_{a}^{0} \frac{\partial}{\partial z} (\mathcal{N}_1^-(z,t) \varphi(z)) \, dz = \mathcal{N}_1^-(0^-,t) \varphi(0), 
\end{align*}
and
\begin{align*}
	- \int_{a}^{b} \frac{\partial}{\partial z} \overline{c}_0(z,t) \mathcal{N}_1(z,t) \varphi(z) \, dz = \int_{0}^{b} \frac{\partial}{\partial z} (\mathcal{N}_1^+(z,t) \varphi(z)) \, dz =  - \mathcal{N}_1^+(0^+,t) \varphi(0). 
\end{align*}
This implies that, due to the continuity condition \eqref{Interface Paper - Interface Variable fI}, the evaluation of the flux variables of the respective subdomain at the interface cancel each other out, and equation \eqref{Interface Paper - Conservation Law of x1 - Nonmoving Interface} becomes
\begin{align}
	\begin{split}
	\partial_t x_1(\cdot,t) &= \mathbf{d}_0 \mathcal{N}_1(x, c_0, \overline{c}_0)(\cdot,t)  \\
	&= - c_0(\cdot,t) \partial_z \mathcal{N}_1^-(x)(\cdot,t) - \partial_z c_0(\cdot,t)\mathcal{N}_1(x, c_0, \overline{c}_0)(\cdot,t) \\
	&\hspace{0.475 cm} -  \overline{c}_0(\cdot,t) \partial_z \mathcal{N}_1^+(x)(\cdot,t) - \partial_z \overline{c}_0(\cdot,t) \mathcal{N}_1(x, c_0, \overline{c}_0)(\cdot,t) \\
	&= - c_0(\cdot,t) \partial_z \mathcal{N}_1^-(x)(\cdot,t) - \overline{c}_0(\cdot,t) \partial_z \mathcal{N}_1^+(x)(\cdot,t).
\end{split}
	\label{Interface Paper - Conservation Law of x1 - Nonmoving Interface - Version 2 in L2}
\end{align}
As a consequence, equation \eqref{Interface Paper - Conservation Law of x1 - Nonmoving Interface - Version 2 in L2} holds, in fact,  in $L^2([a,b], \mathbb{R})$. This motivates to define the operator  $\mathbf{d}_0 \colon D(\mathbf{d}_0) \subset L^2([a,b], \mathbb{R}) \to L^2([a,b], \mathbb{R})$ in the following way:
\begin{align}
	 D(\mathbf{d}_0) &= \left\{ x \in L^2([a,b], \mathbb{R}) \mid x_{|(a,0)} \in H^1((a,0), \mathbb{R}), \hspace{0.1 cm} x_{|(0,b)} \in H^1((0,b), \mathbb{R}), \hspace{0.1 cm} x \in \mathcal{C}([a,b], \mathbb{R}) \right\}  \nonumber \\
	 &\hspace{0.12 cm} = H^1([a,b], \mathbb{R}), \label{Interface Paper - Operator d0} \\
	\mathbf{d}_0 x &= - \left[ \frac{d}{dz} (c_0 x) + \frac{d}{dz} (\overline{c}_0 x) \right], \hspace{0.5 cm} x \in D(\mathbf{d}_0). \nonumber
\end{align}
For a function $x = c_0 x^- + \overline{c}_0 x^+ \in D(\mathbf{d}_0)$ we have in particular that
\begin{align*}
\mathbf{d}_0 x = - c_0 \frac{d}{dz} x^- - \overline{c}_0 \frac{d}{dz} x^+. 
\end{align*} 
This way, the continuity equation \eqref{Interface Paper - Interface Variable fI} of the flux variable $\mathcal{N}_1(x,c_0, \overline{c}_0)$ is incorporated into the domain of the operator $\mathbf{d}_0$. This allows us to define classical solutions of the conservation law of the state variable $x_1$ in the following way.
\begin{definition}
	Let $x = (x_1,x_2) \colon [0, \infty) \to X$. Let $Y = L^2([a,b],\mathbb{R})$. We call $x_1 \colon [0,\infty) \to Y$ a classical solution of the conservation law
	\begin{align}
		\begin{split}
			\partial_t x_1(t) &= \mathbf{d}_0 \mathcal{N}_1 (x, c_0, \overline{c}_0)(\cdot,t),  \hspace{0.5 cm} t > 0, \\
			x_1(0) &= x_{0,1} \in Y, 
		\end{split}
	\label{Interface Paper - Definition Conservation Law of x1 - Nonmoving Interface}
	\end{align}
	where $\mathbf{d}_0$ is the operator given by  \eqref{Interface Paper - Operator d0}, and with the flux variable
	\begin{align*}
		\mathcal{N}_1(x, c_0, \overline{c}_0)  = c_0 \mathcal{N}_1^-(x) + \overline{c}_0 \mathcal{N}_1^+(x)
	\end{align*}
	of the systems of conservation laws \eqref{Interface Paper - Two Systems of Conservation Laws on Respective Domain} coupled at $z = 0$, if:
	\begin{enumerate}
		\item[(i)] $x_1 \in \mathcal{C}([0,\infty), Y)$. 
		\item[(ii)] $x_{1_{ |(0,\infty) }}  \in \mathcal{C}^1((0, \infty), Y)$.
		\item[(iii)] $\mathcal{N}_1(x, c_0, \overline{c}_0)(\cdot,t) \in D(\mathbf{d}_0)$ for all $t > 0$.
		\item[(iv)] $x_1$ satisfies \eqref{Interface Paper - Definition Conservation Law of x1 - Nonmoving Interface}.
	\end{enumerate}
\end{definition}

Note that the operator $\mathbf{d}_0$ acts as the differential operator $-\frac{d}{dz}$ on each subdomain  $[a',b'] \subset [a,b]$ not containing the interface, i.e., with either $a \leq a' < b' < 0$, or with $0 < a' < b' \leq b$. Indeed, according to \eqref{Interface Paper - Conservation Law of x1 - Nonmoving Interface - Version 2 in L2}, for all $t > 0$ it holds that
\begin{align*}
	\frac{d}{dt} \int_{a'}^{b'} x_1(z,t) \, dz = \int_{a'}^{b'} \partial_t x_1^- (z,t) \, dz = \int_{a'}^{b'} - \partial_z \mathcal{N}_1^-(x) (z,t) \, dz = \mathcal{N}_1^-(a',t) - \mathcal{N}_1^- (b',t)
\end{align*}
or
\begin{align*}
	\frac{d}{dt} \int_{a'}^{b'} x_1(z,t) \, dz = \int_{a'}^{b'} \partial_t x_1^+ (z,t) \, dz = \int_{a'}^{b'} - \partial_z \mathcal{N}_1^+(x)(z,t) \, dz =  \mathcal{N}_1^+(a',t) - \mathcal{N}_1^+ (b',t),
\end{align*}
respectively. Now, if we consider an interval $[a',b']$ with $a \leq a' < 0 < b' \leq b$, then the continuity of the flux variable $\mathcal{N}_1(x, c_0 , \overline{c}_0)$ at the interface implies that
\begin{align*}
	\frac{d}{dt} \int_{a'}^{b'} x_1(z,t) \, dz	&= \int_{a'}^{b'} \mathbf{d}_0 \mathcal{N}_1(x, c_0, \overline{c}_0)(z,t) \, dz \, \\
	&= \int_{a'}^{0}  - \partial_z  \mathcal{N}_1^-(x)(z,t)  \, dz + \int_{0}^{b'} - \partial_z \mathcal{N}_1^+(x)(z,t)  \, dz \\
	&=  \mathcal{N}_1^-(a',t) \underbrace{- \mathcal{N}_1^-(0^-,t) + \mathcal{N}_1^+ (0^+,t)}_{\stackrel{(\ref{Interface Paper - Interface Variable fI})}{=0}} - \mathcal{N}_1^+(b',t) \\
	&=   \mathcal{N}_1^-(a',t) - \mathcal{N}_1^+(b',t) \\
	&=  \mathcal{N}_1(a',t) - \mathcal{N}_1(b',t).
\end{align*}
\paragraph*{Conservation law of $\mathbf{x_2}$.} Lastly, let us consider the conservation law of the state variable $x_2$. Recall that at the interface, the associated flux variable $\mathcal{N}_2(x, c_0, \overline{c}_0) = c_0 \mathcal{N}_2^-(x) + \overline{c}_0 \mathcal{N}_2^+(x)$ is supposed to satisfy the balance equation \eqref{Interface Paper - Interface Variable eI}, where $e_I \in \mathbb{R}$ is defined as the difference of the flux variables $\mathcal{N}_2^+(x)$ and $\mathcal{N}_2^-(x)$ localized at the interface. In order to formulate this conservation law, we have to calculate the formal adjoint $\mathbf{d}_0^{\ast} \colon D(\mathbf{d}_0^{\ast}) \subset L^2([a,b], \mathbb{R}) \to L^2([a,b], \mathbb{R})$ of the operator $\mathbf{d}_0$ defined in \eqref{Interface Paper - Operator d0}. This will allow us to define a formally skew-symmetric differential operator, as we will see over the course of this chapter.
\vspace{0.5 cm}\\
Recall from Section \ref{Section Operators} that the domain of the adjoint operator is given by
\begin{align*}
	\left\{ y \in X \mid \langle \mathbf{d}_0 \cdot, y \rangle_{L^2} \colon D(\mathbf{d}_0) \to \mathbb{R}, \hspace{0.1 cm}  x \mapsto \langle \mathbf{d}_0 x, y \rangle_{L^2} \text{ is continuous} \right\}.
\end{align*}
However, we need to find the formally adjoint operator, as the system is considered to be open. Since, in contrast to the differential operators considered in Section \ref{Section Boundary Port-Hamiltonian Systems Associated with Skew-symmetric Operators}, this encompasses both the power flow at the boundary and the interface, we seek an operator $\mathbf{d}_0^{\ast}$ such that
\begin{align*}
	\langle \mathbf{d}_0 x, y \rangle_{L^2} = \langle x, \mathbf{d}_0^{\ast} y \rangle_{L^2} \hspace{0.2 cm} + \text{boundary and interface terms}, \hspace{0.3 cm} x \in D(\mathbf{d}_0), \, y \in D(\mathbf{d}_0^{\ast}).
\end{align*}
Therefore, the general domain of the formal adjoint $\mathbf{d}_0^{\ast}$ is given by
\begin{align}
	\label{Interface Paper - Operator d_0ast General Domain}
	D(\mathbf{d}_0^{\ast}) = \left\{ y \in X \mid \langle \mathbf{d}_0 \cdot, y \rangle_{L^2} \colon \mathbf{H}_0^1(a,b) \to \mathbb{R}, \hspace{0.1 cm}  x \mapsto \langle \mathbf{d}_0 x, y \rangle_{L^2} \text{ is continuous} \right\},
\end{align}
where 
\begin{align}
	\label{Interface Paper - Subspace H(a,b)}
	\mathbf{H}_0^1(a,b) := \left\{ f \in H_0^1([a,b],\mathbb{R}) \mid f(0) = 0 \right\} \subset D(\mathbf{d}_0).
\end{align}
Note that the set $\mathbf{H}_0^1(a,b)$ is dense in $L^2([a,b], \mathbb{R})$. Hence, the formal adjoint of $\mathbf{d}_0$, as we wish to define it, actually exists. 
\vspace{0.5 cm}\\
We first want to specify the domain $D(\mathbf{d}_0^{\ast})$ given in \eqref{Interface Paper - Operator d_0ast General Domain}. To that end, let  $x \in \mathbf{H}_0^1(a,b)$ and let $y \in L^2([a,b], \mathbb{R})$. Then we get
\begin{align*}
	\langle \mathbf{d}_0 x, y \rangle_{L^2} &= \int_{a}^{b} (\mathbf{d}_0 x(z)) y(z) \, dz \\
	&= \int_{a}^{b} - \left[ \frac{d}{dz} (c_0 x)(z) + \frac{d}{dz} (\overline{c}_0x )(z) \right] y(z) \, dz \\
	&= - \int_{a}^{0} \frac{d}{dz} (c_0 x)(z) y(z) \, dz  - \int_{0}^{b} \frac{d}{dz} (\overline{c}_0 x)(z) y (z) \, dz.
\end{align*}
If we additionally assume that $y_{|(a,0)} \in H^1((a,0),\mathbb{R})$ and $y_{|(0,b)} \in H^1((0,b), \mathbb{R})$, then integration by parts yields
\begin{align}
	\begin{split} 
	\langle \mathbf{d}_0 x, y \rangle_{L^2}	&= -x(0^-) y(0^-) + x(a) y(a) + x(0^+) y(0^+) - x(b) y(b) \\
	&\hspace{0.5 cm} + \int_{a}^{0} (x c_0 )(z) \frac{d}{dz} y(z) \, dz + \int_{0}^{b} (x \overline{c}_0)(z) \frac{d}{dz} y(z) \, dz \\
	&= -  \underbrace{\big[ (c_0 + \overline{c}_0)(z) x(z)  y(z) \big]_{a}^{b}}_{=0} + \underbrace{x(0)}_{=0} \left[ y(0^+) - y(0^-) \right] \\
	&\hspace{0.5 cm} + \int_{a}^{0} x(z) \frac{d}{dz} y(z) \, dz + \int_{0}^{b} x(z) \frac{d}{dz} y(z) \, dz. 
	\end{split}
\label{Interface Paper - Operator d_0ast Preparation}
\end{align}
	By virtue of Hölder's inequality, we have for all $x \in \mathbf{H}_0^1(a,b)$ and for all such $y \in L^2([a,b], \mathbb{R})$,
\begin{align*}
	\vert \langle \mathbf{d}_0 x, y \rangle_{L^2} \vert & \leq \int_{0}^{a} \left\vert x(z) \frac{d}{dz} y(z) \right\vert  \, dz + \int_{0}^{b} \left\vert x(z) \frac{d}{dz} y(z) \right\vert \, dz \\
	&\leq \|x\|_{L^2(a,0)} \left\| \frac{d}{dz}  y\right \|_{L^2(a,0)} + \|x \|_{L^2(0,b)} \left\| \frac{d}{dz}  y \right\|_{L^2(0,b)} \\
	&\leq \|x \|_{L^2(a,b)} \left( \left\| \frac{d}{dz} y \right\|_{L^2(a,0)} + \left\| \frac{d}{dz} y \right\|_{L^2(0,b)} \right) \\
	&\leq C_y \|x\|_{L^2(a,b)}, 
\end{align*}
whence
\begin{align*}
	\| \langle \mathbf{d}_0 \cdot , y \rangle_{L^2} \|_{\text{op} \colon \mathbf{H}_0^1(a,b) \to \mathbb{R}} = \sup_{\substack{ x \in \mathbf{H}_0^1(a,b) \\
	\|x\|_{L^2} =1}} \vert \langle \mathbf{d}_0 x, y \rangle_{L^2}  \vert  \leq C_y.
\end{align*}
Consequently, each element $y \in L^2([a,b], \mathbb{R})$ that is an $H^1$-function on the respective subdomains is an element of the domain $D(\mathbf{d}_0^{\ast})$. From the preceding calculations one can readily see that we cannot relax this requirement. As a result, the (maximal) domain of the operator $\mathbf{d}_0^{\ast}$ can be specified as follows:
\begin{align*}
			D(\mathbf{d}_0^{\ast}) =  \left\{ y \in L^2([a,b], \mathbb{R}) \mid y_{|(a,0)} \in H^1((a,0),\mathbb{R}), \hspace{0.1cm} y_{|(0,b)} \in H^1((0,b), \mathbb{R}) \hspace{0.1 cm} \right\}.	
\end{align*}
It is left to determine the elements $\mathbf{d}_0^{\ast}y \in L^2([a,b], \mathbb{R})$ for $y \in D(\mathbf{d}_0^{\ast})$. We have already calculated in \eqref{Interface Paper - Operator d_0ast Preparation} that for all $x \in D(\mathbf{d}_0)$ and for all $y \in D(\mathbf{d}_0^{\ast})$ we get
\begin{align*}
	\langle \mathbf{d}_0 x , y \rangle_{L^2} 	&= -  \big[ (c_0 + \overline{c}_0)(z) x(z)  y(z) \big]_{a}^{b} + x(0) \left[ y(0^+) - y(0^-) \right] \\
	&\hspace{0.5 cm} + \int_{a}^{0} x (z) \frac{d}{dz} y(z) \, dz + \int_{0}^{b} x(z) \frac{d}{dz} y(z) \, dz \\
	&\stackrel{!}{=} -  \big[ (c_0 + \overline{c}_0)(z) x(z)  y(z) \big]_{a}^{b} + x(0) \left[ y(0^+) - y(0^-) \right]  + \langle x, \mathbf{d}_0^{\ast} y \rangle_{L^2}.
\end{align*}
Thus, our objective is to rewrite the second line as a single $L^2$-inner product on the composed domain $[a,b]$ with respect to $x$ and $\mathbf{d}_0^{\ast}y$. We claim that $\mathbf{d}_0^{\ast}y \in L^2([a,b], \mathbb{R}) $ is given as
\begin{align*}
	\mathbf{d}_0^{\ast}y =  \left[ \frac{d}{dz} (c_0  y) + \frac{d}{dz} (\overline{c}_0  y)  \right] -  \left[ \frac{d}{dz} c_0 + \frac{d}{dz}\overline{c}_0  \right] y , \hspace{0.5 cm} y \in D(\mathbf{d}_0^{\ast}). 
\end{align*}
Although at first glance, the co-domain of this operator seems to be $\mathcal{D}'(a, b)$, we will see that the distributional terms of this expression, that is, the Dirac masses arising from taking the derivatives of the characteristic functions $c_0$ and $\overline{c}_0$, cancel each other out, so that the remaining terms lie again in $L^2([a, b], \mathbb{R}) \subset \mathcal{D}'(a,b)$.
\vspace{0.5 cm}\\
Initially, let  $\varphi \in \mathcal{D}(a,b)$, which is a dense subset of  $D(\mathbf{d}_0) = H^1([a,b], \mathbb{R})$. Then for all $y \in D(\mathbf{d}_0^{\ast})$ it holds that (omitting the dependence on the variable $z$)
\begin{align*}
	&\hspace{0.5 cm}  \int_{a}^{b} \left(  \left[ \frac{d}{dz} (c_0y) + \frac{d}{dz} (\overline{c}_0 y) \right] - \left[ \frac{d}{dz} c_0 + \frac{d}{dz} \overline{c}_0 \right] y \right) \varphi \, dz \\
	&= \int_{a}^{b} \frac{d}{dz} (c_0 y) \varphi \, dz + \int_{a}^{b} \frac{d}{dz} (\overline{c}_0 y) \varphi \, dz - \int_{a}^{b} \frac{d}{dz} c_0 \varphi y \, dz - \int_{a}^{b} \frac{d}{dz} \overline{c}_0 \varphi  y \, dz \\
	(\text{part. int.}) \hspace{0.3 cm}	&= \big[ (c_0 + \overline{c}_0) \varphi y \big]_{a}^{b} - \int_{a}^{0} c_0 y \frac{d}{dz} \varphi \, dz - \int_{0}^{b} \overline{c}_0 y \frac{d}{dz}  \varphi \, dz \\ 
	&\hspace{0.45 cm} - \big[ (c_0 + \overline{c}_0) \varphi y \big]_{a}^{b} + \int_{a}^{0} \frac{d}{dz} (\varphi y) \, dz + \int_{0}^{b} \frac{d}{dz} (\varphi y) \, dz \\
	&= - \int_{a}^{0} y \frac{d}{dz} \varphi \, dz - \int_{0}^{b} y \frac{d}{dz} \varphi \, dz \\
	&\hspace{0.5 cm} + \int_{a}^{0} \varphi \frac{d}{dz} y + y \frac{d}{dz} \varphi \, dz + \int_{0}^{b} \varphi \frac{d}{dz} y + y \frac{d}{dz} \varphi \, dz \\  
	&=  \int_{a}^{0} \varphi \frac{d}{dz} y \, dz + \int_{0}^{b} \varphi \frac{d}{dz}  y \, dz. 
\end{align*}

Next, we want to show that for all $x \in D(\mathbf{d}_0)$, $y \in D(\mathbf{d}_0^{\ast})$ the inner product $\langle x, \mathbf{d}_0^{\ast}y \rangle_{L^2}$ exists in a classical sense by proving that the operation $x \mapsto \langle x , \mathbf{d}_0^{\ast} y \rangle_{L^2}$ is closed. Let $y \in D(\mathbf{d}_0^{\ast})$, pick any $x \in D(\mathbf{d}_0)$, and let $(\varphi_n)_{n \in \mathbb{N}} \in \mathcal{D}(a,b)^{\mathbb{N}}$ be a sequence with $\lim\limits_{n \to \infty} \varphi_n = x$ in $L^2$. Since the embedding $L^2([a,b], \mathbb{R}) \hookrightarrow L^1([a,b], \mathbb{R})$ is continuous, this convergence holds in $L^1$ as well. Following \cite[Chapter 6]{Klenke}, we may apply the dominated convergence theorem. This yields
\begin{align*}
	&\hspace{0.5 cm}  \lim_{n \to \infty}  \int_{a}^{b} \left(  \left[ \frac{d}{dz} (c_0y) + \frac{d}{dz} (\overline{c}_0 y) \right] - \left[ \frac{d}{dz} c_0 + \frac{d}{dz} \overline{c}_0 \right] y \right) \varphi_n \, dz \\ 
	&= \lim_{n \to \infty} \int_{a}^{0} \varphi_n \frac{d}{dz} y \, dz + \lim_{n \to \infty} \int_{0}^{b} \varphi_n \frac{d}{dz} y \, dz \\
	&= \int_{a}^{0} x \frac{d}{dz} y \, dz + \int_{0}^{b} x \frac{d}{dz} y \, dz.
\end{align*}
Thus, for all $x \in D(\mathbf{d}_0)$, $y \in D(\mathbf{d}_0^{\ast})$ the inner product $\langle x , \mathbf{d}_0^{\ast}y \rangle_{L^2}$ exists with
\begin{align*}
	\langle x , \mathbf{d}_0^{\ast}y \rangle_{L^2} = \int_{a}^{0} x(z) \frac{d}{dz} y(z) \, dz + \int_{0}^{b} x(z) \frac{d}{dz} y(z) \, dz.
\end{align*} 
Altogether, the formal adjoint $\mathbf{d}_0^{\ast} \colon D(\mathbf{d}_0^{\ast}) \subset L^2([a,b], \mathbb{R}) \to L^2([a,b], \mathbb{R})$ is given by
\begin{align}
	\begin{split}
		D(\mathbf{d}_0^{\ast}) &= \left\{ y \in L^2([a,b], \mathbb{R}) \mid y_{|(a,0)} \in H^1((a,0),\mathbb{R}), \hspace{0.1cm} y_{|(0,b)} \in H^1((0,b), \mathbb{R}) \hspace{0.1 cm} \right\}, \\
	\mathbf{d}_0^{\ast}y &=  \left[ \frac{d}{dz} (c_0  y) + \frac{d}{dz} (\overline{c}_0  y)  \right] -  \left[ \frac{d}{dz} c_0 + \frac{d}{dz}\overline{c}_0  \right] y  \\
	&\hspace{0.12 cm}= \left( - \mathbf{d}_0  - \left[ \frac{d}{dz} c_0 + \frac{d}{dz} \overline{c}_0  \right] \right)y, \hspace{0.5 cm} y \in D(\mathbf{d}_0^{\ast}).
	\end{split}
	\label{Interface Paper - Operator d_0ast} 
\end{align}
In particular, for a function $y = c_0 y^- + \overline{c}_0 y^+ \in D(\mathbf{d}_0^{\ast})$ we have
\begin{align*}
	\mathbf{d}_0^{\ast} y =  c_0 \frac{d}{dz} y^- + \overline{c}_0 \frac{d}{dz} y^+. 
\end{align*}
Moreover, we obtain the following relation between the operators $\mathbf{d}_0$ and $\mathbf{d}_0^{\ast}$:
\begin{align}
	\label{Interface Paper - Relation d_0 and d_0ast}
	\langle \mathbf{d}_0 x , y \rangle_{L^2} =  - \big[ x(z) y(z) \big]_{a}^{b} + x(0) \left[ y(0^+) - y(0^-) \right] + \langle x , \mathbf{d}_0^{\ast} y \rangle_{L^2}, \hspace{0.5 cm} x \in D(\mathbf{d}_0), \,  y \in D(\mathbf{d}_0^{\ast}).
\end{align} 
We are finally able to state the conservation law of the state variable $x_2$ with respect to the operator $\mathbf{d}_0^{\ast}$. Assume that for all $t > 0$, the flux variable $\mathcal{N}_2(x,c_0, \overline{c}_0)(\cdot,t)$ on the composed domain is an element of $D(\mathbf{d}_0^{\ast})$. From the preceding discussion and due to the individual conservation laws \eqref{Interface Paper - Two Systems of Conservation Laws on Respective Domain}, we deduce that for all $\varphi \in \mathcal{D}(a,b)$ we have
\begin{align*}
	 &\hspace{0.5 cm} \int_{a}^{b} \left[ - \mathbf{d}_0^{\ast} \mathcal{N}_2(x,c_0, \overline{c}_0)(z,t) \right] \varphi(z) \, dz \\
	&= \int_{a}^{0} - \partial_z \mathcal{N}_2^- (z,t) \varphi(z) \, dz + \int_{0}^{b} - \partial_z \mathcal{N}_2^+(z,t)  \varphi(z) \, dz \\
	&= \int_{a}^{0} \partial_t x_2^-(z,t) \varphi(z) \, dz + \int_{0}^{b} \partial_t x_2^+(z,t)  \varphi(z) \, dz \\
	&= \int_{a}^{b} \partial_t x_2(z,t) \varphi(z) \, dz. 
\end{align*} 
In conclusion, the conservation law of the state variable $x_2$ is  given by
\begin{align}
	\partial_t x_2 = - \mathbf{d}_0^{\ast} \mathcal{N}_2(x,c_0, \overline{c}_0),
	\label{Interface Paper - Conservation Law of x2 - Nonmoving Interface} 
\end{align}
and it holds in a classical sense.
\begin{remark}
	\label{Interface Paper - Remark Balance Equation}
We need to point out that the balance equation \eqref{Interface Paper - Interface Variable eI} the flux variable $\mathcal{N}_2$ is supposed to satisfy is incorporated neither into the domain of the operator $\mathbf{d}_0^{\ast}$ nor into the conservation law \eqref{Interface Paper - Conservation Law of x2 - Nonmoving Interface} itself. So far, we have only allowed a jump discontinuity of $\mathcal{N}_2$ at the interface. One possible way is to rewrite the conservation law as
\begin{align*}
	\partial_t x_2 = \mathbf{d}_0 \mathcal{N}_2(x,c_0,\overline{c}_0) - e_I \delta_0,
\end{align*}
with $\delta_0 \in \mathcal{D}'(a,b)$ the Dirac mass at $z = 0$, and $e_I \colon [0, \infty) \to \mathbb{R}$. By doing so, $e_I$ may be regarded as an external source term, and the equation holds (in a distributional sense) if and only if $\mathcal{N}_2$ satisfies the balance equation \eqref{Interface Paper - Interface Variable eI}. However, since this will only complicate the modeling in the further course and since we want to define a formally skew-symmetric matrix differential operator later on, we will stick with the conservation law \eqref{Interface Paper - Conservation Law of x2 - Nonmoving Interface} with respect to the formally adjoint operator $\mathbf{d}_0^{\ast}$ of $\mathbf{d}_0$, and take the interface port variable $e_I$ to be implicitly defined by means of the jump discontinuity of $\mathcal{N}_2(x,c_0, \overline{c}_0)$ at the interface. As we will see, this does not affect the dynamics of the resulting system over the one proposed in \cite{Diagne}.
\end{remark}
Let us now define classical solutions of the conservation law of $x_2$ with respect to $\mathcal{N}_2(x,c_0, \overline{c}_0)$.

\begin{definition}
	 Let $x = (x_1,x_2) \colon [0, \infty) \to X$. Let $Y = L^2([a,b], \mathbb{R})$. We call $x_2 \colon [0,\infty) \to Y$ a classical solution of the conservation law
	\begin{align}
		\begin{split}
			\partial_t x_2(t) &= - \mathbf{d}_0^{\ast} \mathcal{N}_2 (x, c_0, \overline{c}_0)(\cdot,t),  \hspace{0.5 cm} t > 0, \\
			x_2(0) &= x_{0,2} \in Y, 
		\end{split}
		\label{Interface Paper - Definition Conservation Law of x2 - Nonmoving Interface}
	\end{align}
	where $\mathbf{d}_0^{\ast}$ is the operator given in \eqref{Interface Paper - Operator d_0ast}, and with the flux variable
	\begin{align*}
		\mathcal{N}_2(x, c_0, \overline{c}_0)  = c_0 \mathcal{N}_2^-(x) + \overline{c}_0 \mathcal{N}_2^+(x)
	\end{align*}
	of the systems of conservation laws \eqref{Interface Paper - Two Systems of Conservation Laws on Respective Domain} coupled at $z=0$, if:
	\begin{enumerate}
		\item[(i)] $x_2 \in \mathcal{C}([0,\infty), Y)$. 
		\item[(ii)] $x_{2_{ |(0,\infty) }}  \in \mathcal{C}^1((0, \infty), Y)$.
		\item[(iii)] $\mathcal{N}_2(x, c_0, \overline{c}_0)(\cdot,t) \in D(\mathbf{d}_0^{\ast})$ for all $t > 0$.
		\item[(iv)] $x_2$ satisfies \eqref{Interface Paper - Definition Conservation Law of x2 - Nonmoving Interface}.
	\end{enumerate}
\end{definition} 
Analogously to the operator $\mathbf{d}_0$, its formal adjoint $\mathbf{d}_0^{\ast}$ acts as the differential operator $ \frac{d}{dz}$ on the respective subdomains. Assuming that the flux variable $\mathcal{N}_2$ satisfies the balance equation \eqref{Interface Paper - Interface Variable eI}, on an arbitrary interval $[a',b'] \subset [a,b]$ containing the interface we get the following for all $t>0$:
\begin{align*}
	\frac{d}{dt} \int_{a'}^{b'} x_2(z,t) \, dz &= \int_{a'}^{b'} \partial_t x_2(z,t) \, dz \\
	&= \int_{a'}^{b'} -\mathbf{d}_0^{\ast} \mathcal{N}_2(x,c_0, \overline{c}_0) (z,t) \, dz \\
	&= \int_{a'}^{0} - \partial_z \mathcal{N}_2^-(x) (z,t) \, dz + \int_{0}^{b'} - \partial_z \mathcal{N}_2^+(x)(z,t) \, dz \\
	&= \mathcal{N}_2^-(a',t)  - \mathcal{N}_2^+(b',t) - \mathcal{N}_2^-(0^-,t) + \mathcal{N}_2^+(0^+,t) \\
	&= \mathcal{N}_2(a',t) - \mathcal{N}_2(b',t) - e_I(t).
\end{align*}
Finally, we can formulate the interconnection of the conservation laws  \eqref{Interface Paper - Two Systems of Conservation Laws on Respective Domain} as a port-Hamiltonian system defined on the extended state space $\tilde{X}$, which we are going to do next.

\subsection{Port-Hamiltonian Formulation and Dirac Structure}
\label{Subsection Port-Hamiltonian Formulation and Dirac Structure}
In the previous section, we have formulated the dynamics of the state variables $x_1$ and $x_2$ on the composed domain $[a,b]$ through the conservation laws   \eqref{Interface Paper - Definition Conservation Law of x1 - Nonmoving Interface} and \eqref{Interface Paper - Definition Conservation Law of x2 - Nonmoving Interface} using the operator $\mathbf{d}_0$ and its formal adjoint $\mathbf{d}_0^{\ast}$ given by \eqref{Interface Paper - Operator d0} and \eqref{Interface Paper - Operator d_0ast}, respectively. Furthermore, in \eqref{Interface Paper - Conservation Laws of Color Functions with Fixed Interface Position} we stated the conservation laws for the characteristic functions $c_0$ and $\overline{c}_0$ defined in \eqref{Interface Paper - Characteristic Functions c0 and c0overline} of the two subdomains separated by the interface, which are constituting the extended state variable as well. Now, we want to formulate a Hamiltonian system off of these equation on the extended state space $\tilde{X}$, and we want to specify the underlying Dirac structure of the fixed interface scenario.
\begin{remark}
	Recall that both the operator $\mathbf{d}_0$ and its formal adjoint $\mathbf{d}_0^{\ast}$ depend on the color functions $c_0$ and $\overline{c}_0$. Strictly speaking, the operators $\mathbf{d}_0$ and $\mathbf{d}_0^{\ast}$ are therefore modulated by the state $\tilde{x} = (x, c_0, \overline{c}_0) \in \tilde{X}$. However, as the color functions remain constant over time, we will neglect this dependence of the aforementioned operators. 
\end{remark}
With the aid of \eqref{Interface Paper - Relation Flux Variable and Variational Derivative - Composed Domain}, we may define the agumented system 
\begin{align*}
	\frac{\partial}{\partial t} \begin{bmatrix}
		x_1 \\
		x_2 \\
		c_0 \\
		\overline{c}_0
	\end{bmatrix} = \begin{bmatrix}
	0 & \mathbf{d}_0 & 0 & 0 \\
	- \mathbf{d}_0^{\ast} & 0 & 0 & 0 \\
	0 & 0 & 0 & 0 \\
	0 & 0 & 0 & 0 
\end{bmatrix} \begin{bmatrix}
c_0 \delta_{x_1} H^-(x) + \overline{c}_0 \delta_{x_1}  H^+(x) \\
c_0 \delta_{x_2} H^-(x) + \overline{c}_0 \delta_{x_2}  H^+(x) \\
\mathcal{H}^-(x) \\
\mathcal{H}^+(x)
\end{bmatrix},
\end{align*}
or short, 
\begin{align}
	\label{Interface Paper - Augmented System wrt Ja}
	\frac{\partial}{\partial t} \tilde{x} =  \tilde{\mathcal{J}}_0 \delta_{\tilde{x}} H(\tilde{x}),
\end{align}
with the Hamiltonian's variational derivative $\delta_{\tilde{x}} H$ given by \eqref{Interface Paper - Extended Variational Derivatives}, and with the operator $\tilde{\mathcal{J}}_0 \colon D(\tilde{\mathcal{J}}_0) \subset \tilde{X} \to \tilde{X}$ given by
\begin{align}
	\begin{split}
	D(\tilde{\mathcal{J}}_0)	&=  D(\mathbf{d}_0^{\ast}) \times D(\mathbf{d}_0) \times L^2([a,b], \mathbb{R}^2), \\
	\tilde{\mathcal{J}}_0 \tilde{x} &= \begin{bmatrix}
		0 & \mathbf{d}_0 & 0 & 0 \\
		- \mathbf{d}_0^{\ast} & 0 & 0 & 0 \\
		0 & 0 & 0 & 0 \\
		0 & 0 & 0 & 0 
	\end{bmatrix} \tilde{x}, \hspace{0.5 cm} \tilde{x}  \in \tilde{X}.
\end{split}
\label{Interface Paper - Operator Ja}
\end{align}
By virtue  of the relation \eqref{Interface Paper - Relation d_0 and d_0ast} between the operators $\mathbf{d}_0$ and $\mathbf{d}_0^{\ast}$, we find for all $\tilde{e}^1, \tilde{e}^2 \in D(\tilde{\mathcal{J}}_0)$ that
\begin{align}
	\begin{split}
	 &\hspace{0.5 cm} \langle \tilde{\mathcal{J}}_0 \tilde{e}^1, \tilde{e}^2 \rangle_{L^2} + \langle \tilde{e}^1, \tilde{\mathcal{J}}_0 \tilde{e^2} \rangle_{L^2}   \\
	 &= \int_{a}^{b} (e^2)^{\top}(z) \begin{bmatrix}
	 	0 & \mathbf{d}_0 \\
	 	-\mathbf{d}_0^{\ast} & 0 
	 \end{bmatrix} e^1(z) + \left( \begin{bmatrix}
	 0 & \mathbf{d}_0 \\
	 -\mathbf{d}_0^{\ast} & 0 
 \end{bmatrix} e^2(z)  \right)^{\top} e^1(z) \, dz \\
&= \int_{a}^{b} e_1^2(z) \mathbf{d}_0 e_2^1(z) - e_2^2(z) \mathbf{d}_0^{\ast} e_1^1(z) + e_1^1(z) \mathbf{d}_0 e_2^2(z) - e_2^1(z) \mathbf{d}_0^{\ast} e_1^2(z) \, dz \\ 
	 &= - \big[ e_1^1(z)e_2^2(z) \big]_{a}^{b} - \big[  e_2^1(z) e_1^2(z) \big]_{a}^{b} \\
	 &\hspace{0.5 cm}+ e_2^1(0) \left[ e_1^2(0^+) - e_1^2(0^-) \right] + e_2^2(0) \left[ e_1^1(0^+) - e_1^1(0^-) \right] \\
	 &= \big[(e^1)^{\top}(z) P_1 e^2(z) \big]_{a}^{b} + e_2^1(0) \left[ e_1^2(0^+) - e_1^2(0^-) \right] + e_2^2(0) \left[ e_1^1(0^+) - e_1^1(0^-) \right],
	\end{split}
\label{Interface Paper - Skew-Symmetry of Ja}
\end{align}
with $P_1$ defined in \eqref{Interface Paper - Matrix P1}. Hence, the operator $\tilde{\mathcal{J}}_0$ is formally skew-symmetric. We will not check whether $\tilde{\mathcal{J}}_0$ satifies the Jacobi identity (see again \cite[Chapter 7]{Olver}), so we refer to the system \eqref{Interface Paper - Augmented System wrt Ja} as a \emph{generalized Hamiltonian system}.
\vspace{0.5 cm}\\
In order to take the energy exchange at the boundary and the power inflow (or outflow) at the interface into account, the generalized Hamiltonian system
\eqref{Interface Paper - Augmented System wrt Ja} will now be extended to a boundary port-Hamiltonian system. For this purpose, we need to define a Stokes-Dirac structure  similar to the one given in \eqref{Stokes-Dirac structure - Linear PH Systems}.
\vspace{0.5 cm}\\
To find the underlying Dirac structure of the system \eqref{Interface Paper - Augmented System wrt Ja}, we need to define a suitable power pairing $ \langle \cdot \mid \cdot \rangle \colon \mathcal{E} \times \mathcal{F} \to \mathbb{R}$ constituting the plus pairing $\ll \cdot , \cdot \gg \colon \mathcal{B} \times \mathcal{B} \to \mathbb{R}$ on the bond space $\mathcal{B} = \mathcal{F} \times \mathcal{E}$. We will basically proceed as in Section \ref{Section Boundary Port-Hamiltonian Systems Associated with Skew-symmetric Operators}. 
\vspace{0.5 cm}\\
To begin with, consider the operator $\mathcal{J}_0 \colon D(\mathcal{J}_0) \subset X \to X$ given by
\begin{align*}
		D(\mathcal{J}_0) &= \left\{ x \in X \mid x_1 \in D(\mathbf{d}_0^{\ast}), \hspace{0.1 cm} x_2 \in D(\mathbf{d}_0)\right\}, \\
		\mathcal{J}_0 x &= \begin{bmatrix}
			0 & \mathbf{d}_0  \\
			- \mathbf{d}_0^{\ast} & 0 
		\end{bmatrix} \begin{bmatrix}
			x_1 \\
			x_2
		\end{bmatrix}, \hspace{0.5 cm} x \in D(\mathcal{J}_0).
\end{align*}
Note that $\mathcal{J}_0$ is the operator we obtain if we reduce the operator $\tilde{\mathcal{J}}_0$ from \eqref{Interface Paper - Operator Ja} to $X$, i.e., if we ignore the color functions we included in the definition of the extended state space $\tilde{X}$. This operator will be useful for finding the power pairing. Note that from \eqref{Interface Paper - Skew-Symmetry of Ja} it immediately follows that $\mathcal{J}_0$ is formally skew-symmetric as well, and that on any interval $(c,d) \subset[a,b]$ not containing the interface position $z = 0$, the operator $\mathcal{J}_0$ simply acts as the matrix differential operator $P_1 \frac{d}{dz}$. We will come across this operator in Section \ref{Section A Simplified System} again. 
\vspace{0.5 cm}\\
For systems associated with a formally skew-symmetric differential operator, we have already defined the boundary port variables in Definition \ref{Definition Boundary Port Variables in The Linear Case}. Since the interface position resides in the interior of the spatial domain $[a,b]$, the trace operator defined in \eqref{Trace Operator - Chapter 3} can be extended to
\begin{align*}
	\trace_0 \colon D(\mathcal{J}_0) \to \mathbb{R}^4, \hspace{0.3 cm}  e \mapsto \begin{bmatrix}
		e(b) \\
		e(a)
	\end{bmatrix}.
\end{align*}
Following \eqref{Boundary Flow and Effort Linear Case}, the boundary flow variable $f_{\partial} = f_{\partial, e} \in \mathbb{R}^2$ and the boundary effort variable $e_{\partial} = e_{\partial, e} \in \mathbb{R}^2$ can, for all co-energy variables $ e  \in D(\mathcal{J}_0)$, be defined as
\begin{align*}
	\begin{bmatrix}
		f_{\partial} \\
		e_{\partial}
	\end{bmatrix} =  R_{\extern} \trace_0(e) = \begin{bmatrix}  \frac{1}{\sqrt{2}} (P_1e(b) - P_1 e(a)) \\
		\frac{1}{\sqrt{2}} ( e(b) + e(a))
	\end{bmatrix}, 
\end{align*} 
with $R_{\extern}$ defined in \eqref{Interface Paper - Matrix Rext}. Furthermore, we have introduced the interface relations in \eqref{Interface Paper - Interface Variable fI} and \eqref{Interface Paper - Interface Variable eI} with respect to the flux variables. In general, the interface flow $f_I = f_{I,e} \in \mathbb{R}$ and the interface effort $e_I = e_{I,e} \in \mathbb{R}$ are for all effort variables $ e \in D(\mathcal{J}_0)$ given as follows:
\begin{align*}
	f_I &= e_2 (0^+) = e_2(0^-), \\ 
	-e_I &= e_1(0^+) - e_1 (0^-). 
\end{align*}
Now, let us choose the power pairing in order to define the Dirac structure associated with the system \eqref{Interface Paper - Augmented System wrt Ja} that is augmented with the boundary and interface port variables. Define the space of flows $\mathcal{F}$ and the space of efforts $\mathcal{E}$  as
\begin{align}
	\mathcal{F} = \mathcal{E} = L^2([a,b], \mathbb{R}^4) \times \mathbb{R} \times \mathbb{R}^2,
	\label{Interface Paper - Flow and Effort Space}
\end{align}
endowed with the canonical inner product $\langle \cdot, \cdot  \rangle_{\mathcal{F}}$. For the right choice of the power pairing $\langle \cdot \mid \cdot \rangle \colon \mathcal{E} \times \mathcal{F} \to \mathbb{R}$ we make use of equation \eqref{Interface Paper - Skew-Symmetry of Ja}. Note that for all $ e^1, e^2 \in D(\mathcal{J}_0)$ we may write the last line of \eqref{Interface Paper - Skew-Symmetry of Ja} with respect to the boundary and interface port variables as follows: 
\begin{align*}
	&\hspace{0.5 cm} \big[ (e^1)^{\top}(z) P_1 e^2(z) \big]_{a}^{b} + e_2^1(0) \left[ e_1^2(0^+) - e_1^2(0^-) \right]  + e_2^2(0) \left[ e_1^1(0^+) - e_1^1(0^-) \right] \\
	&= \langle e_{\partial, e^2} , f_{\partial, e^1} \rangle_2 + \langle e_{\partial, e^1} , f_{\partial, e^2} \rangle_2 - f_{I,e^1} e_{I, e^2} - f_{I, e^2} e_{I, e^1}. 
\end{align*}
The representaion of the evaluation at the boundary with respect to the boundary flows and efforts can be looked up in Subsection \ref{Subsection Boundary Port Variables} (see page \pageref{Theorem Stokes-Like Theorem - Equation Representation 2}). Consequently, we have for all $\tilde{e}^1, \tilde{e}^2 \in D(\tilde{\mathcal{J}}_0)$, 
\begin{align}
	\langle \tilde{\mathcal{J}}_0 \tilde{e}^1, \tilde{e}^2 \rangle_{L^2} + \langle \tilde{e}^1, \tilde{\mathcal{J}}_0 \tilde{e}^2 \rangle_{L^2}  = \langle e_{\partial, e^2} , f_{\partial, e^1} \rangle_2 + \langle e_{\partial, e^1} , f_{\partial, e^2} \rangle_2 - f_{I,e^1} e_{I, e^2} - f_{I, e^2} e_{I, e^1}.
	\label{Interface Paper - Skew-Symmetry of Ja - Representation 2}
\end{align}
Thus, we shall define
\begin{align}
	\left\langle \begin{pmatrix}
		\tilde{e} \\
		e_I \\
		e_{\partial} 
	\end{pmatrix} \Bigg| \begin{pmatrix}
		\tilde{f} \\
		f_I \\
		f_{\partial}
	\end{pmatrix} \right\rangle := \langle \tilde{e}, \tilde{f} \rangle_{L^2} +e_I f_I - \langle e_{\partial}, f_{\partial} \rangle_2, \hspace{0.5 cm} \left( \begin{pmatrix}
		\tilde{e} \\
		e_I \\
		e_{\partial} 
	\end{pmatrix} , \begin{pmatrix}
		\tilde{f} \\
		f_I \\
		f_{\partial}
	\end{pmatrix} \right) \in \mathcal{E} \times \mathcal{F}, 
	\label{Interface Paper - Power Pairing}
\end{align}
and the plus pairing $\ll \cdot , \cdot \gg \colon \mathcal{B} \times \mathcal{B} \to \mathbb{R}$ on the bond space $\mathcal{B} = \mathcal{F} \times \mathcal{E}$ induced by the power pairing \eqref{Interface Paper - Power Pairing} is accordingly defined as in \eqref{Plus Pairing - Distributed-Parameter}. Finally, we can define the Dirac structure associated with the formally skew-symmetric operator $\tilde{\mathcal{J}}_0$ given by \eqref{Interface Paper - Operator Ja}. The idea is based on the proofs of Proposition 4 in \cite{Diagne} and of Theorem 3.6 in \cite{LeGorrec}.
\begin{proposition}
	\label{Interface Paper - Proposition Dirac Structure DI}
 Let $\mathcal{B} = \mathcal{F} \times \mathcal{E}$ be the bond space with respect to the flow space $\mathcal{F}$ and the effort space $\mathcal{E}$ defined in \eqref{Interface Paper - Flow and Effort Space}. The subspace $\mathcal{D}_{\tilde{\mathcal{J}}_0} \subset \mathcal{B}$ defined by 
	\begin{align}
		\mathcal{D}_{\tilde{\mathcal{J}}_0} &:= 
			\left\{ \left( \begin{pmatrix}
				\tilde{f} \\
				f_I \\
				f_{\partial}
			\end{pmatrix} , \begin{pmatrix}
				\tilde{e} \\
				e_I \\
				e_{\partial}
			\end{pmatrix} \right) \in \mathcal{F} \times \mathcal{E} \hspace{0.1 cm} \Bigg| \hspace{0.1 cm}  \tilde{e} \in D(\tilde{\mathcal{J}}_0), \hspace{0.1 cm}
			\tilde{f}  = \tilde{\mathcal{J}}_0 \tilde{e}, \hspace{0.1 cm} f_I = e_2(0), \right. \label{Interface Paper - Dirac Structure DI} \\ 
			&\hspace{5.7 cm} \left. e_I = e_1(0^-) - e_1(0^+), \hspace{0.1 cm} \begin{bmatrix}
				f_{\partial } \\
				e_{\partial}
			\end{bmatrix} = R_{\extern} \trace_0 (e_1,e_2) \hspace{0.1 cm} \right\}, \nonumber
	\end{align}
with the flow and effort variables 
\begin{align*}
	\tilde{f} = \begin{bmatrix}
		f_1 \\
		f_2 \\
		f_{c} \\
		f_{\overline{c}}
	\end{bmatrix} \hspace{0.3 cm} \text{and} \hspace{0.3 cm} \tilde{e} = \begin{bmatrix}
	e_1 \\
	e_2 \\
	e_{c} \\
	e_{\overline{c}}
\end{bmatrix},
\end{align*}
 respectively, associated with a system of two conservation laws defined on the
 spatial domain $[a,b] \subset \mathbb{R}$ with an interface at the point $z = 0$ which imposes the continuity of the effort variable $e_2$ and allows for the discontinuity of the effort variable $e_1$, and associated with the extended state \eqref{Interface Paper - Extended State Variable xtilde}, the differential operator $\tilde{\mathcal{J}}_0$ defined in \eqref{Interface Paper - Operator Ja}, the operator $\mathbf{d}_0$ and its formal adjoint $\mathbf{d}_0^{\ast}$ defined in \eqref{Interface Paper - Operator d0} and \eqref{Interface Paper - Operator d_0ast}, respectively, is a Dirac structure with respect to the plus pairing $\ll \cdot, \cdot \gg \colon \mathcal{B} \times \mathcal{B} \to \mathbb{R}$ induced by the power pairing $\langle \cdot \mid \cdot \rangle \colon \mathcal{E} \times \mathcal{F} \to \mathbb{R}$ given by \eqref{Interface Paper - Power Pairing}.
\end{proposition}

\begin{proof}
	In the following, we denote by
	\begin{align*}
		F = \begin{pmatrix}
			\tilde{f} \\
			f_I \\
			f_{\partial}
		\end{pmatrix} \in \mathcal{F} \hspace{0.3 cm} \text{and} \hspace{0.3 cm} E = \begin{pmatrix}
			\tilde{e} \\
			e_I \\
			e_{\partial}
		\end{pmatrix} \in \mathcal{E}
	\end{align*}
	the vector of flows and the vector of efforts, respectively, for all $(F,E) \in \mathcal{D}_{\tilde{\mathcal{J}}_0}$.
	To prove the claim, we have to verify that the space $\mathcal{D}_{\tilde{\mathcal{J}}_0}$ coincides with its orthogonal complement $\mathcal{D}_{\tilde{\mathcal{J}}_0}^{\perp \!\!\! \perp}$ with respect to the plus pairing $\ll \cdot, \cdot \gg$, see Definition \ref{Definition Dirac Structure - Infinite-dimensional Case}.
	\vspace{0.5 cm}\\
We first show that $\mathcal{D}_{\tilde{\mathcal{J}}_0} \subset \mathcal{D}_{\tilde{\mathcal{J}}_0}^{\perp \!\!\! \perp}$. In order to do so, we have to show that for every $(F^1,E^1)$, $(F^2,E^2) \in \mathcal{D}_{\tilde{\mathcal{J}}_0}$ we have
	\begin{align*}
	\ll (F^1,E^1), (F^2,E^2) \gg = \langle E^1 \mid F^2 \rangle + \langle E^2 \mid F^1 \rangle \stackrel{!}{=} 0
	\end{align*}
	with respect to the pairing $\langle \cdot \mid \cdot \rangle $ defined in \eqref{Interface Paper - Power Pairing}. Using the constitutive relations of $\mathcal{D}_{\tilde{\mathcal{J}}_0}$ and \eqref{Interface Paper - Skew-Symmetry of Ja - Representation 2}, we compute 
	\begin{align*}
		&\hspace{0.5 cm} \ll (F^1,E^1), (F^2,E^2) \gg  \\
		&= \langle \tilde{e}^2, \tilde{f}^1 \rangle_{L^2}  + \langle \tilde{e}^1, \tilde{f}^2 \rangle_{L^2} + e_I^1 f_I^2 + e_I^2 f_I^1 - \langle e_{\partial}^1 , f_{\partial}^2 \rangle_2 - \langle e_{\partial}^2 , f_{\partial}^1 \rangle_2 \\
		&= \langle \tilde{e}^2, \tilde{\mathcal{J}}_0 \tilde{e}^1 \rangle_{L^2}  + \langle \tilde{e}^1, \tilde{\mathcal{J}}_0 \tilde{e}^2 \rangle_{L^2} + e_I^1 f_I^2 + e_I^2 f_I^1 - \langle e_{\partial}^1 , f_{\partial}^2 \rangle_2 - \langle e_{\partial}^2 , f_{\partial}^1 \rangle_2 \\
		&= \langle e_{\partial}^2, f_{\partial}^1 \rangle_2 + \langle e_{\partial}^1 , f_{\partial}^2 \rangle_2 - f_I^1 e_I^2 - f_I^2 e_I^1  + e_I^1 f_I^2 + e_I^2 f_I^1 - \langle e_{\partial}^1 , f_{\partial}^2 \rangle_2 - \langle e_{\partial}^2 , f_{\partial}^1 \rangle_2 \\ 
		&= 0.
	\end{align*}
 As $(F^1,E^1), (F^2, E^2) \in \mathcal{D}_{\tilde{\mathcal{J}}_0}$ can be chosen arbitrarily, we conlcude that $\mathcal{D}_{\tilde{\mathcal{J}}_0} \subset \mathcal{D}_{\tilde{\mathcal{J}}_0}^{\perp \!\!\! \perp}$.
	\vspace{0.5 cm}\\
	Next, we want to show the inclusion $\mathcal{D}_{\tilde{\mathcal{J}}_0}^{\perp \!\!\! \perp} \subset \mathcal{D}_{\tilde{\mathcal{J}}_0}$. To this end, let 
	\begin{align*}
	(F^2, E^2) = \left( \begin{pmatrix}
	\tilde{\phi} \\
	\phi_I \\
	\phi_{\partial}
	\end{pmatrix}, \begin{pmatrix}
	\tilde{\epsilon} \\
	\epsilon_I \\
	\epsilon_{\partial}
\end{pmatrix} \right) \in \mathcal{B} 
\end{align*}
such that for all $(F^1,E^1) \in \mathcal{D}_{\tilde{\mathcal{J}}_0}$ it holds that
	\begin{align*}
		\ll (F^1, E^1), (F^2, E^2) \gg = 0. 
	\end{align*}
	We need to verify that $(F^2, E^2) \in \mathcal{D}_{\tilde{\mathcal{J}}_0}$.  For simplicity, we will omit the dependence on the variable $z$. For any $(F^1, E^1) \in \mathcal{D}_{\tilde{\mathcal{J}}_0}$ we have
	\begin{align}
		\begin{split}
			0 &= \int_{a}^{b} (\tilde{e}^1)^{\top} \tilde{\phi} + \epsilon_1 (\mathbf{d}_0 e_2^1) + \epsilon_2 (- \mathbf{d}_0^{\ast} e_1^1) \, dz  + e_I^1 \phi_I + \epsilon_I f_I^1 \\
			&\hspace{0.5 cm} - \langle e_{\partial}^1, \phi_{\partial} \rangle_{2} - \langle \epsilon_{\partial}, f_{\partial}^1 \rangle_{2}. 
		\end{split}
		\label{Interface Paper - Proposition 4 - Variable Determination}
	\end{align} 
	Note that the choice of the variable $\tilde{e}^1 = (e_1^1,e_2^1, e_{c}^1, e_{\overline{c}}^1)$ determines all remaining variables in $(F^1, E^1) \in \mathcal{D}_{\tilde{\mathcal{J}}_0}$. Since \eqref{Interface Paper - Proposition 4 - Variable Determination} has to hold for all $(F^1, E^1) \in \mathcal{D}_{\tilde{\mathcal{J}}_0}$, we will consider different instances for the previously mentioned variable in order to verify that $(F^2, E^2) \in \mathcal{D}_{\tilde{\mathcal{J}}_0}$. 
	\begin{enumerate}
		\item[(i)] Let $e_1^1 = 0$, $e_2^1 = 0$, and $e_{\overline{c}}^1 = 0$. \\
		From \eqref{Interface Paper - Dirac Structure DI} and \eqref{Interface Paper - Proposition 4 - Variable Determination} we derive
		\begin{align*}
			0 = \int_{a}^{b} (\tilde{e}^1)^{\top} \tilde{\phi} \, dz = \int_{a}^{b} e_{c}^1 \phi_{c} \, dz. 
		\end{align*} 
		As $e_{c}^1$ may be chosen arbitrarily, we conclude that $\phi_{c} = 0$. If we change the roles of $e_{c}^1$ and $e_{\overline{c}}^1$, then we analogously obtain $\phi_{\overline{c}} = 0$. 
		\item[(ii)] Let $e_1^1(a) = e_1^1(b) = 0$, $e_1^1(0^+) - e_1^1(0^-) = 0$, and $e_2^1 = 0$. \\
		Then we deduce from \eqref{Interface Paper - Proposition 4 - Variable Determination} that
		\begin{align*}
			0 = \int_{a}^{b} e_1^1 \phi_1 + \epsilon_2(- \mathbf{d}_0^{\ast} e_1^1) \,  dz. 
		\end{align*}
		Assume that $\epsilon_2 \in D(\mathbf{d}_0)$. Then, by virtue of \eqref{Interface Paper - Relation d_0 and d_0ast}, it follows that
		\begin{align*}
			0 &= \int_{a}^{b} e_1^1 \phi_1 + e_1^1 (- \mathbf{d}_0 \epsilon_2) \, dz - \big[ e_1^1 \epsilon_2 \big]_{a}^{b} + \epsilon_2(0) \left[ e_1^1(0^+) - e_1^1(0^-) \right] \\ 
			&= \int_{a}^{b} e_1^1 (- \mathbf{d}_0 \epsilon_2 + \phi_1) \, dz,
		\end{align*}
		whence $\phi_1 = \mathbf{d}_0 \epsilon_2$. 
		\item[(iii)] Let $e_1^1 = 0$ and $e_2^1(a) = e_2^1(0) = e_2^1(b) = 0$. \\
		If we assume that $\epsilon_1 \in D(\mathbf{d}_0^{\ast})$ and if we apply \eqref{Interface Paper - Relation d_0 and d_0ast} once again, we get
		\begin{align*}
			0 &= \int_{a}^{b} e_2^1 \phi_2 + \epsilon_1 (\mathbf{d}_0 e_2^1) \, dz \\
			&= \int_{a}^{b} e_2^1 \phi_2 + e_2^1( \mathbf{d}_0^{\ast} \epsilon_1) \, dz - \big[e_2^1 \epsilon_1 \big]_{a}^{b} + e_2^1(0) \left[ \epsilon_1(0^+) - \epsilon_1(0^-) \right] \\
			&= \int_{a}^{b} e_2^1 \phi_2 + e_2^1( \mathbf{d}_0^{\ast} \epsilon_1) \, dz,
		\end{align*} 
		whence $\phi_2 = - \mathbf{d}_0^{\ast} \epsilon_1$. 
	\end{enumerate}
	From the preceding instances, we conclude that $\tilde{\epsilon} \in D(\tilde{\mathcal{J}}_0)$ and
	\begin{align*}
		\tilde{\phi} = \tilde{\mathcal{J}}_0 \tilde{\epsilon}.
	\end{align*}
	It is left to show the last three constitutive relations in (\ref{Interface Paper - Dirac Structure DI}) concerning the boundary and the interface port variables. 
	\begin{enumerate}
			\item[(iv)] Let $e_1^1(a) =  e_1^1(b) =0$ and $e_2^1 = 0$. \\
		From \eqref{Interface Paper - Relation d_0 and d_0ast} and \eqref{Interface Paper - Proposition 4 - Variable Determination} we infer
		\begin{align*}
			0 &= \int_{a}^{b} e_1^1(\mathbf{d}_0 \epsilon_2) + \epsilon_2(-\mathbf{d}_0^{\ast} e_1^1) \, dz + e_I^1 \phi_I \\
			&= \epsilon_2(0) \left[ e_1^1(0^+) - e_1^1(0^-) \right] + e_I^1 \phi_I \\
			&= - e_I^1 \epsilon_2(0)  + e_I^1 \phi_I.
		\end{align*}
		Hence, $\phi_I = \epsilon_2(0)$. 
		\item[(v)] Let $e_1^1 = 0$ and $e_2^1(a) = e_2^1(b) = 0$. \\
		With the same argument we find
		\begin{align*}
			0 &= \int_{a}^{b} e_2^1(-\mathbf{d}_0^{\ast} \epsilon_1) + \epsilon_1(\mathbf{d}_0 e_2^1) \, dz +\epsilon_I f_I^1 \\
			&= e_2^1(0) \left[ \epsilon_1(0^+) - \epsilon_1(0^-) \right] + \epsilon_I f_I^1 \\
			&= f_I^1 \left[ \epsilon_1(0^+) - \epsilon_1(0^-) \right] + \epsilon_I f_I^1.
			\end{align*}
		Consequently, $\epsilon_I = \epsilon_1(0^-) - \epsilon_1(0^+)$.
		\item[(vi)] Let $e_1^1(0^+) - e_1^1(0^-) = 0$. \\
		Following the discussion in Subsection \ref{Subsection Boundary Port Variables} and using the relation \eqref{Interface Paper - Skew-Symmetry of Ja}, equation \eqref{Interface Paper - Proposition 4 - Variable Determination} can be written as
		\begin{align*}
			0 	&= \langle \tilde{\epsilon}, \tilde{\mathcal{J}}_0 \tilde{e}^1 \rangle_{L^2}  + \langle \tilde{e}^1, \tilde{\mathcal{J}}_0 \tilde{\epsilon} \rangle_{L^2} + \epsilon_I f_I^1 - \langle e_{\partial}^1 , \phi_{\partial} \rangle_2 - \langle \epsilon_{\partial} , f_{\partial}^1 \rangle_2 \\
			&= \big[ (e^1)^{\top} P_1 \epsilon \big]_{a}^{b} - f_I^1 \epsilon_I  + \epsilon_I f_I^1 - \langle e_{\partial}^1 , \phi_{\partial} \rangle_2 - \langle \epsilon_{\partial} , f_{\partial}^1 \rangle_2 \\
			&= \trace_0(e^1)^{\top} \mathcal{\begin{bmatrix}
				P_1 & 0 \\
				0 & - P_1
			\end{bmatrix} } \trace_0(\epsilon)  - \langle e_{\partial}^1 , \phi_{\partial} \rangle_2 - \langle \epsilon_{\partial} , f_{\partial}^1 \rangle_2 \\
		&=  \trace_0(e^1)^{\top} \mathcal{P}_{\extern} \trace_0(\epsilon)  - \langle e_{\partial}^1 , \phi_{\partial} \rangle_2 - \langle \epsilon_{\partial} , f_{\partial}^1 \rangle_2 \\
		&=  \trace_0(e^1)^{\top} R_{\extern}^{\top} \Sigma R_{\extern} \trace_0(\epsilon)  - \langle e_{\partial}^1 , \phi_{\partial} \rangle_2 - \langle \epsilon_{\partial} , f_{\partial}^1 \rangle_2 \\
		&= \begin{bmatrix}
			(f_{\partial}^1)^{\top } & (e_{\partial}^1)^{\top}
		\end{bmatrix} \Sigma R_{\extern} \trace_0(\epsilon) - (e_{\partial}^1)^{\top} \phi_{\partial} - \epsilon_{\partial}^{\top} f_{\partial}^1 \\
&=  \begin{bmatrix}
	(e_{\partial}^1)^{\top } & (f_{\partial}^1)^{\top} \end{bmatrix} \left[ R_{\extern} \trace_0(\epsilon) - \begin{bmatrix}
	\phi_{\partial} \\
	\epsilon_{\partial}
\end{bmatrix} \right].
		\end{align*}
Since a proper choice for the variable $\tilde{e}$ yields arbitary values for the vectors $e_{\partial}^1$ and $f_{\partial}^1$, we conclude that
\begin{align*}
	\begin{bmatrix}
		\phi_{\partial} \\
		\epsilon_{\partial}
	\end{bmatrix} =  R_{\extern} \trace_0(\epsilon).
\end{align*}
	\end{enumerate}
	Thus, we have $(F^2, E^2) \in \mathcal{D}_{\tilde{\mathcal{J}}_0}$, and so the inclusion $\mathcal{D}_{\tilde{\mathcal{J}}_0}^{\perp \!\!\! \perp} \subset \mathcal{D}_{\tilde{\mathcal{J}}_0}$ holds as well. Summing up, we have $\mathcal{D}_{\tilde{\mathcal{J}}_0} = \mathcal{D}_{\tilde{\mathcal{J}}_0}^{\perp \!\!\! \perp}$. This completes the proof. 	
\end{proof}

Comparing the constitutive relations of the Dirac structure $\mathcal{D}_{\tilde{\mathcal{J}}_0}$ defined in \eqref{Interface Paper - Dirac Structure DI} with the generalized Hamiltonian system \eqref{Interface Paper - Augmented System wrt Ja} together with the interface relations \eqref{Interface Paper - Interface Variable fI}-\eqref{Interface Paper - Interface Variable eI} imposed on the flux variable $\mathcal{N}(x, c_0, \overline{c}_0)$, one may easily see that this system possesses a boundary port-Hamiltonian structure. Note that it is crucial to pick states of the form
\begin{align*}
	\tilde{x} = (x, c_0, \overline{c}_0) \in \tilde{X}, 
\end{align*}
with $c_0$ and $\overline{c}_0$ the color functions defined in \eqref{Interface Paper - Characteristic Functions c0 and c0overline}, as the conservation laws derived in Subsection \ref{Subsection Balance Equations of the State Variables}, which are constituting the generalized Hamiltonian system \eqref{Interface Paper - Augmented System wrt Ja}, have been determined with this particular choice. Denoting $f_I(t) = f_{I, \delta_{\tilde{x}}H(\tilde{x}(\cdot,t))}$ and $e_I(t) = e_{I, \delta_{\tilde{x}}H(\tilde{x}(\cdot,t))}$ for $t \geq 0$, the following holds. 
\begin{corollary}
	\label{Interface Paper - Corollary PH System with Nonmoving Interface}
	The generalized Hamiltonian system \eqref{Interface Paper - Augmented System wrt Ja} of four conservation laws may be defined as a generalized boundary port-Hamiltonian system, and is geometrically specified by the requirement that for all $t > 0$ we have
	\begin{align*}
		\left( \begin{pmatrix}
			\partial_t \tilde{x}(\cdot,t) \\
			f_I(t) \\
			f_{\partial}(t)
		\end{pmatrix}, \begin{pmatrix}
		\delta_{\tilde{x}} H(\tilde{x}(\cdot,t)) \\
		e_I(t) \\
		e_{\partial} (t)
	\end{pmatrix} \right) \in \mathcal{D}_{\tilde{\mathcal{J}}_0},
	\end{align*}
with $\mathcal{D}_{\tilde{\mathcal{J}}_0}$ the Dirac structure defined in \eqref{Interface Paper - Dirac Structure DI}, the extended state vector $\tilde{x} = (x, c_0, \overline{c}_0)$, where $c_0$ and $\overline{c}_0$ are precisely the color functions defined in \eqref{Interface Paper - Characteristic Functions c0 and c0overline}, the variational derivate \eqref{Interface Paper - Extended Variational Derivatives} of the Hamiltonian $H\colon \tilde{X} \to \mathbb{R}$ defined in \eqref{Interface Paper - Hamiltonian H of Coupled System}, the pair of port variables $(f_I,e_I) \in \mathbb{R}^2$ associated with the interface, and the pair of port variables $(f_{\partial}, e_{\partial}) \in \mathbb{R}^4$ associated with the boundary of the spatial domain $[a,b]$. 
\end{corollary}
As a consequence of the boundary port-Hamiltonian structure, we can state a power balance equation along solutions $\tilde{x} \colon [0, \infty) \to \tilde{X}$ of \eqref{Interface Paper - Augmented System wrt Ja} of the form 
\begin{align*}
	\tilde{x}(t) = \begin{bmatrix}
		x_1(\cdot, t) \\
		x_2(\cdot, t) \\
		c_0(\cdot,t) \\
		\overline{c}_0(\cdot,t)
	\end{bmatrix}, \hspace{0.5 cm} t \geq 0,
\end{align*}
by means of the boundary and interface port variables: using \eqref{Interface Paper - Skew-Symmetry of Ja - Representation 2}, for all $t > 0$ it holds that
\begin{align*}
	&\hspace{0.5 cm} \frac{d}{dt} H(\tilde{x}(\cdot,t)) \\
	&= \frac{d}{dt} \int_{a}^{b} c_0(z,t) \mathcal{H}^-(x(z,t)) + \overline{c}_0(z,t) \mathcal{H}^+(x(z,t)) \, dz \\
	&= \int_{a}^{b} (\partial_t c_0 (z,t)) \mathcal{H}^-(x(z,t)) + c_0(z,t) \partial_t \mathcal{H}^-(x(z,t)) \, dz \\
	&\hspace{0.5 cm} + \int_{a}^{b} (\partial_t \overline{c}_0(z,t)) \mathcal{H}^+(x(z,t)) + \overline{c}_0(z,t) \partial_t \mathcal{H}^+(x(z,t)) \, dz \\
	&= \int_{a}^{b} c_0(z,t) \left[ \delta_{x_1}H^-(x(z,t)) \partial_t x_1(z,t) + \delta_{x_2}H^-(x(z,t)) \partial_t x_2(z,t) \right] \, dz  \\
	&\hspace{0.5 cm} + \int_{a}^{b} \overline{c}_0(z,t) \left[ \delta_{x_1}H^+(x(z,t)) \partial_t x_1(z,t) + \delta_{x_2}H^+(x(z,t)) \partial_t x_2(z,t) \right] \, dz \\
	&= \int_{a}^{b} c_0(z,t) \left[ \delta_{x_1}H^-(x(z,t)) (\mathbf{d}_0 \delta_{x_2}H^-(x(z,t))) + \delta_{x_2}H^-(x(z,t)) \left( - \mathbf{d}_0^{\ast} \delta_{x_1}H^-(x(z,t)) \right) \right] \, dz \\
	&\hspace{0.5 cm } + \int_{a}^{b} \overline{c}_0(z,t) \left[ \delta_{x_1}H^+(x(z,t)) (\mathbf{d}_0 \delta_{x_2}H^+(x(z,t))) + \delta_{x_2}H^+(x(z,t)) \left( - \mathbf{d}_0^{\ast} \delta_{x_1}H^+(x(z,t)) \right) \right] \, dz \\
	&= \frac{1}{2} \left(  \langle \tilde{\mathcal{J}}_0 \delta_{\tilde{x}} H(\tilde{x}(\cdot,t)), \delta_{\tilde{x}} H(\tilde{x}(\cdot,t)) \rangle_{L^2} + \langle \delta_{\tilde{x}} H(\tilde{x}(\cdot,t)) , \tilde{\mathcal{J}}_0 \delta_{\tilde{x}} H(\tilde{x}(\cdot,t)) \rangle_{L^2} \right) \\
	&= \langle e_{\partial}(t), f_{\partial}(t) \rangle_2 - e_I(t) f_I(t). 
\end{align*}
This equation expresses the fact that the change of energy is given by the power flow through the boundary minus the power flow at the interface position. An interesting question is whether this class of systems even admits a solution, and whether it is well-posed. We will postpone the answer to questions of this kind to Section \ref{Section A Simplified System}, as in this section, we only wanted to lay the foundation for the upcoming segment.
\vspace{0.5 cm}\\
In this section, we have formulated two systems of conservation laws that are coupled by a fixed interface as a boundary port-Hamiltonian system with some power flow at the interface. In the following, we repeat this procedure, but we assume that the interface position is not fixed anymore, but moving over time.

\section{Port-Hamiltonian Systems Coupled by a Moving Interface} 
\label{Section Port-Hamiltonian Systems Coupled through a Moving Interface}
In this section, we want to consider two systems of conservation laws that are coupled by a moving interface. We denote by $l\colon [0, \infty) \to (a,b)$ the time-varying position of the interface and its velocity by $\dot{l}\colon [0, \infty) \to \mathbb{R}$. In particular, we assume that $l$ is continuously differentiable. We proceed as in the last section. We begin by formulating the conservation laws in Subsection \ref{Subsection Balance Equations of the State Variables - Moving Interface} and, hence, highlight the difference to the fixed interface case. Afterwards, we briefly describe the port-Hamiltonian formulation in case of a moving system in Subsection \ref{Subsection Port-Hamiltonian Formulation - Moving Interface}. This yields an abstract control system, where the input is the velocity $\dot{l}$ of the displacement of the interface. We will see that this system description poses delicate problems. Thus, we will discuss the main issues of this class of systems, and seek a simplified model closely related to the original problem that we are going to analyze in the further course.   

\subsection{Balance Equations of the State Variables}
\label{Subsection Balance Equations of the State Variables - Moving Interface}
Since the position of the interface $l$ is potentially moving, the systems of conservation laws \eqref{Interface Paper - Two Systems of Conservation Laws on Respective Domain} are now given by
\begin{align}
	\begin{split}
		\frac{\partial}{\partial t} x^- + \frac{\partial}{\partial z} \mathcal{N}_{l(t)}^-(x^-) &= 0, \hspace{0.5 cm}  z \in [a,l(t)), \hspace{0.1 cm} t >0,  \\
		\frac{\partial}{\partial t} x^+ + \frac{\partial}{\partial z} \mathcal{N}_{l(t)}^+(x^+) &= 0, \hspace{0.5 cm} z \in (l(t),b], \hspace{0.1 cm} t > 0.
	\end{split} 
	\label{Interface Paper - Two Systems of Conservation Laws on Respective Domain - Moving Interface} 
\end{align}
This implies that the Hamiltonians of the respective system depend on the time as well. For $t \geq 0$, the Hamiltonian functionals $H_{l(t)}^- \colon L^2((a,l(t)), \mathbb{R}^2) \to \mathbb{R}$ and $H_{l(t)}^+ \colon L^2((l(t),b),\mathbb{R}^2) \to \mathbb{R}$ are given as follows:
\begin{align*}
	H_{l(t)}^-(x^-) &= \int_{a}^{l(t)} \mathcal{H}^-(x^-(z)) \, dz, \hspace{0.5 cm} x^- \in  L^2((a,l(t)), \mathbb{R}^2), \\
	H_{l(t)}^+(x^+) &= \int_{l(t)}^{b} \mathcal{H}^+(x^+(z)) \, dz, \hspace{0.5 cm} x^+ \in L^2((l(t),b),\mathbb{R}^2).
\end{align*}
The flux variables $\mathcal{N}_{l(t)}^-(x^-)$ and $\mathcal{N}_{l(t)}^+(x^+)$ appearing in \eqref{Interface Paper - Two Systems of Conservation Laws on Respective Domain - Moving Interface} still satisfy the relation \eqref{Interface Paper - Relation Flux Variable and Variational Derivative}.
As the interface is moving, the two color functions now depend on the time-varying position of the interface. They are given by
\begin{align}
	c_{l}(z,t) = \begin{cases}
		1, & z \in [a,l(t)), \\
		0, & z \in [l(t),b],
	\end{cases} \hspace{0.3 cm} \text{and} \hspace{0.3 cm} \overline{c}_{l} (z,t) = \begin{cases} 
	0, & z \in [a,l(t)], \\
	1, & z \in (l(t),b],
	\end{cases}
\label{Interface Paper - Color Functions Depending on Moving Interface}
\end{align}
respectively. The prolongated state variables, the Hamiltonian functional, and the prolongated flux variables may be defined similarly to the definitions \eqref{Interface Paper - State Variable x Split}, \eqref{Interface Paper - Hamiltonian H of Coupled System}, and \eqref{Interface Paper - Flux Variable N Split}, respectively: for every $t\geq 0$, we have
\begin{align*}
	x(z,t) &= x^-(z,t) + x^+(z,t) \\
	&= c_l(z,t) x(z,t) + \overline{c}_l(z,t) x(z,t).
\end{align*}
The energy density is still given by
\begin{align*}
	\mathcal{H}(\tilde{x}) = \mathcal{H}(x,c, \overline{c}) = c \mathcal{H}^-(x) + \overline{c} \mathcal{H}^+(x), \hspace{0.5 cm} \tilde{x} \in \mathbb{R}^4. 
\end{align*}
Consequently, if we choose $c = c_l$ and $\overline{c} = \overline{c}_l$ as in \eqref{Interface Paper - Color Functions Depending on Moving Interface}, then the Hamiltonian functional is for all $\tilde{x} = (x, c_l, \overline{c}_l) \in \tilde{X}$ given by
 \begin{align*}
  H(\tilde{x}) &= \int_{a}^{b} \mathcal{H}(x(z),c_l(z), \overline{c}_l(z)) \, dz \\
  &= \int_{a}^{b}  c_l(z) \mathcal{H}^-(x(z)) + \overline{c}_l(z) \mathcal{H}^+(x(z)) \, dz \\
  &= \int_{a}^{l} \mathcal{H}^-(x(z)) \, dz + \int_{l}^{b} \mathcal{H}^+(x(z)) \, dz \\
  &= H_l^-(x^-) + H_l^+(x^+).
\end{align*}
Thus, for all $t \geq 0$, the mapping $x \mapsto H(x, c_l(\cdot,t), \overline{c}_l(\cdot,t))$ describes the internally stored energy of the coupled system at time $t$. For any $(x,c_l, \overline{c}_l) \in \tilde{X}$, the variational derivative is given by
\begin{align}
	\label{Interface Paper - Variational Derivative - Moving Interface}
	\delta_{\tilde{x}} H(\tilde{x}) = \begin{bmatrix}
		\delta_{x} H(x, c_l, \overline{c}_l) \\
		\delta_{c_l} H(x, c_l, \overline{c}_l) \\
		\delta_{\overline{c}_l} H(x, c_l, \overline{c}_l)
	\end{bmatrix} = \begin{bmatrix}
		c_l \delta_{x} H_l^-(x) +  \overline{c}_l \delta_{x} H_l^+(x) \\
		\mathcal{H}^-(x) \\
		\mathcal{H}^+(x)
	\end{bmatrix},
\end{align}
and the flux variables are again defined through
\begin{align}
	\begin{split}
	\begin{bmatrix}
		0 & 1 \\
		1 & 0 
	\end{bmatrix} \delta_{x} H(x, c_l, \overline{c}_l) &= \begin{bmatrix}
		0 & 1 \\
		1 & 0 
	\end{bmatrix}  c_l \delta_x H_l^-(x) + \overline{c}_l  \delta_x H_l^+ (x)  \\
	&= c_l \mathcal{N}_l^-(x) + \overline{c}_l \mathcal{N}_l^+(x)  \\
	&= \mathcal{N}(x, c_l, \overline{c}_l).
\end{split}
\label{Interface Paper - Relation Flux Variable and Variational Derivative - Composed Domain - Moving Interface}
\end{align}
Note that by prolongation of the flux variables $\mathcal{N}_1$ and $\mathcal{N}_2$ defined on the composed domain $[a,b]$, they do not explicitly depend on the position $l$ of the moving interface, but implicitly through the color functions \eqref{Interface Paper - Color Functions Depending on Moving Interface}.
\vspace{0.5 cm}\\ 
Due to the change of the position, the interface relations and, with that said, the interface port variables become time-dependent as well. For $t \geq 0$, these are given by
\begin{align}
	f_I(t) &= \mathcal{N}_{l(t), 1}^+(x^+)(l(t)^+) = \mathcal{N}_{l(t), 1}^-(x^-)(l(t)^-), \label{Interface Paper - Interface Flow fI - Moving Interface} \\
	-e_I(t) &= \mathcal{N}_{l(t), 2}^+(x^+)(l(t)^+) - \mathcal{N}_{l(t), 2}^-(x^-)(l(t)^-), \label{Interface Paper - Interface Effort e_I - Moving Interface} 
\end{align}
where $l(t) \in (a,b)$.
\vspace{0.5 cm}\\
Now, we may begin with stating the balance equations for the variables constituting the extended state variable $\tilde{x} = (x,c_l, \overline{c}_l) \in \tilde{X}$ in case of a moving interface. To this end, let $\Omega := (a,b) \times (0, \infty)$, and let $\mathcal{D}(\Omega)$ denote the set of test functions on $\Omega$. Throughout this section, we will show that the balance equations hold in $\mathcal{D}'(\Omega)$.

\paragraph*{Conservation laws of $\mathbf{c_l}$ and $\mathbf{\overline{c}_l}$.} In contrast to the fixed interface case, where the respective color functions $c_0$ and $\overline{c}_0$ satisfied trivial conservations laws, the color functions \eqref{Interface Paper - Color Functions Depending on Moving Interface} are the solutions of the transport equations depending on the velocity $\dot{l}$ of the interface, which are given by
\begin{align}
	\partial_t c_l(z,t) = - \dot{l}(t) \partial_z c_l(z,t) \hspace{0.3 cm} \text{and} \hspace{0.3 cm} \partial_t \overline{c}_l(z,t) = - \dot{l}(t) \partial_z \overline{c}_l(z,t)
	\label{Interface Paper - Balance Equation of Color Function with Moving Interface}
\end{align}
with the initial conditions
\begin{align*}
	c_l(z,0) = c_{l_0} (z) \hspace{0.3 cm} \text{and} \hspace{0.3 cm} \overline{c}_l(z,0) = \overline{c}_{l_0}(z),
\end{align*}
respectively, where $l_0 = l(0) \in (a,b)$. \\
\begin{figure}[h!]
	\centering
	{{\includegraphics[width=7.7764cm]{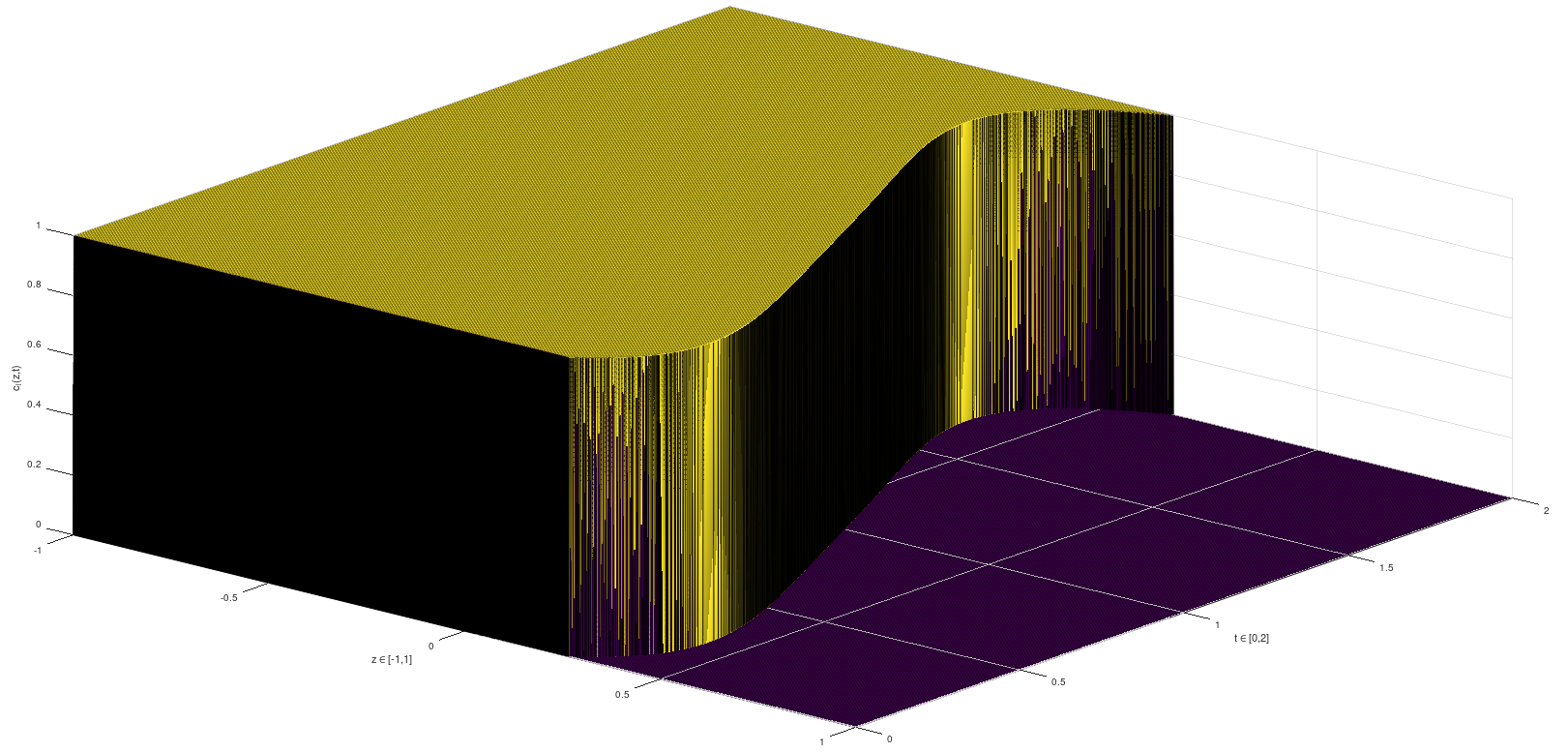} }}%
	\qquad
	{{\includegraphics[width=7.7764cm]{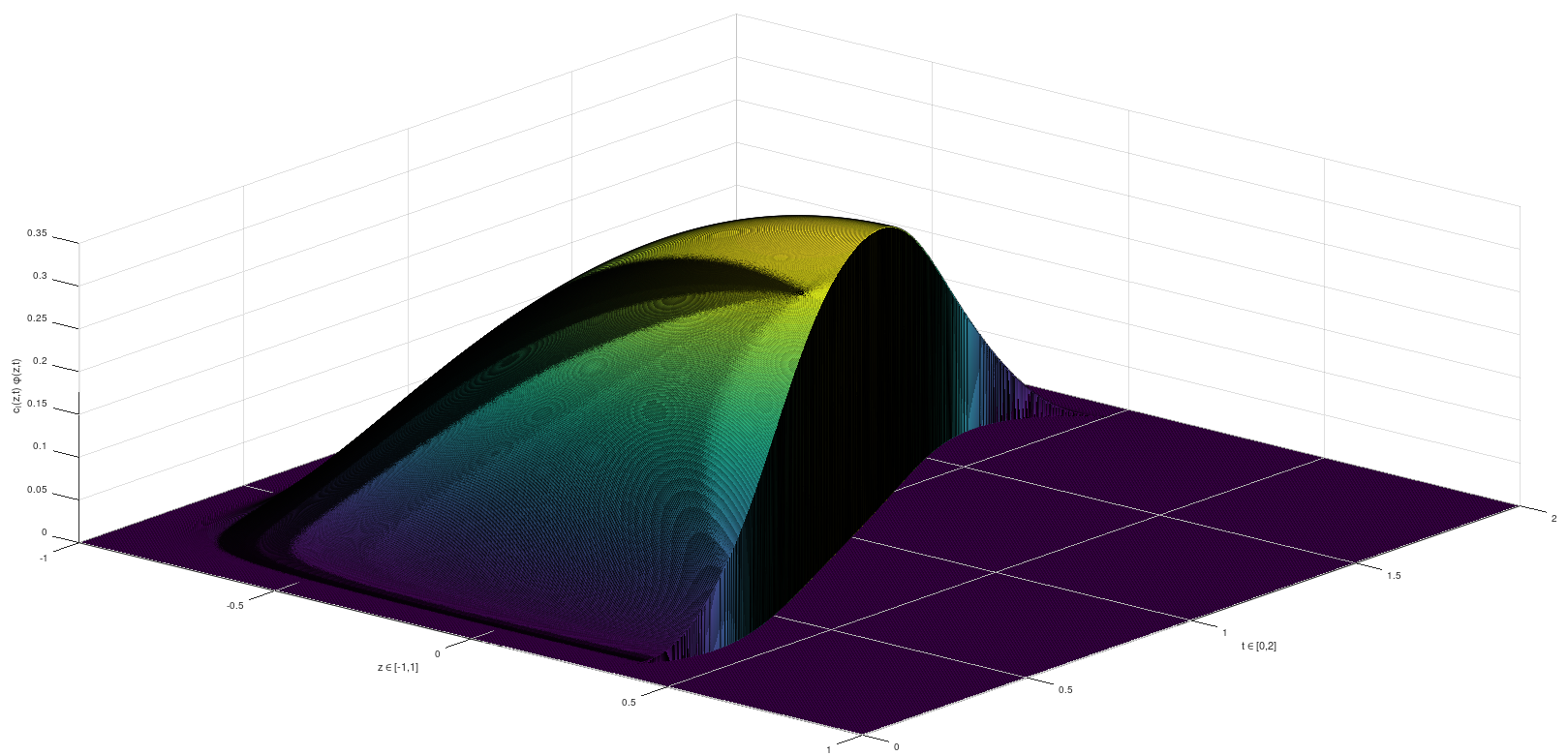} }}%
	\caption{Schematic Plots of the Color Function $c_l$ and of $c_l$ Tested Against some Test Function $\varphi \in \mathcal{D}(\Omega)$, respectively}
	\label{Figure Char and Bump}
\end{figure}

The transport equations given in \eqref{Interface Paper - Balance Equation of Color Function with Moving Interface} hold in a distributional sense. We only show this for the first equation, as the proof for the second equation works analogously. A plot of the color function $c_l$ defined with respect to some $\mathcal{C}^1$-function $l$ and of $c_l$ tested against some test function $\varphi \in \mathcal{D}(\Omega)$ is depicted in Figure \ref{Figure Char and Bump}.
\vspace{0.5 cm}\\
Let $\varphi \in \mathcal{D}(\Omega)$. Then we have
\begin{align*}
	& \hspace{0.5 cm} \int_{0}^{\infty} \int_{a}^{b} - \dot{l}(t) \partial_z c_l(z,t) \varphi(z,t) \, dz \, dt \\
	\text{(part. int.)} \hspace{0.3 cm}	&= \int_{0}^{\infty} - \dot{l}(t) \left( \big[c_l (z,t) \varphi(z,t) \big]_{z=a}^{b} - \int_{a}^{b} c_l(z,t) \partial_z \varphi(z,t) \,  dz \right) \, dt \\
	&= \int_{0}^{\infty} \dot{l}(t) \int_{a}^{l(t)} \partial_z \varphi(z,t) \, dz \, dt\\
	&= \int_{0}^{\infty} \dot{l}(t) \varphi(l(t),t) \, dt. 
\end{align*}
Furthermore, by applying Fubini's theorem and Leibniz's rule we obtain
\begin{align*}
	&\hspace{0.5 cm} \int_{0}^{\infty} \int_{a}^{b} \partial_t c_l(z,t) \varphi(z,t) \, dz \, dt \\
	&= \int_{a}^{b} \big[ c_l(z,t) \varphi(z,t)  \big]_{t = 0}^{\infty} - \int_{0}^{\infty} c_l(z,t) \partial_t \varphi(z,t) \, dt \, dz \\
	&= \int_{0}^{\infty} - \int_{a}^{l(t)} \partial_t \varphi(z,t) \, dz \, dt \\
	&= \int_{0}^{\infty} - \frac{d}{dt} \int_{a}^{l(t)} \varphi(z,t) \, dz + \varphi(l(t),t) \dot{l}(t) \, dt \\
	&= \int_{0}^{\infty} \dot{l}(t) \varphi(l(t),t)  \, dt  \underbrace{- \int_{0}^{\infty} \frac{d}{dt} \int_{a}^{l(t)} \varphi(z,t) \, dz  \, dt}_{=0} \\
	&= \int_{0}^{\infty} \dot{l}(t) \varphi(l(t),t)  \, dt.
\end{align*}
Thus, the conservation laws given in \eqref{Interface Paper - Balance Equation of Color Function with Moving Interface} hold in $\mathcal{D}'(\Omega)$.
\vspace{0.5 cm}\\
For all $t >0$ and on any interval $[a',b'] \subset [a,b]$ not containing the interface position $l(t)$, it holds that
\begin{align*}
	\frac{d}{dt} \int_{a'}^{b'} c_l(z,t) \, dz = 	\frac{d}{dt} \int_{a'}^{b'} \overline{c}_l(z,t) \, dz = 0. 
\end{align*}
If $a \leq a' < l(t) < b' \leq b$, then applying Leibniz's rule once again yields
\begin{align*}
		\frac{d}{dt} \int_{a'}^{b'} c_l(z,t) \, dz = 	\frac{d}{dt} \left( \int_{a'}^{l(t)} 1 \, dz + \int_{l(t)}^{b'} 0 \, dz \right) = \dot{l}(t)
\end{align*}
and, analogously, 
\begin{align*}
	\frac{d}{dt} \int_{a'}^{b'} \overline{c}_l(z,t) \, dz = \frac{d}{dt} \left( \int_{a'}^{l(t)} 0 \, dz + \int_{l(t)}^{b'} 1 \, dz \right) = - \dot{l}(t).
\end{align*}
Next, we want to find the conservation laws for the state variables $x_1$ and $x_2$.
\paragraph*{Conservation law of $\mathbf{x_1}$.}
As in the fixed interface scenario, the flux variable $\mathcal{N}_1(x, c_l, \overline{c}_l)$ is supposed to satisfy the continuity condition \eqref{Interface Paper - Interface Flow fI - Moving Interface}. For $l \in (a,b)$, define the operator $\mathbf{d}_l \colon D(\mathbf{d}_l) \subset L^2([a,b], \mathbb{R}) \to L^2([a,b], \mathbb{R})$ by
\begin{align}
	D(\mathbf{d}_l) &= \left\{ x \in L^2([a,b], \mathbb{R}) \mid x_{|(a,l)} \in H^1((a,l), \mathbb{R}), \hspace{0.1 cm} x_{|(l,b)} \in H^1((l,b),\mathbb{R}), \hspace{0.1 cm} x \in \mathcal{C}([a,b], \mathbb{R}) \right\} \nonumber \\
	&= H^1([a,b], \mathbb{R}), \label{Interface Paper - Operator dl} \\
	\mathbf{d}_lx &= - \left[ \frac{d}{dz} (c_l x) + \frac{d}{dz} (\overline{c}_l x) \right], \hspace{0.5 cm} x \in D(\mathbf{d}_l). \nonumber
\end{align}
As the color functions \eqref{Interface Paper - Color Functions Depending on Moving Interface} are changing over time, for $t >0$ the operator 
\begin{align*}
	\mathbf{d}(t) = \mathbf{d}(\tilde{x}(t)) \colon D(\mathbf{d}(t)) \subset L^2([a,b], \mathbb{R}) \to L^2([a,b], \mathbb{R})
\end{align*} 
is modulated by the state $\tilde{x}(t) = (x(\cdot,t), c_l(\cdot,t), \overline{c}_l(\cdot,t)) \in \tilde{X}$, and is given as follows:
\begin{align}
	\begin{split}
	D(\mathbf{d}(t)) &= D(\mathbf{d}_l) = H^1([a,b], \mathbb{R}), \\
	\mathbf{d}(t) x &= - \left[ \partial_z (c_l(\cdot,t) x) + \partial_z (\overline{c}_l(\cdot,t) x) \right], \hspace{0.5 cm} x \in D(\mathbf{d}(t)).
\end{split}
\label{Interface Paper - Operator d(t)}
\end{align}
Let us now define the conservation law of the state variable $x_1$ on $[a,b]$ in case of a moving interface. For $t >0$, assume that $\mathcal{N}_1(x,c_l, \overline{c}_l)(\cdot,t) \in D(\mathbf{d}(t))$, and let $\varphi \in \mathcal{D}(\Omega)$. Using the individual conservation laws \eqref{Interface Paper - Two Systems of Conservation Laws on Respective Domain - Moving Interface} and the continuity condition \eqref{Interface Paper - Interface Flow fI - Moving Interface}, which is once again included in the domain of $\mathbf{d}(t)$, one computes
\begin{align}
	\begin{split}
	&\hspace{0.5 cm} \int_{0}^{\infty} \int_{a}^{b} [\mathbf{d}(t) \mathcal{N}_1(x, c_l, \overline{c}_l) (z,t)] \varphi (z,t) \, dz \, dt  \\
	&= \int_{0}^{\infty} \mathcal{N}_1^-(l(t)^-,t) \varphi(l(t),t) - \mathcal{N}_1^+(l(t)^+,t) \varphi(l(t),t) \\
	&\hspace{0.5 cm} - \int_{a}^{l(t)} \partial_z \mathcal{N}_1^-(z,t)  \varphi (z,t) \, dz - \int_{l(t)}^{b} \partial_z \mathcal{N}_1^+(z,t) \varphi(z,t) \, dz \, dt \\
	&= \int_{0}^{\infty} \int_{a}^{l(t)} \partial_t x_1^-(z,t)  \varphi (z,t) \, dz +\int_{l(t)}^{b} \partial_t x_1^+(z,t) \varphi(z,t) \, dz \, dt.
\end{split}
\label{Interface Paper - Hilfe1}
\end{align}
Furthermore, we have
\begin{align*}
	&\hspace{0.5 cm} \int_{0}^{\infty} \int_{a}^{b} \partial_t x_1(z,t) \varphi(z,t) \, dz \, dt \\
	&=  \int_{0}^{\infty} \int_{a}^{b} \partial_t \left[ c_l(z,t) x_1(z,t) + \overline{c}_l(z,t) x_1(z,t) \right] \varphi(z,t) \, dz \, dt \\
\text{(Fubini)} \hspace{0.3 cm}	&= -  \int_{0}^{\infty} \int_{a}^{l(t)} x_1(z,t) \partial_t \varphi(z,t) \, dz \, dt \\
	&\hspace{0.5 cm} -  \int_{0}^{\infty} \int_{l(t)}^{b} x_1(z,t) \partial_t \varphi(z,t) \, dz \, dt \\
	&= -  \int_{0}^{\infty} \int_{a}^{l(t)}  \partial_t (x_1\varphi)(z,t) - \partial_t x_1(z,t) \varphi(z,t) \, dz \, dt \\
	&\hspace{0.5 cm} -  \int_{0}^{\infty} \int_{l(t)}^{b}  \partial_t (x_1\varphi)(z,t) - \partial_t x_1(z,t) \varphi(z,t) \, dz \, dt \\
	 \text{(Leibniz)} \hspace{0.3 cm} &=  \int_{0}^{\infty} \int_{a}^{l(t)} \partial_t x_1^-(z,t) \varphi(z,t) \, dz  + \dot{l}(t) x_1^-(l(t)^-,t) \varphi(l(t),t) \, dt \\
	 &\hspace{0.5 cm} + \int_{0}^{\infty} \int_{l(t)}^{b} \partial_t x_1^+(z,t)  \varphi(z,t) \, dz - \dot{l}(t) x_1^+(l(t)^+,t) \varphi(l(t),t) \, dt,
\end{align*}
where in the last step we additionally used the fact that
\begin{align*}
	\int_{0}^{\infty} \frac{d}{dt} \int_{a}^{l(t)} x_1^-(z,t) \varphi(z,t) \, dz \, dt = 
	\int_{0}^{\infty} \frac{d}{dt} \int_{l(t)}^{b} x_1^+(z,t) \varphi(z,t) \, dz \, dt  = 0. 
\end{align*}
Together with \eqref{Interface Paper - Hilfe1}, we conclude that for all $\varphi \in \mathcal{D}(\Omega)$ it holds that
\begin{align*}
&\hspace{0.5 cm} \int_{0}^{\infty} \int_{a}^{b} \partial_t x_1(z,t) \varphi(z,t) \, dz \, dt \\
 &= \int_{0}^{\infty}  \int_{a}^{b} \mathbf{d}(t) \mathcal{N}_1(z,t) \varphi (z,t) \, dz - \dot{l}(t) \left[ x_1^+(l(t)^+,t) - x_1^-(l(t)^-,t) \right] \varphi(l(t),t) \, dt.
\end{align*}
Thus, the equation
\begin{align}
	\partial_t x_1 = \mathbf{d}(t) \mathcal{N}_1(x, c_l, \overline{c}_l) - \dot{l}  \left[ c_l x_1 \partial_z c_l + \overline{c}_l x_1 \partial_z \overline{c}_l \right]
	\label{Interface Paper - Conservation Law x1 - Moving Interface}
\end{align}
holds in $\mathcal{D}'(\Omega)$. We want to stress that, in \eqref{Interface Paper - Conservation Law x1 - Moving Interface}, the equation
\begin{align*}
	c_l \partial_t x_1^- + \overline{c}_l \partial_t x_1^+ = \mathbf{d}(t) \mathcal{N}_1(x, c_l, \overline{c}_l) 
\end{align*}
holds in $L^2([a,b], \mathbb{R})$. Since the color functions $c_l$ and $\overline{c}_l$ do not satisfy the trivial conservation laws anymore, but the transport equation \eqref{Interface Paper - Balance Equation of Color Function with Moving Interface}, equation \eqref{Interface Paper - Conservation Law x1 - Moving Interface} comprises an equality solely holding in $\mathcal{D}'(\Omega)$ as well, namely
\begin{align*}
	x_1 \partial_t c_l + x_1 \partial_t \overline{c}_l =  - \dot{l}  \left[ c_l x_1 \partial_z c_l + \overline{c}_l x_1 \partial_z \overline{c}_l \right].
\end{align*}
Once again, on any interval $[a',b'] \subset [a,b]$ not containing the interface position $l(t)$, $t >0$, the operator $\mathbf{d}(t)$ simply acts as the differential operator $- \frac{d}{dz}$. However, due to the moving interface, on intervals $[a',b']$ with $a \leq a' < l(t) < b' \leq b$ there appears an additional term depending on the velocity $\dot{l}$ of the interface:
\begin{align*}
	&\hspace{0.5 cm} \frac{d}{dt} \int_{a'}^{b'} x_1(z,t) \, dz  \\
	&= \frac{d}{dt} \left( \int_{a'}^{l(t)} x_1^-(z,t) \, dz + \int_{l(t)}^{b'} x_1^+(z,t) \, dz \right) \\
	\text{(Leibniz)} \hspace{0.3 cm}	&= \int_{a'}^{l(t)} \partial_t x_1^-(z,t) \, dz + \int_{l(t)}^{b'} \partial_t x_1^+(z,t) \, dz \\
	&\hspace{0.5 cm} + \dot{l}(t) x_1^-(l(t)^-,t) - \dot{l}(t) x_1^+(l(t)^+,t) \\
	&= \int_{a'}^{b'} \mathbf{d}(t) \mathcal{N}_1(x,c_l, \overline{c}_l)(z,t) \, dz - \dot{l}(t) \left[ x_1^+(l(t)^+,t) - x_1^-(l(t)^-,t) \right] \\
	&= \mathcal{N}_1^-(a',t) - \mathcal{N}_1^+(b',t) - \dot{l}(t)  \left[ x_1^+(l(t)^+,t) - x_1^-(l(t)^-,t) \right].
\end{align*}
\paragraph*{Conservation law of $\mathbf{x_2}$.} Lastly, we want to formulate the conservation law of the state variable $x_2$ in case of a moving interface, where the corresponding flux variable $\mathcal{N}_2(x,c_l, \overline{c}_l)$ is supposed to satisfy the balance equation \eqref{Interface Paper - Interface Effort e_I - Moving Interface} at all times. Beforehand, we once again need to determine the formally adjoint operator of $\mathbf{d}(t)$, since the conservation law of $x_2$ has to be formulated with respect to $\mathbf{d}^{\ast}(t)$, $t>0$. Analogously to the operator $\mathbf{d}_0^{\ast}$ defined in \eqref{Interface Paper - Operator d_0ast}, for $l \in (a,b)$ one defines the formal adjoint $\mathbf{d}_l^{\ast} \colon D(\mathbf{d}_0^{\ast}) \subset X \to X$ of $\mathbf{d}_l$ through
\begin{align}
	\begin{split}
	D(\mathbf{d}_l^{\ast}) &= \left\{ y \in X \mid y_{|(a,l)}  \in H^1((a,l), \mathbb{R}), \hspace{0.1 cm} y_{|(l,b)}  \in H^1((l,b), \mathbb{R}) \right\}, \\
	\mathbf{d}_l^{\ast} y &= \left[ \frac{d}{dz} (c_l y) + \frac{d}{dz} (\overline{c}_l y) \right] - \left[ \frac{d}{dz} c_l + \frac{d}{dz} \overline{c}_l \right] y \\
	&= \left( - \mathbf{d}_l - \left[ \frac{d}{dz} c_l + \frac{d}{dz} \overline{c}_l \right] \right) y. 
	\end{split}
\label{Interface Paper - Operator dlast}
\end{align}
The relation between $\mathbf{d}_l$ and its formal adjoint $\mathbf{d}_l^{\ast}$ defined in \eqref{Interface Paper - Operator dl} and \eqref{Interface Paper - Operator dlast}, respectively, is given as follows: for $x \in D(\mathbf{d}_l)$, $y \in D(\mathbf{d}_l^{\ast})$ we have
\begin{align*}
	\langle \mathbf{d}_l x, y \rangle_{L^2} = - \big[ x(z) y(z) \big]_{a}^{b} + x(l) \left[ y(l^+) - y(l^-) \right] + \langle x, \mathbf{d}_l^{\ast} y \rangle_{L^2}.
\end{align*}

Hence, the operator 
\begin{align*}
\mathbf{d}^{\ast}(t) = \mathbf{d}^{\ast}(\tilde{x}(t)) \colon D(\mathbf{d}^{\ast}(t)) \subset L^2([a,b], \mathbb{R}) \to L^2([a,b], \mathbb{R}), \hspace{0.5 cm} t > 0,
\end{align*} 
is modulated by the color functions $c_l$, $\overline{c}_l$, and thus by the extended state $\tilde{x} = (x,c_l, \overline{c}_l) \in \tilde{X}$. Unlike $\mathbf{d}(t)$ defined in \eqref{Interface Paper - Operator d(t)}, the domain of $\mathbf{d}^{\ast}(t)$ potentially changes over time. However, as the interface position $l$ is assumed to be continuously differentiable, the time-dependence of the domain is somewhat regular. For $t > 0$, the operator $\mathbf{d}^{\ast}(t)$ is defined by
\begin{align}
	\begin{split}
	D(\mathbf{d}^{\ast}(t)) &= \left\{ y \in X \mid y_{|(a,l(t))}  \in H^1((a,l(t)), \mathbb{R}), \hspace{0.1 cm} y_{|(l(t),b)}  \in H^1((l(t),b), \mathbb{R}) \right\}, \\
	\mathbf{d}^{\ast}(t) y &= \left[ \partial_z (c_l(\cdot,t) y) + \partial_z (\overline{c}_l(\cdot,t) y) \right] - \left[ \partial_z c_l(\cdot,t) + \partial_z \overline{c}_l(\cdot,t) \right] y \\
	&= \left( - \mathbf{d}(t) - \left[ \partial_z c_l(\cdot,t) + \partial_z \overline{c}_l(\cdot,t) \right] \right) y.
\end{split}
\label{Interface Paper - Operator d(t)ast}
\end{align}
For $t >0$, $x \in D(\mathbf{d}(t))$, and $y \in D(\mathbf{d}^{\ast}(t))$, the relation between the operators $\mathbf{d}(t)$ and $\mathbf{d}^{\ast}(t)$ is given as follows:
\begin{align*}
	\langle \mathbf{d}(t) x, y \rangle_{L^2}= - \big[ x(z) y(z) \big]_{a}^{b} + x(l(t)) \left[ y(l(t)^+) - y(l(t)^-) \right] + \langle x, \mathbf{d}^{\ast}(t) y \rangle_{L^2}.
\end{align*}

Now we are able to derive the conservation law of $x_2$ with respect to the operator $\mathbf{d}^{\ast}(t)$, $t>0$. Assume that $\mathcal{N}_2(x, c_l, \overline{c}_l)(\cdot,t) \in D(\mathbf{d}^{\ast}(t))$ for all $t > 0$. By virtue of the individual conservation laws \eqref{Interface Paper - Two Systems of Conservation Laws on Respective Domain - Moving Interface} and with the same arguments as for the conservation law of the state variable $x_1$, we have for all $\varphi \in \mathcal{D}(\Omega)$,
\begin{align*}
 &\hspace{0.5 cm} \int_{0}^{\infty} \int_{a}^{b} \partial_t x_2(z,t) \varphi(z,t) \, dz \, dt \\
  &= \int_{0}^{\infty} \int_{a}^{l(t)} \partial_t x_2^-(z,t) \varphi(z,t) \, dz  + \dot{l}(t) x_2^-(l(t)^-,t) \varphi(l(t),t) \, dt \\
	&\hspace{0.5 cm} + \int_{0}^{\infty} \int_{l(t)}^{b} \partial_t x_2^+(z,t) \varphi(z,t) \, dz - \dot{l}(t) x_2^+(l(t)^+,t) \varphi(l(t),t)  \, dt.
\end{align*}
Furthermore, we observe that
\begin{align*}
	&\hspace{0.5 cm} \int_{0}^{\infty} \int_{a}^{b} [- \mathbf{d}^{\ast}(t) \mathcal{N}_2(x,c_l, \overline{c}_l) (z,t)] \varphi(z,t) \, dz \, dt \\
	&= \int_{0}^{\infty} \int_{a}^{l(t)} - \partial_z \mathcal{N}_2^-(z,t) \varphi (z,t) \, dz + \int_{l(t)}^{b} - \partial_z \mathcal{N}_2^+(z,t) \varphi (z,t) \, dz \, dt \\
	&= \int_{0}^{\infty} \int_{a}^{l(t)} \partial_t x_2^-(z,t) \varphi(z,t) \, dz + \int_{l(t)}^{b}  \partial_t x_2^+(z,t) \varphi(z,t) \, dz  \, dt. 
\end{align*}
In conclusion, the equation
\begin{align*}
	&\hspace{0.5 cm} \int_{0}^{\infty} \int_{a}^{b} \partial_t x_2(z,t) \varphi (z,t) \, dz \, dt \\
	 &=  \int_{0}^{\infty} \int_{a}^{b} - \mathbf{d}^{\ast}(t) \mathcal{N}_2(z,t) \varphi(z,t) \, dz - \dot{l}(t) \left[ x_2^+(l(t)^+,t) - x_2^-(l(t)^-,t) \right] \varphi(l(t),t) \, dt
\end{align*}
holds for all $\varphi \in \mathcal{D}(\Omega)$. Thus, the conservation law is given by
\begin{align}
	\partial_t x_2 = - \mathbf{d}^{\ast}(t) \mathcal{N}_2(x, c_l, \overline{c}_l) - \dot{l} \left[ c_l x_2 \partial_z c_l + \overline{c}_l x_2 \partial_z \overline{c}_l \right],
	\label{Interface Paper - Conservation Law x2 - Moving Interface}
\end{align}
and it holds in $\mathcal{D}'(\Omega)$. As for the conservation law of the state variable $x_1$, we want to emphasize that the equation 
\begin{align*}
	c_l \partial_t x_2^- + \overline{c}_l \partial_t x_2^+ = - \mathbf{d}^{\ast}(t) \mathcal{N}_2(x, c_l, \overline{c}_l) 
\end{align*}
holds in a classical sense, whereas
\begin{align*}
	x_2 \partial_t c_l + x_2 \partial_t \overline{c_l } = - \dot{l} \left[ c_l x_2 \partial_z c_l + \overline{c}_l x_2 \partial_z \overline{c}_l \right]
\end{align*}
merely holds in $\mathcal{D}'(\Omega)$. Recall from Remark \ref{Interface Paper - Remark Balance Equation} that the conservation law \eqref{Interface Paper - Conservation Law x2 - Moving Interface} does not encompass the balance equation \eqref{Interface Paper - Interface Effort e_I - Moving Interface}. As proposed, for simplification of the model we do not consider $e_I(t)$ as an external source term, but define it to quantify the jump discontinuity at the moving interface position $l(t)$, $t \geq 0$.  
\vspace{0.5 cm}\\
As usual, on subdomains $[a',b'] \subset [a,b]$ not containing $l(t)$, $t>0$, the operator $\mathbf{d}^{\ast}(t)$ acts as the differential operator $-\frac{d}{dz}$. 
However, due to the moving interface and the jump discontinuity of the flux variable $\mathcal{N}_2$ at the interface position, on $[a',b']$ with $a \leq a' < l(t) < b' \leq b$ it holds that \\
\begin{align*}
	&\hspace{0.5 cm} \frac{d}{dt} \int_{a'}^{b'} x_2 (z,t) \, dz \\
	&= \frac{d}{dt} \left( \int_{a'}^{l(t)} x_2^-(z,t) \, dz + \int_{l(t)}^{b'} x_2^+(z,t) \, dz \right)\\
	&= \int_{a'}^{l(t)} \partial_t x_2^-(z,t) \, dz + \int_{l(t)}^{b'} \partial_t x_2^+(z,t) \, dz \\
	&\hspace{0.5 cm} + \dot{l}(t) x_2^-(l(t)^-,t) - \dot{l}(t) x_2^+(l(t)^+,t) \\
	&= \int_{a'}^{b'} - \mathbf{d}^{\ast}(t) \mathcal{N}_2(x, c_l, \overline{c}_l) \, dz - \dot{l}(t) \left[ x_2^+(l(t)^+,t) - x_2^-(l(t)^-,t) \right] \\
	&= \mathcal{N}_2^-(a',t) - \mathcal{N}_2^+(b',t) - \mathcal{N}_2^-(l(t)^-,t) + \mathcal{N}_2^+(l(t)^+,t) - \dot{l}(t) \left[ x_2^+(l(t)^+,t) - x_2^-(l(t)^-,t) \right] \\ 
	&= \mathcal{N}_2^-(a',t) - N_2^+(b',t) - e_I(t) - \dot{l}(t) \left[ x_2^+(l(t)^+,t) - x_2^-(l(t)^-,t) \right]. 
\end{align*}
After formulating the conservation laws of the variables constituting the extended state $\tilde{x} = (x,c_l, \overline{c}_l) \in \tilde{X}$, we can eventually discuss the port-Hamiltonian formulation of the systems of conservation laws \eqref{Interface Paper - Two Systems of Conservation Laws on Respective Domain - Moving Interface} that are coupled through a moving interface. 

\subsection{Port-Hamiltonian Formulation}
\label{Subsection Port-Hamiltonian Formulation - Moving Interface}

In the previous section, we have derived four conservation laws of the constituting state variables with respect to a moving interface, stated in \eqref{Interface Paper - Balance Equation of Color Function with Moving Interface}, \eqref{Interface Paper - Conservation Law x1 - Moving Interface}, and \eqref{Interface Paper - Conservation Law x2 - Moving Interface}. All of these equations hold in $\mathcal{D}'(\Omega)$. We now want to formulate a (generalized) Hamiltonian system which comprises these conservation laws. As opposed to the fixed interface scenario, this yields (formally) an abstract control system. For a complete port-Hamiltonian formulation, we would need to consider the conjugated output as well (cf. \eqref{Nonlinear PH-System - Finite Dimensional} in the finite-dimensional case). In \cite{Diagne}, the adjoint output map has been outlined, but in this thesis, we will not address this issue.
\vspace{0.5 cm}\\ 
Omitting the dependence on the spatial variable $z$, and with the aid of \eqref{Interface Paper - Relation Flux Variable and Variational Derivative - Composed Domain - Moving Interface} we formally obtain the system
\begin{align*}
	\frac{\partial}{\partial t} \begin{bmatrix}
		x_1(t) \\
		x_2(t)\\
		c_l(t) \\
		\overline{c}_l(t)
	\end{bmatrix} = \begin{bmatrix}
			0 & \mathbf{d}(t) & 0 & 0 \\
			-\mathbf{d}^{\ast}(t) & 0 & 0 & 0 \\
			0 & 0 & 0 & 0 \\
			0 & 0 & 0 & 0 
	\end{bmatrix} \begin{bmatrix}
		\delta_{x_1}H(x, c_l, \overline{c}_l)(t) \\
		\delta_{x_2}H(x, c_l, \overline{c}_l)(t) \\
		\delta_{c_l}H(x, c_l, \overline{c}_l)(t) \\
		\delta_{\overline{c}_l} H(x, c_l, \overline{c}_l)(t) 
	\end{bmatrix} - \begin{bmatrix}
			(c_l x_1)(t) & (\overline{c}_l x_1)(t) \\
			(c_l x_2)(t) & (\overline{c}_l x_2)(t) \\
			1 & 0 \\
			0 & 1 
		\end{bmatrix} \partial_z \begin{bmatrix}
			c_l(t) \\
			\overline{c}_l(t)
	\end{bmatrix} \dot{l}(t),
\end{align*}
or short,
\begin{align}
	\frac{\partial}{\partial t}\tilde{x}(t) = \tilde{\mathcal
	{J}}(\tilde{x}(t)) \delta_{\tilde{x}}H(\tilde{x}(t)) + \tilde{B}(\tilde{x}(t)) \dot{l}(t). 
	\label{Interface Paper - Augmented System wrt Jl}
\end{align}

For $\tilde{x}(t) = (x_1(t),x_2(t), c_l(t), \overline{c}_l(t)) \in \tilde{X}$, the operator $\tilde{\mathcal{J}}(\tilde{x}(t)) \colon D(\tilde{\mathcal{J}}(\tilde{x}(t))) \subset \tilde{X} \to \tilde{X}$ given by
\begin{align*}
D(\tilde{\mathcal{J}}(\tilde{x}(t))) &= D(\mathbf{d}^{\ast}(t)) \times D(\mathbf{d}(t)) \times L^2([a,b], \mathbb{R}^2), \\
\tilde{\mathcal{J}}(\tilde{x}(t)) \tilde{y} &= \begin{bmatrix}
	0 & \mathbf{d}(t) & 0 & 0 \\
	-\mathbf{d}^{\ast}(t) & 0 & 0 & 0 \\
	0 & 0 & 0 & 0 \\
	0 & 0 & 0 & 0 
\end{bmatrix} \tilde{y}, \hspace{0.5 cm} \tilde{y} \in D(\tilde{\mathcal{J}}(\tilde{x}(t))), 
\end{align*}
is modulated by the extended state $\tilde{x}(t) \in \tilde{X}$, since the operator $\mathbf{d}(t)$ and its formal adjoint $\mathbf{d}^{\ast}(t)$ defined in \eqref{Interface Paper - Operator d(t)} and \eqref{Interface Paper - Operator d(t)ast}, respectively, depend on the time-varying color functions $c_l$ and $\overline{c}_l$ defined in \eqref{Interface Paper - Color Functions Depending on Moving Interface}. The variational derivative $\delta_{\tilde{x}} H \colon \tilde{X} \to \tilde{X}$ is defined as in \eqref{Interface Paper - Variational Derivative - Moving Interface}. 
\vspace{0.5 cm}\\
If we take the velocity $\dot{l}$ of the interface position to be the input of the system \eqref{Interface Paper - Augmented System wrt Jl}, then one may define the mapping $B(\tilde{x}(t))$ as the input map with respect to the input $\dot{l}$, i.e., 
\begin{align}
	B(\tilde{x}(t)) \colon \dot{l} \mapsto - \begin{bmatrix}
		(c_l x_1)(t) & (\overline{c}_l x_1)(t) \\
		(c_l x_2)(t) & (\overline{c}_l x_2)(t) \\
		1 & 0 \\
		0 & 1 
	\end{bmatrix} \partial_z \begin{bmatrix}
		c_l(t) \\
		\overline{c}_l(t)
\end{bmatrix} \dot{l}, \hspace{0.5 cm} \dot{l} \in \mathbb{R}.
\label{Interface Paper - Operator Bx}
\end{align}
In the following, we want to give some comments on the obtained system description \eqref{Interface Paper - Augmented System wrt Jl} of two systems of conservation laws that are coupled by some moving interface defined on the extended state space $\tilde{X} = L^2([a,b], \mathbb{R}^4)$.
\vspace{0.5 cm}\\
First and foremost, recall that on any interval not containing the interface position, each conservation law stated in \eqref{Interface Paper - Balance Equation of Color Function with Moving Interface}, \eqref{Interface Paper - Conservation Law x1 - Moving Interface}, and \eqref{Interface Paper - Conservation Law x2 - Moving Interface} holds in a classical sense. However, in contrast to the fixed interface scenario discussed in Section \ref{Section Two Port-Hamiltonian Systems Coupled by an Interface}, these conservation laws defined on the composed domain $[a,b]$ contain an additional term consisting of evaluations of the extended state variables at the interface position multiplied by the velocity of the moving interface. Regarding the system \eqref{Interface Paper - Augmented System wrt Jl}, these point evaluations are captured in the entries of the input mapping $B(\tilde{x})$. Thus, one possible approach to deal with this class of systems is to define another state space comprising the following information:
\begin{enumerate}[label = (\roman*)]
	\item the position of the moving interface $l(t)$ and its velocity $\dot{l}(t)$, $t > 0$,
	\item the solutions on the respective subdomains $[a,l(t))$ and $(l(t),b]$, $t >0$, and
	\item the right-hand limits $x_1^+(l(t)^+,t)$, $x_2^+(l(t)^+,t)$ and the left-hand limits $x_1^-(l(t)^-,t)$, $x_2^-(l(t)^-,t)$ of the state variables $x_1$ and $x_2$, $t > 0$.
\end{enumerate}
Nevertheless, for the analysis, we wish to write the system \eqref{Interface Paper - Augmented System wrt Jl} as an abstract control system of the form
\begin{align*}
	\dot{x}(t) = Ax(t) + Bu(t), \hspace{0.5 cm} t > 0,
\end{align*}
as discussed in Section \ref{Section Linear Control Systems}, and check the sufficient conditions to prove well-posedness of the system. Unfortunately, this is not offhandedly possible due to a number of reasons. We have identified the following main issues:
\begin{enumerate}[label =(P\arabic*)]
\item \label{Problem1MovingInterface} From \eqref{Interface Paper - Augmented System wrt Jl} we infer that the operator $A$ is modulated by the state and, in particular, changes over time. So for each time instance $t >0$ and $\tilde{x}(t)$, we need to define the operator $\tilde{A}(\tilde{x}(t)) \colon D(\tilde{A}(\tilde{x}(t))) \subset \tilde{X} \to \tilde{X}$  with
\begin{align}
	\tilde{A}(\tilde{x}(t))\tilde{y} = \tilde{\mathcal{J}}(\tilde{x}(t)) \delta_{\tilde{y}} H(\tilde{y}), \hspace{0.5 cm} \tilde{y} \in D(\tilde{A}(\tilde{x}(t))).
	\label{Interface Paper - Operator tildeA}
\end{align}
This yields a family $(\tilde{A}( \tilde{x}(t)))_{t \geq 0}$ of state-dependent operators. As a consequence, this system does not belong to the class of systems considered in Section \ref{Section Linear Control Systems}. Furthermore, the state $\tilde{x}(t)$ has to be of the form
\begin{align*}
	\tilde{x}(t) = \begin{bmatrix}
		x_1(t) \\
		x_2(t) \\
		c_l(t) \\
		\overline{c}_l(t)
	\end{bmatrix},
\end{align*} 
i.e., it has to be defined with respect to the color functions $c_l$, and $\overline{c}_l$. That is because the operators $\mathbf{d}(t)$ and $\mathbf{d}^{\ast}(t)$ constituting the matrix differential operator $\tilde{\mathcal{J}}(\tilde{x}(t))$ depend (in order do keep track of the interface position) on the color functions, and each conservation law has been defined with respect to $c_l$ and $\overline{c}_l$. With that said, the domain $D(\tilde{A}(\tilde{x}(t)))$ has to be vastly restricted, such that $\tilde{A}(\tilde{x}(t))$ cannot be densely defined, let alone the infinitesimal generator of a $C_0$-semigroup. This is a huge problem, as this is a basic requirement to prove well-posedness of abstract control systems, see Proposition~\ref{Proposition Well-Posedenss of Abstract Control Systems}.

\item \label{Problem2MovingInterface} The input map $B(\tilde{x}(t))$ is modulated by the state $\tilde{x}(t)$ as well. Aside from that, there would be several problems even if $B$ were independent of the extended state $\tilde{x}$. Recall that we claimed that the conservation laws of the constituting state variables derived in Subsection \ref{Subsection Balance Equations of the State Variables - Moving Interface} hold in a distributional sense. Moreover, it is easy to check that for all $t > 0$ and for a reasonable choice of $x_1$ and $x_2$, each entry of $B(\tilde{x})$ given in \eqref{Interface Paper - Operator Bx} lies in $\mathcal{D}'(a,b)$. Following the discussion in Section \ref{Section Linear Control Systems}, we know that the assumption that $B$ is a bounded linear operator from $\mathbb{R}$ to $\tilde{X}$ is too restrictive. In order to apply the concept of admissibility, we need to check whether $B$ is an element of $\mathcal{L}(\mathbb{R}, \tilde{X}_{-1})$, where we may identify $\mathcal{L}(\mathbb{R}, \tilde{X}_{-1})$ with the extrapolation space $\tilde{X}_{-1} = (D(\tilde{A}(\tilde{x})), \| \cdot \|_{-1})$ associated with the operator $\tilde{A}(\tilde{x})$ defined in \eqref{Interface Paper - Operator tildeA} (see Definition \ref{Definition Extrapolation Space}). Assuming that there is a way to show that $A(\tilde{x})$ generates a $C_0$-semigroup (see \ref{Problem1MovingInterface}), for a further treatment we would need to check whether the vector defined in \eqref{Interface Paper - Operator Bx} lies in $\tilde{X}_{-1}$, and, as the case may be, whether $B$ is admissible, see Definition \ref{Definition Admissible Control Operator}. This is the second condition that has to be satisfied in order to show well-posedness, see again Proposition \ref{Proposition Well-Posedenss of Abstract Control Systems}. Since in the present scenario it is unavoidable that the input map is modulated by the state, we need to find another approach. Although there is some literature dealing with such input maps (see \cite{Ball}), the treatment of this problem is once again beyond the scope of this thesis.
	
\item \label{Problem3MovingInterface} There are also some difficulties concerning the port-Hamiltonian formulation of the obtained system \eqref{Interface Paper - Augmented System wrt Jl}. In \cite{Diagne}, the adjoint output map and the underlying Dirac structure of the system have been outlined.  However, according to \ref{Problem1MovingInterface}-\ref{Problem2MovingInterface}, the co-domain of $B(\tilde{x})$ is, at best, the extrapolation space $\tilde{X}_{-1}$, and it is not clear how one shall incorporate the constitutive relation between the flow and effort variables that represents the dynamics of the system \eqref{Interface Paper - Augmented System wrt Jl} into the definition of the Dirac structure. Possibly, one has to introduce the notion of a generalized Dirac structure that preserves the main characteristic of Dirac structures, namely the power-conserving interconnection between the system components through the respective ports. Furthermore, the Dirac structure is not constant anymore, but is modulated by the extended state $\tilde{x}$, as the position of the interface has to be captured at all times.

\end{enumerate}
In conclusion, it does not seem promising to further investigate the quite problematic system description \eqref{Interface Paper - Augmented System wrt Jl}, since the treatment of (some of) these delicate problems is beyond the scope of this thesis.
\vspace{0.5 cm}\\ 
Alternatively, we are going to consider a simplified model that is still closely related to the original one, and seek some fruitful results that might help for further investigating the formulation derived in \eqref{Interface Paper - Augmented System wrt Jl}. A plausible approach is to simply remove the color functions from the state variables again, and to consider the adjusted system description on the state space $X = L^2([a,b], \mathbb{R}^2)$. This is what we are going to do next.

\section{A Simplified Model}
\label{Section A Simplified System}
For the remainder of this thesis, we want to study a simplified version of the systems \eqref{Interface Paper - Augmented System wrt Ja} and \eqref{Interface Paper - Augmented System wrt Jl}, where we assume that the color functions are not part of the state variable anymore. We seek to find useful results that are beneficial for the further analysis of the respective original problem. Our main focus relies on proving that under certain (passivity) conditions, for arbitrary interface positions $l \in (a,b)$, the port-Hamiltonian operator associated with the simplified system generates a contraction semigroup on its corresponding energy space, as this result sets the foundation for further investigations.
\vspace{0.5 cm}\\
We start with the fixed interface position $l = 0$. Let $X = L^2([a,b], \mathbb{R}^2)$. Consider the system
\begin{align}
	\label{Interface - Simplified PH-System}
	\partial_t x(z,t) = \mathcal{J}_0 \left( \mathcal{Q}_0(z) x(z,t) \right), \hspace{0.5 cm} t >0, 
\end{align}
where $\mathcal{Q}_0 = c_0 \mathcal{Q}^- + \overline{c}_0 \mathcal{Q}^+ \in \mathcal{L}(X)$ is a coercive matrix multiplication operator. In particular, there exist constants $0 < m \leq M$ such that $mI \leq \mathcal{Q}^{\pm}(z) \leq MI$ holds for almost all $z \in [a,b]$. The functions $c_0$ and $\overline{c}_0$ are the color functions defined in \eqref{Interface Paper - Characteristic Functions c0 and c0overline}. 
\vspace{0.5 cm}\\
We endow $X$ with $\langle \cdot, \cdot \rangle_{\mathcal{Q}_0} = \frac{1}{2} \langle \cdot, \mathcal{Q}_0 \cdot \rangle_{L^2}$, i.e., 
\begin{align}
	\langle x , y \rangle_{\mathcal{Q}_0} = \frac{1}{2} \int_{a}^{b} y^{\top}(z) \mathcal{Q}_0(z) x(z) \, dz, \hspace{0.5 cm} x, y \in X.
	\label{Interface - Inner Product wrt Q0}
\end{align}
The Hamiltonian $H \colon X \to \mathbb{R}$ on the energy space $(X, \| \cdot \|_{\mathcal{Q}_0})$ is defined as
\begin{align}
	\label{Interface - Simplified PH-System - Hamiltonian}
	H(x) = \frac{1}{2} \int_{a}^{b} x^{\top}(z) \mathcal{Q}_0(z) x(z) \, dz =  \|x\|_{\mathcal{Q}_0}^2, \hspace{0.5 cm} x \in X.
\end{align}
The operator $\mathcal{J}_0 \colon D(\mathcal{J}_0) \subset X \to X$ has already been introduced in Subsection \ref{Subsection Port-Hamiltonian Formulation and Dirac Structure}, and is given by
\begin{align}
	\begin{split}
	D(\mathcal{J}_0) &= \left\{ x \in X \mid x_1 \in D(\mathbf{d}_0^{\ast}), \hspace{0.1 cm} x_2 \in D(\mathbf{d}_0)\right\}, \\
	\mathcal{J}_0 x &= \begin{bmatrix}
		0 & \mathbf{d}_0  \\
		- \mathbf{d}_0^{\ast} & 0 
	\end{bmatrix} \begin{bmatrix}
		x_1 \\
		x_2
	\end{bmatrix}, \hspace{0.5 cm} x \in D(\mathcal{J}_0),
\end{split}
\label{Interface - Operator J0}
\end{align}
where the operator $\mathbf{d}_0$ and its formal adjoint $\mathbf{d}_0^{\ast}$ have been defined in \eqref{Interface Paper - Operator d0} and \eqref{Interface Paper - Operator d_0ast}, respectively: they are given by
\begin{align*}
	D(\mathbf{d}_0) &= H^1([a,b], \mathbb{R}), \\
	\mathbf{d}_0x &= - \left[ \frac{d}{dz} (c_0 x) + \frac{d}{dz} (\overline{c}_0 x ) \right], \hspace{0.5 cm} x \in D(\mathbf{d}_0),
\end{align*}
and
\begin{align*}
	D(\mathbf{d}_0^{\ast}) &= \left\{ y \in L^2([a,b], \mathbb{R}) \mid y_{|(a,0)} \in H^1((a,0), \mathbb{R}), \hspace{0.1 cm} y_{|(0,b)} \in H^1((0,b), \mathbb{R}) \right\}, \\
	\mathbf{d}_0^{\ast}y &= \left( - \mathbf{d}_0 - \left[ \frac{d}{dz} c_0 + \frac{d}{dz} \overline{c}_0 \right] \right) y, \hspace{0.5 cm} y \in D(\mathbf{d}_0^{\ast}).
\end{align*}
Note that on any interval $[a',b'] \subset[a,b]$ not containing the interface position $l = 0$, the operator $\mathcal{J}_0$ simply acts as the matrix differential operator $P_1 \frac{d}{dz}$, where
\begin{align}
	P_1 = \begin{bmatrix}
		0 & -1 \\
		-1 & 0 
	\end{bmatrix} \in \mathbb{R}^{2 \times 2}.
\label{Interface - Matrix P1}
\end{align}
We will proceed as follows: In Subsection \ref{Subsection Dirac Structure - Infinite Dimensional Case} we begin with the definition of the underlying Dirac structure of the system \eqref{Interface - Simplified PH-System}. This can be easily generalized to arbitrary interface positions $l \in (a,b)$. Following this, in Subsection \ref{Subsection Generation of a Contraction Semigroup} we will show that under certain boundary and interface conditions, the port-Hamiltonian operator associated with the system \eqref{Interface - Simplified PH-System} generates a contraction semigroup on the energy space. If we add some more assumptions concerning the boundary conditions and the matrix operators constituting $\mathcal{Q}_0$, we can even guarantee exponential stability of the semigroup, as we will work out in Subsection \ref{Subsection Exponential Stability}. Furthermore, we determine the adjoint operator of the port-Hamiltonian operator in Subsection \ref{Subsection The Adjoint Operator}, and the corresponding resolvent operator in Subsection \ref{Subsection Resolvent}. Lastly, in Subsection \ref{Subsection Stability of the Family of Infinitesimal Generators} we shall define a family of port-Hamiltonian operators that keeps track of the moving interface position, and analyze the resulting time-variant evolution problem associated with this family. We present sufficient conditions for the stability of this family, as discussed in Section \ref{Section Evolution Equations}. This is a key property and our last contribution made in this thesis for the further analysis of the system formulation suggested in \cite{Diagne}.
\subsection{Dirac Structure}
\label{Subsection Dirac Structure - Infinite Dimensional Case}
As usual, to define the underlying Dirac structure of the system \eqref{Interface - Simplified PH-System}, we need to find a suitable power pairing $ \langle \cdot \mid \cdot \rangle \colon \mathcal{E} \times \mathcal{F} \to \mathbb{R}$ which will constitute the plus pairing $\ll \cdot , \cdot \gg \colon \mathcal{B} \times \mathcal{B} \to \mathbb{R}$ on the bond space $\mathcal{B} = \mathcal{F} \times \mathcal{E}$. We will mainly repeat the steps we have seen in Subsection \ref{Subsection Port-Hamiltonian Formulation and Dirac Structure}, and therefore will not go into detail.
\vspace{0.5 cm}\\
Analogously to the verification that the operator $\tilde{\mathcal{J}}_0$ given in \eqref{Interface Paper - Operator Ja} is formally skew-symmetric, see \eqref{Interface Paper - Skew-Symmetry of Ja}, we can show that the operator $\mathcal{J}_0$ defined in \eqref{Interface - Operator J0} is formally skew-symmetric as well. Applying the relation \eqref{Interface Paper - Relation d_0 and d_0ast} between the operators $\mathbf{d}_0$ and $\mathbf{d}_0^{\ast}$, for $e^1, e^2 \in D(\mathcal{J}_0)$ we have
\begin{align}
	\begin{split}
 \langle \mathcal{J}_0 e^1 , e^2 \rangle_{L^2} + \langle e^1 , \mathcal{J}_0 e^2 \rangle_{L^2} &= - \big[ e_1^2(z) e_2^1(z) \big]_{a}^{b} + e_2^1(0) \left[ e_1^2(0^+) - e_1^2(0^-) \right] \\
	&\hspace{0.5 cm} - \big[ e_1^1(z) e_2^2(z) \big]_{a}^{b} + e_2^2(0) \left[ e_1^1(0^+) - e_1^1(0^-) \right] \\
	&= \big[ (e^1)^{\top}(z) P_1 e^2(z) \big]_{a}^{b} \\
	&\hspace{0.5 cm}  + e_2^1(0) \left[ e_1^2(0^+) - e_1^2(0^-) \right]  + e_2^2(0) \left[ e_1^1(0^+) - e_1^1(0^-) \right],
	\end{split} 
\label{Interface - Skew-Symmetry of J0}
\end{align}
with $P_1$ defined in \eqref{Interface - Matrix P1}. This relation will help us finding the appropriate  power pairing later on. Prior to that, we need to augment the system \eqref{Interface - Simplified PH-System} with boundary and interface port variables, which we will introduce next.  
\vspace{0.5 cm}\\
Again, since the interface position resides in the interior of the spatial domain $[a,b]$, the trace operator defined in \eqref{Trace Operator - Chapter 3} can be extended to $\trace_0 \colon D(\mathcal{J}_0) \to \mathbb{R}^4$ with 
\begin{align*}
 e \mapsto \begin{bmatrix}
		e(b) \\
		e(a)
	\end{bmatrix}, \hspace{0.5 cm} e \in D(\mathcal{J}_0).
\end{align*}
Consequently, for all co-energy variables $ e \in D(\mathcal{J}_0)$, the boundary flow $f_{\partial} = f_{\partial, e} \in \mathbb{R}^2$ and the boundary effort $e_{\partial} = e_{\partial, e} \in \mathbb{R}^2$ are defined as
\begin{align}
	\label{Interface - Boundary Flow and Effort}
	\begin{bmatrix}
		f_{\partial} \\
		e_{\partial}
	\end{bmatrix} =  R_{\extern} \trace_0(e) = \begin{bmatrix}  \frac{1}{\sqrt{2}} (P_1e(b) - P_1 e(a)) \\
		\frac{1}{\sqrt{2}} ( e(b) + e(a))
	\end{bmatrix}, 
\end{align} 
where
\begin{align}
	\label{Interface - Matrix Rext}
	R_{\extern} = \frac{1}{\sqrt{2}} \begin{bmatrix}
		P_1 & - P_1 \\
		I & I 
	\end{bmatrix} \in \mathbb{R}^{4 \times 4}. 
\end{align}
Next, we define the interface port variables $e_{I} = e_{I, e}$, $f_I = f_{I, e} \in \mathbb{R}$, which have already been introduced in Section \ref{Section Two Port-Hamiltonian Systems Coupled by an Interface}.  For all effort variables $e \in D(\mathcal{J}_0)$, the interface flow and the interface effort are given as follows:
\begin{align}
	f_I &= e_2 (0^+) = e_2(0^-), \label{Interface - Simplified Continuity Equation} \\
	-e_I &= e_1(0^+) - e_1 (0^-). \label{Interface - Simplified Balance Equation}
\end{align}
It is left to find the power pairing in order to define the Dirac structure associated with the system \eqref{Interface - Simplified PH-System} that is augmented with the previously defined boundary and interface port variables. To this end, we write the right-hand side of \eqref{Interface - Skew-Symmetry of J0} with respect to the boundary and interface port variables for all $ e^1, e^2 \in D(\mathcal{J}_0)$ as follows: 
\begin{align}
	\begin{split}
	&\hspace{0.5 cm} \big[ (e^1)^{\top}(z) P_1 e^2(z) \big]_{a}^{b} + e_2^1(0) \left[ e_1^2(0^+) - e_1^2(0^-) \right]  + e_2^2(0) \left[ e_1^1(0^+) - e_1^1(0^-) \right] \\
	&= \langle e_{\partial, e^2} , f_{\partial, e^1} \rangle_2 + \langle e_{\partial, e^1} , f_{\partial, e^2} \rangle_2 - f_{I,e^1} e_{I, e^2} - f_{I, e^2} e_{I, e^1}. 
	\end{split}
\label{Interface - Auxiliary Equation Power Pairing}
\end{align}
Now, define the space of flows $\mathcal{F}$ and the space of efforts $\mathcal{E}$ for the reduced system  as
\begin{align}
	\mathcal{F} = \mathcal{E} = L^2([a,b], \mathbb{R}^2) \times \mathbb{R} \times \mathbb{R}^2,
	\label{Interface - Flow and Effort Space}
\end{align}
endowed with the canonical inner product $\langle \cdot, \cdot  \rangle_{\mathcal{F}}$. For the right choice of the power pairing $\langle \cdot \mid \cdot \rangle \colon \mathcal{E} \times \mathcal{F} \to \mathbb{R}$ we simply substitute the relation \eqref{Interface - Auxiliary Equation Power Pairing} into equation \eqref{Interface - Skew-Symmetry of J0}. This yields for all $e^1, e^2 \in D(\mathcal{J}_0)$, 
\begin{align}
	\langle \mathcal{J}_0 e^1, e^2 \rangle_{L^2} + \langle e^1, \mathcal{J}_0 e^2 \rangle_{L^2}  = \langle e_{\partial, e^2} , f_{\partial, e^1} \rangle_2 + \langle e_{\partial, e^1} , f_{\partial, e^2} \rangle_2 - f_{I,e^1} e_{I, e^2} - f_{I, e^2} e_{I, e^1}.
	\label{Interface - Skew-Symmetry of J0 - Representation 2}
\end{align}
Thus, we define
\begin{align}
	\left\langle \begin{pmatrix}
		e \\
		e_I \\
		e_{\partial} 
	\end{pmatrix} \Bigg| \begin{pmatrix}
	f \\
	f_I \\
	f_{\partial}
	\end{pmatrix} \right\rangle := \langle e, f \rangle_{L^2} +e_I f_I - \langle e_{\partial}, f_{\partial} \rangle_2, \hspace{0.5 cm} \left( \begin{pmatrix}
	e \\
	e_I \\
	e_{\partial} 
\end{pmatrix} , \begin{pmatrix}
f \\
f_I \\
f_{\partial}
\end{pmatrix} \right) \in \mathcal{E} \times \mathcal{F}, 
\label{Interface - Power Pairing}
\end{align}
and the plus pairing $\ll \cdot , \cdot \gg \colon \mathcal{B} \times \mathcal{B} \to \mathbb{R}$ on the bond space $\mathcal{B} = \mathcal{F} \times \mathcal{E}$ induced by the power pairing \eqref{Interface - Power Pairing} is accordingly defined as in \eqref{Plus Pairing - Distributed-Parameter}. The Dirac structure associated with the formally skew-symmetric operator $\mathcal{J}_0$ defined in \eqref{Interface - Operator J0} is given as follows.

\begin{corollary}
	\label{Interface - Corollary Dirac Structure}
	Consider the formally skew-symmetric operator $\mathcal{J}_0$ given by \eqref{Interface - Operator J0}. Let $\mathcal{B} = \mathcal{F} \times \mathcal{E}$ be the bond space with respect to the flow space $\mathcal{F}$ and the effort space $\mathcal{E}$ defined in \eqref{Interface - Flow and Effort Space}. Then the subspace $\mathcal{D}_{\mathcal{J}_0} \subset \mathcal{B}$ defined as
	\begin{align}
		\mathcal{D}_{\mathcal{J}_0} :=& \left\{ \left( \begin{pmatrix}
			f \\
			f_I \\
			f_{\partial}
		\end{pmatrix} , \begin{pmatrix}
		e \\
		e_I \\
		e_{\partial}
	\end{pmatrix} \right) \in \mathcal{B} \hspace{0.1 cm} \Bigg| \hspace{0.1 cm}  e \in D(\mathcal{J}_0), \, f = \mathcal{J}_0e, \, f_I = e_2(0), \right. \label{Interface - Dirac Structure DJ0} \\
&\hspace{4.22 cm} \left. \, e_I = e_1(0^-) - e_1(0^+), \, \begin{bmatrix}
	f_{\partial} \\
	e_{\partial}
\end{bmatrix} = R_{\extern} \trace_0(e) \right\} \nonumber
	\end{align}
is a Dirac structure with respect to the plus pairing $\ll \cdot , \cdot \gg$ induced by the power pairing \eqref{Interface - Power Pairing}. 
\end{corollary}

\begin{proof}
	This follows immediately from the proof of Proposition \ref{Interface Paper - Proposition Dirac Structure DI}.
	\end{proof}

As usual, we can geometrically specify the dynamics of the system \eqref{Interface - Simplified PH-System} with respect to its underlying Dirac structure. 

\begin{corollary}
	Consider the system \eqref{Interface - Simplified PH-System} with Hamiltonian \eqref{Interface - Simplified PH-System - Hamiltonian} augmented with the boundary and interface port variables given in \eqref{Interface - Boundary Flow and Effort} and \eqref{Interface - Simplified Continuity Equation}-\eqref{Interface - Simplified Balance Equation}, respectively. The dynamics of the resulting port-Hamiltonian system are geometrically specified by the requirement that for all $t > 0$ it holds that
	\begin{align*}
		\left( \begin{pmatrix}
			\frac{\partial}{\partial t} x(\cdot,t) \\
			f_I(t) \\
			f_{\partial}(t)
		\end{pmatrix} ,  \begin{pmatrix}
		\mathcal{Q}_0x(\cdot,t) \\
		e_I(t) \\
		e_{\partial}(t)
	\end{pmatrix} \right) \in \mathcal{D}_{\mathcal{J}_0},
	\end{align*}
with $f_I(t) = f_{I, \mathcal{Q}_0x(\cdot,t)}$, $e_I(t) = e_{I, \mathcal{Q}_0 x(\cdot,t)}$, $f_{\partial}(t) = f_{\partial, \mathcal{Q}x(\cdot,t)}$, $e_{\partial}(t) = e_{\partial, \mathcal{Q}x(\cdot,t)}$, and $\mathcal{D}_{\mathcal{J}_0}$ the Dirac structure defined in \eqref{Interface - Dirac Structure DJ0}. 
\end{corollary}
Lastly, we want to state a balance equation for classical solutions of the first-order system \eqref{Interface - Simplified PH-System} by means of the boundary and interface port variables.
\begin{corollary}
	\label{Interface - Theorem Balance Equation}
	Let $x \colon [0, \infty) \to X$ be a classical solution of the system \eqref{Interface - Simplified PH-System} with Hamiltonian \eqref{Interface - Simplified PH-System - Hamiltonian}. Then we have for all $t \geq 0$, 
	\begin{align}
		\label{Interface - Power Balance Equation wrt Boundary and Interface Port}
	\frac{d}{dt} \|x(t)\|_{\mathcal{Q}_0}^2 =	\frac{d}{dt}H(x(t)) =  \langle e_{\partial}(t), f_{\partial}(t) \rangle_2 - e_I(t) f_I(t).
	\end{align}
\end{corollary}
\begin{proof}
	By exploiting the fact that $x$ satisfies \eqref{Interface - Simplified PH-System} and by using \eqref{Interface - Skew-Symmetry of J0 - Representation 2}, we get for every $t \geq 0$, 
	\begin{align*}
		\frac{d}{dt}H(x(\cdot,t)) &= \frac{1}{2} \frac{d}{dt} \langle x(\cdot,t), \mathcal{Q}_0x(\cdot,t) \rangle_{L^2} \\
		&= \frac{1}{2} \langle \mathcal{J}_0( \mathcal{Q}_0x(\cdot,t) ), \mathcal{Q}_0 x(\cdot,t) \rangle_{L^2} + \frac{1}{2} \langle \mathcal{Q}_0 x(\cdot,t), \mathcal{J}_0(\mathcal{Q}_0x(\cdot,t)) \rangle_{L^2} \\
		&= \langle e_{\partial}(t), f_{\partial}(t) \rangle_2 - e_I(t) f_I(t).
	\end{align*}
\end{proof}
Equation \eqref{Interface - Power Balance Equation wrt Boundary and Interface Port} states that the change of energy is equal to the power flow both at the boundary and at the interface position. Note that this is the same balance equation as for the system \eqref{Interface Paper - Augmented System wrt Ja}.
\vspace{0.5 cm}\\
Now that we have established the port-Hamiltonian formulation of the simplified system \eqref{Interface - Simplified PH-System}, we want to discuss under which conditions this system is well-defined. 
\subsection{Generation of a Contraction Semigroup}
\label{Subsection Generation of a Contraction Semigroup}

In this section, we want to prove well-posedness of the system \eqref{Interface - Simplified PH-System}. To this end, we investigate the port-Hamiltonian operator associated with the system \eqref{Interface - Simplified PH-System} with Hamiltonian \eqref{Interface - Simplified PH-System - Hamiltonian}. As mentioned at the end of Chapter \ref{Chapter Infinite-dimensional Port-Hamiltonian Systems}, in \cite[Section 4]{LeGorrec} it has been proved that formulating the boundary conditions with respect to the boundary port variables, one can characterize those boundary conditions for which the associated port-Hamiltonian operator generates a contraction semigroup. We will therefore follow this approach. However, we need to include interface conditions as well, which we want to specify by means of the interface port variables. These conditions will be incorporated into the definition of the domain of the port-Hamiltonian operator. We will show that if we impose a passivity condition at the interface, then these operators generate a contraction semigroup and therefore guarantee well-posedness of the system \eqref{Interface - Simplified PH-System}. Furthermore, we will show that, under some more restrictive conditions, these operators even generate a unitary semigroup. At the end of this section, we apply the former result by showing well-posedness of the model of two transmission lines with different physical properties that are coupled through a resistor. 
\vspace{0.5 cm}\\
Let us start with the boundary conditions for the system \eqref{Interface - Simplified PH-System}. We impose boundary conditions of the form 
\begin{align}
	\tilde{W}_B \begin{bmatrix}
		(\mathcal{Q}_0x)(b) \\
		(\mathcal{Q}_0x)(a) 
	\end{bmatrix} = 0
\label{Interface - Boundary Condition}
\end{align}
for some $\tilde{W}_B \in \mathbb{R}^{2 \times 4}$. These are boundary conditions with respect to the co-energy variables of the system. In order to formulate them with respect to the boundary port variables, we just use their definition \eqref{Interface - Boundary Flow and Effort}. As the matrix  $R_{\extern}$ defined in \eqref{Interface - Matrix Rext} is invertible, we may equivalently write \eqref{Interface - Boundary Condition} as
\begin{align}
	W_B \begin{bmatrix}
		f_{\partial} \\
		e_{\partial}
	\end{bmatrix} = 0, 
	\label{Interface - Boundary Condition wrt Boundary Flow and Effort}
\end{align}
where $W_B = \tilde{W}_B R_{\extern}^{-1} \in \mathbb{R}^{2 \times 4}$. Recall that we want to show that the port-Hamiltonian operator associated with the system \eqref{Interface - Simplified PH-System} generates a contraction semigroup. As depicted in Example \ref{Example Contraction Semigroup}, this implies that the energy of the state $x \in X$ is dissipating. By virtue of the balance equation \eqref{Interface - Power Balance Equation wrt Boundary and Interface Port}, we see that we need to choose the matrix $W_B$ in \eqref{Interface - Boundary Condition wrt Boundary Flow and Effort} such that $\langle e_{\partial}, f_{\partial} \rangle_{2} \leq 0$ holds. This is possible, as we will see shortly. 
\vspace{0.5 cm}\\
Next, we need to impose an interface condition constituted by the interface port variables \eqref{Interface - Simplified Continuity Equation}-\eqref{Interface - Simplified Balance Equation}. Again, from the balance equation \eqref{Interface - Power Balance Equation wrt Boundary and Interface Port} we infer that the interface relation
\begin{align}
	\label{Interface - Interface Passivity Relation}
	f_I = r e_I
\end{align}
for some $r \geq 0$ is a good choice. Recall that the interface port variables are precisely given by
\begin{align*}
	f_I &= \left( \mathcal{Q}^+x\right)_2 (0^+) = \left( \mathcal{Q}^- x \right)_2 (0^-), \\
	-e_I &= \left( \mathcal{Q}^+x \right)_1 (0^+) - \left( \mathcal{Q}^- x \right)_1 (0^-).   
\end{align*}
\begin{remark}
Note that, due to the minus sign on the right-hand side of the balance equation \eqref{Interface - Power Balance Equation wrt Boundary and Interface Port}, the relation \eqref{Interface - Interface Passivity Relation} is in fact a passivity relation with passivity constant $r \geq 0$. This does not contradict the discussion from Section \ref{Section Energy-dissipating Elements}: If we consider the power pairing \eqref{Power-Balance with Dissipation} and the balance equation \eqref{Power-Balance Energy-Storing and Energy-Dissipating Elements}, then it is easy to check that passivity is a matter of sign convention. The important thing is that the inner product (including the sign) of the resistive port variables appearing in the balance equation is less than or equal to $0$, indicating energy dissipation. Thus, in contrast to the resistive structure \eqref{Resistive Relation Finite-Dimensional Case} in the finite-dimensional case, the set of resistive relations $\mathcal{R}_r$ concerning the interface port variables is now given by
\begin{align*}
	\mathcal{R}_r = \left\{ (f_I,e_I) \in \mathbb{R}^2 \mid f_I = re_I\right\}.
\end{align*}
\end{remark}
Now we are able to define the port-Hamiltonian operator  $A_{\mathcal{Q}_0} \colon D(A_{\mathcal{Q}_0}) \subset X \to X$ associated with the system \eqref{Interface - Simplified PH-System} with Hamiltonian \eqref{Interface - Simplified PH-System - Hamiltonian}. It is given by
\begin{align}
	\begin{split}
		D(A_{\mathcal{Q}_0}) &= \left\{ x \in X \, \big| \, \mathcal{Q}_0x \in D(\mathcal{J}_0), \hspace{0.1 cm} f_{I, \mathcal{Q}_0x} = r e_{I, \mathcal{Q}_0x}, \hspace{0.1 cm} W_B \begin{bmatrix}
			f_{\partial, \mathcal{Q}_0 x} \\
			e_{\partial, \mathcal{Q}_0 x}
		\end{bmatrix} = 0 \right\}, \\
		A_{\mathcal{Q}_0}x &= \mathcal{J}_0 (\mathcal{Q}_0x), \hspace{0.5 cm}  x \in D(A_{\mathcal{Q}_0}).
	\end{split}
	\label{Interface - Simplified PH System - Operator A}
\end{align}
We wish to prove that $A_{\mathcal{Q}_0}$ generates a contraction semigroup $(T(t))_{t \geq 0}$ on the energy space $(X, \| \cdot \|_{\mathcal{Q}_0})$. Before we prove the statement, we need some auxiliary lemmas. The proofs of those lemmas can be found in Section 7.2 and Section 7.3 in \cite{JacobZwart}.
\begin{lemma}[Lemma 7.2.3 in \cite{JacobZwart}]
	\label{Interface - Auxiliary Lemma 1}
	Let $Z$ be a Hilbert space with inner product $\langle \cdot,  \cdot \rangle$, and let $\mathcal{P} \in \mathcal{L}(Z)$ be a coercive operator on $Z$. We define $Z_{\mathcal{P}}$ as the Hilbert space $Z$ endowed with $\langle \cdot , \cdot \rangle_{\mathcal{P}} := \langle \cdot , \mathcal{P} \cdot \rangle$. Then the operator $A \colon D(A) \subset Z \to Z$ generates a contraction semigroup on $Z$ if and only if 
	\begin{align*}
		A\mathcal{P} \colon D(A\mathcal{P}) = \left\{ x \in Z \mid \mathcal{P}z \in D(A) \right\} \subset Z_{\mathcal{P}} \to Z_{\mathcal{P}}
	\end{align*}
generates a contraction semigroup on $Z_{\mathcal{P}}$. 
\end{lemma}
Moreover, we need two technical lemmas concerning matrix representations. 
\begin{lemma}[Lemma 7.3.1 in \cite{JacobZwart}]
	\label{Interface - Auxiliary Lemma 2}
	For a matrix $W \in \mathbb{R}^{n \times 2n}$, the following are equivalent:
	\begin{enumerate}
		\item[(i)] $\rank(W) = n$ and $W\Sigma W^{\top} \geq 0$, with $\Sigma \in \mathbb{R}^{4 \times 4}$ defined in \eqref{Matrix Sigma}. 
		\item[(ii)] There exist an invertible matrix $S \in \mathbb{R}^{n \times n}$ and a matrix $V \in \mathbb{R}^{n \times n}$ with $V V^{\top} \leq I$ such that
		\begin{align}
			\label{Interface - Auxiliary Lemma 2 - Matrix Representation}
			W = S \begin{bmatrix}
				I + V & I - V
			\end{bmatrix}.
		\end{align}
	\end{enumerate} 
\end{lemma}

\begin{lemma}[Lemma 7.3.2 in \cite{JacobZwart}]
	\label{Interface - Auxiliary Lemma 3}
	Suppose that a matrix $W \in \mathbb{R}^{n \times 2n}$ can be written as in \eqref{Interface - Auxiliary Lemma 2 - Matrix Representation} for some matrices $S, V \in \mathbb{R}^{n \times n}$, where $S$ is invertible. Then it holds that
	 \begin{align*}
	 	\ker (W) = \ran \left( \begin{bmatrix}
	 		I - V \\
	 		-I - V
	 	\end{bmatrix}\right).
	 \end{align*}
\end{lemma}
Now we are able to prove that the port-Hamiltonian operator $A_{\mathcal{Q}_0}$ defined in \eqref{Interface - Simplified PH System - Operator A} generates a contraction semigroup. The idea is based on the proof of Theorem 7.2.4 in \cite{JacobZwart}. 

\begin{theorem}
	\label{Interface - Theorem A Generates a Contraction Semigroup}
	Consider the port-Hamiltonian operator $A_{\mathcal{Q}_0} \colon D(A_{\mathcal{Q}_0}) \subset X \to X$ defined in \eqref{Interface - Simplified PH System - Operator A} with passivity constant $r \geq 0$ and with a full rank matrix $W_B \in \mathbb{R}^{2 \times 4}$ satisfying $W_B \Sigma W_B^{\top} \geq 0$. Then $A_{\mathcal{Q}_0}$ is the infinitesimal generator of a contraction semigroup on the energy space $(X, \|\cdot\|_{\mathcal{Q}_0})$. In particular, the abstract Cauchy problem
	\begin{align}
		\begin{split}
		\dot{x}(t) &= A_{\mathcal{Q}_0}x(t), \hspace{0.5 cm} t > 0, \\
		x(0) &= x_0 \in X,
	\end{split}
\label{Interface - Theorem A Generates a Contraction Semigroup - ACP}
	\end{align}
is well-defined.
\end{theorem}

\begin{proof}
	For the proof, we need to show two properties:
\begin{enumerate}
	\item[(i)] $A_{\mathcal{Q}_0}$ is dissipative (see Definition \ref{Definition Dissipative Operator}).
	\item[(ii)] ran$(I - A_{\mathcal{Q}_0}) = X$. 
\end{enumerate}
Applying Theorem \ref{Theorem Lumer-Phillips} then yields that $A_{\mathcal{Q}_0}$ generates a contraction semigroup. The proof of these properties is split into three steps.
\vspace{0.5 cm}\\
\textit{Step 1.} Recall that the energy space $X$ is endowed with the inner product $\langle \cdot, \cdot \rangle_{\mathcal{Q}_0} = \frac{1}{2} \langle \cdot , \mathcal{Q}_0 \cdot \rangle_{L^2}$. Since $\mathcal{Q}_0 \in \mathcal{L}(X)$ is a coercive operator, by Lemma \ref{Interface - Auxiliary Lemma 1} we may assume for the rest of the proof that $\mathcal{Q}_0 \equiv I$. In particular, $X$ is endowed with the inner product $\frac{1}{2} \langle \cdot , \cdot \rangle_{L^2}$. In the following, we will write $A$ instead of $A_I$. 
\vspace{0.5 cm}\\
\textit{Step 2.} \label{Step 2} We begin by proving that $A$ is dissipative. From \eqref{Interface - Skew-Symmetry of J0 - Representation 2} one can readily see that for all $x \in D(A)$ it holds that
\begin{align*}
	\langle Ax, x \rangle_{L^2} + \langle x, Ax \rangle_{L^2} &=   \langle \mathcal{J}_0x , x \rangle_{L^2} + \langle x, \mathcal{J}_0 x \rangle_{L^2} \\
	&= 2 \langle e_{\partial}, f_{\partial} \rangle_2 - 2 e_If_I.
\end{align*}
We want to estimate the right-hand side. Note that by assumption, for all $x \in D(A)$ we have $\begin{bmatrix}
	f_{\partial} \\
	e_{\partial}
\end{bmatrix} \in \ker(W_B)$. Since $\rank(W_B) = 2$ and since we assume that $W_B \Sigma W_B^{\top} \geq 0$, Lemma~\ref{Interface - Auxiliary Lemma 2} implies that $W_B$ may be written as
\begin{align*}
	W_B = S \begin{bmatrix}
		I+ V & I - V
	\end{bmatrix},
\end{align*}
with $S \in \mathbb{R}^{2 \times 2}$ invertible and $V \in \mathbb{R}^{2 \times 2}$ such that $VV^{\top} \leq I$, or equivalently, $V^{\top}V \leq I$. Thus, by Lemma \ref{Interface - Auxiliary Lemma 3}, for all $x \in D(A)$ there exists some $\lambda_x \in \mathbb{R}^2$ such that
\begin{align}
	\label{Interface - Theorem Contraction Semigroup - Boundary Flow and Effort Representation}
	\begin{bmatrix}
		f_{\partial} \\
		e_{\partial}
	\end{bmatrix} = \begin{bmatrix}
		I-V \\
		-I - V
	\end{bmatrix} \lambda_x.
\end{align}
Using the passivity relation \eqref{Interface - Interface Passivity Relation}  and \eqref{Interface - Theorem Contraction Semigroup - Boundary Flow and Effort Representation}, we compute for all $x \in D(A)$,
\begin{align*}
	 \langle Ax, x \rangle_{L^2}  &=  \langle e_{\partial}, f_{\partial} \rangle_2 -  e_If_I \\
	&=  \lambda_x^{\top} (-I-V)^{\top}(I -V) \lambda_x  - e_I f_I \\
	&= \lambda_x^{\top} (-I + V^{\top}V) \lambda_x - e_I f_I \\
	&\leq - e_I f_I \\
	&= - r e_I^2 \\
	&\leq 0.
\end{align*}
Thus, the operator $A$ is dissipative.
\vspace{0.5 cm}\\
\textit{Step 3.} It is left to show that ran$(I - A) = X$. Let $y = c_0y^- + \overline{c}_0 y^+ \in X$ and consider the equation $(I - A)x = y$, i.e.,
\begin{align}
	\label{Interface - ODE of (I-A)x equals y}
	\begin{bmatrix}
		x_1 \\
		x_2
	\end{bmatrix} - \underbrace{\begin{bmatrix}
		0 & \mathbf{d}_0 \\
		- \mathbf{d}_0^{\ast} & 0 
	\end{bmatrix}}_{=\mathcal{J}_0} \begin{bmatrix}
		x_1 \\
		x_2
	\end{bmatrix} = \begin{bmatrix}
		y_1 \\
		y_2
	\end{bmatrix}.
\end{align}
We need to show that the solution of this differential equation is an element of $D(A)$. As the operator $\mathcal{J}_0$ acts as the differential operator $P_1 \frac{d}{dz}$ on any subdomain not containing the interface position, with $P_1$ given in \eqref{Interface - Matrix P1}, this can be equivalently written as
\begin{align*}
	x^- - P_1 \frac{d}{dz} x^- &= y^- \hspace{0.3 cm} \text{on } [a,0), \\
	x^+ - P_1 \frac{d}{dz} x^+ &= y^+ \hspace{0.3 cm} \text{on } (0,b]
\end{align*}
with suitable initial conditions at $z =a$ and $z = 0^+$, respectively. We will shortly discuss this in more detail. Noticing that $P_1^{-1} = P_1$ and writing
\begin{align*}
	\frac{d}{dz} x^{\pm} = P_1 (x^{\pm} - y^{\pm}),
\end{align*}
the variation of constants formula yields the solutions
\begin{align}
	\varphi_y^-(z ; x(a)) &= e^{P_1(z-a)}x(a) - \int_{a}^{z} e^{P_1(z-s)} P_1 y^-(s) \, ds, \hspace{0.5 cm} z \in [a,0), \label{Interface - Solution of ODE left} \\
	\varphi_y^+(z; x(0^+)) &= e^{P_1z} x(0^+) - \int_{0}^{z} e^{P_1(z-s)} P_1 y^+(s) \, ds, \hspace{0.5 cm } z \in (0, b]. \label{Interface - Solution of ODE right}
\end{align}
In particular, we have
\begin{align*}
	\varphi_y^-(0^-; x(a)) = e^{-P_1a}x(a) - \int_{a}^{0} e^{-P_1 s} P_1 y^-(s) \, ds.
\end{align*}
In the following, we will consider two different instances concerning the passivity constant $r \geq 0$. 
\vspace{0.5 cm}\\
\underline{Case 1: $r > 0$}

Due to the continuity condition \eqref{Interface - Simplified Continuity Equation}, we have to choose \begin{align}
	x_2(0) = f_I = \left( \varphi_y^-(0^-; x(a))\right)_2.
	\label{Interface - ODE Interface Continuity Equation}
\end{align}
Furthermore, the balance equation \eqref{Interface - Simplified Balance Equation} is given by
\begin{align}
	\label{Interface - ODE Interface Balance Equation}
	e_I = - \left( x_1(0^+) - \left( \varphi_y^-(0^-; x(a)) \right)_1 \right).
\end{align}
Due to the  passivity relation \eqref{Interface - Interface Passivity Relation}, we have to choose $x_1(0^+)$ such that
\begin{align*}
	f_I = r e_I.
\end{align*}
Thus, the initial condition $x(0^+) =(x_1(0^+), x_2(0))$ of $\varphi_y^+(\cdot; x(0^+))$ is uniquely determined through the interface relations and conditions, and depends on the choice of $x(a)$. The solution of \eqref{Interface - ODE of (I-A)x equals y} is therefore given by
\begin{align}
	\label{Interface - Solution of ODE complete}
	\varphi_y(z; x(a), x(0^+)) = c_0(z) \varphi_y^-(z; x(a)) + \overline{c}_0(z) \varphi_y^+(z; x(0^+)), \hspace{0.5 cm} z \in [a,b].
\end{align}
It is left to show that $\varphi_y(\cdot; x(a),x(0^+))$ is indeed an element of $D(A)$ for some initial value $x(a) \in \mathbb{R}^2$. To that end, we need to verify the last constitutive relation of $D(A)$ in \eqref{Interface - Simplified PH System - Operator A}. By using \eqref{Interface - Boundary Flow and Effort} and \eqref{Interface - Boundary Condition wrt Boundary Flow and Effort}, the solution $\varphi_y(\cdot; x(a), x(0^+))$ is an element of $D(A)$ if and only if
\begin{align}
	\label{Interface - D(A) condition - representation 1}
	0 = W_B \begin{bmatrix}
		f_{\partial} \\
		e_{\partial}
	\end{bmatrix} = W_B R_{\extern} \begin{bmatrix}
		\varphi_y(b; x(a), x(0^+)) \\
		\varphi_y(a; x(a), x(0^+))
	\end{bmatrix}. 
\end{align} 
Obviously, $\varphi(a;x(a), x(0^+)) = x(a)$. We wish to write $\varphi(b; x(a), x(0^+))$ as $E x(a) + q$ for some $E \in \mathbb{R}^{2 \times 2}$, $q \in \mathbb{R}^2$. Note that
\begin{align*}
	\varphi_y(b; x(a), x(0^+)) = \varphi_y^+(b; x(0^+)) = e^{P_1 b}x(0^+) \underbrace{- \int_{0}^{b} e^{P_1(b-s)} P_1 y^+(s) \, ds}_{=: q_1}.
\end{align*}
Using the interface conditions \eqref{Interface - ODE Interface Continuity Equation}-\eqref{Interface - ODE Interface Balance Equation}, we get
\begin{align*}
	f_I &= r e_I \\
	\Leftrightarrow \hspace{0.3 cm} \left(\varphi_y^-(0^-; x(a)) \right)_2 &= r \left[ - \left( x_1(0^+) - \left( \varphi_y^-(0^+; x(a)) \right)_1 \right) \right] \\
	\Leftrightarrow \hspace{1.955 cm} x_1(0^+) &= - \frac{1}{r} \left( \varphi_y^-(0^-; x(a)) \right)_2 + \left( \varphi_y^-(0^-; x(a)) \right)_1.
\end{align*}
Hence, we may write
\begin{align*}
	&\hspace{0.5 cm} e^{P_1b}x(0^+) \\
	&= e^{P_1b}\begin{bmatrix}
		- \frac{1}{r} \left( e^{-P_1a} x(a) - \int_{a}^{0} e^{-P_1 s} P_1 y^-(s) \, ds \right)_2 + \left( e^{-P_1a}x(a) - \int_{a}^{0} e^{-P_1s}P_1 y^-(s) \, ds \right)_1 \\
		\left(  e^{-P_1a} x(a) - \int_{a}^{0} e^{-P_1 s} P_1 y^-(s) \, ds \right)_2  
	\end{bmatrix} \\
	&= \underbrace{ e^{P_1b} \begin{bmatrix}
			- \frac{1}{r}  \left( e^{-P_1a}x(a) \right)_2 + \left( e^{-P_1a} x(a) \right)_1  \\
			\left( e^{-P_1a}x(a) \right)_2
	\end{bmatrix}}_{=: Ex(a)} \\
	&\hspace{0.5 cm}  \underbrace{+ e^{P_1b} \begin{bmatrix}
			- \frac{1}{r} \left( - \int_{a}^{0} e^{-P_1s} P_1 y^-(s) \, ds \right)_2 + \left( - \int_{a}^{0} e^{-P_1s}P_1 y^-(s) \, ds \right)_1  \\
			- \left( \int_{a}^{0} e^{-P_1s}P_1 y^-(s) \, ds \right)_2
	\end{bmatrix}}_{=: q_2}. 
\end{align*}
Letting $q := q_1 + q_2$, we finally have
\begin{align*}
	\varphi_y(b; x(a),x(0^+)) = Ex(a) + q, 
\end{align*}
with only $q$ depending on $y \in X$. Now we can come back to the boundary condition \eqref{Interface - D(A) condition - representation 1}. Equation \eqref{Interface - D(A) condition - representation 1} is equivalent to 
\begin{align}
	\label{Interface - D(A) condition - representation 2}
	0 = W_B R_{\extern} \begin{bmatrix}
		E x(a) + q \\
		x(a) 
	\end{bmatrix} \hspace{0.3 cm} \Leftrightarrow \hspace{0.3 cm} W_B R_{\extern} \begin{bmatrix}
		E \\
		I
	\end{bmatrix} x(a) = - W_B R_{\extern} \begin{bmatrix}
		q \\
		0
	\end{bmatrix}.
\end{align}
Therefore, we want to show that the matrix $W_B R_{\extern} \begin{bmatrix}
	E \\
	I
\end{bmatrix} \in \mathbb{R}^{2 \times 2}$ is invertible, since this implies that there exists a unique initial value $x(a) \in \mathbb{R}^2$ such that $\varphi_y(\cdot; x(a), x(0^+)) \in D(A)$. As $y \in X$ was chosen arbitrarily, ran$(I - A) = X$, and we may apply Theorem~\ref{Theorem Lumer-Phillips}.  
\vspace{0.5cm}\\
Assume now that there exists $0 \neq r_0 \in \ker\left( W_B R_{\extern} \begin{bmatrix}
	E \\
	I
\end{bmatrix} \right)$, and let $y = 0$. Consider the solution of equation \eqref{Interface - ODE of (I-A)x equals y} given by
\begin{align*}
	\varphi_0(z; r_0, x(0^+)) = c_0(z) e^{P_1(z-a)}r_0 + \overline{c}_0(z) e^{P_1z} \begin{bmatrix}
		x_1(0^+) \\
		\left( \varphi^-(0^-; r_0) \right)_2
	\end{bmatrix},
\end{align*}
with $x_1(0^+) \in \mathbb{R}$ accordingly chosen. Note that $\varphi_0(\cdot; r_0, x(0^+)) \in D(\mathcal{J}_0)$, and, since $y = 0$, that $q = 0$. Consequently, $\varphi_0(\cdot; r_0, x(0^+)) \in D(A)$, since by assumption and \eqref{Interface - D(A) condition - representation 2}, we have
\begin{align*}
	W_B R_{\extern} \begin{bmatrix}
		E \\
		I
	\end{bmatrix} r_0 = 0. 
\end{align*}
As $\varphi_0(\cdot; r_0, x(0^+))$ is a non-zero solution of $(I - A)x = 0$, it is, in particular, an eigenfunction of $A$ with eigenvalue $\lambda = 1$. This contradicts the fact that $A$ is dissipative, and therefore $W_B R_{\extern} \begin{bmatrix}
	E \\
	I
\end{bmatrix}$ is in fact invertible.
\vspace{0.5 cm}\\
\underline{Case 2: $r = 0$}

The passivity relation \eqref{Interface - Interface Passivity Relation} together with the continuity condition \eqref{Interface - Simplified Continuity Equation} yield that we have to choose
\begin{align*}
	0 = f_I = x_2(0) = \left( \varphi_y^-(0^-; x(a)) \right)_2. 
\end{align*}
Note that for the matrix exponential $e^{P_1 \cdot}$, where $P_1$ is given in \eqref{Interface - Matrix P1}, we have
\begin{align*}
	e^{P_1z} = \frac{1}{2 e^z} \begin{bmatrix}
		e^{2z} +1 & -e^{2z} +1 \\
		-e^{2z} +1 & e^{2z} +1
	\end{bmatrix}, \hspace{0.5 cm} z \in [a,b]. 
\end{align*}
Hence, it holds that
\begin{align*}
	0 &= \left( \varphi_y^-(0^-; x(a)) \right)_2 \\
	&=  \frac{-e^{-2a} +1}{2e^{-a}} x_1(a) + \frac{e^{-2a} +1}{2e^{-a}} x_2(a) - \left( \int_{a}^{0} e^{-P_1s} P_1 y^-(s) \, ds \right)_2  \\
	\Leftrightarrow \hspace{0.5 cm} x_2(a) &= \frac{2 e^{-a}}{e^{-2a} +1} \left( \int_{a}^{0} e^{-P_1 s} P_1 y^-(s) \, ds \right)_2 - \frac{- e^{-2a}+1}{e^{-2a} +1} x_1(a). 
\end{align*}
In conclusion, the interface conditions specify the values $x_2(0)$ and $x_2(a)$. Furthermore, $\varphi_y$ is an element of $D(\mathcal{J}_0)$ and it satisfies the interface relation. Therefore, it is left to find the values of $x_1(a)$ and $x_1(0^+)$ such that the solution $\varphi_y$ satisfies the last constitutive relation of $D(A)$, namely
\begin{align}
	\label{Interface - Theorem A Generates a Unitary Semigroup - Boundary Condition in Proof}
	W_B R_{\extern} \begin{bmatrix}
		\varphi_y^+(b; x(0^+)) \\
		\varphi_y^-(a; x(a)) 
	\end{bmatrix} = 0,  
\end{align}
where $R_{\extern}$ is defined in \eqref{Interface - Matrix Rext}. We have $\varphi_y^-(a;x(a)) = x(a)$ and
\begin{align*}
	\varphi_y^+(b;x(0^+)) = e^{P_1b}x(0^+) - \int_{0}^{b} e^{P_1(b-s)} P_1 y^+(s) \, ds. 
\end{align*}
Using the definition of $R_{\extern}$, $e^{P_1 \cdot}$, and the values $x_2(a)$ and $x_2(0)$, we compute
\begin{align*}
	&\hspace{0.5 cm} \sqrt{2} R_{\extern} \begin{bmatrix}
		\varphi_y^+(b; x(0^+)) \\
		\varphi_y^-(a; x(a)) 
	\end{bmatrix}  \\
	&= \begin{bmatrix}
		0 & -1 & 0 & 1 \\
		-1 & 0 & 1 & 0 \\
		1 & 0 & 1 & 0 \\
		0 & 1 & 0 & 1
	\end{bmatrix} \begin{bmatrix}
		e^{P_1b}x(0^+) - \int_{0}^{b} e^{P_1(b-s)} P_1 y^+(s) \, ds \\
		x(a)
	\end{bmatrix} \\
	&= \begin{bmatrix}
		- \frac{- e^{2b} + 1}{2e^{b}} x_1(0^+) - \frac{-e^{-2a} +1}{e^{-2a} +1} x_1(a) + \frac{2e^{-a}}{e^{-2a} +1} \left( \int_{a}^{0} e^{-P_1s}P_1 y^-(s) \, ds \right)_2 + \left( \int_{0}^{b} e^{P_1(b-s)} P_1 y^+(s) \, ds \right)_2 \\
		- \frac{ e^{2b} + 1}{2e^{b}} x_1(0^+) + x_1(a) + \left( \int_{0}^{b} e^{P_1(b-s)} P_1 y^+(s) \, ds \right)_1 \\
		\frac{e^{2b} +1}{2e^{b}} x_1(0^+) + x_1(a) - \left( \int_{0}^{b} e^{P_1(b-s)} P_1 y^+(s) \, ds \right)_1 \\
		\frac{- e^{2b} +1}{2e^{b}} x_1(0^+) - \frac{-e^{-2a} +1}{e^{-2a} +1} x_1(a) + \frac{2e^{-a}}{e^{-2a} +1} \left( \int_{a}^{0} e^{-P_1s}P_1 y^-(s) \, ds \right)_2   - \left( \int_{0}^{b} e^{P_1(b-s)} P_1 y^+(s) \, ds \right)_2
	\end{bmatrix}.
\end{align*}
Now, let the matrix $M \in \mathbb{R}^{2 \times 2}$ be defined such that
\begin{align*}
	 W_B  \begin{bmatrix}
		- \frac{- e^{2b} + 1}{2e^{b}} x_1(0^+) - \frac{-e^{-2a} +1}{e^{-2a} +1} x_1(a)  \\
		- \frac{ e^{2b} + 1}{2e^{b}} x_1(0^+) + x_1(a)  \\
		\frac{e^{2b} +1}{2e^{b}} x_1(0^+) + x_1(a)  \\
		\frac{- e^{2b} +1}{2e^{b}} x_1(0^+) - \frac{-e^{-2a} +1}{e^{-2a} +1} x_1(a) 
	\end{bmatrix} = M \begin{bmatrix}
	x_1(a) \\
	x_1(0^+)
\end{bmatrix}.
\end{align*}
Furthermore, define $N \in \mathbb{R}^2$ by
\begin{align*}
	N = - W_B \begin{bmatrix}
		\frac{2e^{-a}}{e^{-2a} +1} \left( \int_{a}^{0} e^{-P_1s}P_1 y^-(s) \, ds \right)_2 + \left( \int_{0}^{b} e^{P_1(b-s)} P_1 y^+(s) \, ds \right)_2  \\
		\left( \int_{0}^{b} e^{P_1(b-s)} P_1 y^+(s) \, ds \right)_1  \\
		- \left( \int_{0}^{b} e^{P_1(b-s)} P_1 y^+(s) \, ds \right)_1 \\
		\frac{2e^{-a}}{e^{-2a} +1} \left( \int_{a}^{0} e^{-P_1s}P_1 y^-(s) \, ds \right)_2   - \left( \int_{0}^{b} e^{P_1(b-s)} P_1 y^+(s) \, ds \right)_2
	\end{bmatrix}.
\end{align*}
Then the boundary condition \eqref{Interface - Theorem A Generates a Unitary Semigroup - Boundary Condition in Proof} is satisfied if and only if
\begin{align}
	M \begin{bmatrix}
		x_1(a) \\
		x_1(0^+)
	\end{bmatrix} = N. 
	\label{Interface - Theorem A Generates a Unitary Semigroup - Linear System of Equations in Proof}
\end{align}
We shall verify that the matrix $M$ is invertible. If so, then for every $y \in X$ the initial values are uniquely determined and can be calculated through \eqref{Interface - Theorem A Generates a Unitary Semigroup - Linear System of Equations in Proof}. Moreover, the corresponding solution $\varphi_y(\cdot; x(a), x(0^+))$ of $(I-A)x = y$ is an element of $D(A)$, and so $\ran(I - A) = X$. The Lumer-Phillips Theorem \ref{Theorem Lumer-Phillips} then yields the claim. 
\vspace{0.5 cm}\\
As in the first case, assume that there is some $0 \neq r_0 \in \ker(M)$, and let $y = 0$. Consider the solution of equation \eqref{Interface - ODE of (I-A)x equals y} given by
\begin{align*}
	\varphi_0(z; x(a), x(0^+)) = c_0(z) e^{P_1(z-a)}x(a) + \overline{c}_0(z) e^{P_1z} \begin{bmatrix}
		x_1(0^+) \\
		0
	\end{bmatrix}, 
\end{align*}
with $r_0 = \begin{bmatrix}
	x_1(a) \\
	x_1(0^+)
\end{bmatrix}$, and $x_2(a) = - \frac{- e^{-2a}+1}{e^{-2a}+1} x_1(a)$. Then $\varphi_0(\cdot; x(a), x(0^+)) \in D(\mathcal{J}_0)$, and with this choice of the initial values, the interface relation \eqref{Interface - Interface Passivity Relation} is satisfied as well. Note that $N = 0$, since $y = 0$. Furthermore, \eqref{Interface - Theorem A Generates a Unitary Semigroup - Linear System of Equations in Proof} is satisfied, since by assumption we have
\begin{align*}
	M r_0 = 0.
\end{align*}
As explained in the first case, this contradicts the dissipativity of $A$, which is why the matrix $M$ is indeed invertible.
\vspace{0.5 cm}\\ 
In conclusion, the operator $A_{\mathcal{Q}_0}$ defined in \eqref{Interface - Simplified PH System - Operator A} is the infinitesimal generator of a contraction semigroup on the energy space $(X, \| \cdot \|_{\mathcal{Q}_0})$. The fact that for all $x_0 \in D(A_{\mathcal{Q}_0})$, the abstract Cauchy problem \eqref{Interface - Theorem A Generates a Contraction Semigroup - ACP} has a unique classical solution $x \colon [0, \infty) \to X$, depending continuously on the initial data, follows immediately from Theorem \ref{Theorem Characterization of Well-Posedness}.
\end{proof}
Theorem \ref{Interface - Theorem A Generates a Contraction Semigroup} shows that, under certain boundary and interface conditions, the port-Hamil- tonian model of two systems of conservation laws coupled by a rigid interface is well-defined.  If we compare the system \eqref{Interface Paper - Augmented System wrt Ja} defined on the extended state space $\tilde{X}$ in Subsection \ref{Subsection Port-Hamiltonian Formulation and Dirac Structure} with the simplified system \eqref{Interface - Simplified PH-System} treated in this section, we notice that no valuable information gets lost if we take the color functions $c_0$ and $\overline{c}_0$ defined in \eqref{Interface Paper - Characteristic Functions c0 and c0overline} not to be part of the state variables. For the moving interface scenario, this is not true, as the corresponding color functions \eqref{Interface Paper - Color Functions Depending on Moving Interface} do not satisfy a trivial conservation law, but the transport equation in a distributional sense, see Subsection \ref{Subsection Balance Equations of the State Variables - Moving Interface}. We will come back to the moving interface scenario in Subsection \ref{Subsection Stability of the Family of Infinitesimal Generators}.
\vspace{0.5 cm}\\
Based on Theorem \ref{Interface - Theorem A Generates a Contraction Semigroup}, we may find conditions which ensure that the internally stored energy of the system is conserved. 
\begin{corollary}
	\label{Interface - Corollary A Generates a Unitary Semigroup}
		Consider the port-Hamiltonian operator $A_{\mathcal{Q}_0} \colon D(A_{\mathcal{Q}_0}) \subset X \to X$ defined in \eqref{Interface - Simplified PH System - Operator A} with passivity constant $r=0$ and with a full rank matrix $W_B \in \mathbb{R}^{2 \times 4}$ satisfying $W_B \Sigma W_B^{\top} = 0$. Then $A_{\mathcal{Q}_0}$ is the infinitesimal generator of a unitary semigroup on the energy space $(X, \|\cdot\|_{\mathcal{Q}_0})$.
\end{corollary}

\begin{proof}
	Theorem \ref{Interface - Theorem A Generates a Contraction Semigroup} yields that $A_{\mathcal{Q}_0}$ generates a contraction semigroup $(T(t))_{t \geq 0}$ on the energy space $(X,  \|\cdot\|_{\mathcal{Q}_0})$. So, we only need to show that $(T(t))_{t \geq 0}$ is a unitary semigroup.
	\vspace{0.5 cm}\\
	By Lemma \ref{Interface - Auxiliary Lemma 2}, the matrix $W_B$ is of the form \eqref{Interface - Auxiliary Lemma 2 - Matrix Representation}, with an invertible matrix $S \in \mathbb{R}^{2 \times 2}$ and with $V \in \mathbb{R}^{2 \times 2}$ satisfying $VV^{\top} \leq I$. Since we have
		\begin{align*}
		0 = 	W_B \Sigma W_B^{\top} = S \begin{bmatrix}
				I + V & I - V
			\end{bmatrix} \begin{bmatrix}
			0 & I \\
			I & 0 
		\end{bmatrix} \begin{bmatrix}
		I + V^{\top} \\
		I- V^{\top}
	\end{bmatrix} S^{\top} = 2S (I- VV^{\top}) S^{\top}, 
		\end{align*}
	we deduce that $VV^{\top} = V^{\top}V = I$. Thus, by repeating Step 2 of the proof of Theorem \ref{Interface - Theorem A Generates a Contraction Semigroup}, we get for all $x \in D(A_{\mathcal{Q}_0})$, 
	\begin{align}
		2 \langle A_{\mathcal{Q}_0}x, x \rangle_{\mathcal{Q}_0} = \langle e_{\partial}, f_{\partial} \rangle_2 - e_I f_I = \lambda_x^{\top} (-I + V^{\top}V) \lambda_x = 0. 
		\label{Interface - Theorem A Generates a Unitary Semigroup - Dissipativity Property in Proof}
	\end{align}
Now, for $x_0 \in D(A_{\mathcal{Q}_0})$, consider the mapping $t \mapsto \| T(t)x_0\|_{\mathcal{Q}_0}^2$, $t \geq 0$. This mapping describes the energy of the solution of the abstract Cauchy problem \eqref{Interface - Theorem A Generates a Contraction Semigroup - ACP}. Recall from Theorem~\ref{Theorem Properties of C0-Semigroups}~\ref{C0SemigroupProp3} that $T(t)x_0 \in D(A_{\mathcal{Q}_0})$ for all $t\geq 0$ and for all $x_0 \in  D(A_{\mathcal{Q}_0})$. By taking the derivative and by exploiting the dissipativity \eqref{Interface - Theorem A Generates a Unitary Semigroup - Dissipativity Property in Proof} of $A_{\mathcal{Q}_0}$, we obtain for all $t > 0$,
\begin{align*}
	\frac{d}{dt} \| T(t)x_0\|_{\mathcal{Q}_0}^2 &=  \frac{d}{dt} \langle T(t)x_0 , T(t)x_0 \rangle_{\mathcal{Q}_0} \\
	&= \langle A_{\mathcal{Q}_0}T(t)x_0, T(t)x_0 \rangle_{\mathcal{Q}_0} + \langle T(t)x_0, A_{\mathcal{Q}_0}T(t)x_0 \rangle_{\mathcal{Q}_0}  \\
	&=  2 \langle A_{\mathcal{Q}_0}T(t)x_0, T(t)x_0 \rangle_{\mathcal{Q}_0} \\
	&= 0. 
\end{align*}
Thus, the mapping $t \mapsto \| T(t)x_0\|_{\mathcal{Q}_0}^2$ is constant for all $x_0 \in D(A_{\mathcal{Q}_0})$. Since $\|T(0) x_0 \|_{\mathcal{Q}_0} = \| x_0\|_{\mathcal{Q}_0}$ and since $A_{\mathcal{Q}_0}$ is densely defined, we have 
\begin{align*}
	\|T(t) x\|_{\mathcal{Q}_0} = \|x\|_{\mathcal{Q}_0}, \hspace{0.5 cm} t \geq 0, \, x \in X. 
\end{align*}
Therefore, $(T(t))_{t \geq 0}$ is a unitary semigroup. This completes the proof. 
\end{proof}

To finish this section, we model two transmission lines that are coupled through a resistor.  
\begin{example}[Coupled Transmission Lines]
	\label{Interface - Example Coupled Transmission Lines}
	Consider two transmission lines defined on the spatial intervals $[a,0)$ and $(0,b]$, respectively. We assume that the distributed capacity and the distributed inductance of the transmission lines are different. Recall from Subsection~\ref{Subsection The Lossless Transmission Line} that the conservation laws are given by
	\begin{align*}
		\frac{\partial}{\partial t} \begin{bmatrix}
			Q^- \\
			\varphi^-
		\end{bmatrix} &= \begin{bmatrix}
		0 & -\frac{\partial}{\partial z} \\
		- \frac{\partial}{\partial z} & 0 
	\end{bmatrix} \begin{bmatrix}
	\frac{Q^-}{C^-} \\
	\frac{\varphi^-}{L^-}
\end{bmatrix}, \hspace{0.5 cm} z \in [a,0), \, t > 0, \\
	\frac{\partial}{\partial t} \begin{bmatrix}
	Q^+ \\
	\varphi^+
\end{bmatrix} &= \begin{bmatrix}
	0 & -\frac{\partial}{\partial z} \\
	- \frac{\partial}{\partial z} & 0 
\end{bmatrix} \begin{bmatrix}
	\frac{Q^+}{C^+} \\
	\frac{\varphi^+}{L^+}
\end{bmatrix}, \hspace{0.5 cm} z \in (0,b], \, t > 0.
		\end{align*}
	Letting $x^{\pm} = (Q^{\pm}, \varphi^{\pm})$, the Hamiltonian functionals are given by 
	\begin{align*}
		H^-(x^-) &= \int_{a}^{0} \mathcal{H}^-(x^-(z)) \, dz = \frac{1}{2} \langle x^-, \mathcal{Q}^- x^- \rangle_{L^2(a,0)}, \hspace{0.5 cm} x^- \in X^- = L^2((a,0), \mathbb{R}^2), \\
			H^+(x^+) &= \int_{a}^{0} \mathcal{H}^+(x^+(z)) \, dz = \frac{1}{2} \langle x^+, \mathcal{Q}^+ x^+ \rangle_{L^2(0,b)}, \hspace{0.5 cm} x^+ \in X^+ = L^2((0,b), \mathbb{R}^2), 
	\end{align*}
where $\mathcal{Q}^{\pm} \in \mathcal{L}(X^{\pm})$ are the coercive multiplication operators given by
\begin{align*}
	(\mathcal{Q}^{\pm}x^{\pm})(\cdot) = \begin{bmatrix}
		\frac{1}{C^{\pm}(\cdot)} & 0 \\
		0 & \frac{1}{L^{\pm}(\cdot)} \end{bmatrix} \begin{bmatrix}
			Q^{\pm}(\cdot) \\
			\varphi^{\pm}(\cdot)
		\end{bmatrix}, \hspace{0.5 cm} x ^{\pm} \in X^{\pm},
\end{align*}
respectively. Recall that the co-energy variables are given by the
voltages
\begin{align*}
	V^{\pm} = \frac{Q^{\pm}}{C^{\pm}} = \frac{\delta H^{\pm}}{\delta Q^{\pm}} = \left( \mathcal{Q}^{\pm} x^{\pm} \right)_1
\end{align*} 
and the currents 
\begin{align*}
I^{\pm} = \frac{\varphi^{\pm}}{L^{\pm}} = \frac{\delta H^{\pm}}{\delta Q^{\pm}} = \left( \mathcal{Q}^{\pm} x^{\pm} \right)_2.
\end{align*}
We want to couple these two transmission lines at $z = 0$, and assume that a resistor with resistance $R_I >0$ is placed at the interface. Furthermore, we set the voltage at $z = a$ to zero, and we put another resistor with resistance $R_b > 0$ at $z = b$. We want to show by means of Theorem \ref{Interface - Theorem A Generates a Contraction Semigroup} that the coupled system is well-posed. 
\vspace{0.5 cm}\\
The model of the coupling is precisely given by the system \eqref{Interface - Simplified PH-System} defined on the state space $X = L^2([a,b], \mathbb{R}^2)$, with $[a,b]$ the composed spatial domain. Defining the coercive matrix operator $\mathcal{Q}_0 = c_0 \mathcal{Q}^- + \overline{c}_0 \mathcal{Q}^+ \in \mathcal{L}(X)$, we may endow $X$ with the inner product $\langle \cdot , \cdot \rangle_{\mathcal{Q}_0}$ defined in \eqref{Interface - Inner Product wrt Q0}.
\vspace{0.5 cm}\\
Due to the resistor placed at the interface, we expect a voltage drop, and therefore have a discontinuity of the effort variable $V = c_0 V^- + \overline{c}_0 V^+$, whereas the current $I = c_0 I^- + \overline{c}_0 I^+$ flows continuously through the resistor. Thus, the interface port variables are given by
\begin{align*}
	f_{I, \mathcal{Q}_0 x} &= I(0^+) = I(0^-), \\
	-e_{I, \mathcal{Q}_0 x} &= V(0^+) - V(0^-), 
\end{align*}
and, by virtue of Ohm's law, the interface relation is given by the passivity relation
\begin{align*}
	f_I(t) = \frac{1}{R_I} e_I(t), \hspace{0.5 cm} t \geq 0.
\end{align*}
Furthermore, we have the boundary conditions
\begin{align*}
V(a,t) = 0 \hspace{0.3 cm} \text{and} \hspace{0.3 cm} V(b,t) = R_b I(b,t), \hspace{0.5 cm} t \geq 0.
\end{align*}
Neglecting the time-dependence and using \eqref{Interface - Boundary Flow and Effort}, we may equivalently write
\begin{align*}
	0 = \begin{bmatrix}
		0 & 0 & 1 & 0 \\
		-1 & R_b & 0 & 0 
	\end{bmatrix} \begin{bmatrix}
	V(b) \\
	I(b) \\
	V(a) \\
	I(a)
\end{bmatrix} = \tilde{W}_B \begin{bmatrix}
(\mathcal{Q}_0x)(b) \\
(\mathcal{Q}_0x)(a)
\end{bmatrix} = W_B \begin{bmatrix}
f_{\partial, \mathcal{Q}_0 x} \\
e_{\partial, \mathcal{Q}_0 x}
\end{bmatrix},
\end{align*}
with $W_B = \tilde{W}_B R_{\extern}^{-1} = \frac{1}{\sqrt{2}} \begin{bmatrix}
	0 & 1 & 1 & 0 \\
	-R_b & 1 & -1 & R_b
\end{bmatrix} \in \mathbb{R}^{2 \times 4}$. Note that $\rank(W_B) = 2$. Furthermore, one computes that
\begin{align*}
	W_B \Sigma W_B^{\top} = \begin{bmatrix}
		0 & 0 \\
		0 & 2R_b
	\end{bmatrix} \geq 0.
\end{align*}
Summing up, the constituting relations of the domain of the port-Hamiltonian operator $A_{\mathcal{Q}_0}$ associated with the system of transmission lines coupled through a resistor satisfy the assumptions of Theorem \ref{Interface - Theorem A Generates a Contraction Semigroup}. Hence, $A_{\mathcal{Q}_0}$ generates a contraction semigroup on the energy space $(X, \|\cdot\|_{\mathcal{Q}_0})$, and therefore, this system is well-defined. Its dynamics are geometrically specified by the requirement that for all $t> 0$ it holds that
	\begin{align*}
	\left( \begin{pmatrix}
		\frac{\partial}{\partial t} x(\cdot,t) \\
		f_I(t) \\
		f_{\partial}(t)
	\end{pmatrix} ,  \begin{pmatrix}
		\mathcal{Q}_0x(\cdot,t) \\
		e_I(t) \\
		e_{\partial}(t)
	\end{pmatrix} \right) \in \mathcal{D}_{\mathcal{J}_0}, \hspace{0.3 cm} (f_I(t), e_I(t)) \in \mathcal{R}_{\frac{1}{R_I}},
\end{align*}
where the resistive structure $\mathcal{R}_{\frac{1}{R_I}}$ is given by
\begin{align*}
	\mathcal{R}_{\frac{1}{R_I}} = \left\{ (f_I, e_I) \in \mathbb{R}^2 \mid f_I = \frac{1}{R_I}e_I \right\}.
\end{align*}
\QEDA
\end{example}

\subsection{Exponential Stability}
\label{Subsection Exponential Stability}
In this section, we want to find conditions that guarantee the exponential stability (cf. Definition \ref{Definition Exponential Stability}) of the abstract Cauchy problem \eqref{Interface - Theorem A Generates a Contraction Semigroup - ACP} associated with the port-Hamiltonian operator $A_{\mathcal{Q}_0}$ defined in \eqref{Interface - Simplified PH System - Operator A}. Throughout this section, we impose the following assumptions:
\begin{enumerate}[label = (A\arabic*)]
	\item \label{MatrixQ0Assumption1} The matrix operators $\mathcal{Q}^{\pm}$ defining the coercive operator $\mathcal{Q}_0 \in \mathcal{L}(X)$ satisfy $\mathcal{Q}^{\pm} \in \mathcal{C}^1([a,b], \mathbb{R}^{2 \times 2})$. 
	\item \label{ContractionAssumption} $r \geq 0$, $\rank(W_B) = 2$ and $W_B \Sigma W_B^{\top} \geq 0$, with $\Sigma$ defined in \eqref{Matrix Sigma}. In other words, the port-Hamiltonian operator $A_{\mathcal{Q}_0}$ generates a contraction semigroup (see Theorem~\ref{Interface - Theorem A Generates a Contraction Semigroup}).
\end{enumerate}

We begin by showing an essential property in order to verify the exponential stability of the port-Hamiltonian system \eqref{Interface - Theorem A Generates a Contraction Semigroup - ACP}. The following lemma is a simple extension of \cite[Lemma~9.1.2]{JacobZwart}.
\begin{lemma}
	\label{Interface - Lemma Auxiliary Lemma for Exponential Stability}
	Let the assumptions \ref{MatrixQ0Assumption1}-\ref{ContractionAssumption} hold. Denote by $(T(t))_{t \geq 0}$ the contraction semigroup generated by $A_{\mathcal{Q}_0}$. Then there exist constants $\tau >0$ and $C >0$ such that for every initial value $x_0 \in D(A_{\mathcal{Q}_0})$, the trajectory $x(\cdot) = T(\cdot)x_0$ satisfies
	\begin{align}
		\label{Interface - Lemma Auxiliary Lemma for Exponential Stablity - Inequality}
		\| x(\tau) \|_{\mathcal{Q}_0}^2 &\leq C \int_{0}^{\tau} \| \mathcal{Q}^-(a)x(a,t) \|_2^2 + \| \mathcal{Q}^+(b) x(b,t) \|_2^2 \, dt.
	\end{align} 
\end{lemma}

\begin{proof}
	Since this is clear for $x_0 = 0$, assume that $0 \neq x_0 \in D(A_{\mathcal{Q}_0})$. Define the function $F_1 \colon (0,b] \to (0, \infty)$ by
	\begin{align}
		\label{Interface - Lemma Auxiliary Lemma for Exponential Stability - Function F1}
		F_1(z) := \int_{\gamma_1(b-z)}^{\tau_1 - \gamma_1(b-z)} x^{\top}(z,t) \mathcal{Q}^+(z) x(z,t) \, dt, \hspace{0.5 cm} z \in (0,b], 
	\end{align}
with $\gamma_1 > 0$ and $\tau_1 > 0$ chosen such that $\tau_1 > 2 \gamma_1 b$. Recall that on domains not containing the interface position, the operator $\mathcal{J}_0$ defined in \eqref{Interface - Operator J0} acts as the differential operator $P_1 \frac{d}{dz}$, where $P_1$ is given in \eqref{Interface - Matrix P1}. Following the proof of Lemma 9.1.2 in \cite{JacobZwart}, by exploiting the fact that $x$ satisfies \eqref{Interface - Simplified PH-System} and that the matrices $\mathcal{Q}^+(z)$ and $P_1$ are symmetric, we find
\begin{align*}
	\frac{d}{dz} F_1(z) &= - \int_{\gamma_1(b-z)}^{\tau_1 - \gamma_1(b-z)} x^{\top}(z,t) \frac{d}{dz} \mathcal{Q}^+(z) x(z,t) \, dt \\
	&\hspace{0.5 cm} + x^{\top}(z,\tau_1 - \gamma_1(b-z)) \left( P_1^{-1} + \gamma_1 \mathcal{Q}^+(z) \right) x(z, \tau_1 - \gamma_1(b-z)) \\
	&\hspace{0.5 cm} + x^{\top}(z, \gamma_1(b-z)) \left( -P_1^{-1} + \gamma_1 \mathcal{Q}^+(z) \right) x(z,\gamma_1(b-z)). 
\end{align*}
Since $\frac{d}{dz} \mathcal{Q}^+$ is bounded, we may find a constant $\kappa_1 >0 $ such that
\begin{align*}
	 \frac{d}{dz} \mathcal{Q}^+(z)  \leq  \kappa_1 \mathcal{Q}^+(z), \hspace{0.5 cm} z \in (0,b]. 
\end{align*}
Furthermore, we may choose $\gamma_1$ large enough such that $ P_1^{-1} + \gamma_1 \mathcal{Q}^+(z)$ and $ -P_1^{-1} + \gamma_1 \mathcal{Q}^+(z)$ are positive definite. The constant $\tau_1$ has to be changed accordingly. This yields
\begin{align*}
	\frac{d}{dz} F_1(z) \geq - \kappa_1 \int_{\gamma_1(b-z)}^{\tau_1 - \gamma_1(b-z)} x^{\top}(z,t) \mathcal{Q}^+(z) x(z,t) \, dt \stackrel{\eqref{Interface - Lemma Auxiliary Lemma for Exponential Stability - Function F1} }{=} - \kappa_1 F_1(z). 
\end{align*}
This implies that for all $z_1 \in (0,b]$ it holds that
\begin{align*}
	\int_{z_1}^{b} \frac{\frac{d}{dz}F_1(z)}{F_1(z)} \, dz \geq - \kappa_1 \int_{z_1}^{b} 1 \, dz.
\end{align*}
Simple computations yield that
\begin{align}
	\label{Interface - Lemma Auxiliary Lemma for Exponential Stability - Estimate F1}
	F_1(b)  \geq F_1(z) e^{-\kappa_1(b-z)} \geq  F_1(z) e^{-\kappa_1 b}, \hspace{0.5 cm} z\in (0,b]. 
\end{align}
Next, we define $F_2 \colon [a,0) \to (0,\infty)$ by
\begin{align*}
	F_2 (z) = \int_{\gamma_2 \frac{b}{\vert a \vert}(z-a)}^{\tau_2 - \gamma_2  \frac{b}{\vert a \vert }(z-a)} x^{\top}(z,t) \mathcal{Q}^-(z) x(z,t) \, dt, \hspace{0.5 cm} z  \in [a,0), 
\end{align*}
where $\gamma_2 > 0$ and $\tau_2 > -2\gamma_2a \frac{b}{\vert a \vert} = 2 \gamma_2 b$. With similar arguments we may choose $\gamma_2$, $\tau_2$, and $\kappa_2 > 0 $ large enough such that 
\begin{align}
		\label{Interface - Lemma Auxiliary Lemma for Exponential Stability - Estimate F2}
	F_2(z) \leq F_2(a) e^{\kappa_2 (z-a)} \leq F_2(a) e^{-\kappa_2 a}, \hspace{0.5 cm} z \in [a,0).
\end{align}
Now define $\gamma = \max\{ \gamma_1, \gamma_2\}$, $\tau = \max\{ \tau_1, \tau_2\}$, $\kappa = \max\{ \kappa_1, \kappa_2\}$, and let $F_1$ and $F_2$ be defined with respect to $\gamma$ and $\tau$. 
\vspace{0.5 cm} \\
As $(T(t))_{t \geq 0}$ is of type $C_0(1,0)$, we have $\|x(t_2)\|_{\mathcal{Q}_0} \leq \| x(t_1) \|_{\mathcal{Q}_0}$ for all $t_1 \leq t_2$. Thus, 
\begin{align}
		\label{Interface - Lemma Auxiliary Lemma for Exponential Stability - Estimate T(tau)}
	\int_{\gamma b}^{\tau - \gamma b} \|x(t)\|_{\mathcal{Q}_0}^2 \, dt \geq \|x(\tau - \gamma b )\|_{\mathcal{Q}_0}^2  \int_{\gamma b}^{\tau - \gamma b} 1 \, dt = (\tau - 2 \gamma b) \| x(\tau - \gamma b)  \|_{\mathcal{Q}_0}^2. 
\end{align}
Applying Fubini's theorem, using the estimates \eqref{Interface - Lemma Auxiliary Lemma for Exponential Stability - Estimate F1}-\eqref{Interface - Lemma Auxiliary Lemma for Exponential Stability - Estimate T(tau)}, and since $\mathcal{Q}^+(z), \mathcal{Q}^-(z) \geq mI$ for all $z \in [a,b]$, we infer that  
\begin{align*}
	&\hspace{0.5 cm} 2(\tau - 2 \gamma b)  \|x(\tau)\|_{\mathcal{Q}_0}^2 \\
	&\leq 2 (\tau - 2 \gamma b) \| x(\tau - \gamma b)\|_{\mathcal{Q}_0}^2 \\
	&\leq 2 \int_{\gamma b}^{\tau - \gamma b} \|x(t) \|_{\mathcal{Q}_0}^2 \, dt \\
	&= \int_{\gamma b}^{\tau - \gamma b} \int_{a}^{0} x^{\top}(z,t) \mathcal{Q}^-(z) x(z,t) \, dz + \int_{0}^{b} x^{\top}(z,t) \mathcal{Q}^+(z) x(z,t) \, dz \, dt \\
	&=  \int_{a}^{0} \int_{\gamma (b-0)}^{\tau - \gamma (b-0)} x^{\top}(z,t) \mathcal{Q}^-(z) x(z,t) \, dt \, dz + \int_{0}^{b} \int_{\gamma (b-0)}^{\tau - \gamma (b-0)} x^{\top}(z,t) \mathcal{Q}^+(z) x(z,t) \, dt \, dz \\
	&=  \int_{a}^{0} \int_{\gamma \frac{b}{\vert a \vert}(0- a)}^{\tau - \gamma \frac{b}{\vert a \vert}(0- a)} x^{\top}(z,t) \mathcal{Q}^-(z) x(z,t) \, dt \, dz + \int_{0}^{b} \int_{\gamma (b-0)}^{\tau - \gamma (b-0)} x^{\top}(z,t) \mathcal{Q}^+(z) x(z,t) \, dt \, dz 
\end{align*}
\begin{align*}
	& \leq \int_{a}^{0} \int_{\gamma \frac{b}{\vert a \vert}(z- a)}^{\tau - \gamma \frac{b}{\vert a \vert}(z-a)} x^{\top}(z,t) \mathcal{Q}^-(z) x(z,t) \, dt \, dz + \int_{0}^{b} \int_{\gamma (b-z)}^{\tau - \gamma (b-z)} x^{\top}(z,t) \mathcal{Q}^+(z) x(z,t) \, dt \, dz \\
	&= \int_{a}^{0} F_2(z) \, dz + \int_{0}^{b} F_1(z)  \, dz \\
	&\leq -a F_2(a) e^{- \kappa a} + b F_1(b) e^{\kappa b} \\
	&= -a e^{- \kappa a} \int_{0}^{\tau} x^{\top}(a,t) \mathcal{Q}^-(a) x(a,t) \, dt + b e^{\kappa b} \int_{0}^{\tau} x^{\top}(b,t) \mathcal{Q}^+(b) x(b,t) \, dt \\
	&\leq \frac{-a e^{- \kappa a}}{m} \int_{0}^{\tau} \| \mathcal{Q}^- (a) x(a,t) \|_2^2 \, dt + \frac{b e^{\kappa b}}{m}  \int_{0}^{\tau} \| \mathcal{Q}_0(b) x(b,t) \|_2^2 \, dt.
 \end{align*}
Letting $C = \frac{\max \left\{-a e^{- \kappa a}, b e^{\kappa b} \right\} }{2m (\tau - 2 \gamma b)}$, we conclude that
	\begin{align*}
		\|x(\tau) \|_{\mathcal{Q}_0}^2 \leq C \int_{0}^{\tau}  \| \mathcal{Q}^-(a) x(a,t) \|_2^2 + \| \mathcal{Q}^+(b) x(b,t) \|_2^2 \, dt.
	\end{align*}
This shows \eqref{Interface - Lemma Auxiliary Lemma for Exponential Stablity - Inequality} and, hence, the claim. 
\end{proof}

Lemma \ref{Interface - Lemma Auxiliary Lemma for Exponential Stability} shows that there exists some time $\tau >0 $ such that the energy of a state can be bounded by the sum of the accumulated energies at both boundaries during the time interval $[0, \tau]$. It allows us to prove the following sufficient condition for exponential stability.
\begin{theorem}
	\label{Interface - Theorem Exponential Stability}
		Let the assumptions \ref{MatrixQ0Assumption1}-\ref{ContractionAssumption} hold. Assume there exists some $ k > 0$ such that 
	\begin{align}
		\label{Interface - Theorem Exponential Stability - Inequality}
		\langle A_{\mathcal{Q}_0} x_0, x_0 \rangle_{\mathcal{Q}_0} &\leq  - k \left( \| (\mathcal{Q}_0 x_0)(a) \|_2^2 + \| (\mathcal{Q}_0 x_0)(b) \|_2^2 \right)
	\end{align}
holds for all $x_0 \in D(A_{\mathcal{Q}_0})$. Then $A_{\mathcal{Q}_0}$ generates an exponentially stable $C_0$-semigroup. 
\end{theorem}
\begin{proof}
	This follows immediately from Lemma \ref{Interface - Lemma Auxiliary Lemma for Exponential Stability} and the proof of Theorem 9.1.3 in \cite{JacobZwart}. 
\end{proof}
Lastly, we want to present another sufficient condition for the exponential stability of the port-Hamiltonian system \eqref{Interface - Simplified PH-System}. Once again, we follow the exposition in \cite[Chapter 9]{JacobZwart}.
\begin{theorem}
	\label{Interface - Theorem 2 Exponential Stability}
		Let the assumptions \ref{MatrixQ0Assumption1}-\ref{ContractionAssumption} hold. If the matrix $W_B$ satisfies $W_B \Sigma W_B^{\top} > 0$, then $A_{\mathcal{Q}_0}$ generates an exponentially stable $C_0$-semigroup. 
\end{theorem}

\begin{proof}
	By Lemma \ref{Interface - Auxiliary Lemma 2}, we may write
	\begin{align*}
		W_B = S \begin{bmatrix}
			I + V & I - V
		\end{bmatrix}, 
	\end{align*}
where $S \in \mathbb{R}^{2 \times 2}$ is invertible and $V \in \mathbb{R}^{2 \times 2}$ satisfies $VV^{\top} < I$. Moreover, define the matrix
\begin{align*}
	W_C := \begin{bmatrix}
		I + V^{\top} & - I + V^{\top}
	\end{bmatrix} \in \mathbb{R}^{2 \times 4}.
\end{align*}
As shown in the proof of Lemma 9.1.4 in \cite{JacobZwart}, the matrix $W_1 := \begin{bmatrix}
	W_B \\
	W_C
\end{bmatrix}$ is invertible. 
\vspace{0.5 cm}\\
Now, let $ x \in D(A_{\mathcal{Q}_0})$. By Lemma \ref{Interface - Auxiliary Lemma 3}, we have
\begin{align*}
	\begin{bmatrix}
		f_{\partial} \\
		e_{\partial}
	\end{bmatrix} = \begin{bmatrix}
	I - V \\
	-I - V
\end{bmatrix} \lambda_x
\end{align*}
for some $\lambda_x \in \mathbb{R}^2$. By virtue of \eqref{Interface - Skew-Symmetry of J0 - Representation 2}, we compute that
\begin{align*}
	\langle A_{\mathcal{Q}_0}x, x \rangle_{\mathcal{Q}_0} + \langle x , A_{\mathcal{Q}_0}x \rangle_{\mathcal{Q}_0} &= \frac{1}{2} \left( \langle \mathcal{J}_0\mathcal{Q}_0 x, \mathcal{Q}_0 x \rangle_{L^2} + \langle \mathcal{Q}_0 x , \mathcal{J}_0 \mathcal{Q}_0 x \rangle_{L^2} \right) \\
	&= \langle e_{\partial}, f_{\partial} \rangle_2 - e_If_I \\
	&= \lambda_x^{\top}(-I + V^{\top}V) \lambda_x - e_I f_I. 
\end{align*}
Furthermore, it holds that
\begin{align*}
	y_x :=  W_C \begin{bmatrix}
		f_{\partial} \\
		e_{\partial}
	\end{bmatrix} = \begin{bmatrix}
	I + V^{\top} & -I + V^{\top}
\end{bmatrix} \begin{bmatrix}
I - V \\
-I - V
\end{bmatrix} \lambda_x = 2(I- V^{\top}V) \lambda_x . 
\end{align*}
Since $V^{\top} V - I <0$, the matrix $-I + V^{\top} V$ is invertible, and $(-I + V^{\top}V)^{-1} \leq -\hat{m}_1 I$ for some $\hat{m}_1 > 0$. Combining the preceding equations and using the passivity condition \eqref{Interface - Interface Passivity Relation} of the interface port variables, we obtain the following estimate:
\begin{align}
	\langle A_{\mathcal{Q}_0}x, x \rangle_{\mathcal{Q}_0} = \frac{1}{8} y_x^{\top} (-I + V^{\top}V)^{-1}y_x  - \frac{1}{2}e_If_I \leq -m_1 \|y_x\|_2^2,
	\label{Interface - Lemma Exponential Stability Proof - Estimate 1}
\end{align}
with $m_1 = \frac{\hat{m}_1}{8}$. Using the definition of $y_x$ and the boundary port variables \eqref{Interface - Boundary Flow and Effort}, we infer
\begin{align}
	\begin{bmatrix}
		0 \\
		y_x
	\end{bmatrix} = \begin{bmatrix}
	W_B \\
	W_C
\end{bmatrix} \begin{bmatrix}
f_{\partial} \\
e_{\partial}
\end{bmatrix} = W_1 R_{\extern}  \begin{bmatrix}
	(\mathcal{Q}_0x)(b) \\
	(\mathcal{Q}_0x)(a)
\end{bmatrix} := W  \begin{bmatrix}
(\mathcal{Q}_0x)(b) \\
(\mathcal{Q}_0x)(a)
\end{bmatrix}. 
\label{Interface - Lemma Exponential Stability Proof - Definition W}
\end{align}
As both $R_{\extern}$ and $W_1$ are invertible, the matrix $W = W_1 R_{\extern} \in \mathbb{R}^{4 \times 4}$ is invertible as well, and there exists some $m_2 > 0$ such that $\|W w \|_2^2 \geq m_2 \|w\|_2^2$ holds for all $w \in \mathbb{R}^4$. Consequently, taking the norm on both sides of \eqref{Interface - Lemma Exponential Stability Proof - Definition W} yields
\begin{align}
	\|y_x\|_2^2  = \left\|  W  \begin{bmatrix}
		(\mathcal{Q}_0x)(b) \\
		(\mathcal{Q}_0x)(a)
	\end{bmatrix} \right\|_2^2 \geq m_2 \left\|  \begin{bmatrix}
	(\mathcal{Q}_0x)(b) \\
	(\mathcal{Q}_0x)(a)
\end{bmatrix} \right\|_2^2 = m_2 \left(  \| (\mathcal{Q}_0x)(a) \|_2^2 + \| (\mathcal{Q}_0x)(b) \|_2^2 \right). 
\label{Interface - Lemma Exponential Stability Proof - Estimate y}
\end{align}
Finally, inserting \eqref{Interface - Lemma Exponential Stability Proof - Estimate y} into \eqref{Interface - Lemma Exponential Stability Proof - Estimate 1}, we obtain for all $x \in D(A_{\mathcal{Q}_0})$, 
\begin{align*}
		\langle A_{\mathcal{Q}_0}x, x \rangle_{\mathcal{Q}_0} \leq -m_1 \|y_x\|_2^2 \leq - m_1 m_2 \left(  \| (\mathcal{Q}_0x)(a) \|_2^2 + \| (\mathcal{Q}_0x)(b) \|_2^2 \right). 
\end{align*}
This is exactly inequality \eqref{Interface - Theorem Exponential Stability - Inequality}, and so Lemma \ref{Interface - Theorem Exponential Stability} yields the exponential stability of the contraction semigroup generated by $A_{\mathcal{Q}_0}$.
\end{proof}
In this section, we have shown that if we impose certain boundary conditions, then we can ensure that the semigroup $(T(t))_{t \geq 0}$ generated by the port-Hamiltonian operator $A_{\mathcal{Q}_0}$ is of type $C_0(1, -\alpha)$ for some $\alpha > 0$, i.e.,
\begin{align*}
	\|T(t)\|_{\mathcal{L}(X, \|\cdot\|_{\mathcal{Q}_0})} \leq e^{-\alpha}, \hspace{0.5 cm} t \geq 0. 
\end{align*}
As opposed to the boundary conditions, the interface condition plays a minor role in the proof of Theorem \ref{Interface - Theorem 2 Exponential Stability}. One only has to ensure that there is no power inflow at the interface position.

\subsection{The Adjoint Operator}
\label{Subsection The Adjoint Operator}
This section aims to specify the Hilbert space adjoint $A_{\mathcal{Q}_0}^{\ast}$ of the port-Hamiltonian operator $A_{\mathcal{Q}_0}$ defined in \eqref{Interface - Simplified PH System - Operator A}. To this end, we need an auxiliary lemma. For the proof, we follow the idea of the second assertion in \cite[Lemma~7.2.1]{JacobZwart}.
\begin{lemma}
	\label{Interface - Auxiliary Lemma Adjoint Operator}
		Consider the port-Hamiltonian operator $A_{\mathcal{Q}_0}$ defined in \eqref{Interface - Simplified PH System - Operator A}. Then for every vector
		\begin{align*}
			\begin{bmatrix}
				f \\
				e
			\end{bmatrix} \in \ker(W_B)
		\end{align*}
	there exists an element $x \in D(A_{\mathcal{Q}_0})$ such that
	\begin{align}
		\begin{bmatrix}
			f_{\partial, \mathcal{Q}_0 x} \\
			e_{\partial, \mathcal{Q}_0 x}
		\end{bmatrix} = 	\begin{bmatrix}
		f \\
		e
	\end{bmatrix}.
\label{Interface - Auxiliary Lemma Adjoint Operator - Target Equation}
	\end{align}
\end{lemma}

\begin{proof}
Let $\begin{bmatrix}
	f \\
	e
\end{bmatrix} \in \ker(W_B)$, and define
\begin{align*}
	\begin{bmatrix}
		u \\
		y
	\end{bmatrix} = R_{\extern}^{-1} \begin{bmatrix}
		f \\
		e
	\end{bmatrix},
\end{align*}
with $R_{\extern}$ given in \eqref{Interface - Matrix Rext}. Now, consider $x_0 \in X$ defined as
\begin{align*}
	x_0(z) = \mathcal{Q}_0^{-1}(z) \left( \frac{(b-z)z}{(b-a)a} y  + \frac{(z-a)z}{(b-a)b} u \right), \hspace{0.5 cm} z \in [a,b]. 
\end{align*}
One can readily see that $\mathcal{Q}_0x_0 \in H^1([a,b], \mathbb{R}^2) \subset D(\mathcal{J}_0)$. Furthermore, we compute
\begin{align*}
	(\mathcal{Q}_0x_0)(a) &= y, \\
	(\mathcal{Q}_0x_0)(0) &= 0, \\
	(\mathcal{Q}_0x_0)(b) &= u.
\end{align*}
Thus, $x_0$ satisfies the interface condition \eqref{Interface - Interface Passivity Relation}, and by definition of the boundary port variables \eqref{Interface - Boundary Flow and Effort} we have
\begin{align*}
\begin{bmatrix}
	f_{\partial, \mathcal{Q}_0x_0} \\
	e_{\partial, \mathcal{Q}_0x_0}
\end{bmatrix} =	R_{\extern}	\trace_0 (\mathcal{Q}_0x_0) = R_{\extern} \begin{bmatrix}
		(\mathcal{Q}_0x_0)(b) \\
		(\mathcal{Q}_0x_0)(a)
	\end{bmatrix} = \begin{bmatrix}
	f \\
	e
\end{bmatrix} \in \ker(W_B).
\end{align*}
Altogether, $x_0$ is an element of $D(A_{\mathcal{Q}_0})$ and it satisfies \eqref{Interface - Auxiliary Lemma Adjoint Operator - Target Equation}. 
\end{proof}

 The proof of the following theorem is based on the proof of Theorem 2.24 in \cite{Villegas}. 
\begin{theorem}
	\label{Interface - Theorem Adjoint Operator}
	Consider the port-Hamiltonian operator $A_{\mathcal{Q}_0}$ defined in \eqref{Interface - Simplified PH System - Operator A}. Assume that the matrix $W_B$ may be written as \eqref{Interface - Auxiliary Lemma 2 - Matrix Representation}, with $S \in \mathbb{R}^{2 \times 2}$ invertible and $V \in \mathbb{R}^{2 \times 2}$ satisfying $V^{\top}V \leq I$. On the energy space $(X, \|\cdot\|_{\mathcal{Q}_0})$, the Hilbert space adjoint $A_{\mathcal{Q}_0}^{\ast} \colon D(A_{\mathcal{Q}_0}^{\ast}) \subset X \to X$ of the operator $A_{\mathcal{Q}_0}$ is given by
	\begin{align}
		\begin{split}
		D(A_{\mathcal{Q}_0}^{\ast}) &= \left\{ y \in X \hspace{0.1 cm} \Big| \hspace{0.1 cm} \mathcal{Q}_0y \in D(\mathcal{J}_0), \hspace{0.1 cm} f_{I, \mathcal{Q}_0y} = -r e_{I, \mathcal{Q}_0y}, \hspace{0.1 cm}  \begin{bmatrix}
			-(I+V^{\top}) & I + V^{\top}
		\end{bmatrix}  \begin{bmatrix}
		f_{\partial, \mathcal{Q}_0y} \\
		e_{\partial, \mathcal{Q}_0 y}
	\end{bmatrix} = 0 \right\}, \\
		A_{\mathcal{Q}_0}^{\ast} y &= - \mathcal{J}_0 (\mathcal{Q}_0 y), \hspace{0.5 cm} y \in D(A_{\mathcal{Q}_0}^{\ast}). 
	\end{split}
\label{Interface - Adjoint Operator}
	\end{align}
\end{theorem}
\begin{proof}
	Recall from Section \ref{Section Operators} that the domain of the Hilbert space adjoint can be written as \eqref{Adjoint Operator Domain wrt Frechet Riesz}. Let $x \in \mathbf{H}_0^1(a,b) \times H_0^1([a,b], \mathbb{R}) \subset D(A_{\mathcal{Q}_0})$, with $\mathbf{H}_0^1(a,b)$ the space of $H^1$-functions vanishing both at the boundary and the interface defined in \eqref{Interface Paper - Subspace H(a,b)}. If $y \in D(A_{\mathcal{Q}_0}^{\ast})$, then the following holds:
	\begin{align*}
	 \frac{1}{2} \langle \mathcal{J}_0 (\mathcal{Q}_0 x) , \mathcal{Q}_0 y \rangle_{L^2} =	\langle A_{\mathcal{Q}_0} x , y \rangle_{\mathcal{Q}_0}  = \langle x, A_{\mathcal{Q}_0}^{\ast} y  \rangle_{\mathcal{Q}_0} = \frac{1}{2} \langle \mathcal{Q}_0 x , A_{\mathcal{Q}_0}^{\ast}y  \rangle_{L^2}. 
	\end{align*}
From this equation we infer that $\mathcal{Q}_0 y$ has to be an element of $D(\mathcal{J}_0)$ for all $y \in D(A_{\mathcal{Q}_0}^{\ast})$. This is the first constitutive relation of $D(A_{\mathcal{Q}_0}^{\ast})$, and it allows us to define boundary and interface port variables with respect to $\mathcal{Q}_0y$. By virtue of equation \eqref{Interface - Skew-Symmetry of J0 - Representation 2}, we get for all $x \in D(A_{\mathcal{Q}_0})$, $y \in D(A_{\mathcal{Q}_0}^{\ast})$,
\begin{align}
	\begin{split}
	2 \langle A_{\mathcal{Q}_0} x, y \rangle_{\mathcal{Q}_0} &=  \langle \mathcal{J}_0 (\mathcal{Q}_0 x), \mathcal{Q}_0 y \rangle_{L^2} \\
	&= -  \langle \mathcal{Q}_0 x, \mathcal{J}_0 (\mathcal{Q}_0 y) \rangle_{L^2} + \langle e_{\partial, \mathcal{Q}_0 y} , f_{\partial, \mathcal{Q}_0 x} \rangle_2 + \langle e_{\partial, \mathcal{Q}_0 x}, f_{\partial, \mathcal{Q}_0 y} \rangle_2 \\
	&\hspace{0.5 cm} - f_{I, \mathcal{Q}_0 x} e_{I, \mathcal{Q}_0 y} - f_{I, \mathcal{Q}_0 y} e_{I, \mathcal{Q}_0 x}. 
	\end{split} 
\label{Interface - Theorem Adjoint Operator - Relation in Proof}
\end{align}
We need to get rid of the boundary and interface terms. With $\Sigma$ defined in \eqref{Matrix Sigma}, we may write
\begin{align*}
	\langle e_{\partial, \mathcal{Q}_0 y} , f_{\partial, \mathcal{Q}_0 x} \rangle_2 + \langle e_{\partial, \mathcal{Q}_0 x}, f_{\partial, \mathcal{Q}_0 y} \rangle_2 =   \begin{bmatrix}
		f_{\partial, \mathcal{Q}_0x}^{\top} &
		e_{\partial, \mathcal{Q}_0y}^{\top} 
	\end{bmatrix} \Sigma \begin{bmatrix}
		f_{\partial, \mathcal{Q}_0y} \\
		e_{\partial, \mathcal{Q}_0y}
	\end{bmatrix}.
\end{align*}
By assumption, $\begin{bmatrix} 
	f_{\partial, \mathcal{Q}_0x} \\
	e_{\partial, \mathcal{Q}_0x} 
\end{bmatrix} \in \ker (W_B)$ for all $x \in D(A_{\mathcal{Q}_0})$. By virtue of Lemma \ref{Interface - Auxiliary Lemma 3},  we have
\begin{align*}
	\begin{bmatrix} 
		f_{\partial, \mathcal{Q}_0x} \\
		e_{\partial, \mathcal{Q}_0x} 
	\end{bmatrix} = \begin{bmatrix}
		I + V \\
		- (I + V) 
	\end{bmatrix} \lambda_x
\end{align*}
for some $\lambda_x \in \mathbb{R}^2$. Lemma \ref{Interface - Auxiliary Lemma Adjoint Operator} yields that for every $\begin{bmatrix}
	f \\
	e
\end{bmatrix} \in \ker (W_B)$ there exists some $x \in D(A_{\mathcal{Q}_0})$ such that $\begin{bmatrix} 
	f_{\partial, \mathcal{Q}_0x} \\
	e_{\partial, \mathcal{Q}_0x} 
\end{bmatrix} = \begin{bmatrix}
	f \\
	e
\end{bmatrix}$. Now, since $\dim(\ker(W_B)) = 2$, the equality
\begin{align*}
	\lambda_x^{\top} \begin{bmatrix}
		I + V^{\top} & - (I+V^{\top}) 
	\end{bmatrix} \Sigma \begin{bmatrix}
		f_{\partial, \mathcal{Q}_0y } \\
		e_{\partial, \mathcal{Q}_0y}
	\end{bmatrix} = 0
\end{align*}
has to hold for every $\lambda_x \in \mathbb{R}^2$. Hence, the second constitutive condition of $D(A_{\mathcal{Q}_0}^{\ast})$ is given by
\begin{align*}
	\begin{bmatrix} 
		f_{\partial, \mathcal{Q}_0y} \\
		e_{\partial, \mathcal{Q}_0y} 
	\end{bmatrix} \in \ker \left( \begin{bmatrix}
		-(I + V^{\top}) &  I + V^{\top}
	\end{bmatrix}  \right).
\end{align*}

Lastly, we need to get rid of the interface terms. Due to the interface relation \eqref{Interface - Interface Passivity Relation}, we can write

\begin{align*}
 f_{I, \mathcal{Q}_0x} e_{I, \mathcal{Q}_0y} + f_{I, \mathcal{Q}_0y} e_{I, \mathcal{Q}_0x} =
	 r e_{I, \mathcal{Q}_0x} e_{I,\mathcal{Q}_0y} + f_{I, \mathcal{Q}_0y} e_{I, \mathcal{Q}_0x} \stackrel{!}{=} 0.
\end{align*}
Hence, the last constitutive condition of $D(A_{\mathcal{Q}_0}^{\ast})$ is given by
\begin{align*}
	f_{I, \mathcal{Q}_0y} = - r e_{I, \mathcal{Q}_0y}.
\end{align*}
Altogether, the domain of the adjoint operator $A_{\mathcal{Q}_0}^{\ast}$ is precisely given as in \eqref{Interface - Adjoint Operator}. For all $x \in D(A_{\mathcal{Q}_0})$, $y \in D(A_{\mathcal{Q}_0}^{\ast})$, equation \eqref{Interface - Theorem Adjoint Operator - Relation in Proof} becomes
\begin{align*}
	\langle A_{\mathcal{Q}_0} x , y \rangle_{\mathcal{Q}_0} = \frac{1}{2} \langle \mathcal{J}_0 (\mathcal{Q}_0 x), \mathcal{Q}_0 y \rangle_{L^2} = - \frac{1}{2} \langle \mathcal{Q}_0 x, \mathcal{J}_0 (\mathcal{Q}_0 y) \rangle_{L^2} = \langle x , -\mathcal{J}_0 (\mathcal{Q}_0 y) \rangle_{\mathcal{Q}_0}.  
\end{align*}
Thus, it holds that
\begin{align*}
	A_{\mathcal{Q}_0}^{\ast} y = - \mathcal{J}_0 (\mathcal{Q}_0 y), \hspace{0.5 cm} y \in D(A_{\mathcal{Q}_0}^{\ast}).
\end{align*}
This completes the proof. 
\end{proof}

\subsection{Resolvent associated with $\mathbf{A_{\mathcal{Q}_0}}$}
\label{Subsection Resolvent}
In this section, we want to determine the resolvent operator $R(\lambda, A_{\mathcal{Q}_0})$ of the port-Hamiltonian operator $A_{\mathcal{Q}_0}$ defined in \eqref{Interface - Simplified PH System - Operator A} for $\lambda \in (0, \infty) \subset \rho(A)$. For that purpose, we will basically replicate Step 3 in the proof of Theorem \ref{Interface - Theorem A Generates a Contraction Semigroup}. Recall that for $\mathcal{Q}_0 \equiv I$, the solution of $(I-A_I)x = y$ for some $y \in X$ is given by $\varphi_y(\cdot; x(a), x(0^+)) \in D(A_I)$ defined in \eqref{Interface - Solution of ODE left}, \eqref{Interface - Solution of ODE right}, and \eqref{Interface - Solution of ODE complete}. The initial value $x(a)$ can be calculated through \eqref{Interface - D(A) condition - representation 2}, and $x(0^+)$ is given by
\begin{align*}
	x(0^+) = \begin{bmatrix}
		x_1(0^+) \\
		x_2(0) 
	\end{bmatrix} = \begin{bmatrix}
		- \frac{1}{r} \left( \varphi^-(0^-; x(a)) \right)_2 + \left( \varphi^-(0^-; x(a)) \right)_1 \\
		\left( \varphi^-(0^-; x(a)) \right)_2 
	\end{bmatrix}. 
\end{align*}
This yields
\begin{align*}
	\| (I-A)^{-1} y \|_{L^2}^2 &= \|\varphi_y(\cdot; x(a), x(0^+)) \|_{L^2}^2 \\
	&= \int_{a}^{b} \|\varphi_y(z; x(a), x(0^+)) \|_2^2 \, dz \\
	&= \int_{a}^{b} \left[ \left( c_0(z) \varphi^-(z; x(a)) + \overline{c}_0(z)\varphi^+(z; x(0^+)) \right)_1 \right]^2 \\
	&\hspace{0.5 cm} + \left[  \left( c_0(z) \varphi^-(z; x(a)) + \overline{c}_0(z)\varphi^+(z; x(0^+)) \right)_2 \right]^2 \, dz \\
	&= \int_{a}^{0} \| \varphi^-(z;x(a)) \|_2^2 \, dz + \int_{0}^{b} \| \varphi^+(z;x(0^+))\|_2^2 \, dz.
\end{align*}
Thus, the solution of the equation $(\lambda I - A_{\mathcal{Q}_0})x = y$ for some $y \in X$, $\lambda >0$ coincides with $R(\lambda, A_{\mathcal{Q}_0})y$. Subsequently, we will exploit this fact in order to compute the resolvent operator of $A_{\mathcal{Q}_0}$.
\vspace{0.5 cm}\\
To this end, consider again the energy space $(X, \| \cdot\|_{\mathcal{Q}_0})$ with $\mathcal{Q}_0= c_0 \mathcal{Q}^- + \overline{c}_0 \mathcal{Q}^+ \in \mathcal{L}(X)$ a coercive matrix operator, and let the assumptions \ref{MatrixQ0Assumption1}-\ref{ContractionAssumption} hold, where $r > 0$. Let $y \in X$ and $\lambda > 0$ be arbitary. Recall that on the complementary domains $[a,0)$ and $(0,b]$, the operator $\mathcal{J}_0$ given in \eqref{Interface - Operator J0} acts as the differential operator $P_1 \frac{d}{dz}$, with $P_1$ defined in \eqref{Interface - Matrix P1}. Thus, the equation $(\lambda I-A_{\mathcal{Q}_0})x = y$ can be equivalently written as
\begin{align*}
	\lambda x^- - P_1 \frac{d}{dz} (\mathcal{Q}^- x^-) &= y^- \hspace{0.5 cm} \text{on } [a,0), \\
	\lambda x^+ - P_1 \frac{d}{dz}(\mathcal{Q}^+ x^+) &= y^+  \hspace{0.5 cm} \text{on } (0,b].
\end{align*}
Since $\mathcal{Q}^{\pm} \in \mathcal{C}^1([a,b], \mathbb{R}^{2 \times 2})$, observe that we may write
\begin{align*}
	\frac{d}{dz} x^{\pm}(z) = (\mathcal{Q}^{\pm})^{-1} (z) \left( \lambda P_1 - \frac{d}{dz} \mathcal{Q}^{\pm}(z) \right) x^{\pm}(z) - (\mathcal{Q}^{\pm})^{-1}(z) P_1 y^{\pm}(z)
\end{align*}
on the respective subdomain. 
The variation of constants formula yields
\begin{align*}
	\varphi_{\lambda, y}^-(z; x(a)) = \Lambda_{\lambda}^-(z,a) x(a) - \int_{a}^{z} \Lambda_{\lambda}^-(z,s)(\mathcal{Q}^-)^{-1}(s) P_1 y^-(s) \, ds,  \hspace{0.5 cm}  z \in [a,0),
\end{align*}
and
\begin{align*}
	\varphi_{\lambda, y}^+(z; x(0^+)) = \Lambda_{\lambda}^+(z,0) x(0^+) - \int_{0}^{z} \Lambda_{\lambda}^+(z,s) (\mathcal{Q}^+)^{-1}(s) P_1 y^+(s) \, ds, \hspace{0.5 cm} z \in (0,b],
\end{align*}
with $\Lambda_{\lambda}^-(\cdot,a)$ and $\Lambda_{\lambda}^+(\cdot,0^+)$ the respective transition matrices defined on the respective subdomain. 
The composed solution of $(I-A)x = y$ is therefore given by
\begin{align*}
	\varphi_{\lambda, y}(z; x(a), x(0^+)) = c_0(z) \varphi_{\lambda, y}^-(z; x(a)) + \overline{c}_0(z) \varphi_{\lambda, y}^+(z; x(0^+)), \hspace{0.5 cm} z \in [a,b]. 
\end{align*}
For the remainder of this section, we will omit indicating the dependence on $\lambda$ and $y$. Analogously to the proof of Theorem \ref{Interface - Theorem A Generates a Contraction Semigroup}, the initial value $x(a)$ is chosen such that the boundary condition
\begin{align*}
	W_B R_0 \begin{bmatrix}
		\mathcal{Q}(b) \varphi(b; x(a)) \\
		\mathcal{Q}(a) \varphi(a; x(a))
	\end{bmatrix} = 0
\end{align*}
is satisfied, and the value $x(0^+)$ is determined by means of the interface relations. In the following, we will compute those values.
\vspace{0.5 cm}\\
According to \eqref{Interface - Simplified Continuity Equation}-\eqref{Interface - Simplified Balance Equation}, the interface port variables are given as follows:
	\begin{align*}
		f_I &= \left( \mathcal{Q}^-(0) \varphi^-(0^-; x(a)) \right)_2 = \left( \mathcal{Q}^+(0) x(0^+) \right)_2 \\
		- e_I &=  \left( \mathcal{Q}^+(0) x(0^+) \right)_1 - \left( \mathcal{Q}^-(0) \varphi^-(0^-; x(a)) \right)_1. 
	\end{align*}
	More precisely, 
	\begin{align}
		\begin{split}
		f_I &= Q_{12}^-(0) \left( \varphi^-(0^-; x(a)) \right)_1 + Q_{22}^- (0) \left( \varphi^-(0^-; x(a)) \right)_2 \\
		&= Q_{12}^+(0) x_1(0^+) + Q_{22}^+(0) x_2(0^+). 
		\end{split}
	\label{Interface - Calculation x(0+) - fI}
	\end{align}
	The exact value of $f_I$ is defined through the first equation.  Furthermore, 
	\begin{align*}
		- e_I &=  Q_{11}^+(0) x_1(0^+) + Q_{12}^+(0) x_2(0^+) \\
		&\hspace{0.5 cm} - \left( Q_{11}^-(0) \left( \varphi^-(0^-; x(a)) \right)_1 + Q_{12}^-(0) \left( \varphi^-(0^-; x(a)) \right)_2 \right). 
	\end{align*}
	The entries of the initial value $x(0^+)$ can be determined by exploiting the passivity relation \eqref{Interface - Interface Passivity Relation}: we have 
	\begin{align}
		&f_I = r e_I  \nonumber \\
		\Leftrightarrow \hspace{0.5 cm} &Q_{11}^+(0) x_1(0^+) + Q_{12}^+(0) x_2(0^+) = - \frac{1}{r} f_I + \left( \mathcal{Q}^-(0) \varphi^-(0^-; x(a)) \right)_1 \nonumber \\
		\Leftrightarrow \hspace{0.5 cm} &x_1(0^+) = - \frac{1}{r Q_{11}^+(0)} f_I + \frac{1}{Q_{11}^+(0)} \left( \mathcal{Q}^-(0) \varphi^-(0^-; x(a))  \right)_1 - \frac{Q_{12}^+(0)}{Q_{11}^+(0)} x_2(0^+). \label{Interface - Calculation x(0+) - x1(0+)}
	\end{align}
	Plugging this into the second equation of \eqref{Interface - Calculation x(0+) - fI} yields
	\begin{align*}
		f_I &= \frac{Q_{12}^+(0)}{Q_{11}^+(0)} \left( - \frac{1}{r} f_I  + \left( \mathcal{Q}^-(0) \varphi^-(0^-; x(a)) \right)_1 \right) \\
		&\hspace{0.5 cm} + \left( - \frac{(Q_{12}^+(0))^2}{Q_{11}^+(0)} + Q_{22}^+(0) \right) x_2(0^+), 
	\end{align*}
	whence
	\begin{align*}
		x_2(0^+) = \frac{Q_{11}^+(0)}{(Q_{12}^+(0))^2 + Q_{11}^+(0) Q_{22}^+(0)} \left( f_I - \frac{Q_{12}^+(0)}{Q_{11}^+(0)} \left(  - \frac{1}{r} f_I + \left( \mathcal{Q}^-(0) \varphi^-(0^-; x(a)) \right)_1 \right) \right). 
	\end{align*}
	Thus, the initial value $x(0^+)$ is uniquely defined and depends on the initial value $x(a)$. Letting $\tilde{c}_1 = \frac{Q_{11}^+(0)}{(Q_{12}^+(0))^2 + Q_{11}^+(0) Q_{22}^+(0)} \left( 1 + \frac{Q_{12}^+(0)}{r Q_{11}^+(0)} \right)$, $\tilde{c}_2 = - \frac{Q_{12}^+(0)}{(Q_{12}^+(0))^2 + Q_{11}^+(0) Q_{22}^+(0)}$, we may write
	\begin{align*}
		x_2(0^+) = \tilde{c}_1 f_I + \tilde{c}_2 \left( \mathcal{Q}^-(0) \varphi^-(0^-; x(a)) \right)_1.
	\end{align*}
	Furthermore, letting $\tilde{c}_3 = - \frac{1 + r Q_{12}^+(0) \tilde{c}_1}{r Q_{11}^+(0)}$, $\tilde{c}_4 = \frac{1 - Q_{12}^+(0) \tilde{c}_2}{Q_{11}^+(0)}$, and inserting $x_2(0^+)$ in \eqref{Interface - Calculation x(0+) - x1(0+)}, we get
	\begin{align*}
		x_1(0^+) = \tilde{c}_3 f_I + \tilde{c}_4 \left( \mathcal{Q}^-(0) \varphi^-(0^-; x(a)) \right)_1. 
	\end{align*}
	Finally, the choices
	\begin{align*}
		c_1 &= \tilde{c}_3 Q_{12}^-(0) + \tilde{c}_4 Q_{11}^-(0), \hspace{0.3 cm } c_2 = \tilde{c}_3 Q_{22}^-(0) + \tilde{c}_4 Q_{12}^-(0), \\
		c_3 &= \tilde{c}_1 Q_{12}^-(0) + \tilde{c}_2 Q_{11}^-(0), \hspace{0.3 cm} c_4 = \tilde{c}_1 Q_{22}^-(0) + \tilde{c}_2 Q_{12}^-(0)
	\end{align*}
	yield
	\begin{align}
		x_1(0^+) &= c_1 \left( \varphi^-(0^-; x(a)) \right)_1 + c_2 \left( \varphi^-(0^-; x(a)) \right)_2 \label{Interface - x1(0+)} \\
		x_2(0^+) &= c_3 \left( \varphi^-(0^-; x(a)) \right)_1 + c_4 \left( \varphi^-(0^-; x(a)) \right)_2.
		\label{Interface - x2(0+)}
	\end{align}
Defining $C_0 = \begin{bmatrix}
	c_1 & c_2 \\
	c_3 & c_4
\end{bmatrix} \in \mathbb{R}^{2 \times 2}$, we may write
\begin{align*}
	x(0^+) = C_0 \varphi^-(0^-; x(a)). 
\end{align*}
We continue with the determination of the initial value $x(a) \in \mathbb{R}^2$. It has to be chosen such that
	\begin{align}
		\label{Interface - Calculation x(a) Reference}
		0 = W_B \begin{bmatrix}
			f_{\partial} \\
			e_{\partial}
		\end{bmatrix} = W_B R_{\extern} \begin{bmatrix}
			\mathcal{Q}_0(b)\varphi^+(b; x(0^+)) \\
			\mathcal{Q}_0(a) \varphi^-(a; x(a))
		\end{bmatrix} = W_B R_{\extern} \begin{bmatrix}
			\mathcal{Q}^+(b) & 0 \\
			0 & \mathcal{Q}^-(a) 
			\end{bmatrix}
		\begin{bmatrix}
			\varphi^+(b; x(0^+)) \\
			 x(a)
		\end{bmatrix}.
	\end{align}
	We again wish to write $\varphi^+(b; x(0^+))$ as $E x(a) + q$ for some $E \in \mathbb{R}^{2 \times 2}$ and $q \in \mathbb{R}^2$. Similarly to the proof of Theorem \ref{Interface - Theorem A Generates a Contraction Semigroup}, we compute that
	\begin{align*}
		E x(a)&= \Lambda^+(b,0) \begin{bmatrix}
			c_1 \left( \Lambda^-(0^-,a) x(a) \right)_1 + c_2 \left( \Lambda^-(0^-,a) x(a) \right)_2 \\
			c_3 \left( \Lambda^-(0^-,a) x(a) \right)_1 + c_4 \left( \Lambda^-(0^-,a) x(a) \right)_2
		\end{bmatrix} \\
	&= \left( \Lambda^+(b,0)C_0\Lambda^-(0^-,a) \right) x(a)	\end{align*}
	and 
	\begin{align*}
		q &= \Lambda^+(b,0) \begin{bmatrix} c_1 \left( \varphi_{\text{inhom}}^-(0^-, x(a)) \right)_1  + c_2 \left( \varphi_{\text{inhom}}^-(0^-, x(a)) \right)_2 \\
			c_3 \left( \varphi_{\text{inhom}}^-(0^-, x(a)) \right)_1 + c_4 \left( \varphi_{\text{inhom}}^-(0^-, x(a)) \right)_2
		\end{bmatrix} \\
		&\hspace{0.5 cm} - \int_{0}^{b} \Lambda^+(b,s) (\mathcal{Q}^+)^{-1}(s) P_1 y^+(s) \, ds \\
		&= \Lambda^+(b,0)C_0\varphi_{\text{inhom}}^-(0^-;x(a)) - \int_{0}^{b} \Lambda^+(b,s) (\mathcal{Q}^+)^{-1}(s) P_1 y^+(s) \, ds, 
	\end{align*}
where
\begin{align*}
	\varphi_{\text{inhom}}^-(0^-, x(a)) = - \int_{a}^{0} \Lambda^-(0^-,s) (\mathcal{Q}^-)^{-1}(s) P_1 y^-(s) \, ds.
\end{align*}
	Following the proof, the initial value $x(a)$ can be calculated from \eqref{Interface - Calculation x(a) Reference} and is given by
	\begin{align}
		\label{Interface - x(a)}
		x(a) = - \left( W_B R_{\extern} \begin{bmatrix}
			\mathcal{Q}^+(b) & 0 \\
			0 & \mathcal{Q}^-(a)
		\end{bmatrix} \begin{bmatrix}
			E \\
			I
		\end{bmatrix} \right)^{-1} W_B R_{\extern} \begin{bmatrix}
			\mathcal{Q}^+(b) & 0 \\
			0 & \mathcal{Q}^-(a)
		\end{bmatrix} \begin{bmatrix}
			q \\
			0
		\end{bmatrix}.
	\end{align}
Finally, we can specify the resolvent operator of $A_{\mathcal{Q}_0}$. For every $\lambda > 0$ and for every $y \in X$ we have
\begin{align*}
	&\hspace{0.5 cm} \|R(\lambda, A_{\mathcal{Q}_0})y \|_{\mathcal{Q}_0}^2 \\
	&= \| \varphi_{\lambda, y}(\cdot; x(a), x(0^+)) \|_{\mathcal{Q}_0}^2 \\
	&= \frac{1}{2} \int_{a}^{b} \left( c_0(z) \varphi^-(z;x(a)) + \overline{c}_0(z) \varphi^+(z; x(0^+)) \right)^{\top} \mathcal{Q}_0(z) \left( c_0(z) \varphi^-(z;x(a)) + \overline{c}_0(z) \varphi^+(z; x(0^+)) \right) \, dz \\
	&= \frac{1}{2} \left(  \int_{a}^{0}  \left( \varphi^-(z; x(a)) \right)^{\top} \mathcal{Q}^-(z) \varphi^-(z;x(a)) \, dz + \int_{0}^{b} \left( \varphi^+(z; x(0^+)) \right)^{\top} \mathcal{Q}^+(z) \varphi^+(z; x(0^+)) \, dz \right),
\end{align*}
with the initial values $x(a)$ and $x(0^+)$ given by \eqref{Interface - x(a)} and \eqref{Interface - x1(0+)}-\eqref{Interface - x2(0+)}, respectively. 
\vspace{0.5 cm}\\
We skip specifying the resolvent in the case that the passivity constant is given by $r =0$. To this end, one would basically need to replicate Case 2 in Step 3 of the proof of Theorem \ref{Interface - Theorem A Generates a Contraction Semigroup}. Instead, we will come back to the scenario where the interface position is moving over time, and prove a crucial stability property of the resulting system. 

\subsection{Stability of the Family of Infinitesimal Generators}
\label{Subsection Stability of the Family of Infinitesimal Generators}
In the last segment of this chapter we shall analyze the simplified port-Hamiltonian system \eqref{Interface - Simplified PH-System} in case of a moving interface. We want to define a family of port-Hamiltonian operators that encompasses the position of the interface, and we want to associate this family with an evolution problem of the form
\begin{align*}
	\dot{x}(t) &= A(t)x(t), \hspace{0.5 cm} 0 \leq s < t \leq \tau, \\
	x(s) &= x_0 \in X.
\end{align*}
We shall impose assumptions such that this family is a stable family in the sense of Definition~\ref{Definition Stable Family of Generators}.
\vspace{0.5 cm}\\
Everything we have proved in Section \ref{Section A Simplified System} can be analogously shown for arbitrary interface positions $l \in (a,b)$. We have already come across the operators $\mathbf{d}_l$ and its formal adjoint $\mathbf{d}_l^{\ast}$ in Subsection \ref{Subsection Balance Equations of the State Variables - Moving Interface}. They are defined by 
\begin{align*}
	D(\mathbf{d}_l) &= H^1([a,b], \mathbb{R}), \\
	\mathbf{d}_lx &= - \left[ \frac{d}{dz} (c_l x) + \frac{d}{dz} (\overline{c}_l x ) \right], \hspace{0.5 cm} x \in D(\mathbf{d}_l),
\end{align*}
and
\begin{align*}
	D(\mathbf{d}_l^{\ast}) &= \left\{ y \in L^2([a,b], \mathbb{R}) \mid y_{|(a,l)} \in H^1((a,l), \mathbb{R}), \hspace{0.1 cm} y_{|(l,b)} \in H^1((l,b), \mathbb{R}) \right\}, \\
	\mathbf{d}_l^{\ast}y &= \left( - \mathbf{d}_l - \left[ \frac{d}{dz} c_l + \frac{d}{dz} \overline{c}_l \right] \right) y, \hspace{0.5 cm} y \in D(\mathbf{d}_l^{\ast}),
\end{align*}
respectively, where $c_l$ and $\overline{c}_l$ are the color functions with a jump discontinuity at the interface position $l$. This allows us to define the formally skew-symmetric matrix differential operator $\mathcal{J}_l \colon D(\mathcal{J}_l) \subset X \to X$ as follows:
\begin{align*}
	D(\mathcal{J}_l) &= \left\{ x \in X \mid x_1 \in D(\mathbf{d}_l^{\ast}), \hspace{0.1 cm} x_2 \in D(\mathbf{d}_l)\right\}, \\
	\mathcal{J}_l x &= \begin{bmatrix}
		0 & \mathbf{d}_l  \\
		- \mathbf{d}_l^{\ast} & 0 
	\end{bmatrix} \begin{bmatrix}
		x_1 \\
		x_2
	\end{bmatrix}, \hspace{0.5 cm} x \in D(\mathcal{J}_l).
\end{align*}
Then the port-Hamiltonian operator $A_{\mathcal{Q}_l} \colon D(A_{\mathcal{Q}_l}) \subset X \to X$ given by
\begin{align*}
		D(A_{\mathcal{Q}_l}) &= \left\{ x \in X \mid \mathcal{Q}_lx \in D(\mathcal{J}_l), \hspace{0.1 cm} f_{\mathcal{Q}_l x} = r e_{\mathcal{Q}_l x}, \hspace{0.1 cm} W_B \begin{bmatrix}
			f_{\partial, \mathcal{Q}_lx} \\
			e_{\partial, \mathcal{Q}_l x}
		\end{bmatrix} = 0 \right\}, \\
		A_{\mathcal{Q}_l}x &= \mathcal{J}_l (\mathcal{Q}_lx), \hspace{0.5 cm}  x \in D(A_{\mathcal{Q}_l}),
\end{align*}
associated with the abstract Cauchy problem
\begin{align*}
	\dot{x}(t) &= A_{\mathcal{Q}_l}x(t), \hspace{0.5 cm} t > 0, \\
	x(0) &= x_0 \in X,
\end{align*}
generates a contraction semigroup on the energy space $(X, \| \cdot \|_{\mathcal{Q}_l})$, where
\begin{align*}
	\mathcal{Q}_l = c_l \mathcal{Q}^- + \overline{c}_l \mathcal{Q}^+ \in \mathcal{L}(X)
\end{align*}
is coercive, if the assumptions of Theorem \ref{Interface - Theorem A Generates a Contraction Semigroup} are satisfied. More precisely, it generates a contraction semigroup $(T_l(t))_{t \geq 0}$ provided that $W_B \Sigma W_B^{\top} \geq 0$, $\rank(W_B) =2$, and $r \geq 0$. Let now the assumptions \ref{MatrixQ0Assumption1}-\ref{ContractionAssumption} hold (see page \pageref{MatrixQ0Assumption1}). If $W_B \Sigma W_B^{\top} > 0$, then  the semigroup $(T_l(t))_{t \geq 0}$ is even of type $C_0(1,- \alpha)$ for some $\alpha >0$, see Theorem~\ref{Interface - Theorem 2 Exponential Stability}. Its Hilbert space adjoint $A_{\mathcal{Q}_l}^{\ast} \colon D(A_{\mathcal{Q}_l}) \subset X \to X$ is accordingly defined as in \eqref{Interface - Adjoint Operator}. Furthermore, the underlying Dirac structure is given by
\begin{align*}
	\mathcal{D}_{\mathcal{J}_l} :=& \left\{ \left( \begin{pmatrix}
		f \\
		f_I \\
		f_{\partial}
	\end{pmatrix} , \begin{pmatrix}
		e \\
		e_I \\
		e_{\partial}
	\end{pmatrix} \right) \in \mathcal{B} \hspace{0.1 cm} \Bigg| \hspace{0.1 cm}  e \in D(\mathcal{J}_l), \, f = \mathcal{J}_le, \, f_I = e_2(0), \right. \\ 
	&\hspace{4.22 cm} \left. \, e_I = e_1(0^-) - e_1(0^+), \, \begin{bmatrix}
		f_{\partial} \\
		e_{\partial}
	\end{bmatrix} = R_{\extern} \trace_l(e) \right\}, 
\end{align*}
where
\begin{align*}
	\trace_l \colon D(\mathcal{J}_l) \to \mathbb{R}^4, \hspace{0.3 cm} e \mapsto \begin{bmatrix}
		e(b) \\
		e(a)
	\end{bmatrix}.
\end{align*}
Together with the resistive structure
\begin{align*}
	\mathcal{R}_r = \{ (f_I,e_I) \in \mathbb{R}^2 \mid f_I = re_I \},
\end{align*}
the dynamics are geometrically specified by the requirement that for all $t> 0$ it holds that
\begin{align*}
	\left( \begin{pmatrix}
		\frac{\partial}{\partial t} x(\cdot,t) \\
		f_I(t) \\
		f_{\partial}(t)
	\end{pmatrix} ,  \begin{pmatrix}
		\mathcal{Q}_lx(\cdot,t) \\
		e_I(t) \\
		e_{\partial}(t)
	\end{pmatrix} \right) \in \mathcal{D}_{\mathcal{J}_l}, \hspace{0.3 cm} (f_I(t), e_I(t)) \in \mathcal{R}_{r}.
\end{align*}
Lastly, if $r >0$, then for the resolvent it holds that
\begin{align*}
	&\hspace{0.5 cm} \|R(\lambda, A_{\mathcal{Q}_l})y \|_{\mathcal{Q}_l}^2 \\
	&= \frac{1}{2} \left(  \int_{a}^{l}  \left( \varphi^-(z; x(a)) \right)^{\top} \mathcal{Q}^-(z) \varphi^-(z;x(a)) \, dz + \int_{l}^{b} \left( \varphi^+(z; x(l^+)) \right)^{\top} \mathcal{Q}^+(z) \varphi^+(z; x(l^+)) \, dz \right),
\end{align*}
with $y \in X$, $\lambda >0$, and the solutions $\varphi^-(\cdot; x(a))$, $\varphi^+(\cdot; x(l^+))$ as well as the values $x(a)$ and $x(l^+)$ accordingly defined, see Subsection \ref{Subsection Resolvent}.
\vspace{0.5 cm}\\
Unfortunately, the state space $(X, \|\cdot\|_{\mathcal{Q}_l})$ has to be changed for different interface positions $l \in (a,b)$ to guarantee that the aforementioned statements hold true. However, we have a norm equivalence of the family $(\|\cdot\|_{\mathcal{Q}_l})_{l \in (a,b)}$ to the norm $\| \cdot \|_{\mathcal{Q}_0}$: Since $mI \leq \mathcal{Q}^{\pm}(z) \leq MI$ for all $z \in [a,b]$, we get for all $l \in (a,b)$,
\begin{align*}
	\frac{m}{2} \| x\|_{L^2}^2 \leq \| x \|_{\mathcal{Q}_l}^2 \leq \frac{M}{2} \| x \|_{L^2}^2, \hspace{0.5 cm} x \in X.
\end{align*}
In particular, we have
\begin{align*}
	\frac{2}{M} \| x \|_{\mathcal{Q}_0}^2 \leq \| x \|_{L^2}^2 \leq \frac{2}{m} \| x \|_{\mathcal{Q}_0}^2, \hspace{0.5 cm} x \in X. 
\end{align*}
The norm equivalence between $\| \cdot \|_{\mathcal{Q}_0}$ and $\| \cdot \|_{\mathcal{Q}_l}$ for some $l \in (a,b)$ follows immediately from the preceding inequalities:
\begin{align}
	\label{Interface - Norm Equivalence}
	\frac{m}{M} \| x \|_{\mathcal{Q}_0}^2 \leq \frac{m}{2} \| x \|_{L^2}^2 \leq \| x\|_{\mathcal{Q}_l}^2 \leq \frac{M}{2} \| x \|_{L^2}^2 \leq \frac{M}{m} \|x \|_{\mathcal{Q}_0}^2, \hspace{0.5 cm} x \in X.
\end{align}
Note that the constants $\underline{M} = \sqrt{\frac{m}{M}}$ and $\overline{M}= \sqrt{\frac{M}{m}}$ are independent of the interface position $l \in (a,b)$.
\vspace{0.5 cm}\\
We now want to show the following: Let  $0 < \tau < \infty$, and let $l \colon [0, \tau] \to (a,b)$ be continuously differentiable with $\dot{l} \in \mathcal{C}([0,\tau], \mathbb{R})$. Let $(A(t))_{t \in [0, \tau]}$ be a family of port-Hamiltonian operators, where  $A(t) = A_{\mathcal{Q}_{l(t)}}\colon D(A(t))\subset X \to X$ is given by
\begin{align}
	\begin{split}
	D(A(t)) &= \left\{ x \in X \mid \mathcal{Q}_{l(t)} x \in D(\mathcal{J}_{l(t)}), \hspace{0.1 cm} f_I(t) = re_I(t), \hspace{0.1 cm} W_B \begin{bmatrix}
		f_{\partial}\\
		e_{\partial}
		\end{bmatrix} = 0 \right\}, \\
	A(t)x &= \mathcal{J}_{l(t)} (\mathcal{Q}_{l(t)} x)  \\
	&= \begin{bmatrix}
		0 & \mathbf{d}_{l(t)} \\
		- \mathbf{d}_{l(t)}^{\ast} & 0 
	\end{bmatrix} (\mathcal{Q}_{l(t)} x), \hspace{0.5 cm} x \in D(A(t)),
\end{split}
\label{Interface - Operator Family A(t)}
\end{align} 
with $e_I(t) = e_{I, \mathcal{Q}_{l(t)}x}$, $f_I(t) =  f_{I, \mathcal{Q}_{l(t)}x}$. The operators $\mathbf{d}_{l(t)}$ and $\mathbf{d}_{l(t)}^{\ast}$ constituting the matrix differential operator $\mathcal{J}_{l(t)} \colon D(\mathcal{J}_{l(t)}) \subset X \to X$ given by
\begin{align}
	\begin{split}
		D(\mathcal{J}_{l(t)}) &= D(\mathbf{d}_{l(t)}^{\ast}) \times D(\mathbf{d}_{l(t)}), \\
		\mathcal{J}_{l(t)}x &= \begin{bmatrix}
			0 & \mathbf{d}_{l(t)} \\
			- \mathbf{d}_{l(t)}^{\ast} & 0 
		\end{bmatrix} x, \hspace{0.5 cm} x \in D(\mathcal{J}_{l(t)}), 
	\end{split}
\label{Interface - Operator Jl(t)}
\end{align} 
are precisely defined as the operators $\mathbf{d}(t)$ and $\mathbf{d}^{\ast}(t)$ given in \eqref{Interface Paper - Operator d(t)} and \eqref{Interface Paper - Operator d(t)ast}, respectively. Furthermore, the coercive operator $\mathcal{Q}_{l(t)} \in \mathcal{L}(X)$, $t \in [0, \tau]$, is given by 
\begin{align}
	\mathcal{Q}_{l(t)} = c_{l}(\cdot,t) \mathcal{Q}^- + \overline{c}_{l}(\cdot,t) \mathcal{Q}^+.
	\label{Interface - Operator Ql(t)}
\end{align}
We claim that, under certain conditions, this is a stable family on the state space $(X, \| \cdot \|_{\mathcal{Q}_0})$.
\begin{remark}
	We changed the notation of the operators $\mathbf{d}(t)$ and $\mathbf{d}^{\ast}(t)$ to $\mathbf{d}_{l(t)}$ and $\mathbf{d}_{l(t)}^{\ast}$, respectively, to emphasize that the operators explicitly depend on the moving interface position $l(t)$, $t >0$. Recall that the operators $\mathbf{d}(t)$ and $\mathbf{d}^{\ast}(t)$ introduced in Subsection \ref{Subsection Balance Equations of the State Variables - Moving Interface} depend implicitly on the interface position through their definition with respect to the color functions $c_l$ and $\overline{c}_l$ that we took to be part of the extended state variable $\tilde{x} \in \tilde{X} = L^2([a,b], \mathbb{R}^4)$. Since this is not the case for the simplified system discussed throughout Section \ref{Section A Simplified System}, the new notation helps to distinguish the different scenarios.
\end{remark}
If the assumptions \ref{MatrixQ0Assumption1}-\ref{ContractionAssumption} hold, then, due to the norm equivalence  \eqref{Interface - Norm Equivalence} of the norms of the respective energy spaces $(X, \|\cdot\|_{\mathcal{Q}_{l(t)}})$, we can at least guarantee that the family $(A(t))_{t \in [0,\tau]}$ is a family of infinitesimal generators of $C_0$-semigroups on $(X, \| \cdot \|_{\mathcal{Q}_0})$ (or any other state space $(X, \|\cdot\|_{\mathcal{Q}_{\hat{l}}})$ with $\hat{l} \in (a,b)$), which are of type $C_0(\overline{M},0)$. However, as we will see, this is not sufficient to prove the stability of this family. Thus, we need to impose more restrictive assumptions. These are given as follows:
\begin{enumerate}[label =($\tilde{\mathrm{A}}$\arabic*)]
	\item \label{AssumptionP1MatricesQ} The matrix operators $\mathcal{Q}^{\pm} \in \mathcal{C}^1([a,b], \mathbb{R}^{2 \times 2})$ defining the coercive operator $\mathcal{Q}_{l(t)} \in \mathcal{L}(X)$ are diagonal matrix operators, i.e., 
	\begin{align*}
		\mathcal{Q}^{\pm}(z) = \begin{bmatrix}
			Q_{11}^{\pm}(z) & 0 \\
			0 & Q_{22}^{\pm}(z)
		\end{bmatrix}, \hspace{0.5 cm} z \in [a,b],
	\end{align*} 
and they satisfy
\begin{align*}
	\frac{Q_{11}^+}{Q_{11}^-}(0) = 1 \hspace{0.3 cm} \text{and} \hspace{0.3 cm} \frac{Q_{11}^+}{Q_{11}^-}(z) = \frac{Q_{22}^+}{Q_{22}^-}(z), \hspace{0.5 cm} z \in [a,b].
\end{align*}
	\item \label{AssumptionP2ConservationAtInterface} $r =0 $, $\rank(W_B) = 2$, and $W_B \Sigma W_B^{\top} \geq 0$.
\end{enumerate}
Assumption \ref{AssumptionP1MatricesQ} implies that the ratios of the respective physical properties of the subsystems are identical on the spatial domain $[a,b]$, and that the physical properties even coincide in $z = 0$. Further, assumption \ref{AssumptionP2ConservationAtInterface} implies that $A(t)$ generates a contraction semigroup on the respective energy space $(X, \|\cdot\|_{\mathcal{Q}_{l(t)}})$, where there is no power flow at the interface position, see Theorem~\ref{Interface - Theorem A Generates a Contraction Semigroup}.
\vspace{0.5 cm}\\
Since we can prove the stability neither directly nor by means of Theorem \ref{Theorem 5.2.2 in Pazy}, we make use of an extension of a Lyapunov theorem \cite{Datko}.

\begin{theorem}
	\label{Interface - Theorem Stability of A(t)}
	Let $0 < \tau < \infty$, and let $l \colon [0, \tau] \to (a,b)$ be continuously differentiable with $\dot{l} \in \mathcal{C}([0,\tau], \mathbb{R})$. Let the assumptions \ref{AssumptionP1MatricesQ}-\ref{AssumptionP2ConservationAtInterface} hold. Consider the family $(A(t))_{t \in[0,\tau]}$ of port-Hamiltonian operators given by \eqref{Interface - Operator Family A(t)}-\eqref{Interface - Operator Ql(t)} defined on $(X, \|\cdot\|_{\mathcal{Q}_0})$. Then there exists some $\omega > 0$ such that for every $t \in [0, \tau]$ and for every $x \in D(A(t))$,
	\begin{align}
		\langle A(t)x, x \rangle_{\mathcal{Q}_0} \leq \omega \| x\|_{\mathcal{Q}_0}^2.
		\label{Interface - Necessary Condition Stablity of Family}
	\end{align} 
\end{theorem}

\begin{proof}
Let $t \in [0, \tau]$ such that for the interface position we have $a < l(t) < 0$. Recall that in this case, the operator $\mathcal{J}_{l(t)}$ defined in \eqref{Interface - Operator Jl(t)} acts as the operator $P_1 \frac{d}{dz}$ on the subdomains $[a, l(t))$ and $(l(t),b]$, with $P_1$ defined in \eqref{Interface - Matrix P1}. Furthermore, as $A(t)$ generates a contraction semigroup on $(X, \|\cdot\|_{\mathcal{Q}_{l(t)}})$, it is dissipative, see Theorem \ref{Theorem Lumer-Phillips}. More precisely, for all $x \in D(A(t))$ it holds that
\begin{align*}
&\hspace{0.5 cm} \langle A(t)x, x \rangle_{\mathcal{Q}_{l(t)}} \\
&= \frac{1}{2} \int_{a}^{b} x^{\top}(z) \mathcal{Q}_{l(t)}(z) A(t) x(z) \, dz \\ 
	&= \frac{1}{2} \int_{a}^{l(t)} x^{\top}(z)  \mathcal{Q}^-(z) P_1 \frac{d}{dz} (\mathcal{Q}^-(z) x(z)) \, dz + \frac{1}{2} \int_{l(t)}^{0} x^{\top}(z)  \mathcal{Q}^+(z) P_1 \frac{d}{dz} (\mathcal{Q}^+(z) x(z)) \, dz \\
	&\hspace{0.5 cm} + \frac{1}{2} \int_{0}^{b} x^{\top}(z)  \mathcal{Q}^+(z) P_1 \frac{d}{dz} (\mathcal{Q}^+(z) x(z)) \, dz \\
	&\leq 0.
\end{align*}
Thus, we have
\begin{align*}
	&\hspace{0.5 cm} \langle A(t)x, x \rangle_{\mathcal{Q}_0} \\
	&= \frac{1}{2} \int_{a}^{l(t)} x^{\top}(z)  \mathcal{Q}^-(z) P_1 \frac{d}{dz} (\mathcal{Q}^-(z) x(z)) \, dz + \frac{1}{2} \int_{l(t)}^{0} x^{\top}(z)  \mathcal{Q}^-(z) P_1 \frac{d}{dz} (\mathcal{Q}^+(z) x(z)) \, dz \\
	&\hspace{0.5 cm} + \frac{1}{2} \int_{0}^{b} x^{\top}(z)  \mathcal{Q}^+(z) P_1 \frac{d}{dz} (\mathcal{Q}^+(z) x(z)) \, dz \\
	&= \langle A(t)x,x \rangle_{\mathcal{Q}_{l(t)}} \\
	&\hspace{0.5 cm } + \frac{1}{2} \int_{l(t)}^{0} (\mathcal{Q}^-x)^{\top}(z) P_1 \frac{d}{dz} (\mathcal{Q}^+x)(z) - (\mathcal{Q}^+x)^{\top}(z) P_1 \frac{d}{dz} (\mathcal{Q}^+x)(z) \, dz \\
	&\leq \frac{1}{2} \int_{l(t)}^{0} x^{\top}(z) (\mathcal{Q}^-(z) - \mathcal{Q}^+(z)) P_1 \frac{d}{dz} (\mathcal{Q}^+x)(z) \, dz \\
	&= \frac{1}{2} \int_{l(t)}^{0} x^{\top}(z)\left[ (\mathcal{Q}^-(z) - \mathcal{Q}^+(z))P_1 \frac{d}{dz}\mathcal{Q}^+(z)\right] x(z) \, dz \\
	&\hspace{0.5 cm} + \frac{1}{2} \int_{l(t)}^{0}  x^{\top}(z) \left[ (\mathcal{Q}^-(z) - \mathcal{Q}^+(z))P_1 \mathcal{Q}^+(z) \right] \frac{d}{dz} x(z) \, dz.
\end{align*}
We may choose $\omega_1^- > 0$ large enough such that (in the sense of the symmetric part)
\begin{align*}
 (\mathcal{Q}^-(z) - \mathcal{Q}^+(z))P_1 \frac{d}{dz}\mathcal{Q}^+(z) \leq \omega_1^- \mathcal{Q}^-(z), \hspace{0.5 cm} z \in [a,b]. 
\end{align*}
Now, let us define the matrix operator $\tilde{Q} \colon [a,b] \to \mathbb{R}^{2 \times 2}$ by
\begin{align}
	\label{Interface - Theorem Stability of A(t) - Matrix TildeQ}
	\tilde{Q}(z) =  (\mathcal{Q}^-(z) - \mathcal{Q}^+(z))P_1 \mathcal{Q}^+(z) = - \begin{bmatrix}
		0 & Q_{22}^+ (Q_{11}^- - Q_{11}^+) \\
		Q_{11}^+(Q_{22}^- - Q_{22}^+)  & 0 
	\end{bmatrix}(z), \hspace{0.5 cm} z \in [a,b]. 
\end{align}
By assumption \ref{AssumptionP1MatricesQ}, we have $\mathcal{Q}^+, \mathcal{Q}^- \in \mathcal{C}^1([a,b], \mathbb{R}^{2 \times 2}$), and so the matrix operator $\tilde{Q}$ is continuously differentiable as well. Moreover, by \ref{AssumptionP1MatricesQ}, it is easy to check that $\tilde{Q}(z)$ is symmetric for all $z \in [a,b]$. Thus, integration by parts yields that
\begin{align*}
	&\hspace{0.5 cm} \int_{l(t)}^{0} x^{\top}(z) \tilde{Q}(z) \frac{d}{dz} x(z) \, dz  \\
	&=  \frac{1}{2} \big[x^{\top}(z)  \tilde{Q}(z) x(z) \big]_{l(t)}^{0} - \frac{1}{2} \int_{l(t)}^{0} x^{\top}(z) \frac{d}{dz} \tilde{Q}(z) x(z) \, dz.
\end{align*}
Once again, we may choose $\omega_2^- > 0$ large enough such that
\begin{align*}
	- \frac{d}{dz} \tilde{Q}(z) \leq \omega_2^- \mathcal{Q}^-(z), \hspace{0.5 cm} z \in [a,b]. 
\end{align*}
Altogether, we obtain for all $x \in D(A(t))$, 
\begin{align}
	\begin{split}	
	\langle A(t)x, x \rangle_{\mathcal{Q}_0}	&\leq  \frac{\omega_1^-}{2} \int_{l(t)}^{0} x^{\top}(z) \mathcal{Q}^-(z) x(z) \, dz \\
	&\hspace{0.5 cm} +  \frac{1}{4} \big[x^{\top}(z)  \tilde{Q}(z) x(z) \big]_{l(t)}^{0} + \frac{\omega_2^-}{4} \int_{l(t)}^{0}  x^{\top}(z)\mathcal{Q}^-(z)  x(z) \, dz. 
\end{split}
\label{Interface - Theorem Auxiliary - A(t) is Stable - Proof Estimate 1}
\end{align}
The interface relation \eqref{Interface - Interface Passivity Relation} with $r = 0$ together with the assumption that $\mathcal{Q}^{\pm}$ are diagonal matrix operators imply that $x_2(l(t)^+) = x_2(l(t)^-) = 0$. Furthermore, by assumption \ref{AssumptionP1MatricesQ}, $\mathcal{Q}^{\pm}$ coincide in $z = 0$. 
Consequently, it holds that
\begin{align*}
	  \big[x^{\top}(z)  \tilde{Q}(z) x(z) \big]_{l(t)}^{0} &= -  x_1(0)x_2(0) \left[  Q_{22}^+ (Q_{11}^- - Q_{11}^+) + 	Q_{11}^+(Q_{22}^- - Q_{22}^+) \right] (0) \\
	&\hspace{0.5 cm} + x_1(l(t)^+) x_2(l(t)^+) \left[ Q_{22}^+ (Q_{11}^- - Q_{11}^+) + 	Q_{11}^+(Q_{22}^- - Q_{22}^+) \right] (l(t)) \\
	&= 0.
\end{align*}
With the choice $\omega^- = \max\{\omega_1^-, \frac{\omega_2^-}{2}\}$ we deduce from \eqref{Interface - Theorem Auxiliary - A(t) is Stable - Proof Estimate 1} that for all $x \in D(A(t))$, where $t \in [0, \tau]$ is chosen such that $a < l(t) < 0$, it holds that
\begin{align*}
	\langle A(t)x, x \rangle_{\mathcal{Q}_0} \leq \omega^- \|x\|_{\mathcal{Q}_0}^2. 
\end{align*}
Now, let $t \in [0, \tau]$ such that for the interface position we have $0 < l(t) < b$. Analogously to the previous case, we have for all $x \in D(A(t))$, 
\begin{align*}
	\langle A(t)x, x \rangle_{\mathcal{Q}_0} &\leq \frac{1}{2} \int_{0}^{l(t)} x^{\top}(z) \left[ (\mathcal{Q}^+(z) - \mathcal{Q}^-(z))P_1 \frac{d}{dz} \mathcal{Q}^-(z) \right] x(z) \, dz \\
	&\hspace{0.5 cm} + \frac{1}{2} \int_{0}^{l(t)} x^{\top}(z) \left[ (\mathcal{Q}^+(z) - \mathcal{Q}^-(z))P_1 \mathcal{Q}^-(z)  \right] \frac{d}{dz} x(z) \, dz. 
\end{align*}
One can readily see that the same arguments as in the former case apply, and so there exists some $\omega^+ > 0$ such that 
\begin{align*}
	\langle A(t)x, x \rangle_{\mathcal{Q}_0} \leq \omega^+ \|x\|_{\mathcal{Q}_0}^2
\end{align*}
for all $x \in D(A(t))$ with $0 < l(t) < b$. If $l(t) = 0$ for some $t \in [0,\tau]$, then $A(t)$ is dissipative on $(X, \|\cdot\|_{\mathcal{Q}_0})$. Altogether, the choice $\omega = \max\{\omega^-, \omega^+\}$ yields the estimate \eqref{Interface - Necessary Condition Stablity of Family} for all $t \in [0, \tau]$ and for all $x \in D(A(t))$. This proves the claim. 
\end{proof}

\begin{remark}
	We want to emphasize that the choice of the state space $(X, \|\cdot\|_{\mathcal{Q}_0})$ the family $(A(t))_{t \in [0, \tau]}$ is defined on was arbitrary. In fact, we may endow $X$ with any norm of the form $\|\cdot\|_{\mathcal{Q}_{\hat{l}}}$, where $\hat{l} \in (a,b)$, and change the assumption \ref{AssumptionP1MatricesQ} as follows:
	\begin{enumerate}[label = ($\tilde{\mathrm{A}}$1)]
		\item The matrix operators $\mathcal{Q}^{\pm} \in \mathcal{C}^1([a,b], \mathbb{R}^{2 \times 2})$ defining the coercive operator $\mathcal{Q}_{l(t)} \in \mathcal{L}(X)$ are diagonal matrix operators and they satisfy
		\begin{align*}
			\frac{Q_{11}^+}{Q_{11}^-}(\hat{l}) = 1 \hspace{0.3 cm} \text{and} \hspace{0.3 cm} \frac{Q_{11}^+}{Q_{11}^-}(z) = \frac{Q_{22}^+}{Q_{22}^-}(z), \hspace{0.5 cm} z \in [a,b].
		\end{align*}
	\end{enumerate}
In particular, if there is some $\hat{l} \in (a,b)$ such that the matrices $\mathcal{Q}^-(\hat{l})$ and $\mathcal{Q}^+(\hat{l})$ coincide, then it is advisable to consider the family $(A(t))_{t \in [0, \tau]}$ defined on the space $(X, \|\cdot\|_{\mathcal{Q}_{\hat{l}}})$. 
\end{remark}
Note that the inequality \eqref{Interface - Necessary Condition Stablity of Family} implies that for all $t \in [0, \tau]$, the $C_0$-semigroup $(S_t(s))_{s \geq 0}$ on $(X, \|\cdot\|_{\mathcal{Q}_0})$ generated by $A(t)$ is of type $C_0(1, \omega)$: By virtue of \eqref{Interface - Necessary Condition Stablity of Family}, for all $x \in D(A(t))$ it holds that
\begin{align}
	\begin{split}
	 \frac{d}{ds} \| S_t(s)x \|_{\mathcal{Q}_0}^2 &= \frac{d}{ds}  \langle S_t(s)x, S_t(s)x \rangle_{\mathcal{Q}_0} \\
	 &= 2 \langle AS_t(s)x, S_t(s)x \rangle_{\mathcal{Q}_0} \\
	 &\leq 2\omega \|S_t(s)x\|_{\mathcal{Q}_0}^2.
	\end{split}
\label{Interface - Semigroup Estimate}
\end{align} 
Since $s \mapsto e^{2 \omega s}\|x\|_{\mathcal{Q}_0}^2$ is the unique solution of the initial value problem
\begin{align*}
	\frac{d}{ds} y(s) &= 2 \omega y(s), \hspace{0.5 cm} s > 0, \\
	y(0) & =  \|x\|_{\mathcal{Q}_0}^2,
\end{align*}
we infer from \eqref{Interface - Semigroup Estimate} that
\begin{align*}
	\frac{d}{ds} \| S_t(s)x \|_{\mathcal{Q}_0}^2 \leq \frac{d}{ds} \left( e^{\omega s} \|x\|_{\mathcal{Q}_0} \right)^2
\end{align*}
holds for all $x \in D(A(t))$. As $D(A(t))$ is a dense subset of $X$, we conclude that
\begin{align*}
	\|S_t(s)\|_{\mathcal{L}(X, \|\cdot\|_{\mathcal{Q}_0})} \leq e^{\omega s}, \hspace{0.5 cm} s \geq 0. 
\end{align*}
 Recall from Remark \ref{Remark Stable Familiy of Generators} that this yields that the family $(A(t))_{t \in [0, \tau]}$ is stable in the sense of Definition \ref{Definition Stable Family of Generators} with stability constants $M=1$ and $\omega > 0$. Indeed, by the Hille-Yosida Theorem \ref{Theorem Hille-Yosida - C0 Semigroup} we have $(\omega, \infty) \subset \rho(A(t))$ for all $t \in [0, \tau]$. Furthermore, for every finite sequence $0 \leq t_1 \leq t_2 \leq \ldots \leq t_k \leq \tau$, $k \in \mathbb{N}$, and for all $\lambda > \omega$ it holds that
\begin{align*}
	\left\| \prod_{j = 1}^{k} R(\lambda, A(t_j)) \right\|_{ \mathcal{L}(X, \|\cdot\|_{\mathcal{Q}_0})} &\leq \prod_{j=1}^{k} \left\| R(\lambda, A(t_j)) \right\|_{ \mathcal{L}(X, \|\cdot\|_{\mathcal{Q}_0})} \\
	&\leq \prod_{j=1}^{k} \frac{1}{\lambda - \omega} \\
	&= \frac{1}{(\lambda - \omega)^k}.
\end{align*}
Together with Theorem \ref{Interface - Theorem Stability of A(t)} we can summarize these facts as follows:
\begin{corollary}
	\label{Interface - Corollary Stability of A(t)}
	Let $0 < \tau < \infty$, and let $l \colon [0, \tau] \to (a,b)$ be continuously differentiable with $\dot{l} \in \mathcal{C}([0,\tau], \mathbb{R})$. Let the assumptions \ref{AssumptionP1MatricesQ}-\ref{AssumptionP2ConservationAtInterface} hold. Then the family $(A(t))_{t \in[0,\tau]}$ of port-Hamiltonian operators given by \eqref{Interface - Operator Family A(t)}-\eqref{Interface - Operator Ql(t)} is stable on $(X, \|\cdot\|_{\mathcal{Q}_0})$.
	\QEDB
\end{corollary}

Lastly, we wish to prove the same statement as in Corollary \ref{Interface - Corollary Stability of A(t)} for the system where we allow for a power flow at the moving interface. We impose the following assumptions:
\begin{enumerate}[label =($\tilde{\mathrm{B}}$\arabic*)]
	\item \label{AssumptionB1MatricesQ} The matrix operators $\mathcal{Q}^{\pm} \in \mathcal{C}^1([a,b], \mathbb{R}^{2 \times 2})$ defining the coercive operator $\mathcal{Q}_{l(t)} \in \mathcal{L}(X)$ are diagonal matrix operators and they satisfy
	\begin{align*}
		\frac{Q_{11}^+}{Q_{11}^-}(0) = 1 \hspace{0.3 cm} \text{and} \hspace{0.3 cm} \frac{Q_{11}^+}{Q_{11}^-}(z) = \frac{Q_{22}^+}{Q_{22}^-}(z), \hspace{0.5 cm} z \in [a,b].
	\end{align*}
	\item \label{AssumptionB2ExpStability} $r >0 $, $\rank(W_B) = 2$, and $W_B \Sigma W_B^{\top} > 0$.
\end{enumerate}
Assumption \ref{AssumptionB2ExpStability} implies that $A(t)$ generates an exponentially stable contraction semigroup on the respective energy space $(X, \|\cdot\|_{\mathcal{Q}_{l(t)}})$, see Theorem \ref{Interface - Theorem 2 Exponential Stability}.
\vspace{0.5 cm}\\
Let $t \in [0, \tau]$ such that for the interface position we have $a < l(t) < 0$. From assumption \ref{AssumptionB2ExpStability} and from the proof of Theorem \ref{Interface - Theorem 2 Exponential Stability} we deduce that for all $x \in D(A(t))$ it holds that
\begin{align*}
	\langle A(t)x, x \rangle_{\mathcal{Q}_{l(t)}} &=\frac{1}{2} \langle e_{\partial}, f_{\partial} \rangle_2 - \frac{1}{2} e_I(t)f_I(t) \\
	&\leq - k \|\trace_{l(t)}(\mathcal{Q}_{l(t)}x)\|_{2}^2 - \frac{1}{2} e_I(t)f_I(t)
\end{align*}
for some $k > 0$ small enough. With the same arguments as in the proof of Theorem \ref{Interface - Theorem Stability of A(t)}, the following estimate holds for all $x \in D(A(t))$:
\begin{align}
	\begin{split}	
		\langle A(t)x, x \rangle_{\mathcal{Q}_0}	&\leq  - k \|\trace_{l(t)}(\mathcal{Q}_{l(t)}x)\|_{2}^2  - \frac{1}{2} e_I(t)f_I(t) + \frac{\omega_1^-}{2} \int_{l(t)}^{0} x^{\top}(z) \mathcal{Q}^-(z) x(z) \, dz \\
		&\hspace{0.5 cm} +  \frac{1}{4} \big[x^{\top}(z)  \tilde{Q}(z) x(z) \big]_{l(t)}^{0} + \frac{\omega_2^-}{4} \int_{l(t)}^{0}  x^{\top}(z) \mathcal{Q}^-(z)  x(z) \, dz, 
	\end{split}
	\label{Interface - A(t) is Stable - Proof Estimate 1}
\end{align}
with $\omega_1^- >0$ and $\omega_2^- >0$ large enough, and with $\tilde{Q}$ defined in \eqref{Interface - Theorem Stability of A(t) - Matrix TildeQ}. In the following, we estimate the expression
\begin{align*}
	- f_I(t) e_I(t) + \frac{1}{2}  \big[x^{\top}(z)  \tilde{Q}(z) x(z) \big]_{l(t)}^{0}.
\end{align*}
By assumption \ref{AssumptionB1MatricesQ}, it holds that
\begin{align*}
	\frac{1}{2}  \big[x^{\top}(z)  \tilde{Q}(z) x(z) \big]_{l(t)}^{0} =  \frac{1}{2} x_1(l(t)^+) x_2(l(t)^+) \left[ Q_{22}^+ (Q_{11}^- - Q_{11}^+) + 	Q_{11}^+(Q_{22}^- - Q_{22}^+) \right] (l(t)).
\end{align*}
By definition of the interface port variables \eqref{Interface - Simplified Continuity Equation}-\eqref{Interface - Simplified Balance Equation}, we have 
\begin{align*}
	&\hspace{0.5 cm} - f_I(t)e_I(t) + \frac{1}{2} x_1(l(t)^+) x_2(l(t)^+) \left[ Q_{22}^+ (Q_{11}^- - Q_{11}^+) + 	Q_{11}^+(Q_{22}^- - Q_{22}^+) \right] (l(t)) \color{white}\text{abcdefghii} \\
	&=  f_I(t) \left[ Q_{11}^+(l(t)) x_1(l(t)^+) - Q_{11}^-(l(t)) x_1(l(t)^-) \right]  \\
	&\hspace{0.5 cm} + \frac{1}{2}  x_1(l(t)^+) x_2(l(t)^+) \left[ (Q_{11}^-Q_{22}^+)(l(t))  -2 (Q_{11}^+ Q_{22}^+)(l(t)) + 	(Q_{11}^+Q_{22}^-)(l(t))  \right] 
\end{align*}
\begin{align*}
	&=  f_I(t) \left[ Q_{11}^+(l(t)) x_1(l(t)^+) - Q_{11}^-(l(t)) x_1(l(t)^-) \right] - Q_{11}^+(l(t)) x_1(l(t)^+) \underbrace{Q_{22}^+(l(t)) x_2(l(t)^+)}_{= f_I(t)} \\
	&\hspace{0.5 cm}  + \frac{1}{2}  x_1(l(t)^+) x_2(l(t)^+) \left[ (Q_{11}^+Q_{22}^-)(l(t)) + 	(Q_{11}^-Q_{22}^+)(l(t)) \right] \\
	&= - f_I(t) Q_{11}^-(l(t)) x_1(l(t)^-) + \frac{1}{2}  x_1(l(t)^+) x_2(l(t)^+) \left[ (Q_{11}^-Q_{22}^+)(l(t)) + 	(Q_{11}^+Q_{22}^-)(l(t)) \right] \\
	&= - f_I(t) Q_{11}^-(l(t)) x_1(l(t)^-) + \frac{1}{2} f_I(t) Q_{11}^-(l(t)) x_1(l(t)^+) + \frac{1}{2} f_I(t) \frac{Q_{22}^-}{Q_{22}^+}(l(t)) Q_{11}^+(l(t) x_1(l(t)^+) \\
	&= f_I(t) \left[   \frac{1}{2} \left( Q_{11}^-(l(t)) + \frac{Q_{22}^-}{Q_{22}^+}(l(t)) Q_{11}^+(l(t)) \right) x_1(l(t)^+) - Q_{11}^-(l(t)) x_1(l(t)^-) \right].
\end{align*}
Let us define 
\begin{align}
	\label{Interface - Eta-}
	\eta^-(t) := \frac{1}{2} \left( \frac{Q_{11}^-}{Q_{11}^+}(l(t)) + \frac{Q_{22}^-}{Q_{22}^+}(l(t)) \right) \in \left[ \frac{m}{M}, \frac{M}{m} \right]. 
\end{align}
Now, by virtue of the passivity relation \eqref{Interface - Interface Passivity Relation}, we compute
\begin{align*}
	&\hspace{0.5 cm} f_I(t) \left[   \frac{1}{2} \left( Q_{11}^-(l(t)) + \frac{Q_{22}^-}{Q_{22}^+}(l(t)) Q_{11}^+(l(t)) \right) x_1(l(t)^+) - Q_{11}^-(l(t)) x_1(l(t)^-) \right] \\
	&= r e_I(t) \left[ \eta^-(t) Q_{11}^+(l(t)) x_1(l(t)^+) - Q_{11}^-(l(t)) x_1(l(t)^-) \right] \\
	&= - r \left( Q_{11}^+(l(t)) x_1(l(t)^+) - Q_{11}^-(l(t)) x_1(l(t)^-) \right)  \left[ \eta^-(t) Q_{11}^+(l(t)) x_1(l(t)^+) - Q_{11}^-(l(t)) x_1(l(t)^-) \right] \\
	&= -r \left[ \eta^-(t) (Q_{11}^+(l(t)) x_1(l(t)^+))^2 - (\eta^-(t) +  1) Q_{11}^+(l(t)) x_1(l(t)^+) Q_{11}^-(l(t)) x_1(l(t)^-)  \right. \\
	&\hspace{1.2 cm} \left. + (Q_{11}^-(l(t)) x_1(l(t)^-))^2 \right]. 
\end{align*}
Put $e_1^+(z) = Q_{11}^+(z) x_1(z)$, $z \in (l(t),b]$, and $e_1^-(z) = Q_{11}^-(z) x_1(z)$, $z \in [a,l(t))$. By virtue of the inequality $xy \leq \frac{1}{2}(x^2 + y^2)$ for all $x,y \in \mathbb{R}$, we find
\begin{align*}
	&\hspace{0.5 cm} 	r \left[ - \eta^-(t) e_1^+(l(t))^2 - e_1^-(l(t))^2 + (\eta^-(t) +1) e_1^+(l(t)) e_1^-(l(t)) \right] \\
	&\leq r \left[  - \eta^-(t) e_1^+(l(t))^2 - e_1^-(l(t))^2 + \frac{\eta^-(t) +1}{2}e_1^+(l(t))^2 + \frac{\eta^-(t) +1}{2}e_1^-(l(t))^2 \right] \\
	&= r \left[ \frac{1 - \eta^-(t)}{2} e_1^+(l(t))^2 + \frac{\eta^-(t) -1}{2}e_1^-(l(t))^2 \right].
\end{align*}
Altogehter, we have
\begin{align*}
-f_I(t) e_I(t) + \frac{1}{2} \big[ x^{\top}(z) \tilde{Q}(z) x(z) \big]_{l(t)}^{0} = r \left[ \frac{1 - \eta^-(t)}{2} e_1^+(l(t))^2 + \frac{\eta^-(t) -1}{2}e_1^-(l(t))^2 \right].
\end{align*}
We wish to find some $\omega_3^- >0$ large enough that guarantees that (cf. \eqref{Interface - A(t) is Stable - Proof Estimate 1})
\begin{align*}
	- k \|\trace_{l(t)}(\mathcal{Q}_{l(t)}x)\|_{2}^2 - \frac{1}{2} f_I(t) e_I(t) + \frac{1}{4}  \big[x^{\top}(z)  \tilde{Q}(z) x(z) \big]_{l(t)}^{0}  \stackrel{!}{\leq} \omega_3^- \| x|_{\mathcal{Q}_0}^2.
\end{align*}

Define 
\begin{align*}
	e_2^-(z) &= Q_{22}^-(z)x_2(z), \hspace{0.5 cm} z \in [a,l(t)], \\
	e_2^+(z) &= Q_{22}^+(z) x_2(z), \hspace{0.5 cm} z \in [l(t),b]. 
\end{align*}
Using the definition of $e_1^{\pm}$ and $e_2^{\pm}$, we have
\begin{align*}
2 \|x\|_{\mathcal{Q}_0}^2 &= \int_{a}^{l(t)} Q_{11}^-(z) x_1(z)^2 + Q_{22}^-(z) x_2(z)^2 \, dz +  \int_{l(t)}^{0} Q_{11}^-(z) x_1(z)^2 + Q_{22}^-(z) x_2(z)^2 \, dz \\
	&\hspace{0.5 cm} + \int_{0}^{b} Q_{11}^+(z) x_1(z)^2 + Q_{22}^+(z) x_2(z)^2 \, dz \\
	&= \int_{a}^{l(t)} \frac{1}{Q_{11}^-}(z) e_1^-(z)^2 + \frac{1}{Q_{22}^-}(z) e_2^-(z)^2 \, dz + \int_{l(t)}^{0} \frac{Q_{11}^-}{(Q_{11}^+)^2}(z) e_1^+(z)^2 + \frac{Q_{22}^-}{(Q_{22}^+)^2}(z) e_2^+(z)^2 \, dz \\
	&\hspace{0.5 cm} + \int_{0}^{b} \frac{1}{Q_{11}^+}(z) e_1^+(z)^2 + \frac{1}{Q_{22}^+}(z) e_2^+(z)(z)^2 \, dz \\
	&\geq \delta \left( \int_{a}^{l(t)}  e_1^-(z)^2 +  e_2^-(z)^2 \, dz + \int_{l(t)}^{b} e_1^+(z)^2 + e_2^+(z)^2 \, dz \right)
\end{align*}
for some $\delta >0$ small enough. Hence, it amounts to show that there is some $\tilde{\omega}_3^- > 0$ such that for all $x \in D(A(t))$, 
\begin{align}
	\begin{split}
		&\hspace{0.5 cm} - k \|\trace_{l(t)}(\mathcal{Q}_{l(t)}x)\|_{2}^2 + \frac{r}{4} \left[	(1 - \eta^-(t)) e_1^+(l(t))^2 + (\eta^-(t) - 1) e_1^-(l(t))^2 \right] \\ &\stackrel{!}{\leq} \tilde{\omega}_3^- \left( \int_{a}^{l(t)}  e_1^-(z)^2 +  e_2^-(z)^2 \, dz + \int_{l(t)}^{b} e_1^+(z)^2 + e_2^+(z)^2 \, dz \right).
	\end{split}
	\label{Interface - A(t) is Stable - Proof Estimate 2}
\end{align}
Note that the continuity condition \eqref{Interface - Simplified Continuity Equation} implies that
\begin{align*}
	&\hspace{0.5 cm} 	(1 - \eta^-(t)) e_1^+(l(t))^2 + (\eta^-(t) - 1) e_1^-(l(t))^2 \\
	&= 	(1 - \eta^-(t)) [e_1^+(l(t))^2 + e_2^+(l(t))^2] + (\eta^-(t) - 1) [e_1^-(l(t))^2 + e_2^-(l(t))^2].
\end{align*}
Thus, by using the definition of $e_1^{\pm}$ and $e_2^{\pm}$ again, \eqref{Interface - A(t) is Stable - Proof Estimate 2} is equivalent to
\begin{align*}
	&\hspace{0.5 cm} - k \left( \| (\mathcal{Q}^- x)(a) \|_2^2 + \| (\mathcal{Q}^+ x)(b) \|_2^2 \right) \\
	&\hspace{0.5 cm} + \frac{r}{4} \left[ (\eta^-(t) -1) \| \mathcal{Q}^-(l(t)) x(l(t)^-) \|_2^2 + (1 - \eta^-(t)) \|\mathcal{Q}^+(l(t)) x(l(t)^+) \|_2^2 \right] \\
	&\stackrel{!}{\leq} \tilde{\omega}_3^- \left( \int_{a}^{l(t)}  \|(\mathcal{Q}^-x)(z)\|_2^2  \, dz + \int_{l(t)}^{b} \|(\mathcal{Q}^+x)(z)\|_2^2 \, dz \right) \\
	&= \tilde{\omega}_3^- \|\mathcal{Q}_{l(t)}x\|_{L^2}^2.
\end{align*}
Similarly, for $t \in [0, \tau]$ such that $0 < l(t) < b$, the following estimate has to hold for all $ x \in D(A(t))$:
\begin{align*}
	&\hspace{0.5 cm} - k \left( \| (\mathcal{Q}^- x)(a) \|_2^2 + \| (\mathcal{Q}^+ x)(b) \|_2^2 \right)  \\
	&\hspace{0.5 cm} + \frac{r}{4} \left[ (1- \eta^+(t)) \| \mathcal{Q}^-(l(t)) x(l(t)^-) \|_2^2 + (\eta^+(t) -1) \|\mathcal{Q}^+(l(t)) x(l(t)^+) \|_2^2 \right] \\
	&\stackrel{!}{\leq} \tilde{\omega}_3^+ \|\mathcal{Q}_{l(t)}x\|_{L^2}^2,
\end{align*}
where 
\begin{align}
	\label{Interface - Eta+}
	\eta^+(t) = \frac{1}{2} \left( \frac{Q_{11}^+}{Q_{11}^-}(l(t)) + \frac{Q_{22}^+}{Q_{22}^-}(l(t)) \right) \in \left[\frac{m}{M}, \frac{M}{m} \right]. 
\end{align}
This can be summarized as follows:
\begin{corollary}
		\label{Interface - Corollary 2 Stability of A(t)}
	Let $0 < \tau < \infty$, and let $l \colon [0, \tau] \to (a,b)$ be continuously differentiable with $\dot{l} \in \mathcal{C}([0,\tau], \mathbb{R})$. Let the assumptions \ref{AssumptionB1MatricesQ}-\ref{AssumptionB2ExpStability} hold. Let $k > 0$ be chosen such that \eqref{Interface - Theorem Exponential Stability - Inequality} is satisfied. Then the family $(A(t))_{t \in[0,\tau]}$ of port-Hamiltonian operators given by \eqref{Interface - Operator Family A(t)}-\eqref{Interface - Operator Ql(t)} is stable on $(X, \|\cdot\|_{\mathcal{Q}_0})$, if the following assertions are true:
	\begin{enumerate}[label = (\roman*)]
		\item There exists some $\tilde{\omega}^-  >0$ such that for all $t \in [0, \tau]$ with $a < l(t) < 0$ and for all $x \in D(A(t))$ we have
		\begin{align*}
			&\hspace{0.5 cm} - k \left( \| (\mathcal{Q}^- x)(a) \|_2^2 + \| (\mathcal{Q}^+ x)(b) \|_2^2 \right) \\
			&\hspace{0.5 cm} + \frac{r}{4} \left[ (\eta^-(t) -1) \| \mathcal{Q}^-(l(t)) x(l(t)^-) \|_2^2 + (1 - \eta^-(t)) \|\mathcal{Q}^+(l(t)) x(l(t)^+) \|_2^2 \right] \\
			&\leq \tilde{\omega}^- \|\mathcal{Q}_{l(t)}x\|_{L^2}^2,
		\end{align*}
	where $\eta^-(t)$ is defined in \eqref{Interface - Eta-}.
	\item  There exists some $\tilde{\omega}^+  >0$ such that for all $t \in [0, \tau]$ with $0 < l(t) < b$ and for all $x \in D(A(t))$ we have
	\begin{align*}
		&\hspace{0.5 cm} - k \left( \| (\mathcal{Q}^- x)(a) \|_2^2 + \| (\mathcal{Q}^+ x)(b) \|_2^2 \right) \\
		&\hspace{0.5 cm} + \frac{r}{4} \left[ (1 - \eta^+(t)) \| \mathcal{Q}^-(l(t)) x(l(t)^-) \|_2^2 + (\eta^+(t) -1) \|\mathcal{Q}^+(l(t)) x(l(t)^+) \|_2^2 \right] \\
		&\leq \tilde{\omega}^+ \|\mathcal{Q}_{l(t)}x\|_{L^2}^2,
	\end{align*}
where $\eta^+(t)$ is defined in \eqref{Interface - Eta+}.
	\end{enumerate}
	\QEDB
\end{corollary}
As we have discussed in Section \ref{Section Evolution Equations}, the stability of the family $(A(t))_{t \in [0, \tau]}$ is not sufficient for well-posedness of the corresponding evolution problem. However, we will not further analyze this system description and conclude this segment with the preceding stability results.
\vspace{0.5 cm}\\
In this section, we have studied a simplified model closely related to the one formulated in \cite{Diagne}. We have defined its underlying Dirac structure, and we have worked out passivity conditions that guarantee that the associated port-Hamiltonian operator generates a contraction semigroup on the respective energy space. Those conditions have been defined with respect to the interface and boundary port variables. Moreover, we have presented criteria for exponential stability of the $C_0$-semigroup generated by the port-Hamiltonian operator, and we have specified the associated Hilbert space adjoint as well as the resolvent. All these results hold in case of a fixed interface. At the end of this section, we have defined a family of port-Hamiltonian operators that keeps track of the position of the moving interface. We have shown under which assumptions this family is stable, and  we have learned that the power flow at the position of the moving interface is a critical factor for the stability.

\chapter{Conclusions and Future Work}
\label{Chapter Conclusion}
In this thesis, we have provided a functional-analytic foundation for the port-Hamiltonian formulation of two systems of two conservation laws that are coupled by a moving interface worked out in \cite{Diagne}. Furthermore, we have analyzed a simplified version of the proposed model by means of the powerful framework of strongly continuous semigroups. The results presented in this work can be considered as a firm basis for the further analysis of this quite new system class.
\vspace{0.5 cm}\\ 
In Chapter \ref{Chapter Finite-dimensional Port-Hamiltonian Systems} we have introduced the concept of port-based modeling and the necessary terminology related to this framework by reference to finite-dimensional systems. We have studied some basic properties of the underlying Dirac structure, i.e., the power-conserving interconnection structure, as well as the constituting port variables of a multi-physics system, and depicted how the dynamics of a port-Hamiltonian system can be defined with respect to the Dirac structure. Subsequently, in Chapter \ref{Chapter Some Background} we have briefly recalled the definition of the Hilbert space adjoint of a linear operator $A \colon D(A) \subset X \to X$ defined on some Hilbert space $X$, and we have highlighted the difference between the Hilbert space adjoint and the formal adjoint of a matrix differential operator. This allowed us to define formally skew-symmetric operators, which play a fundamental role in the framework of infinite-dimensional port-Hamiltonian systems. Afterwards, we have reviewed the basics of infinite-dimensional systems theory. For the analysis we have decided for the theory of semigroups of linear operators, so we gave a concise overview of this framework in Section \ref{Section Strongly Continuous Semigroups}. Our emphasis was to recall sufficient conditions for well-posedness of abstract control systems (see Section \ref{Section Linear Control Systems}) and of evolution problems (see Section \ref{Section Evolution Equations}). In Chapter \ref{Chapter Infinite-dimensional Port-Hamiltonian Systems} we have extended the port-Hamiltonian framework to infinite-dimensional systems, where we have restricted ourselves to systems with a 1-dimensional spatial domain. We have investigated port-Hamiltonian systems that are associated with formally skew-symmetric matrix differential operators, and we have highlighted the main differences compared to its finite-dimensional counterpart. The most important difference is that one has to take the energy flow at the spatial boundary into account. This is accomplished by introducing boundary port variables, which has allowed us to define the infinite-dimensional power-conserving interconnection structure, called the Stokes-Dirac structure, associated with formally skew-symmetric operators. 
\vspace{0.5 cm}\\
The preceding chapters have set the foundation for the main part of this thesis, namely the port-Hamiltonian formulation and the analysis of two first-order port-Hamiltonian systems that are coupled by a moving interface, which we have dealt with in Chapter \ref{Chapter Boundary Port-Hamiltonian Systems with a Moving Interface}. In Section \ref{Section Two Port-Hamiltonian Systems Coupled by an Interface} we have started by formulating the aforementioned systems that are coupled by an interface fixed at $z = 0$ as a single boundary port-Hamiltonian system, and we have imposed interface conditions on the co-energy variables. Furthermore, we have introduced so-called color functions, which are characteristic functions of the complementary subdomains that keep track of the position of the interface. We have taken them to constitute the extended state space $\tilde{X} = L^2([a,b], \mathbb{R}^4)$. Having regard to the interface relations, we have defined balance equations of the state variables on the composed domain, and we have defined a generalized Hamiltonian system of off these equations. This has allowed us to define a Dirac structure for the fixed interface scenario with respect to a formally skew-symmetric differential operator, see Proposition \ref{Interface Paper - Proposition Dirac Structure DI}. 
\vspace{0.5 cm}\\
In Section \ref{Section Port-Hamiltonian Systems Coupled through a Moving Interface} we have repeated the exposition from Section \ref{Section Two Port-Hamiltonian Systems Coupled by an Interface} in case of a moving interface position. As opposed to the fixed interface case, the proper formulation as a port-Hamiltonian system is fraught with problems. We have stated the main issues in \ref{Problem1MovingInterface}-\ref{Problem3MovingInterface}. Instead of tackling these problems directly, we have studied a simplified model that is closely related to the original system in Section \ref{Section A Simplified System}, where we have merely assumed that the color functions are not part of the state variables anymore. Furthermore, we have assumed that the variational derivative of the Hamiltonian functional is linear. For the analysis, we have started again by assuming that the interface is fixed. After specifying the Dirac structure of this system (see Corollary \ref{Interface - Corollary Dirac Structure}), we have analyzed the port-Hamiltonian operator associated with this system. We have shown that if we impose passivity conditions at the boundary and at the interface, then the port-Hamiltonian operator is the infinitesimal generator of a contraction semigroup. In particular, the coupling of two first-order port-Hamiltonian systems coupled by a fixed interface is well-defined, see Theorem \ref{Interface - Theorem A Generates a Contraction Semigroup}. We have illustrated this result by reference to the coupling of two transmission lines with distinct physical properties, see Example \ref{Interface - Example Coupled Transmission Lines}. Furthermore, we have been able to show that under more restrictive assumptions regarding the boundary conditions, the port-Hamiltonian operator generates an exponentially stable contraction semigroup, see Theorem \ref{Interface - Theorem 2 Exponential Stability}. After specifying the Hilbert space adjoint as well as the resolvent of the port-Hamiltonian operator in Subsection \ref{Subsection The Adjoint Operator} and Subsection \ref{Subsection Resolvent}, respectively, we have studied a family of port-Hamiltonian operators that encompasses the position of the moving interface. We have provided criteria that ensure that this family is stable (see Subsection \ref{Subsection Stability of the Family of Infinitesimal Generators}), which is an essential property for the proof of well-posedness of the evolution problem associated with this family.
\vspace{0.5 cm}\\ 
One of these criteria is that we have assumed that there is no power flow at the interface (see Corollary~\ref{Interface - Corollary Stability of A(t)}), which is quite restrictive. If we allow for a power flow at the interface, there is room for improvement concerning the stability result stated in this thesis (see Corollary~\ref{Interface - Corollary 2 Stability of A(t)}). Furthermore, there are several problems remaining open. For instance, one needs to check for the sufficient conditions for well-posedness apart from stability. Moreover, we have not discussed the input operator associated with the simplified system in case of a moving interface at all. Hence, plenty of control related problems have to be clarified in the future. Additionally, for a complete port-Hamiltonian formulation, we need to find the output port variable conjugated to the velocity of the interface (the input). This is the major topic for further research. After these problems have been successfully tackled for the simplified system, one may eventually deal with the original system \eqref{Interface Paper - Augmented System wrt Jl} accompanied by the delicate problems  \ref{Problem1MovingInterface}-\ref{Problem3MovingInterface} again (see pages \pageref{Interface Paper - Augmented System wrt Jl}-\pageref{Problem3MovingInterface}). 
\vspace{0.5 cm}\\
Besides the demanding task of proving well-posedness of the boundary port-Hamiltonian interface control system, as one may call the model proposed in \cite{Diagne}, there are countless possibilities to generalize the (simplified) system description, e.g., admitting a broader class of color functions, deploying a passive network at the interface position, etc. In conclusion, the port-Hamiltonian model of two systems coupled by a moving interface offers a vast field of potential research.

\newpage


\bibliography{references}

\end{document}